\def\switch{} 
\def\key{7}
\def\picture{1} 
\def\full{1} 
\def\margin{0} 

\ifnum0\key=7
\documentclass[11pt,a4paper]{report}
\usepackage{setspace}
\else
\documentclass[11pt,twoside,a4paper]{report}
\fi
\usepackage[english]{babel}
\usepackage[nodayofweek]{datetime}
\usepackage[utf8]{inputenc}
\usepackage{times}
\usepackage[T1]{fontenc}
\usepackage{graphicx}
\usepackage{amscd,amsmath,amsfonts,amstext,amssymb,amsbsy,amsopn,amsthm,eucal}
\usepackage{txfonts}

\usepackage{ifpdf}

\usepackage{fancyhdr}
\usepackage{hyperref}

\usepackage{makeidx}
\makeindex
\ifnum0\key=7

\makeatletter

 \ifnum0\full=1
 
 \renewcommand*\@makechapterhead[1]{%
  {\parindent \z@ \raggedright \normalfont
    \vspace*{10pt}
    \Huge\bfseries
    \ifnum \c@secnumdepth >\m@ne
         Chapter \thechapter\\
    \fi
    #1\par\nobreak
    \vskip 60\p@
  }}
 
 \fi

    {\if@twocolumn
       \@restonecolfalse
     \else
       \@restonecoltrue
     \fi
          \refstepcounter{chapter}
          \twocolumn[\ifnum\full=1 \else \vspace*{2\topskip} \fi
          \phantomsection
          \addcontentsline{toc}{chapter}{\protect\numberline{\thechapter}\indexname}
                     \@makechapterhead{\indexname}]%
          \chaptermark{\indexname}
   \thispagestyle{headings}\parindent\z@
   \fancyhead[L]{\bfseries Index }
   \parskip\z@ \@plus .3\p@\relax
   \let\item\@idxitem}
   {\if@restonecol\onecolumn\else\clearpage\fi}
   
\addto\captionsenglish{%
  \renewcommand{\contentsname}%
    {Table of Contents}%
}
\renewcommand\tableofcontents { \chapter*{\contentsname}
  \addcontentsline{toc}{chapter}{\contentsname}\thispagestyle{myheadings}\markright{}\fancyhead[L]{\bfseries\contentsname}
  \@starttoc{toc} \clearpage \fancyhead[L]{\bfseries\leftmark} }

\makeatother

\fi

\ifnum0\key=7
\hypersetup{
  pdftitle={PhD Thesis in mathematics},
  pdfauthor={Daniele Valtorta},
  pdfsubject={On the p-Laplace operator on Riemannian manifolds},
  pdfkeywords={PhD Thesis, thesis, Setti, 7i, Naber, Cheeger, J, Giona, Veronelli, Lucio, Luciano, Mari, Milan, 2012},
  pdfpagelayout=SinglePage,
  pdfpagemode=UseOutlines,
  colorlinks,
  linkcolor=[rgb]{0,0,0.7},
  urlcolor=[rgb]{0,0,0.4},
  citecolor=[rgb]{0.4,0.1,0}
}

\else
\usepackage[nottoc]{tocbibind}
\hypersetup{
  pdftitle={PhD Thesis in mathematics},
  pdfauthor={Daniele Valtorta},
  pdfsubject={On the p-Laplace operator on Riemannian manifolds, PRINTED VERSION, PDF-optimized version available on arXiv},
  pdfkeywords={2012, Cheeger, Mari, Milan, Naber, PhD Thesis, Setti, Veronelli},
  pdfpagelayout=TwoPageRight,
  pdfpagemode=UseOutlines,
  colorlinks,
  linkcolor=[rgb]{0,0,0},
  urlcolor=[rgb]{0,0,0},
  citecolor=[rgb]{0,0,0}
}
\fi

\newcommand{\N}{\mathbb{N}}
\newcommand{\R}{\mathbb{R}}

\newcommand{\norm}[1]{\left\|#1\right\|}
\newcommand{\ps}[2]{\left\langle#1\middle\vert#2\right\rangle}
\newcommand{\pps}[2]{\langle\left\langle#1\middle\vert#2\right\rangle\rangle}
\newcommand{\psp}{\ps{\cdot}{\cdot}}
\newcommand{\ppsp}{\pps{\cdot}{\cdot}}
\newcommand{\pst}[3]{#2\ton{#1,#3}}
\newcommand{\cur}[1]{\left\{#1\right\}}
\newcommand{\qua}[1]{\left[#1\right]}
\newcommand{\ton}[1]{\left(#1\right)}
\newcommand{\abs}[1]{\left|#1\right|}
\newcommand{\A}{\mathcal{A}}

\newcommand{\D}{\mathcal{D}}
\newcommand{\E}{\mathcal{E}}
\newcommand{\K}{\mathcal{K}}

\newcommand{\F}{\mathcal{F}}
\renewcommand{\L}{\mathcal{L}}
\newcommand{\Si}{\mathcal{S}}
\newcommand{\Cr}{\mathcal{C}}

\renewcommand{\O}{\mathcal{O}}
\newcommand{\Ha}{\mathcal{H}}
\newcommand{\Mi}{\mathcal{Mi}}
\newcommand{\W}{\mathcal{W}}
\newcommand{\Ric}{\operatorname{Ric}}
\newcommand{\pdw}{\dot w ^{(p-1)}}
\newcommand{\pw}{w^{(p-1)}}
\newcommand{\pG}{G^{(p-1 )}}
\newcommand{\pu}{ u^{(p-1)}}
\newcommand{\sinp}{\operatorname{sin_p}}
\newcommand{\cosp}{\operatorname{cos_p}}
\newcommand{\Vol}{\text{Vol}}

\newcommand{\cp}{\operatorname{Cap_p}}

\newcommand{\dive}{\operatorname{div}}
\newcommand{\ka}{Khas'minskii }
\newcommand{\ho}{H\"older }
\newcommand{\pf}{Pr\"{u}fer }

\newcommand{\loc}{\mathrm{loc}}
\newcommand{\ra}{\rightarrow}
\newcommand{\oF}{\mathcal{F}}

\newcommand{\eps}{\varepsilon}
\newcommand{\esssup}{\mathrm{esssup}}

\newcommand{\disp}{\displaystyle}

\newcommand{\cc}{C^{\infty}_c(\Omega)}
\newcommand{\wup}{W^{1,p}(\Omega)}

\newcommand{\wupz}{W^{1,p}_0(\Omega)}

\newcommand{\Kpt}{\mathcal W_{\psi,\theta}}

\newcommand{\V}{\operatorname V}

\newcommand{\T}{\mathcal{T}}

\newcommand{\cC}{\mathcal{C}}
\newcommand{\cN}{\mathcal{N}}
\newcommand{\cW}{\mathcal{W}}

\newcommand{\dR}{\mathbb{R}}

\newcommand{\cS}{\mathcal{S}}

\numberwithin{equation}{section}
\newtheorem{prop}{Proposition}[section]
\newtheorem{teo}[prop]{Theorem}
\newtheorem{deph}[prop]{Definition}
\newtheorem{lemma}[prop]{Lemma}
\newtheorem{oss}[prop]{Remark}
\newtheorem{ex}[prop]{Example}
\newtheorem{rem}[prop]{Remark}
\newtheorem{remark}[prop]{Remark}
\newtheorem{cor}[prop]{Corollary}

\newtheorem{theorem}[prop]{Theorem}
\newtheorem{definition}[prop]{Definition}
\newtheorem{proposition}[prop]{Proposition}
\newtheorem{example}[prop]{Example}
\newtheorem{corollary}[prop]{Corollary}

\renewcommand{\chaptermark}[1]{\markboth{#1}{}}

\ifnum0\key=7
\else
\fi

\fancypagestyle{plain}{
  \fancyhead{}
   
}

\ifnum0\key=7

     \def\odd{0pt}
     \def\eve{0pt} 
 
 \ifnum0\full=1
 
     \def\vof{-52pt}
     \def\hea{14pt}
     \def\texh{771pt}
     \def\hof{-52pt}
     \def\texw{557pt}

 \else
  
   \ifnum0\margin=1
     \def\hof{-52pt}
   \else
     \def\hof{13pt}
   \fi
     
     \def\vof{-21pt}
     \def\hea{15.2pt}
     \def\texh{662pt}
     \def\texw{426pt}
 
 \fi

\else

     \def\hof{13pt}
     \def\eve{-20pt}
     \def\odd{20pt}
     \def\vof{-20pt}
     \def\hea{15.2pt}
     \def\texh{660pt}
     \def\texw{426pt}
     
\fi

\begin{document}

\begin{titlepage}
\ifnum0\key=7
\phantomsection \addcontentsline{toc}{chapter}{On the p-Laplace operator on Riemannian manifolds}
\else
\setcounter{page}{-3}
\fi
\begin{center}
 \normalsize{UNIVERSIT\`A DEGLI STUDI DI MILANO \ifnum0\key=7 \ \includegraphics[height=4mm]{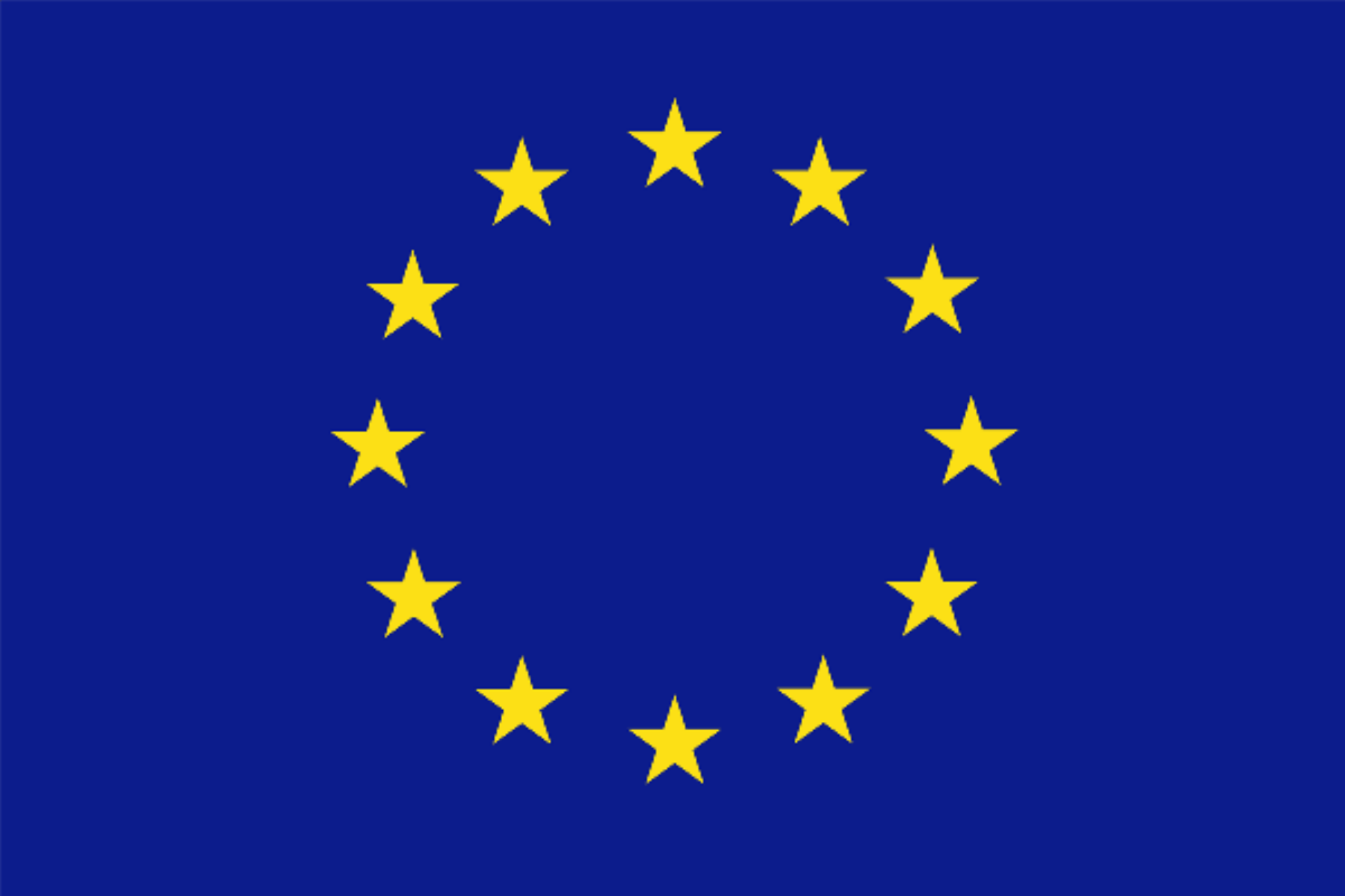} \ \ \else \ - \ \ \fi DEPARTMENT OF MATHEMATICS}\\
\end{center}

\begin{center}
 \vspace{1cm}
{\Huge 

PhD Thesis in Mathematics\\
}
\ \\

\Huge{\textbf{On the p-Laplace operator on\\ Riemannian manifolds}}

\vspace{1cm}

\vspace{1.7cm}
\includegraphics[width=.5\textwidth]{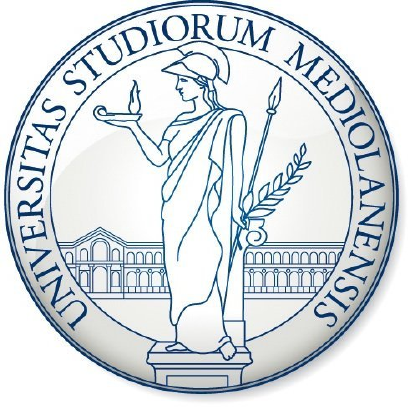}
\vspace{1cm}
\end{center}

\vspace*{\fill}

\begin{flushright}

Date: \today \\
\vspace{0.2cm}
Author: Daniele Valtorta\\
ID number R08513\\
\ifnum0\key=7 e-mail: \href{mailto:danielevaltorta@gmail.com}{danielevaltorta@gmail.com}\\ \fi
\vspace{0.2cm}
Supervisor: Alberto Giulio Setti\ifnum\key=7\\
\vspace{0.5cm}
This thesis is available on \href{http://arxiv.org/find/math/1/au:+Valtorta_D/0/1/0/all/0/1}{arXiv}
\fi

\end{flushright}
\end{titlepage}

\ifnum0\key=7
\setcounter{page}{2} 
\thispagestyle{myheadings} \markright{}
\vspace*{\fill}
\begin{center}
 \includegraphics[height=200pt]{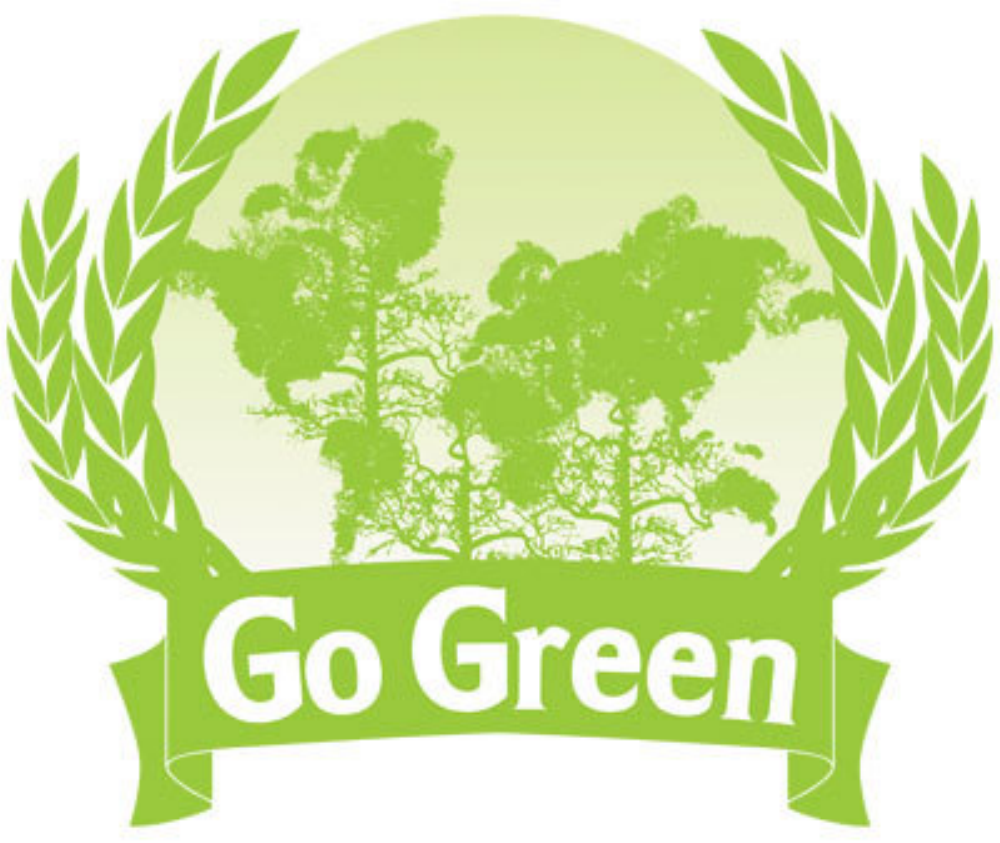}
 \vspace{3cm}

 \Large {\fontfamily{augie}\selectfont
Please do not print this thesis, not just because it's a waste of resources, but also because I spent a lot of time to optimizing it for PDF view (especially the links in the \hyperref[ch_bib]{Bibliography}).
\begin{flushright}
 - \href{http://www.danielevaltorta.eu}{valto} - \hspace{50pt} \ 
\end{flushright}
}
\end{center}
\vspace*{\fill}

\clearpage

\else
\thispagestyle{empty}
\cleardoublepage
\fi

\tableofcontents
\ifnum0\key=7
\else
\thispagestyle{empty}
\clearpage
\setcounter{page}{1}
\fi

 \ifnum0\switch=1

\else
 \chapter{Introduction}
 \ifnum0\key=7
  \thispagestyle{headings}
 \fi
\fi
This work deals with different aspects of the $p$-Laplace equation on Riemannian manifolds. Particular emphasis will be given to some new results obtained in this field, and on the techniques employed to prove them. Some of these techniques are of particular relevance, especially because of their versatile nature and mathematical simplicity.

\section{Potential theoretic aspects}
The first section of this thesis addresses questions related to potential-theoretic aspects of the $p$-Laplacian on manifolds.

A noncompact Riemannian manifold $M$ is said to be parabolic if every bounded above subharmonic function is constant, or equivalently if all compact sets have zero capacity. In symbols, a noncompact Riemannian manifold is parabolic if every function $u\in W^{1,2}_{\loc}(M)$ essentially bounded above and such that $ \Delta (u) = \dive\ton{\nabla u} \geq 0
$ in the weak sense is constant. If $M$ is not parabolic, by definition it is said to be hyperbolic. Parabolicity is an important and deeply studied property of Riemannian manifolds. We cite the nice survey article \cite{gri} for a detailed account of the subject and its connections to stochastic properties of the underlying manifold. On this matter we only recall that parabolicity is equivalent to the recurrence of Brownian motion.

Parabolicity can also be characterized by the existence of special exhaustion function. In the case where $p=2$, a way to characterize this property is the \ka condition.
\begin{prop}
Let $M$ be a complete Riemannian manifold. $M$ is parabolic if and only if, for every compact set $K$ with smooth boundary and nonempty interior, there exists a continuous function $f:M\setminus K\to \R$ which satisfies
\begin{gather}
 f|_{\partial K} =0\notag\, ,\\
 \Delta f \leq 0 \ \ \ \text{on} \ \ M\setminus K\, ,\\
 \lim_{x\to \infty} f(x)=\infty \notag\, .
\end{gather}
\end{prop}
A stronger statement is the existence of an Evans potential, which is essentially similar to the \ka condition where superharmonicity of the potential function $f$ is replaced by harmonicity.
\begin{prop}
Let $M$ be a complete Riemannian manifold. $M$ is parabolic if and only if, for every compact set $K$ with smooth boundary and nonempty interior, there exists a continuous function $f:M\setminus K\to \R$ which satisfies
\begin{gather}
 f|_{\partial K} =0 \notag\, ,\\
 \Delta f =0 \ \ \ \text{on} \ \ M\setminus K\, ,\\
 \lim_{x\to \infty} f(x)=\infty  \notag \, .
\end{gather}
\end{prop}
This second statement has been proved by M. Nakai (see \cite{naka} and also \cite{SN}) \footnote{Actually, the authors prove in detail the existence of Evans potential on parabolic surfaces. However, the case of a generic manifold is similar.}.

Even though there is no counterpart of Brownian motion if $p\neq 2$, it is possible to generalize the concept of parabolicity for a generic $p$ in a very natural way. Indeed, we say that a manifold $M$ is $p$-parabolic if and only if every function $u\in W^{1,p}(M)$ essentially bounded above and such that
\begin{gather}
 \Delta_p(u)=\dive\ton{\abs{\nabla u }^{p-2} \nabla u} \geq 0
\end{gather}
in the weak sense is constant.

In the first part of this thesis we prove that the \ka condition continues to hold for any $p\in (1,\infty)$. 
The proof is based on properties of solutions to the obstacle problem
\ifnum0\switch=1
.
\else
 (see sections \ref{sec_obs} and \ref{sec_tech}).
\fi

The existence of Evans potentials is more difficult to prove when $p\neq 2$. Indeed, the nonlinearity of the $p$-Laplacian makes it hard to generalize the technique described by Nakai in the linear case. However, some partial results can be easily obtained on special classes of manifolds.

The existence results for \ka and Evans potentials are the core of the article \cite{mathz}.

In collaboration with L. Mari, we were able to adapt (and actually improve and simplify) the techniques used in \cite{mathz}, obtaining a proof of the \ka characterization valid for a wider class of operators and applicable also to stochastic completeness, not just to parabolicity. These results
\ifnum0\switch=1
are the core of the article \cite{lucio}.
\else
are described in section \ref{sec_kagen} and are the core of the article \cite{lucio}.
\fi

Parabolicity has also a very strong link with volume growth of geodesic balls and integrals. For instance, V. Gol'dshtein and M. Troyanov in \cite{GolTro} proved that a manifold $M$ is $p$-parabolic if and only if every vector field $X \in L^{p/(p-1)}(M)$ with $\dive (X)\in L^1(M)$ satisfies
\begin{gather}
 \int_M \dive \ton{X} dV =0 \, .
\end{gather}
This equivalence is known as Kelvin-Nevanlinna-Royden condition. Using Evans potentials and other special exhaustion functions,
\ifnum0\switch=1
\else
 in section \ref{sec_J} 
\fi
we discuss how it is possible to improve this result by relaxing the integrability conditions on $X$, and describe some applications of this new result. These extensions have been obtained in collaboration with G. Veronelli and are published in \cite{VV}.

\section{First Eigenvalue of the p-Laplacian}
In this thesis we also study estimates on the first positive eigenvalue of the $p$-Laplacian on a compact Riemannian manifold 
under assumptions on the Ricci curvature and diameter.

Given a compact manifold $M$, we say that $\lambda$ is an eigenvalue for the $p$-Laplacian if there exists a nonzero function $u\in W^{1,p}(M)$ such that
\begin{gather}
 \Delta_p u = -\lambda \abs u ^{p-2} u
\end{gather}
in the weak sense. If $M$ has boundary, we assume Neumann boundary conditions on $u$, i.e., if $\hat n$ is the outer normal vector to $\partial M$,
\begin{gather}
 \ps{\nabla u}{\hat n}=0\, .
\end{gather}
Since $M$ is compact, a simple application of the divergence theorem forces $\lambda \geq 0$. Following the standard convention, we denote by $\lambda_{1,p}$ the first \textit{positive} eigenvalue of the $p$-Laplacian on $M$.

In the linear case, i.e., if $p=2$, eigenvalue estimates, gap theorems and relative rigidity results are well studied topics in Riemannian geometry. Perhaps two of the most famous results in these fields are the Zhong-Yang sharp estimates and the Lichnerowicz-Obata theorem.
\begin{theorem}(see \cite{ZY}). 
Let $M$ be a compact Riemannian manifold with nonnegative Ricci curvature and diameter $d$, and possibly with convex boundary. Define $\lambda_{1,2}$ to be the first positive eigenvalue of the Laplace operator on $M$, then the following sharp estimates holds:
\begin{gather}
 \lambda_{1,2} \geq \ton{\frac{\pi}{d}}^2\, .
\end{gather}
\end{theorem}
Later on, F. Hang and X. Wang in \cite{hang} proved that equality in this estimate holds only if $M$ is a one-dimensional manifold.

\begin{theorem}(see \cite{lich} and \cite{obata}).
 Let $M$ be an $n$-dimensional Riemannian manifold with Ricci curvature bounded from below by $n-1$, and define $\lambda_{1,2}$ to be the smallest positive eigenvalue of the Laplace operator on $M$. Then
 \begin{gather}
  \lambda_{1,2}\geq n
 \end{gather}
and equality can be achieved if and only if $M$ is isometric to the standard Riemannian sphere of radius $1$.
\end{theorem}
Both these results have been extensively studied and improved over the years. For instance, D. Bakry and Z. Qian in \cite{new} studied eigenvalue estimates on weighted Riemannian manifolds, replacing the lower bound on the Ricci curvature with a bound on the Bakry-Emery Ricci curvature. It is also worth mentioning that in \cite{wu}, \cite{pete} and \cite{aubry}, the authors study some rigidity results assuming a positive Ricci curvature lower bound. 

Moreover, A. Matei in \cite{matei} extended the Obata theorem to a generic $p\in(1,\infty)$, and obtained other not sharp estimates with zero or negative lower bound on the curvature. Her proof is based on the celebrated Levy-Gromov's isoperimetric inequality.

Using the same isoperimetric technique, in this work we discuss a very simple rigidity condition valid when $\Ric\geq R>0$. However, this technique looses some of its strength when the lower bound on Ricci is nonpositive. In order to deal with this cases, we use the gradient comparison technique introduced in \cite{kro}, and later developed in \cite{new}, and adapt it to the nonlinear setting. 

The driving idea behind this method is to find the right one dimensional model equation to describe the behaviour of an eigenfunction $u$ on $M$. In order to extend this technique to the nonlinear case, we introduce the linearized $p$-Laplace operator and use it to prove a generalized Bochner formula. 

As a result, we obtain the following sharp estimate:
\begin{theorem} \cite[main Theorem]{svelto}
Let $M$ be a compact Riemannian manifold with nonnegative Ricci curvature and diameter $d$, and possibly with convex boundary. Define $\lambda_{1,p}$ to be the first positive eigenvalue of the $p$-Laplace operator on $M$, then, for any $p\in (1,\infty)$, the following sharp estimates holds:
\begin{gather}
 \frac{\lambda_{1,p}}{p-1} \geq \ton{\frac{\pi_p}{d}}^p\, ,
\end{gather} 
where equality holds only if $M$ is a one dimensional manifold.
\end{theorem}

With a similar method, we also study the case where the Ricci curvature lower bound is a ne\-ga\-tive constant. These results have been obtained in collaboration with A. Naber and are published in \cite{svelto,nava}.

\section[Monotonicity, stratification and critical sets]{Monotonicity formulas, quantitative stratification and estimates on the critical sets}
The last part of the thesis is basically the application of two interesting and versatile techniques studied in collaboration with J. Cheeger and A. Naber (see \cite{chnava}).

The first one is the celebrated Almgren's frequency formula for harmonic functions, which was originally introduced in \cite{alm} for functions defined in $\R^n$.
\begin{definition}
 Given a function $u:B_1(0)\to \R$, $u\in W^{1,2}(B_1(0))$ define
 \begin{gather}
  N(r)= \frac{r\int_{B_r(0)} \abs{\nabla u}^2 \ dV}{\int_{\partial B_r(0)} u^2 \ dS}
 \end{gather}
 for those $r$ such that the denominator does not vanish.
\end{definition}
It is easily proved that, if $u$ is a nonzero harmonic function, then $N(r)$ is monotone increasing with respect to $r$, and the monotonicity is strict unless $u$ is a homogeneous harmonic polynomial. This monotonicity gives doubling conditions on the weighted integral
\begin{gather}\label{eq_dou}
 r^{1-n} \int_{\partial B_r(0)} u^2 \ dS\, ,
\end{gather}
which in turn imply growth estimates for the function $u$ and, as a corollary, the unique continuation property for harmonic functions. Adapting the definition of the frequency, it is also possible to prove growth estimates for solutions to more general elliptic equations.

One could hope that, with a suitably modified definition, it should be possible to prove a monotonicity formula for $p$-harmonic functions, and get unique continuation for the $p$-Laplacian as a corollary. However, this problem is more difficult than expected and still remains unanswered. Note that the subject is a very active field of research, we quote for instance the recent work \cite{uc} by S. Grandlund and N. Marola.

Using the monotonicity of the frequency formula, one can also estimate the size of the critical set of harmonic functions (or solutions to more general elliptic PDEs).

Let $\Cr(u)$ be the critical set of a harmonic function $u$, and let $\Si(u)$ be the singular set of $u$, i.e., $\Si(u)=\Cr(u)\cap u^{-1}(0)$. The unique continuation principle for harmonic functions is equivalent to the fact that $\Cr(u)$ has empty interior. So it is reasonable to think that, refining the estimates and using the monotonicity of the frequency function, it should be possible to have better control on the dimension and Hausdorff measure of the critical set.

Very recent developments in this field have been made by Q. Han, R. Hardt and F. Lin. In \cite{hanhardtlin} (see also \cite{hanlin}), the authors use the doubling conditions on \eqref{eq_dou} (in particular some compactness properties implied by them) to prove quantitative estimates on $\Si(u)$ of the form
\begin{gather}
 \Ha^{n-2}\ton{\Si(u)\cap B_{1/2}(0)}\leq C(n,N(1))\, ,
\end{gather}
where $\Ha^{n-2}$ is the $n-2$-dimensional Hausdorff measure. The result was then extended to critical sets in \cite{HLrank}. With similar, although technically more complicated, techniques they recover a similar result also for solutions to more general PDEs.

In this thesis we obtain Minkowski estimates on the critical set of harmonic functions, and extend the results also to solutions to more general elliptic equations. The proof of these estimates is based on the quantitative stratification technique, which was introduced by Cheeger and Naber to study the structure of singularities of harmonic mappings between Riemannian manifolds (see \cite{ChNa1,ChNa2}).

In the case of harmonic functions, using suitable blowups and rescaling, it is well-known that locally around each point every harmonic function is close in some sense to a homogeneous harmonic polynomial. This suggests a stratification of the critical set according to how close the function $u$ is to a $k$-symmetric homogeneous harmonic polynomial at different scales. Exploiting the monotonicity of the frequency, we estimate in a quantitative way the $k+\epsilon$ Minkowski content of each stratum. As a Corollary we also get the following effective bounds on the whole critical set.
\begin{theorem}
 Let $u:B_1(0)\subset \R^n\to \R$ be a harmonic function with $  \frac{\int_{B_1}\abs{\nabla u}dV}{\int_{\partial B_1} u^2 dS} \leq \Lambda$. Then for every $\eta>0$, there exists a constant $C(n,\Lambda,\eta)$ such that
\begin{gather}
 \Vol (\T_r(\Cr(u))\cap B_{1/2})\leq C r^{2-\eta}\, ,
\end{gather}
where $\T_r(A)$ is the tubular neighborhood of radius $r$ of the set $A$.
\end{theorem}

With some extra technicalities, we obtain similar results for solutions to elliptic equations of the form
\begin{gather}
 \L(u)=\partial_i\ton{a^{ij}\partial_{j} u} + b^i \partial_i u =0 \, ,
\end{gather}
where we assume Lipschitz regularity on $a^{ij}$ and boundedness on $b^i$.

A more detailed description of the techniques employed in this thesis and more bibliographic references will be presented in the introductions to each chapter of this work.

\paragraph{REVISION}
\textit{After this thesis was defended, an anonymous referee brought to our attention the article \cite{HLrank}, which was not cited in the previous version of this thesis, and of the article \cite{chnava}. It seemed fare to add this citation to the thesis. Note that the main contribution of this chapter is not the estimates on the $n-2$ Hausdorff measure of the critical set (which have already been proved in \cite{HLrank}), but the quantitative stratification and the Minkowski estimates on the critical set. The Hausdorff estimates are just secondary results.}

\ifnum0\switch=1
\else

\vspace{0.5cm}

\section{Collaborations}
Even though it is self-evident from the bibliographic references, I would like to remark that some of the results presented in this thesis have been obtained in collaboration with other colleagues\ifnum0\key=7 \ and friends\fi. First of all, I have always benefited from the invaluable help and support of my advisor, prof. Alberto Setti\ifnum0\key=7 , and of some other of my colleagues (for example Dr. Debora Impera, prof. Stefano Pigola and Dr. Michele Rimoldi)\fi.

Moreover \footnote{following the order given by the sections of this thesis}, the study of the generalized \ka condition has been carried out in collaboration with \ifnum0\key=7\else Dr. \fi Luciano Mari in \cite{lucio}, the study on the Stokes Theorem with \ifnum0\key=7 Jona \else Prof. Giona \fi Veronelli in \cite{VV}, the sharp estimates on the first eigenvalue of the $p$-Laplacian with \ifnum0\key=7 \else Prof. \fi Aaron Naber in \cite{nava}, and the study on the critical set of harmonic and elliptic functions with \ifnum0\key=7 \else Prof. \fi Jeff Cheeger and \ifnum0\key=7 \else Prof. \fi Aaron Naber in \cite{chnava}.

\section{Editorial note}
Very often in my experience, one starts to tackle a mathematical problem in a simple situation, where the technical details are easier to deal with, and then moves to a more general situation. Usually this second case does not need many new ideas to be implemented, but it certainly requires more caution and attention to sometimes subtle technical details. In this thesis, in order to separate the ideas from the technicalities, I will often discuss some simple case before dealing with the most general case, trying to present almost self-contained arguments in both cases. Even though this choice will make the thesis longer without adding any real mathematical content to it, I hope it will also make it easier to study.

\fi

 \chapter{Potential theoretic aspects}
 \ifnum0\key=7
  \thispagestyle{headings}
 \fi

As explained in the introduction, the main objective of this chapter is to prove the \ka characterization for $p$-parabolic manifolds and to discuss some applications of this characterization to improve Stokes' type theorems. Moreover, we will also see how to adapt the technique used for $p$-parabolicity in order to prove other \ka type characterizations for more general operators.

\section{\ka condition for p-Laplacians}\label{sec_khas_easy}
In order to fix some ideas before dealing with complicated technical details, we will first state and prove the \ka condition only for $p$-Laplacians (or equivalently for $p$-parabolic manifolds). Hereafter, we assume that $M$ is a complete smooth noncompact Riemannian manifold without boundary, and we denote its metric tensor by $g_{ij}$ and its volume form by $dV$.

Here we recall some standard definitions related to the $p$-Laplacian and state without proof some of their basic properties. A complete and detailed reference for all the properties mentioned here is, among others, \cite{HKM} \footnote{even though this book is set in Euclidean environment, it is easily seen that all properties of local nature in $\R^n$ are easily extended to Riemannian manifolds}.

\begin{deph}
\index{p-Laplacian @$p$-Laplacian!p-harmonic@$p$-harmonic}
 A function $h$ defined on a domain $\Omega\subset M$ is said to be $p$-harmonic if $h\in W^{1,p}_{loc}(\Omega)$ and $\Delta_p h = 0$ in the weak sense, i.e.,
\begin{gather*}
 \int_\Omega \abs{\nabla h}^{p-2}\ps{\nabla h }{\nabla \phi }dV=0 \ \ \ \forall \ \phi \in C^\infty_c(\Omega)\, .
\end{gather*}
The space of $p$-harmonic functions on an open set $\Omega$ is denoted by $H_p(\Omega)$.
\end{deph}
A standard result is that, for every $K\Subset \Omega$, every $p$-harmonic function $h$ belongs to $C^{1,\alpha}(K)$ for some positive $\alpha$ (see (\cite{reg}).
\begin{deph}
\index{p-Laplacian @$p$-Laplacian!p-harmonic@$p$-harmonic!p-sub/superersolution@$p$-sub/supersolution}
A function $s\in W^{1,p}_{loc}(\Omega)$ is a $p$-supersolution if $\Delta_p h \leq 0$ in the weak sense, i.e.
\begin{gather*}
 \int_{\Omega} \abs{\nabla s}^{p-2}\ps{\nabla s}{\nabla \phi }dV\geq 0 \ \ \ \forall \ \phi \in C^\infty_c(\Omega), \phi \geq 0\, .
\end{gather*}
A function $s:\Omega\to \R\cup\{+\infty\}$ (not everywhere infinite) is said to be $p$-superharmonic if it is lower semicontinuous and for every open $D\Subset \Omega$ and every $h\in H_p(D)\cap C(\overline D)$
\begin{gather}\label{eq_comp}
h|_{\partial D}\leq s|_{\partial D} \ \Longrightarrow h\leq s \quad \text{  on all  } D \, .
\end{gather}
The space of $p$-superharmonic functions is denoted by $S_p(\Omega)$.
If $-s$ is $p$-superharmonic (or respec\-tively a $p$-supersolution), then $s$ is $p$-subharmonic (or respectively a $p$-subsolution).
\end{deph}
\begin{remark}
\rm Although technically $p$-supersolutions and $p$-superharmonic functions do not coincide, there is a strong relation between the two concepts. Indeed, every $p$-supersolution has a lower-semicontinuous representative in $W^{1,p}_{loc}(\Omega)$ which is $p$-superharmonic, and if $s\in W^{1,p}_{loc}(\Omega)$ is $p$-superharmonic then it is also a $p$-supersolution.
\end{remark}

Note that the comparison principle in equation \eqref{eq_comp} is valid also if the comparison is made in the $W^{1,p}$ sense. Actually, this comparison principle is one of the defining property of $p$-super and subsolutions.
\begin{proposition}[Comparison principle]
\index{comparison principle}
 Let $s,w\in W^{1,p}(\Omega)$ be respectively a $p$-supersolution and a $p$-subsolution. If $\min\{s-w,0\}\in W^{1,p}_0(\Omega)$, then $s\geq w$ a.e. in $\Omega$.
\end{proposition}
\begin{proof}
 The proof follows easily from the definition of $p$-super and subsolutions. For a detailed reference, see for example \cite[Lemma 3.18]{HKM}.
\end{proof}

A simple argument shows that the family of $p$-supersolutions is closed relative to the $\min$ operation, see \cite[Theorem 3.23]{HKM} for the details.
\begin{proposition}
 Let $s,w\in W^{1,p}(\Omega)$ $p$-supersolution. Then also $\min\{s,w\}$ belongs to $W^{1,p}$ and is a $p$-supersolution.
\end{proposition}

We recall that the family of $p$-superharmonic functions is closed also under right-directed convergence, i.e. if $s_n$ is an increasing sequence of $p-$superharmonic functions with pointwise limit $s$, then either $s=\infty$ everywhere or $s$ is $p$-superharmonic.
By a truncation argument, it is easily seen that that every $p$-superharmonic function is the pointwise limit of an increasing sequence of $p$-supersolutions.

Here we briefly recall the concept of $p$-capacity of a compact set.
\index{p-Laplacian @$p$-Laplacian!p-capacity@$p$-capacity}

\begin{deph}
 Let $\Omega$ be a domain in $M$ and let $K\Subset \Omega$. Define the $p$-capacity of the couple $(K,\Omega)$ by
\begin{gather*}
 \operatorname{Cap_p}(K,\Omega)\equiv \inf_{\varphi \in C^\infty_c(\Omega), \ \varphi(K)=1 }\int_\Omega \abs{\nabla \varphi}^p dV\, .
\end{gather*}
If $\Omega=M$, then we set $\cp(K,M)\equiv \cp(K)$.
\end{deph}
Note that by a standard density argument for Sobolev spaces, the definition of $p$-capacity is unchanged if we take the infimum over all functions $\varphi$ such that $\varphi -\psi\in W^{1,p}_0(\Omega\setminus K)$, where $\psi$ is a cutoff function with support in $\Omega$ and equal to $1$ on $K$.

It is well-known that $p$-harmonic function can be characterized also as the minimizers of the $p$-Dirichlet integral in the class of functions with fixed boundary value. In other words, given $h\in W^{1,p}(\Omega)$, for any other $f\in W^{1,p}(\Omega)$ such that $f-h\in W^{1,p}_0(\Omega)$,
\begin{gather}
 \int_{\Omega} \abs{\nabla f}^p dV \geq \int_{\Omega} \abs{\nabla h}^p dV\, .
\end{gather}
This minimizing property allows us to define the $p$-potential of $(K,\Omega)$ as follows.
\index{p-Laplacian @$p$-Laplacian!p-potential@$p$-potential}
\begin{prop}
 Given $K\subset \Omega\subset M$ with $\Omega$ bounded and $K$ compact, and given $\psi\in C^{\infty}_c(\Omega)$ s.t. $\psi\vert_K=1$, there exists a unique function
\begin{gather*}
 h\in W^{1,p}(\Omega\setminus K) \, , \quad h-\psi\in W^{1,p}_0(\Omega\setminus K)\, .
\end{gather*}
This function is a minimizer for the $p$-capacity, explicitly
\begin{gather*}
 \cp(K,\Omega)=\int_\Omega \abs{\nabla h}^p dV\, .
\end{gather*}
For this reason, we define $h$ to be the $p$-potential of the couple $(K,\Omega)$.\\
Note that if  $\Omega$ is not bounded, it is still possible to define its $p$-potential by a standard exhaustion argument. 
\end{prop}

As observed before, the $p$-potential is a function in $C^1(\Omega\setminus K)$. However, continuity of the $p$-potential up to the boundary of $\Omega\setminus K$ is not a trivial property.
\index{p-Laplacian @$p$-Laplacian!p-regular set@$p$-regular set}
\begin{definition}
 A couple $(K,\Omega)$ is said to be $p$-regular if its $p$-potential is continuous up to $\overline{\Omega \setminus K}$.
\end{definition}
$p$-regularity depends strongly on the geometry of $\Omega$ and $K$, and there exist at least two characte\-rizations of this property: the Wiener criterion and the barrier condition \footnote{As references for these two criteria, we cite \cite{HKM}, \cite{KM} and \cite{BB}, a very recent article which deals with $p$-harmonicity and $p$-regularity on general metric spaces.} . For the aim of this section, we simply recall that $p$-regularity is a local property and that if $\Omega\setminus K$ has smooth boundary, then it is $p$-regular. Further remarks on this issue will be given in the following section. 

Using only the definition, it is easy to prove the following elementary estimates on the capacity.
\begin{lemma}\label{lemma_cap}
Let $K_1\subset K_2\subset \Omega_1\subset \Omega_2\subset M$. Then
\begin{gather*}
 \cp(K_2,\Omega_1)\geq \cp(K_1,\Omega_1) \quad \text{and} \quad \cp(K_2,\Omega_1)\geq \cp(K_2,\Omega_2)\, .
\end{gather*}
Moreover, if $h$ is the $p$-potential of the couple $(K,\Omega)$, for $0\leq t<s\leq 1$ we have
\begin{gather*}
 \cp\ton{\{h\leq s\},\{h<t\}}=\frac{\cp(K,\Omega)}{(s-t)^{p-1}}\, .
\end{gather*}

\end{lemma}
\begin{proof}
The proof of these estimates follows quite easily from the definitions, and it can be found in \cite[Propositions 3.6, 3.7, 3.8]{holop}, or in \cite[Section 2]{HKM}.
\end{proof}

As mentioned in the introduction, we recall the definition of $p$-parabolicity.
\begin{deph}\label{deph_para_easy}
 A Riemannian manifold $M$ is $p$-parabolic if and only if every bounded above $p$-subharmonic function is constant \footnote{or equivalently if any bounded below $p$-superharmonic function is constant}.
\end{deph}
This property if often referred to as the Liouville property for $\Delta_p$, or simply the $p$-Liouville property. There are many equivalent definitions of $p$-parabolicity. For example, $p$-parabolicity is related to the $p$-capacity of compact sets.
\begin{definition}\label{deph_para_2}
\index{p-Laplacian @$p$-Laplacian!p-parabolic@$p$-parabolic}
A Riemannian manifold $M$ is $p$-parabolic if and only if for every $K\Subset M$, $\cp(K)=0$. Equivalently, $M$ is $p$-parabolic if and only if there exists a compact set $\bar K$ with nonempty interior such that $\cp(\bar K)=0$.
\end{definition}
\begin{proposition}
 The two definitions of $p$-parabolicity \ref{deph_para_easy} and \ref{deph_para_2} are equivalent.
\end{proposition}
\begin{proof}
It is easily seen that, given a compact set $K$, its $p$-harmonic potential $h$ (extended to $1$ on $K$) is a bounded $p$-superharmonic function. This implies that $-h$ is a bounded above subharmonic function, and since it has to be constant, we can conclude immediately that $\cp(K)=0$.

To prove the reverse implication, suppose that there exists a $p$-superharmonic function $s$ bounded from below. Without loss of generality, we can suppose that $\operatorname{essinf}_M(s)=0$, and that $s\geq 1$ on an open relatively compact set $K$.

By the comparison principle, the $p$-harmonic potential of $K$ is $h\leq s$, so that $h$ cannot be constant, and $\cp(K)>0$. 
\end{proof}

\subsection{Obstacle problem}\label{sec_obs}
In this section, we report some technical results that will be essential in our proof of the reverse \ka condition, related in particular to the obstacle problem. 
\begin{deph}
\index{obstacle problem}
 Let $M$ be a Riemannian manifold and $\Omega\subset M$ be a bounded domain. Given $\theta\in W^{1,p}(\Omega)$ and $\psi:\Omega\to [-\infty,\infty]$, we define the convex set
\begin{gather*}
 K_{\theta,\psi}=\{\varphi\in W^{1,p}(\Omega) \ s.t. \ \varphi\geq \psi \ a.e. \ \ \ \varphi-\theta\in W^{1,p}_0(\Omega)\}\, .
\end{gather*}
We say that $s\in K_{\theta,\psi}$ solves the obstacle problem relative to the $p$-Laplacian if for any $\varphi\in K_{\theta,\psi}$
\begin{gather*}
 \int_\Omega \ps{\abs {\nabla s }^{p-2}\nabla s}{\nabla \varphi - \nabla s}dV \geq 0\, .
\end{gather*}
\end{deph}

It is evident that the function $\theta$ defines in the Sobolev sense the boundary values of the solution $s$, while $\psi$ plays the role of obstacle, i.e., the solution $s$ must be $\geq \psi$ at least almost everywhere. Note that if we set $\psi\equiv -\infty$, the obstacle problem turns into the classical Dirichlet problem. Moreover, it follows easily from the definition that every solution to the obstacle problem is a $p$-supersolution on its domain.

For our purposes the two functions $\theta$ and $\psi$ will always coincide, so hereafter we will write for simplicity $K_{\psi,\psi}= K_{\psi}$.\\
The obstacle problem is a very important tool in nonlinear potential theory, and with the deve\-lopment of calculus on metric spaces it has been studied also in this very general setting. In the following we cite some results relative to this problem and its solvability.
\begin{prop}\label{prop_obs}
 If $\Omega$ is a bounded domain in a Riemannian manifold $M$, the obstacle prob\-lem $K_{\theta,\psi}$ has always a unique (up to a.e. equivalence) solution if $K_{\theta,\psi}$ is not empty \footnote{which is always the case if $\theta=\psi$}. Moreover the lower semicontinuous regularization of $s$ coincides a.e. with $s$ and it is the smallest $p$-superharmonic function in $K_{\theta,\psi}$, and also the function in $K_{\theta,\psi}$ with smallest $p$-Dirichlet integral. If the obstacle $\psi$ is continuous in $\Omega$, then $s\in C(\Omega)$.
\end{prop}
For more detailed propositions and proofs, see \cite{HKM}.

A corollary to the previous Proposition is a minimizing property of $p$-supersolutions. It is well-known that $p$-harmonic functions (which are the unique solutions to the standard $p$-Dirichlet problem \footnote{or equivalently, solutions to the obstacle problems with obstacle $\psi=.\infty$}) minimize the $p$-Dirichlet energy among all functions with the same boundary values. A similar property holds for $p$-supersolutions.
\begin{rem}\label{rem_min}
\rm Let $\Omega\subset M$ be a bounded domain and let $s\in W^{1,p}(\Omega)$ be a $p$-supersolution. Then for any function $f\in W^{1,p}(\Omega)$ with $f\geq s$ a.e. and $f-s\in W^{1,p}_0(\Omega)$ we have
\begin{gather*}
 \D_p(s)\leq \D_p(f)\, .
\end{gather*}
\end{rem}
\begin{proof}
 This remark follows easily form the minimizing property of the solution to the obstacle problem. In fact, the previous Proposition shows that $s$ is the solution to the obstacle problem relative to $K_s$, and the minimizing property follows.
\end{proof}

Also for the obstacle problem with a continuous obstacle, continuity of the solution up to the boundary is an interesting and well-studied property. Indeed, also in this more general setting the Wiener criterion and the barrier condition are necessary and sufficient conditions for such regularity. More detailed propositions and references will be given in the next section. For the moment we just remark that, as expected, if $\partial \Omega$ is smooth and $\psi$ is continuous up to the boundary, then the solution to the obstacle problem $K_\psi$ belongs to $C(\overline \Omega)$ (see \cite[Theorem 7.2]{BB}).
\begin{prop}\label{prop_cont}
 Given a bounded $\Omega\subset M$ with smooth boundary and given a function  $\psi\in W^{1,p}(\overline \Omega)\cap C(\overline \Omega)$, then the unique solution to the obstacle problem $K_\psi$ is continuous up to $\partial \Omega$.
\end{prop}

In the following we will need this Lemma about uniform convexity in Banach spaces. This Lemma doesn't seem very intuitive at first glance, but a simple two dimensional drawing of the vectors involved shows that in fact it is quite natural.
\begin{lemma}\label{lemma_*}
\index{uniformly convex Banach space}
 Given a uniformly convex Banach space $E$, there exists a function $\sigma:[0,\infty)\to [0,\infty)$ strictly positive on $(0,\infty)$ with $\lim_{x\to 0} \sigma(x)=0$ such that for any $v,w\in E$ with $\norm{v+1/2 w}\geq \norm v$
\begin{gather*}
 \norm{v+w}\geq \norm v \ton{1+\sigma\ton{\frac{\norm w}{\norm v +\norm w}}}\, .
\end{gather*}
\end{lemma}
\begin{proof}
 Note that by the triangle inequality $\norm{v+1/2 w}\geq \norm v$ easily implies $\norm{v+ w}\geq \norm v$. Let $\delta$ be the modulus of convexity of the space $E$. By definition we have
\begin{gather*}
 \delta(\epsilon)\equiv \inf\left\{ 1-\norm{\frac{x+y}{2}} \ s.t. \ \norm x, \norm y \leq 1 \ \ \ \norm {x-y}\geq \epsilon \right\}\, .
\end{gather*}
Consider the vectors $x=\alpha v$ $y=\alpha (v+w)$ where $\alpha=\norm{v+w}^{-1}\leq \norm{v}^{-1}$. Then
\begin{gather*}
 1-\norm{\frac{x+y}{2}}=1-\alpha \norm{v+\frac{w}{2}}\geq \delta(\alpha \norm{w})\geq \delta\ton{\frac{\norm w}{\norm v + \norm w}}\\
\norm{v+w}\geq \norm{v+\frac{w}{2}} \ton{1-\delta\ton{\frac{\norm w}{\norm v + \norm w}}}^{-1}\, .
\end{gather*}
Since $\norm{v+\frac{w}{2}}\geq \norm v$ and by the positivity of $\delta$ on $(0,\infty)$, if $E$ is uniformly convex the thesis follows.
\end{proof}

\begin{remark}\rm{Recall that all $L^p(X,\mu)$ spaces with $1<p<\infty$ are uniformly convex by Clarkson's
\index{Clarkson's inequalities}
inequalities, and their modulus of convexity is a function that depends only on $p$ and not on the underling measure space $(X,\mu)$. For a reference on uniformly convex spaces, modulus of convexity and Clarkson's inequality, we cite his original work \cite{clarkson}.}
\end{remark}

\subsection{\ka condition}\label{sec_ka}
In this section, we prove the \ka condition for a generic $p>1$ and show that it is not just a sufficient condition, but also a necessary one. Even though it is applied in a different context, the proof is inspired by the techniques used in \cite[Theorem 10.1]{HKM}
\begin{prop}[\ka condition]
\index{Khasm@\ka condition}

 Given a Riemannian manifold $M$ and a compact set $K\subset M$, if there exists a $p$-superharmonic finite-valued function $\K:M\setminus K\to \R$ such that
\begin{gather*}
 \lim_{x\to \infty} \K(x)=\infty\, ,
\end{gather*}
then $M$ is $p$-parabolic.
\end{prop}
\begin{proof}
 This condition was originally stated and proved in \cite{khascond} in the case $p=2$. However, since the only tool necessary for this proof is the comparison principle, it is easily extended to any $p>1$. An alternative proof can be found in \cite{PRS}, in the following we sketch it.

Fix a smooth exhaustion $D_n$ of $M$ such that $K\Subset D_0$. Set $m_n\equiv \min_{x\in \partial D_n} \K(x)$, and consider for every $n\geq 1$ the $p$-capacity potential $h_n$ of the couple $(\overline D_0,D_n)$. Since the potential $\K$ is superharmonic, it is easily seen that $h_n(x)\geq 1-\K(x)/m_n$ for all $x\in D_n\setminus \overline D_0$. By letting $n$ go to infinity, we obtain that $h(x)\geq 1$ for all $x\in M$, where $h$ is the capacity potential of $(\overline D_0, M)$. Since by the maximum principle $h(x)\leq 1$ everywhere, $h(x)=1$ and so $\cp(\overline D_0)=0$.
\end{proof}

Observe that the hypothesis of $\K$ being finite-valued can be dropped. In fact if $\K$ is $p$-superharmonic, the set $\K ^{-1}(\infty)$ has zero $p$-capacity, and so the reasoning above would lead to $h(x)= 1$ except on a set of $p$-capacity zero, but this implies $h(x)=1$ everywhere (see \cite{HKM} for the details).

Before proving the reverse of \ka condition for any $p>1$, we present a short simpler proof if $p=2$, which in some sense outlines the proof of the general case.

In the linear case, the sum of $2$-superharmonic functions is again $2$-superharmonic, but of course this fails to be true for a generic $p$. Using this linearity, it is easy to prove that
\begin{prop}
Given a $2$-parabolic Riemannian manifold, for any $2$-regular compact set $K$, there exists a $2$-superharmonic continuous function $\K:M\setminus K\to \R^+$ with $f|_{\partial K}=0$ and $\lim_{x\to \infty}\K(x)=\infty$.

\end{prop}
\begin{proof}
Consider a smooth exhaustion $\{K_n\}_{n=0}^\infty$ of $M$ with $K_0\equiv K$. For any $n\geq1$ define $h_n$ to be the $2$-potential of $(K,K_n)$. By the comparison principle, the sequence $\tilde h_n =1-h_n$ is a decreasing sequence of continuous functions, and since $M$ is $2$-parabolic the limit function $\tilde h$ is the zero function. By Dini's theorem, the sequence $\tilde h_n$ converges to zero locally uniformly, so it is not hard to choose a subsequence $\tilde h_{n(k)}$ such that the series $\sum_{k=1}^\infty \tilde h_{n(k)}$ converges locally uniformly to a continuous function. It is straightforward to see that $\K=\sum_{k=1}^{\infty} \tilde h_{n(k)}$ has all the desired properties.
\end{proof}

It is evident that this proof fails in the nonlinear case. However also in the general case, by making a careful use of the obstacle problem, we will build an increasing locally uniformly converging sequence of $p$-superharmonic functions, whose limit is going to be the \ka potential $\K$.


We first prove that if $M$ is $p$-parabolic, then there exists a proper function $f:M\to \R$ with finite $p$-Dirichlet integral.
\begin{prop}
 Let $M$ be a $p$-parabolic Riemannian manifold. Then there exists a positive continuous function $f:M\to \R$ such that
\begin{gather*}
 \int_{M} \abs{\nabla f}^p dV <\infty \ \ \ \ \ \ \lim_{x\to \infty} f(x)=\infty\, .
\end{gather*}
\end{prop}
\begin{proof}
 Fix an exhaustion $\{D_n\}_{n=0}^{\infty}$ of $M$ such that every $D_n$ has smooth boundary, and let $\{h_n\}_{n=1}^{\infty}$ be the $p$-capacity potential of the couple $(D_0,D_n)$. Then by an easy application of the comparison principle the sequence
\begin{gather*}
 \tilde h_n(x) \equiv \begin{cases}
                    0 & \text{if }x\in D_0\\
1-h_n(x) & \text{if }x\in D_n\setminus D_0\\
1 & \text{if }x\in D_n^C
                   \end{cases}
\end{gather*}
 is a decreasing sequence of continuous functions converging pointwise to $0$ (and so also locally uniformly by Dini's theorem) and also $\int_{M}\abs{\nabla \tilde h_n}^p dV \to 0$. So we can extract a subsequence $\tilde h_{n(k)}$ such that
\begin{gather*}
 0\leq \tilde h_{n(k)}(x)\leq \frac{1}{2^k} \ \ \ \ \forall x\in D_k \ \ \ \ \text{and} \ \ \ \ \ \int_{M}\abs{\nabla \tilde h_{n(k)}}^p dV<\frac1 {2^k}\,\ .
\end{gather*}
It is easily verified that $f(x)=\sum_{k=1}^{\infty} \tilde h _{n(k)}(x)$ has all the desired properties.
\end{proof}

We are now ready to prove the reverse \ka condition, i.e.,
\begin{teo}\label{teo_khas_easy}
\index{Khasm@\ka condition}
Given a $p$-parabolic manifold $M$ and a compact set $K\Subset M$ with smooth bound\-a\-ry, there exists a continuous positive $p$-superharmonic function $\K:M\setminus \overline K\to \R$ such that
\begin{gather*}
 \K|_{\partial M} =0 \quad \text{and} \quad \lim_{x\to \infty} \K(x)=\infty\, .
\end{gather*}
Moreover
\begin{gather}
 \int_M \abs{\nabla \K}^p dV <\infty\, .
\end{gather}

\end{teo}
\begin{proof}
 Fix a continuous proper function $f:M\to \R^+$ with finite Dirichlet integral such that $f=0$ on a compact neighborhood of $K$, and let $D_n$ be a smooth exhaustion of $M$ such that $f|_{D_n^C}\geq n$. We are going to build by induction an increasing sequence of continuous functions $s^{(n)} \in L^{1,p}(M)$ such that
\begin{enumerate}
 \item $s_n$ is $p$-superharmonic in $\overline{K}^C$,
 \item $s^{(n)}|_K=0$,
 \item $s^{(n)}\leq n$ everywhere and $s^{(n)}=n$ on $S_n^C$, where $S_n$ is a compact set.
\end{enumerate}
Moreover, the sequence $s^{(n)}$ will also be locally uniformly bounded, and the locally uniform limit $\K(x)\equiv \lim_n s^{(n)}(x)$ will be a finite-valued function in $M$ with all the desired properties.

Let $s^{(0)}\equiv 0$, and suppose by induction that an $s^{(n)}$ with the desired property exists. Hereafter $n$ is fixed, so for simplicity we will write $s^{(n)}\equiv s$, $s^{(n+1)}\equiv s^+$ and $S_n = S$. 
Define the functions $ f_j(x)\equiv \min\{j^{-1}f(x),1\}$, and consider the obstacle problems on $\Omega_j\equiv D_{j+1}\setminus \overline D_0$ given by the obstacle $\psi_j=s+ f_j$.
For any $j$, the solution $h_j$ to this obstacle problem is a $p$-superharmonic function defined on $\Omega_j$ bounded above by $n+1$ and whose restriction to $\partial D_0$ is zero. 
If $j$ is large enough such that $s=n$ on $D_{j}^C$ (i.e. $S\subset D_j$), then the function $h_j$ is forced to be equal to $n+1$ on $D_{j+1}\setminus D_{j}$ and so we can easily extend it by setting
\begin{gather*}
 \tilde h_j(x)\equiv \begin{cases}
                   h_j(x) & x\in \Omega_j\\
0 & x\in \overline D_0\\
n+1 & x\in D_{j+1}^C
                  \end{cases}
\end{gather*}
Note that $\tilde h_j$ is a continuous function on $M$, $p$-superharmonic in $\overline{D_0}^C$. Once we have proved that, as $j$ goes to infinity, $\tilde h_j$ converges locally uniformly to $s$, we can choose an index $\bar j$ large enough to have $\sup_{x\in D_{n+1}} \abs{\tilde h_{\bar j} (x)-s(x)}<2^{-n-1}$. Thus $s^+=\tilde h_{\bar j}$ has all the desired properties.

We are left to prove the statement about uniform convergence. For this aim, define $\delta_j \equiv h_{j}-s$. Since the obstacle $\psi_j$ is decreasing, it is easily seen that the sequence $h_{j}$ is decreasing as well, and so is $\delta_j$. Therefore $\delta_j$ converges pointwise to a function $\delta\geq 0$. By the minimizing properties of $h_j$, we have that
\begin{gather*}
 \norm{\nabla h_j}_p\leq \norm{\nabla s + \nabla f_j}_p\leq \norm{\nabla s}_p + \norm{\nabla f}_p\, ,
\end{gather*}
and so
\begin{gather*}
\norm{\nabla \delta_j}_p\leq 2\norm{\nabla s} + \norm{\nabla f}\leq C\, .
\end{gather*}
A standard weak-compactness argument in reflexive spaces \footnote{see for example \cite[Lemma 1.3.3]{HKM}} proves that $\delta\in W^{1,p}_\loc(M)$ with $\nabla \delta \in L^{p}(M)$, and also $\nabla \delta_j \to \nabla \delta$ in the weak $L^p(M)$ sense. 

In order to estimate the $p$-Dirichlet integral $\D_p(\delta)$, define the function
\begin{gather*}
g\equiv \min\{s+ \delta_j/2,n\} \, .
\end{gather*}
It is quite clear that $s$ is the solution to the obstacle problem relative to itself on $S\setminus \overline {D_0}$, and since $\delta_j\geq0$ with $\delta_j=0$ on $D_0$, we have $g\geq s$ and $g-s\in W^{1,p}_0(S\setminus \overline {D_0})$. The minimizing property for solutions to the $p$-Laplace equation then guarantees that
\begin{gather*}
\norm{\nabla s + \frac 1 2 \nabla \delta_j}_p^p\equiv \int_{M} \abs{\nabla s + \frac 1 2 \nabla \delta_j}^p dV \geq \int_{S\setminus \overline{D_0}} \abs{\nabla g}^p dV\geq\\
\geq \int_{S\setminus \overline{D_0}} \abs{\nabla s}^p dV = \norm{\nabla s}_p^p\, .
\end{gather*}
Recalling that also $h_j=s+\delta_j$ is solution to an obstacle problem on $D_{j+1}\setminus D_0$, we get
\begin{gather*}
 \norm{\nabla s + \nabla \delta_j}_p=\ton{\int_{M} \abs{\nabla \tilde h_j}^p dV}^{1/p}=\ton{\int_{\Omega_j} \abs{\nabla h_j}^p dV}^{1/p}\leq \\
\leq \ton{\int_{\Omega_j} \abs{\nabla s + \nabla f_j}^p dV}^{1/p}\leq \norm{\nabla s +\nabla f_j}_p \leq \norm{\nabla s }_p + \norm{\nabla {f_j}}_p\, .
\end{gather*}
Using Lemma \ref{lemma_*} we conclude
\begin{gather*}
 \norm{\nabla s}_p \ton{1+\sigma\ton{\frac{\norm{\nabla \delta_j}_p}{ \norm{\nabla s}_p +\norm{\nabla \delta_j}_p}}}\leq \norm{\nabla s}_p + \norm{\nabla f_j}_p\, .
\end{gather*}
Since $\norm{\nabla f_j}_p \to 0$ as $j$ goes to infinity and by the properties of the function $\sigma$, we have
\begin{gather*}
\lim_{j\to \infty} \norm{\nabla \delta_j}_p^p =0\, .
\end{gather*}

By weak convergence, we have $\norm{\nabla \delta}_p^p =0$. Since $\delta=0$ on $D_0$, we can conclude $\delta=0$ everywhere, or equivalently $\tilde h_j \to s$. Note also that since the limit function $s$ is continuous, by Dini's theorem the convergence is also locally uniform.
\end{proof}

\begin{rem}
 \rm Since $\norm{\nabla \tilde h_j}_p\leq \norm{\nabla s^{(n)}}_p+\norm{\nabla \delta_j}_p$, if for each induction step we choose $\bar j$ such that $\norm{\nabla \delta_{\bar j}}_p<2^{-n}$, the function $\K=\lim_{n} s^{(n)}$ has finite $p$-Dirichlet integral.
\end{rem}

\begin{rem}
 \rm In the following section, we are going to prove the \ka condition in a more general setting. In particular, we will address more thoroughly issues related to the continuity on the boundary of the potential function $\K$, and more importantly we generalize the results to a wider class of operators. Moreover, exploiting finer properties of the solutions to the obstacle problems, we will give a slightly different (and perhaps simpler, though technically more challenging) proof of existence for \ka potentials.
\end{rem}

\subsection{Evans potentials}\label{sec_ev_e}
\index{Evans potential}
We conclude this section with some remarks on the Evans potentials for $p$-parabolic manifolds. Given a compact set with nonempty interior and smooth boundary $K\subset M$, we call $p$-Evans potential a function $\E:M\setminus K\to \R$ $p$-harmonic where defined such that
\begin{gather*}
 \lim_{x\to \infty } \E(x)=\infty \quad \text{and} \quad \lim_{x\to \partial K} \E(x)=0 \, .
\end{gather*}
It is evident that if such a function exists, then $M$ is $p$-parabolic by the \ka condition. It is interesting to investigate whether also the reverse implication holds. In \cite{naka} and \cite{SN} \footnote{see in particular \cite[Theorems 12.F and 13.A]{SN}}, Nakai and Sario prove that the $2$-parabolicity of Riemannian surfaces is completely characterized by the existence of such functions. In particular they prove that
\begin{teo}\label{teo_sn}
Given a $p$-parabolic Riemannian surface $M$, and an open precompact set $M_0$, there exists a ($2-$)harmonic function $\E:M\setminus M_0\to \R^+$ which is zero on the boundary of $M_0$ and goes to infinity as $x$ goes to infinity. Moreover
\begin{gather}\label{eq_evp}
 \int_{\{0\leq \E(x)\leq c\}} \abs{\nabla \E(x)}^2dV\leq 2\pi c\, .
\end{gather}
\end{teo}
Clearly the constant $2\pi$ in equation \eqref{eq_evp} can be replaced by any other positive constant. As noted in \cite[Appendix, pag 400]{SN}, with similar arguments and with the help of the classical potential theory \footnote{\cite{helms} might be of help to tackle some of the technical details. \cite{thesis} focuses exactly on this problem, but unfortunately it is written in Italian.}, it is possible to prove the existence of $2$-Evans potentials for a generic $n$-dimensional $2$-parabolic Riemannian manifold.

This argument however is not easily adapted to the nonlinear case ($p\neq 2$). Indeed, Nakai builds the Evans potential as a limit of carefully chosen convex combinations of Green kernels defined on the Royden compactification of $M$. While convex combinations preserve $2$-harmonicity, this is evidently not the case when $p\neq 2$.

Since $p$-harmonic functions minimize the $p$-Dirichlet of functions with the same boundary values, it would be interesting from a theoretical point of view to prove existence of $p$-Evans potentials and maybe also to determine some of their properties. From the practical point of view such potentials could be used to get informations on the underlying manifold $M$, for example they can be used to improve the Kelvin-Nevanlinna-Royden criterion for $p$-parabolicity as shown in the article \cite{VV} and in section \ref{sec_J}.

Even though we are not able to prove the existence of such potentials in the generic case, some special cases are easier to manage. We briefly discuss the case of model manifolds hoping that the ideas involved in the following proofs can be a good place to start for a proof in the general case.

First of all we recall the definition model manifolds.
\begin{deph}
\index{model manifold}
A complete Riemannian manifold $M$ is a model manifold (or a spherically symmetric manifold) if it is diffeomorphic to $\R^n$ and if there exists a point $o\in M$ such that in exponential polar coordinates the metric assumes the form
\begin{gather*}
ds^2 = dr^2 + \sigma^2(r) d\theta^2\, ,
\end{gather*}
where $\sigma$ is a smooth positive function on $(0,\infty)$ with $\sigma(0)=0$ and $\sigma'(0)=1$, and $d\theta^2$ is the standard metric on the Euclidean sphere.
\end{deph}
Following \cite{gri}, on model manifolds it is possible to define a very easy $p$-harmonic radial function.

Set $g$ to be the determinant of the metric tensor $g_{ij}$ in radial coordinates, and define the function $A(r)=\sqrt{g(r)}=\sigma(r)^{n-1}$. Note that, up to a constant depending only on $n$, $A(r)$ is the area of the sphere of radius $r$ centered at the origin. On model manifolds, the radial function
\begin{gather}\label{eq_fr}
 f_{p,\bar r}(r)\equiv\int_{\bar r} ^r A(t)^{-\frac 1 {(p-1)}} dt
\end{gather}
is a $p$-harmonic function away from the origin $o$. Indeed,
\begin{gather*}
 \Delta_p (f)=\frac{1}{\sqrt g} \operatorname{div}(\abs{\nabla f}^{p-2}\nabla f)= \frac{1}{\sqrt g} \partial_i \ton{\sqrt g \ton{g^{kl}\partial_k f \partial_l f}^{\frac{p-2} 2} g^{ij}\partial_j f}=\\
=\frac{1}{ {A(r)} }\partial_r \ton{A(r)\ A(r)^{-\frac{p-2}{p-1}}\ A(r)^{-\frac{1}{p-1}}\ \nabla r}=0\, .
\end{gather*}
Thus the function $\max\{f_{p,\bar r},0\}$ is a $p$-subharmonic function on $M$, so if $f_{p,\bar r}(\infty)<~\infty$, $M$ cannot be $p$-parabolic. A straightforward application of the \ka condition shows that also the reverse implication holds, so that a model manifold $M$ is $p$-parabolic if and only if $f_{p,\bar r}(\infty)=\infty$. This shows that if $M$ is $p$-parabolic, then for any $\bar r>0$ there exists a radial $p$-Evans potential $f_{p,\bar r}\equiv\E_{\bar r}:M\setminus B_{\bar r}(0)\to \R^+$, moreover it is easily seen by direct calculation that
\begin{gather*}
 \int_{B_R} \abs{\nabla \E_{\bar r}}^{p}dV=\E_{\bar r}(M) \ \ \ \ \Longleftrightarrow \ \ \ \  \int_{\E_{\bar r}\leq t} \abs{\nabla \E_{\bar r}}^p dV =t\, .
\end{gather*}
By a simple observation we can also estimate that
\begin{gather*}
 \cp(B_{\bar r}, \{\E_{\bar r}\leq t\})=\int_{\{\E_{\bar r}\leq t\}\setminus B_{\bar r}} \abs{\nabla\ton{\frac{\E_{\bar r}}{t}}}^p dV =t^{1-p}\, .
\end{gather*}
Since $M$ is $p$-parabolic, it is clear that $ \cp(B_{\bar r}, \{\E_{\bar r}\leq t\})$ must go to $0$ as $t$ goes to infinity, but this last estimate gives also some quantitative control on how fast the convergence is.

Now we are ready to prove that, on a model manifold, every compact set with smooth boundary admits an Evans potential.
\begin{prop}
\index{Evans potential}
 Let $M$ be a $p$-parabolic model manifold and fix a compact set $K\subset M$ with smooth boundary and nonempty interior. Then there exists an Evans potential $\E_K:M\setminus K \to \R^+$, i.e., a harmonic function on $M\setminus K$ with
\begin{gather*}\label{eq_evp2}
 \E_K|_{\partial K}=0 \quad \text{and} \quad \lim_{x\to \infty}\E_K(x)=\infty\, .
\end{gather*}
Moreover, as $t$ goes to infinity
\begin{gather}
 \cp(K,\{\E_K<t\})\sim t^{1-p}\, .
\end{gather}
\end{prop}
\begin{proof}
 Since $K$ is bounded, there exists $\bar r>0$ such that $K\subset B_{\bar r}$. Let $\E_{\bar r}$ be the radial $p$-Evans potential relative to this ball. For any $n>0$, set $A_n=\{\E_{\bar r}\leq n\}$ and define the function $e_n$ to be the unique $p$-harmonic function on $A_n\setminus K$ with boundary values $n$ on $\partial A_n$ and $0$ on $\partial K$. An easy application of the comparison principle shows that $e_n\geq \E_{\bar r}$ on $A_n\setminus K$, and so the sequence $\{e_n\}$ is increasing. By the Harnack principle, either $e_n$ converges locally uniformly to a harmonic function $e$, or it diverges everywhere to infinity. To exclude the latter possibility, set $m_n$ to be the minimum of $e_n$ on $\partial B_{\bar r}$. By the maximum principle the set $\{0\leq e_n \leq m_n\}$ is contained in the ball $B_{\bar r}$, and using the capacity estimates described in Lemma \ref{lemma_cap}, we get that
\begin{gather*}
 \cp(K,B_{\bar r})\leq \cp(K, \{e_n< m_n\})=\cp(K,\{e_n/n < m_n/n\}) =\\
=\frac{n^{p-1}}{m_n^{p-1}}\cp\ton{K, \{e_n<n\} } \leq \frac{n^{p-1}}{m_n^{p-1}}\cp(B_{\bar r},\{\E_{\bar r}<n\})\, .
\end{gather*}
Thus we can estimate
\begin{gather*}
m_n^{p-1} \leq \frac{n^{p-1}\cp(B_{\bar r}, \{\E_{\bar r}< n\})}{\cp(K,B_{\bar r})}<\infty\, .
\end{gather*}
So the limit function $\E_k=\lim_n e_n$ is a $p$-harmonic function in $M\setminus K$ with $\E_K \geq \E_{\bar r}$. 

Since $K$ has smooth boundary, continuity up to $\partial K$ of $\E_k$ is easily proved. 

As for the estimates of the capacity, set $M=\max\{\E_K(x) \ \ x\in \partial B_{\bar r}\}$, and consider that, by the comparison principle, $\E_K \leq M+\E_{\bar r}$ (where both functions are defined), so that
\begin{gather*}
 \cp(K,\{\E_K<t\})\leq \cp(B_{\bar r},\{\E_K<t\})\leq \\
\leq\cp(B_{\bar r}, \{\E_{\bar r}<t-M\}) =(t-M)^{1-p}\sim t^{1-p}\, .
\end{gather*}
For the reverse inequality, we have
\begin{gather*}
 \cp(K, \{\E_K<t\})=\int_{\{\E_K<t\}\setminus K}\abs {\nabla \ton{\frac {\E_K} t}}^p dV =\\
=\int_{\{\E_K<M\}\setminus K} \abs {\nabla \ton{\frac {\E_K} t}}^p dV + \int_{\{M<\E_K<t\}} \abs {\nabla \ton{\frac {\E_K} t}}^p dV =\\
=\ton{\frac M t }^p\int_{\{e<M\}\setminus K} \abs {\nabla \ton{\frac e M}}^p dV +\ton{\frac {t-M}{t}} ^p\int_{\{M<\E_K<t\}} \abs {\nabla \ton{\frac {\E_K} {t-M}}}^p dV \geq\\
\geq\ton{\frac m t }^p \cp(K, \{\E_K<M\})+ \ton{\frac {t-m}{t}} ^p \cp\ton{B_{\bar r}, \{\E_{\bar r}<t\} }\sim t^{1-p}\, .
\end{gather*}

\end{proof}

\begin{remark}
 \rm There are other special parabolic manifolds in which it is easy to build an Evans potential. For example, \cite{PRS} deals with the case of roughly Euclidean manifolds and manifolds with Harnack ends.
\index{roughly Euclidean manifold}
\index{Harnack end}

Given the examples, and given what happens if $p=2$, it is reasonable to think that Evans potentials exist on every parabolic manifolds, although the proof seems to be more complex than expected.
\end{remark}

%
%
%
%
%

 \section{\ka condition for general operators}\label{sec_kagen}

In this section we will investigate further the \ka condition, and generalize the main results obtained in the previous sections to a wider class of operators. Aside from the technicalities, we will also show that a suitable \ka condition characterizes not only parabolicity but also stochastic completeness for the usual $p$-Laplacian.

\subsection{Definitions}\label{sec_deph}
In this section we define the family of $\L$-type operators and recall the basic definitions of sub-supersolution, Liouville property and generalized \ka condition.

\begin{definition}
 Let $M$ be a Riemannian manifold, denote by $TM$ its tangent bundle. We say that $A:TM\to TM$ is a Caratheodory map if:
 \begin{enumerate}
\index{Caratheodory function/map}
  \item $A$ leaves the base point of $TM$ fixed, i.e., $\pi\circ A = \pi$, where $\pi$ is the canonical projection $\pi:TM\to M$;
  \item for every local representation $\tilde A$ of $A$, $\tilde A(x,\cdot)$ is continuous for almost every $x$, and $\tilde A(\cdot,v)$ is measurable for every $v\in \R^n$. 
 \end{enumerate}
 We say that $B:M\times \R \to \R$ is a Caratheodory function if:
 \begin{enumerate}
  \item for almost every $x\in M$, $B(x,\cdot)$ is continuous
  \item for every $t\in \R$, $B(\cdot,t)$ is measurable
 \end{enumerate}

\end{definition}

\begin{definition}\label{deph_L}
 Let $A$ be a Caratheodory map, and $B$ a Caratheodory function. Suppose that there exists a $p\in(1,\infty)$ and positive constants $a_1, a_2, b_1,b_2$ such that:
\begin{itemize}
 \item[\rm{(A1)}] for almost all $x\in M$, $A$ is strictly monotone, i.e., for all $X,Y \in T_xM$
  \begin{gather}
   \ps{A(X)-A(Y)}{X-Y} \ge 0
  \end{gather}
  with equality if and only if $X=Y$;
 \item[\rm{(A2)}]  for almost all $x\in M$ and $\forall \ X\in T_xM, \quad \ps{A(X)}{X} \ge a_1|X|^p$;
 \item[\rm{(A3)}] for almost all $x\in M$ and $\forall X\in T_xM, \quad \abs{A(X)}\le a_2 |X|^{p-1}$;
  \item[\rm{(B1)}]  $B(x,\cdot)$ is monotone non-decreasing;
 \item[\rm{(B2)}] For almost all $x\in M$, and for all $t\in \R$, $B(x,t)t \ge 0$;
 \item[\rm{(B3)}] For almost all $x\in M$, and for all $t\in \R$, $\abs{B(x,t)} \le b_1+b_2 t^{p-1}$.
 \end{itemize}
Given a domain $\Omega\subset M$, we define the operator $\L$ on $W^{1,p}_{\loc}(\Omega)$ by
\begin{gather}
 \L(u)= \dive(A(\nabla u))- B(x,u)\, ,
\end{gather}
where the equality is in the weak sense. Equivalently, for every function $\phi\in C^{\infty}_C(\Omega)$, we set
\begin{gather}
 \int_{\Omega} \L(u)\phi dV = -\int_\Omega \ton{\ps{A(\nabla u)}{\nabla \phi} + B(x,u)\phi} dV\, .
\end{gather}
We say that an operator is a $\L$-type operator if it can be written in this form.
\index{L type operator@$\L$-type operator}
\end{definition}

\begin{remark}
 \rm{If $u\in W^{1,p}(\Omega)$, then by a standard density argument the last equality remains valid for all $\phi\in W^{1,p}_0(\Omega)$}
\end{remark}
\begin{remark}
 \rm{Using some extra caution, it is possible to relax assumption (B3) to
\begin{enumerate}
 \item [(B3')] there exists a finite-valued function $b(t)$ such that $\abs{B(x,t)}\leq b(t)$ for almost all $x\in M$.
\end{enumerate}
Indeed, inspecting the proof of the main Theorem, it is easy to realize that at each step we deal only with locally bounded functions, making this exchange possible. However, if we were to make all the statements of the theorems with this relaxed condition, the notation would become awkwardly uncomfortable, without adding any real mathematical content.
}
\end{remark}

\begin{ex}\label{ex1}
\rm{It is evident that the $p$-Laplace operator is an $\L$-type operator. Indeed, if we set $A(V)= \abs{V}^{p-2} V$ and $B(x,u)=0$, it is evident that $A$ and $B$ satisfy all the conditions in the previous definition, and $\L(u)=\Delta_p (u)$.

Moreover, for any $\lambda\geq 0$, also the operator $\Delta_p(u)-\lambda \abs{u}^{p-2} u$ is an $\L$-type operator. Indeed, in the previous definition we can choose $A(V)=\abs{V}^{p-2}V $ and $B(x,u)=\lambda \abs{u}^{p-2} u$.}
\end{ex}
\begin{ex}\label{ex2}
\rm{In \cite[Chapter 3]{HKM}, the authors consider the so-called $\A$-Laplacians on $\R^n$. It is easily seen that also these operators are $\L$-type operators with $B(x,u)=0$.}
\end{ex}
\index{L type operator@$\L$-type operator!A Laplacians@$\A$-Laplacians}
\begin{ex}\label{ex3}
\rm Less standard examples of $\L$-type operators are the so-called $(\varphi-h)$-Laplacian. Some references regarding this operator can be found  in \cite{PRS} and \cite{PSZ}. 

\index{L type operator@$\L$-type operator!phi Laplacians@$(\varphi-h)$-Laplacians}
For a suitable $\varphi:[0,\infty)\to \R$ and a symmetric bilinear form $h:TM\times TM\to \R$, define the $(\varphi-h)$-Laplacian by
\begin{gather}
 \L(u)=\dive\ton{\frac{\varphi(\abs {\nabla u})} {\abs {\nabla u}} h(\nabla u,\cdot)^\sharp}\, ,
\end{gather}
where $\ ^\sharp:T^\star M \to TM$ is the musical isomorphism.

It is evident that $\varphi$ and $h$ has to satisfy some conditions in order for the $(\varphi-h)$-Laplacian to be an $\L$-type operator. In particular, a sufficient and fairly general set of conditions is the following:
\begin{enumerate}
 \item $\varphi: [0,\infty)\to [0,\infty)$ with $\varphi(0)=0$,
 \item $\varphi\in C^0[0,\infty)$ and monotone increasing,
 \item there exists a positive constant $a$ such that $a\abs{t}^{p-1}\leq \varphi(t)\leq a^{-1} \abs t ^{p-1}$,
 \item $h$ is symmetric, positive definite and bounded, i.e., there exists a positive $\alpha$ such that, for almost all $x\in M$ and for all $X\in T_x(M)$:
 \begin{gather}
  \alpha \abs X^2 \leq h(X\vert X)\leq \alpha^{-1} \abs X^2\, ,
 \end{gather}
 \item the following relation holds for almost all $x\in M$ and all unit norm vectors $V\neq W\in T_xM$:
  \begin{gather}
   \frac{\varphi(t)}{t}h(V,V)+\ton{\varphi'(t)-\frac{\varphi(t)} t }\ps{V}{W} h(V,W)>0\, .
  \end{gather}
\end{enumerate}
Under these assumptions it is easy to prove that the $(\varphi-h)$-Laplacian in an $\L$-type operator. The only point which might need some discussion is $(A1)$, the strict monotonicity of the operator $A$. In order to prove this property, fix $X,Y\in T_xM$, and set $Z(\lambda)=Y+\lambda(X-Y)$. Define the function $F:[0,1]\to \R$ by
\begin{gather}
 F(\lambda)=\frac{\varphi(\abs{Z})}{\abs Z} h(Z,X-Y)\, .
\end{gather}
It is evident that monotonicity is equivalent to $F(1)-F(0)>0$. Since by condition $(5)$, $\frac{dF}{d\lambda}>0$, monotonicity follows immediately.
\end{ex}
We recall the concept of subsolutions and supersolutions for the operator $\L$.
\begin{deph}
\index{L type operator@$\L$-type operator!L supersolution@$\L$-supersolution}
Given a domain $\Omega\subset M$, we say that $u \in W^{1,p}_\loc(\Omega)$ is a supersolution relative to the operator $\L$ if it solves $\L u \leq 0$ weakly on $\Omega$. Explicitly, $u$ is a supersolution if, for every non-negative $\phi\in C^\infty_C(\Omega)$
\begin{gather}
 -\int_\Omega \ps{A(\nabla u)}{\nabla \phi} - \int_{\Omega} B(x,u)\phi \leq 0\, .
\end{gather}
A function $u\in W^{1,p}_{\loc}(\Omega)$ is a subsolution ($\L u \geq 0$) if and only if $-u$ is a supersolution, and a function $u$ is a solution ($\L u =0$) if it is simultaneously a sub and a supersolution.
\end{deph}
\begin{oss}
\rm{It is easily seen from definition \ref{deph_L} that $B(x,0)=0$ a.e. in $M$, and more generally $B(x,c)$ is either zero or it has the same sign as $c$. Therefore, the constant function $u=0$ solves $\L u=0$, while positive constants are supersolutions and negative constants are subsolutions.}
\end{oss}
Following \cite{PRS} and \cite{PRSA}, we present the analogues of the
$L^\infty$-Liouville property and the \ka property for the
nonlinear operator $\L$ defined above.
\begin{deph}
\index{L Liouville property@$L^{\infty}$ Liouville property}
Given an $\L$-type operator $\L$ on a Riemannian manifold $M$, we say that $\L$ enjoys the $L^\infty$ Liouville property if for every function $u\in W^{1,p}_\loc (M)$
\begin{gather}\label{Liou}\tag{Li}
 u\in L^\infty(M) \ \ \text{and}\ \ \ u\geq 0 \ \ \text{and}\ \  \L(u)\geq 0 \ \Longrightarrow \ u \text{ is constant}\, .
\end{gather}
\end{deph}

\begin{example}
\rm In the previous section, we have seen that for $\L(u)=\Delta_p(u)$ this Liouville property is also called $p$-parabolicity (see definition \ref{deph_para_easy}). In the case where $p=2$, for any fixed $\lambda>0$, if we consider the half-linear operator $\L(u) = \Delta_p (u) - \lambda \abs u ^{p-2} u$, the Liouville property \eqref{Liou} is equivalent to stochastic completeness of the manifold $M$. This property characterizes the non-explosion of the Brownian motion and important properties of the heat kernel on the manifold (see \cite[Section 6]{gri} for a detailed reference). For a generic $p$, given the lack of linearity, many mathematical definitions and tools \footnote{for example the heat kernel} do not have a natural counterpart when $p\neq 2$. However it is still possible to define a Liouville property and a \ka condition.
\end{example}

\begin{deph}\label{deph_ka}
\index{Khasm@\ka condition!Generalized \ka condition@$\L$-type \ka condition}
Given an $\L$-type operator $\L$ on a Riemannian manifold $M$, we say that $\L$ enjoys the \ka property if for every $\epsilon >0$ and every pair of sets $K\Subset \Omega$ bounded and with Lipschitz boundary, there exists a continuous exhaustion function $\K:M\setminus K \to [0,\infty)$ such that
\begin{gather}
 \notag \K\in C^0(\overline{M\setminus K})\cap W^{1,p}_\loc (M\setminus K)\, ,\\
 \L(\K) \leq 0 \quad \ \text{ on } \quad \ M\setminus K \label{ka}\tag{Ka} \, ,\\
 \notag \K|_{\partial K}=0 \, ,\quad \K(\Omega)\subset [0,\epsilon) \quad \text{and} \quad \lim_{x\to \infty} \K(x)=\infty\, .
\end{gather}
In this case, $\K$ is called a \ka potential for the triple $(K,\Omega,\epsilon)$.
\end{deph}
\begin{remark}
\rm The condition $\K(\Omega)\subset [0,\epsilon)$ might sound unnatural in this context. However, as we shall see, it is necessary in order to prove our main Theorem. Moreover, it is easily seen to be superfluous if the operator $\L$ is half linear. Indeed, if $\alpha \K$ remains a supersolution for every positive $\alpha$, then by choosing a sufficiently small but positive $\alpha$, $\alpha \K$ can be made as small as wanted on any fixed bounded set.
\end{remark}

Once all the operators and properties that we need have been defined, the statement of the main Theorem is surprisingly easy.
\begin{teo}\label{mainteo}
Let $M$ be a Riemannian manifold, and $\L$ be an $\L$-type operator. Then the Liouville property \eqref{Liou} and the \ka property \eqref{ka} are equivalent.
\end{teo}
In the following section we will introduce some important properties and technical tools related to $\L$ supersolutions which will be needed for the proof of the main Theorem.

\subsection{Technical tools}\label{sec_tech}
We start our technical discussion by pointing out some regularity properties of $\L$-supersolutions and some mathematical tools available for their study. In particular, we will mainly focus on the Pasting Lemma and the obstacle problem. When possible, we will refer the reader to specific references for the proofs of the results listed in this section, especially for those results considered standard in the field.

Throughout this section, $\L$ will denote an $\L$-type operator. We start by recalling the famous comparison principle.
\begin{prop}[Comparison principle]
Fix an operator $\L$ on a Riemannian manifold $M$, and let $w$ and $s$ be respectively a supersolution and a subsolution. If $\min\{w-s,0\}\in \wupz$, then $w\geq s$ a.e. in $\Omega$.
\end{prop}
\begin{proof}
\index{comparison principle!generalized comparison principle@$\L$-type comparison principle}
This theorem and its proof, which follows quite easily using the right test function in the definition of supersolution, are standard in potential theory. For a detailed proof see \cite[Theorem 4.1]{antoninimugnaipucci}.
\end{proof}
Another important tool in potential theory is the so-called subsolution-supersolution method. Given the assumptions we made on the operator $\L$, it is easy to realize that this method can be applied to our situation.
\index{sub-supersolution method}
\begin{teo}\cite[Theorems 4.1, 4.4 and 4.7]{kuratake}\label{subsuper}
Let $s,w \in L^\infty_\loc(\Omega) \cap W^{1,p}_\loc(\Omega)$ be, respectively, a subsolution and a supersolution for $\L$ on a domain $\Omega$. If $s \leq w$ a.e. on $\Omega$, then there exists a solution $u\in L^\infty_\loc\cap W^{1,p}_\loc$ of $\L u=0$ satisfying $s\leq u \leq w$ a.e. on $\Omega$. 
\end{teo}
The main regularity properties of solutions and supersolutions are summarized in the following Theorem.
\begin{teo}\label{regularity}
Fix $\L$ and let $\Omega\subset M$ be a bounded domain. 
\begin{itemize}
\item[(i)] [\cite{ZM}, Theorem 4.8] If $u$ is a supersolution on a domain $\Omega$, i.e., if $\L u \le 0$ on $\Omega$, there exists a representative of $u$ in
$\wup$ which is lower semicontinuous. In particular, any solution $u$ of $\L(u)=0$ has a continuous representative.
\item[(ii)] [\cite{ladyura}, Theorem 1.1 p. 251] If $u\in L^\infty(\Omega) \cap \wup$ is a bounded solution of $\L u=0$ on a bounded $\Omega$, then there exists $\alpha \in (0,1)$ such that $u\in C^{0,\alpha}(\Omega)$. The parameter $\alpha$ depends only on the geometry of $\Omega$, on the properties of $\L$ \footnote{in particular, on the constants $a_i,b_i$ appearing in definition \ref{deph_L}} and on $\norm u _\infty$. Furthermore, for every $\Omega_0\Subset \Omega$, there exists $C=C(\L,\Omega,\Omega_0,\norm u _\infty)$ such that
\begin{gather}
 \norm u_{C^{0,\alpha}(\Omega_0)}\leq C\, .
\end{gather}
\end{itemize}
\end{teo}
\index{Caccioppoli inequality}
Caccioppoli-type inequalities are an essential tool in potential theory. Here we report a very simple version of this inequality that fits our needs. More refined inequalities are proved, for example, in \cite[Theorem 4.4]{ZM}
\begin{prop}\label{caccio}
Let $u$ be a bounded solution of $\L u \le 0$ on $\Omega$. Then, for every relatively compact, open set $\Omega_0\Subset \Omega$ there is a constant $C>0$ depending on $\L, \, \Omega, \, \Omega_0$ such that
\begin{gather}
\norm{\nabla u}_{L^p(\Omega_0)} \le C(1+\norm{u}_{L^\infty(\Omega)})\, .
\end{gather}
\end{prop}
\begin{proof}
For the sake of completeness, we sketch a simple proof of this Proposition. We will use the notation
\begin{gather*}
 u_\star = \operatorname{essinf}\{u\}\, , \quad \quad u^\star = \operatorname{esssup}\{u\}\, .
\end{gather*}

Given a supersolution $u$, by the monotonicity of $B$ for every positive constant $c$ also $u+c$ is a supersolution, so without loss of generality we may assume that $u_\star \ge 0$. Thus, $u^\star=\norm{u}_{L^\infty(\Omega)}$. Let $\eta\in C^\infty_c(\Omega)$ be such that $0\le \eta \le 1$ on $\Omega$ and $\eta=1$ on $\Omega_0$. If we use the non-negative function $\phi= \eta^p(u^\star-u)$ in the definition of supersolution, we get, after some manipulation
\begin{equation}\label{prima}
\int_\Omega \eta^p\ps{A(\nabla u)}{\nabla u}dV \leq p\int_\Omega \ps{A(\nabla u)}{\nabla \eta}\eta^{p-1}(u^\star-u) dV + \int_\Omega \eta^p B(x,u)u^\star\, .
\end{equation}
Using the properties of $\L$ described in definition \ref{deph_L}, we can estimate
\begin{equation}\label{seconda}
a_1\int_\Omega \eta^p\abs{\nabla u}^p dV \leq pa_2 \int_{\Omega} \eta^{p-1}(u^\star-u)\abs{\nabla u}^{p-1} \abs{\nabla \eta} dV+ \operatorname{Vol}(\Omega)(b_1u^\star + b_2(u^\star)^p)\, .
\end{equation}
By a simple application of Young's inequality, we can estimate that
\begin{equation}\label{terza}
\begin{array}{lcl}
\disp pa_2 \int_\Omega \big(|\nabla
u|^{p-1}\eta^{p-1}(u^\star-u)|\nabla \eta|\big) & \le & \disp pa_2
\int_\Omega \big(|\nabla
u|^{p-1}\eta^{p-1}\big)\big(u^\star|\nabla \eta|\big) \\[0.4cm]
& \le & \frac{a_2}{\eps^p} \norm{\eta \nabla u}_p^p +
\frac{a_2p\eps^q}{q} \norm{\nabla \eta}^p_p(u^\star)^p
\end{array}
\end{equation}
for every $\epsilon>0$. Choosing $\eps$ such that $a_2\eps^{-p}=a_1/2$, and rearranging the terms, we obtain
\begin{gather*}
\frac{a_1}{2} \norm{\eta\nabla u}_p^p \le C \Big[1+(1+\norm{\nabla \eta}_p^p)(u^\star)^p \Big]\, .
\end{gather*}
Since $\eta=1$ on $\Omega_0$ and $\norm{\nabla \eta}_p \le C$, taking the $p$-root the desired estimate follows.
\end{proof}
Now, we fix our attention on the obstacle problem. There are a lot of references regarding this subject \footnote{see for example \cite[Chapter 5]{ZM}, or \cite[Chapter 3]{HKM} in the case $B =0$}, and as often happens, notation can be quite different from one reference to another. Here we follow conventions used in \cite{HKM}.

First of all, some definitions. Consider a domain $\Omega\subset M$, and fix any function $\psi:\Omega\to \R\cup{\pm\infty}$, and $\theta\in W^{1,p}(\Omega)$. Define the closed convex set $\W_{\psi,\theta}\subset W^{1,p}(\Omega)$ by
\begin{gather}
 \W_{\psi,\theta} =\{f\in \wup \ \vert \  \ f\geq \psi  \ \text{ a.e. and } \ f-\theta\in
\wupz\}\, .
\end{gather}
\begin{definition}
\index{obstacle problem!generalized obstacle problem@$\L$-type obstacle problem}
 Given $\psi$ and $\theta$ as above, $u$ is a solution to the obstacle problem relative to $\Kpt$ if $u\in \Kpt$ and for every $\varphi\in \Kpt$
\begin{gather}\label{eq_obs}
\int_{\Omega}\L (u) (\varphi -u) dV = -\int_{\Omega} \ton{\ps{A(\nabla u)}{\nabla(\varphi -u)} + B(x,u) (\varphi -u)}dV  \leq 0\, .
\end{gather}
\end{definition}
Loosely speaking, $\theta$ determines the boundary condition for the solution $u$, while $\psi$ is the ``obstacle''-function. Most of the times, obstacle and boundary function coincide, and in this case we use the convention $\W_{\theta}= \W_{\theta,\theta}$. 

\begin{remark}
\rm Note that for every nonnegative $\phi\in \cc$ the function $\varphi= u+\phi$ belongs to $\Kpt$, and this implies that the solution to the obstacle problem is always a supersolution.
\end{remark}
The obstacle problem is a sort of generalized Dirichlet problem. Indeed, it is easily seen that if we choose the obstacle function $\psi$ to be trivial (i.e., $\psi=-\infty$), then the obstacle problem and the Dirichlet problem coincide.

It is only natural to ask how many different solutions an obstacle problem can have. Given the similarity with the Dirichlet problem, the following uniqueness theorem should not be surprising.
\begin{teo}\label{existostacolo}
If $\Omega$ is relatively compact and $\W_{\psi,\theta}$ is nonempty, then there exists a unique solution to the corresponding obstacle problem.
\end{teo}
\begin{proof}
The proof is a simple application of Stampacchia's theorem. Indeed, using the assumptions in definitions \ref{deph_L}, it is easy to see that the operator $\F:W^{1,p}(\Omega)\to W^{1,p}_0 (\Omega)^\star$ defined by
\begin{gather}
 \ps{\F(u)}{\phi}= \int_{\Omega} \ton{\ps{A(\nabla u)}{\nabla \phi} + B(x,u)\phi} dV
\end{gather}
is weakly continuous, monotone and coercive. 

For a more detailed proof, we refer to \cite[Appendix 1]{HKM} \footnote{even though this books assumes $B(x,u)=0$, there is no substantial difference in the proof}; as for the celebrated Stampacchia's theorem, we refer to \cite[Theorem III.1.8]{KS}.
\end{proof}
Using the comparison theorem, the solution of the obstacle problem can be characterized by the following Proposition.
\begin{prop}\label{corsuper}
The solution $u$ to the obstacle problem in $\Kpt$ is the smallest supersolution in $\Kpt$.
\end{prop}
\begin{proof}
Let $u$ be the unique solution to the obstacle problem $\Kpt$, and let $w$ be a supersolution such that $\min\{u,w\}\in \Kpt$. We will show that $u\leq w $ a.e.

Define $U=\{x \vert \ u(x)>w(x)\}$, and suppose by contradiction that $U$ has positive measure. Since $u$ solves the obstacle problem, we can use the function $\varphi=\min\{u,w\}\in \Kpt$ in equation \eqref{eq_obs} and obtain
\begin{equation}
0 \leq -\int_{U}\L(u)\varphi dV = \int_U \ps{A(\nabla u)}{\nabla w-\nabla u} + \int_U B(x,u)(w-u)\, .
\end{equation}
On the other hand, since $w$ is a supersolution on $U$ we have
\begin{equation}
0 \leq -\int_{U}\L(u)(u-\min\{u,w\}) dV= \int_U \ps{A(\nabla w)}{\nabla u-\nabla w} + \int_U B(x,w)(u-w)\, .
\end{equation}
Adding the two inequalities, and using the monotonicity of $A$ and $B$, we have the chain of inequalities
\begin{gather}
0 \leq \int_U \ps{A(\nabla u)- A(\nabla w)}{\nabla w -\nabla u} +\int_U \big[B(x,u)-B(x,w)\big](w-u) \le 0\, .
\end{gather}

The strict monotonicity of the operator $A$ forces $u\leq w$ a.e. in $\Omega$. 
\end{proof}

As an immediate Corollary of this characterization, we get the following.
\begin{cor}\label{minsupmaxsub}
Let $w_1,w_2\in W^{1,p}_\loc(M)$ be supersolutions for $\L$. Then also $w =\min\{w_1,w_2\}$ is a supersolution. In a similar way, if $u_1,u_2\in W^{1,p}_\loc(M)$ are subsolutions for $\L$, then so is $u = \max\{u_1,u_2\}$.
\end{cor}
\index{pasting lemma}

A similar statement is valid also if the two functions have different domains. This result is often referred to as ``Pasting Lemma'' (see for example \cite[Lemma 7.9]{HKM}).
\begin{lemma}\label{pastingCMPT}
Let $\Omega$ be an open bounded domain, and $\Omega'\subset \Omega$. Let $w_1 \in W^{1,p}(\Omega)$ be a supersolution for $\L$, and $w_2 \in W^{1,p}(\Omega')$ be a supersolution on $\Omega'$.If in addition $\min\{w_2-w_1,0\} \in W^{1,p}_0(\Omega')$, then
\begin{gather}
w= \begin{cases}
            \min\{w_1,w_2\} & \quad \text{on } \Omega' \\
            w_1 & \quad \text{on } \Omega \setminus \Omega'
           \end{cases}
\end{gather}
is a supersolution for $\L$ on $\Omega$. 
\end{lemma}

Being a supersolution is a local property, so it is only natural that, using some extra caution, it is possible to prove a similar statement also if $\Omega$ is not compact.

Note that, since we need to study supersolution on unbounded domains $\Omega$, the space $W^{1,p}(\Omega)$ is too small for our needs. We need to set our generalized Pasting Lemma in $W^{1,p}_\loc(\Omega)$. For the same reason, we need to replace $W^{1,p}_0(\Omega)$ with
\begin{gather}
X^{1,p}_0 (\Omega)= \cur{f\in W^{1,p}_\loc (\Omega) \ \ \vert  \ \ \exists \phi_n \in C^{\infty}_C(\Omega) \ \text{s.t.} \  \phi_n\to f \ \ \text{in } W^{1,p}_\loc (\Omega)} \, .
\end{gather}

With this definition, it is easy to state a generalized Pasting Lemma.

\begin{lemma}\label{pasting}
Let $\Omega$ be an domain (possibly not bounded), and $\Omega'\subset \Omega$. Let $w_1 \in W^{1,p}_\loc(\Omega)$ be a supersolution for $\L$, and $w_2 \in W^{1,p}_\loc(\Omega')$ be a supersolution on $\Omega'$. If $\min\{w_2-w_1,0\} \in X^{1,p}_0(\Omega')$, then
\begin{gather}
w= \begin{cases}
            \min\{w_1,w_2\} & \quad \text{on } \Omega' \\
            w_1 & \quad \text{on } \Omega \setminus \Omega' .
           \end{cases}
\end{gather}
is a supersolution for $\L$ on $\Omega$.
\end{lemma}

\begin{proof}
It is straightforward to see that $w \in W^{1,p}_\loc(M)$. In order to prove that $w$ is a supersolution, only a simple adaptation of the proof of Proposition \ref{corsuper} is needed. \end{proof}

Now we turn our attention to the regularity of the solutions to the obstacle problem. Without adding more regularity assumption on $\L$, the best we can hope for is \ho continuity.
\begin{teo}[\cite{ZM}, Theorem 5.4 and Corollary 5.6]\label{continuosta}
If the obstacle $\psi$ is continuous in $\Omega$, then the solution $u$ to $\Kpt$ has a continuous representative in the Sobolev sense.

Moreover if $\psi \in C^{0,\alpha}(\Omega)$ for some $\alpha\in (0,1)$, then for every $\Omega_0\Subset \Omega$ there exist $C,\beta>0$ depending only on $\alpha,\Omega,\Omega', \L $ and $\norm{u}_{L^\infty(\Omega)}$ such that
\begin{gather}
\norm{u}_{C^{0,\beta}(\Omega_0)}\leq C(1+\norm{\psi}_{C^{0,\alpha}(\Omega)})\, .
\end{gather}
\end{teo}
\begin{oss}
\rm{We note that stronger results, for instance $C^{1,\alpha}$ regularity, can be obtained from stronger assumptions on $\psi$ and $\L$. See for instance \cite[Theorem 5.14]{ZM}.}
\end{oss}
A solution $u$ to an obstacle problem is always a supersolution. However, something more can be said if $u$ is strictly above the obstacle $\psi$.
\begin{prop}\label{esoluzione}
 Let $u$ be the solution of the obstacle problem $\Kpt$ with continuous obstacle $\psi$. If $u>\psi$ on an open set $D$, then $\L u=0$ on $D$.
\end{prop}
\begin{proof}
Consider any test function $\phi\in C^{\infty}_c(D)$. Since $u>\psi$ on $D$, and since $\phi$ is bounded, by continuity there exists $\delta>0$ such that $u\pm \delta\phi \in \Kpt$. From the definition of solution to the obstacle problem we have that
\begin{gather*}
\pm \int_{D}\L(u)\phi dV = \pm \frac{1}{\delta} \int_{D}\L(u)(\delta\phi) dV  =  \frac{1}{\delta} \int_{D}\L(u)[(u\pm \delta \phi)-u] dV   \geq 0\, ,
\end{gather*}
hence $\L u =0$ as required.
\end{proof}

Also in this general case, boundary regularity for solutions of the Dirichlet and obstacle problems is a well-studied field in potential theory. However, since we are only marginally interested in boundary regularity, we will only deal with domains with Lipschitz boundary.
\begin{teo}\label{teobjorn}
 Consider the obstacle problem $\W_{\psi,\theta}$ on $\Omega$, and suppose that $\Omega$ has Lipschitz boundary and both $\theta$ and $\psi$ are continuous up to the boundary. Then the solution $w$ to $\W_{\psi,\theta}$ is continuous up to the boundary.
\end{teo}
\begin{proof}
 For the proof of this theorem, we refer to \cite[Theorem 2.5]{GZ}. Note that this reference proves boundary regularity only for the solution to the Dirichlet problem with $1<p\leq m$ \footnote{recall that $m$ is the dimension of the manifold $M$}. However, the case $p>m$ is absolutely trivial because of the Sobolev embedding theorems. Moreover, by a simple trick involving the comparison theorem it is possible to use the continuity of the solution to the Dirichlet problem to prove continuity of the solution to the obstacle problem.
\end{proof}

Finally, we present some results on convergence of supersolutions and their approximation with regular ones.
\begin{prop}\label{convergence}
Let $w_j$ be a sequence of supersolutions on some open set $\Omega$. Suppose that either $w_j\uparrow w$ or $w_j \downarrow w$ pointwise monotonically, where $w$ is locally bounded function. Then $w$ is a supersolution in $\Omega$.
\end{prop}
\begin{proof}
We follow the scheme outlined in \cite[Theorems 3.75, 3.77]{HKM}. Suppose that $w_j\uparrow~w$, and assume without loss of generality that $w_j$ are lower semicontinuous. Hence, $\norm {w_j}_\infty$ is uniformly bounded on compact subsets of $\Omega$. By the elliptic estimates in Proposition \ref{caccio}, $w_j$ are also locally bounded in the $\wup$ sense.

Fix a smooth exhaustion $\{\Omega_n\}$ of $\Omega$. For each $n$, up to passing to a subsequence, $w_j \rightharpoonup z_n$ weakly in $W^{1,p}(\Omega_n)$ and strongly in $L^p(\Omega_n)$. By uniqueness of the limit, $z_n=w$ for every $n$, hence $w\in W^{1,p}_\loc(\Omega)$.

With a Cantor argument, we can select a subsequence \footnote{which we will still denote by $w_j$} that converges to $w$ both weakly in $W^{1,p}(\Omega_n)$ and strongly in $L^p(\Omega_n)$ for every fixed $n$. To prove that $w$ is a supersolution, fix $0\le\eta \in C^\infty_c(\Omega)$, and choose a smooth relatively compact open set $\Omega_0\Subset \Omega$ that contains the support of $\eta$. Define $M =\max_j \norm{w_j}_{W^{1,p}(\Omega_0)}<+\infty$.\\
Since $w_j$ is a supersolution and $w\geq w_j$ for every $j$
\begin{gather}
 \int_{\Omega} \L(w_j) \eta (w-w_j) dV \leq 0\, ,
\end{gather}
or equivalently
\begin{equation}\label{primostep}
\int \ps{A(\nabla w_j)}{\eta(\nabla w-\nabla w_j)} \ge -\int
\Big[B(x, w_j)+ \ps{A(\nabla w_j)}{\nabla \eta}\Big](w-w_j)\, .
\end{equation}
By the properties of $\L$, we can bound the RHS from below with
\begin{equation}
\begin{array}{l}
\disp -b_1\norm{\eta}_{L^\infty(\Omega)}\int_{\Omega_0}(w-w_j) -
b_2\norm{\eta}_{L^\infty(\Omega)}
\int_{\Omega_0}|w_j|^{p-1}|w-w_j| \\[0.4cm]
\disp  - a_2\norm{\nabla\eta}_{L^\infty(\Omega)}
\int_{\Omega_0}|\nabla
w_j|^{p-1}|w-w_j| \geq\\[0.4cm]
\ge -\norm{\eta}_{C^1(\Omega)}\Big[ b_1|\Omega_0|^{\frac{p-1}{p}}
- b_2 M^{p-1} - a_2M^{p-1}\Big]\norm{w-w_j}_{L^p(\Omega_0)}\, .
\end{array}
\end{equation}
Combining with \eqref{primostep} and the fact that $w_j \rightharpoonup w$ weakly on $W^{1,p}(\Omega_0)$, and using the mon\-o\-tonicity of $A$, we can prove that
\begin{equation}\label{intazero}
0 \le \int \eta \ps{A(\nabla w)-A(\nabla w_j)}{\nabla w-\nabla
w_j} \le o(1) \qquad \text{as } \ j\ra +\infty\, .
\end{equation}
By a lemma due to F. Browder (see \cite{browder}, p.13 Lemma 3), this implies that $w_j \rightarrow w$ strongly in $W^{1,p}(\Omega_0)$. As a simple consequence, we have that for any $\eta\in C^{\infty}_C(\Omega)$
\begin{gather}
 0\geq \int_{\Omega} \L(w_j)\eta dV \to \int_\Omega \L(w) \eta dV\, .
\end{gather}

The case $w_j \downarrow w$ is simpler. Let $\{\Omega_n\}$ be a smooth exhaustion of $\Omega$, and let $u_n$ be the solution of the obstacle problem relative to $\Omega_n$ with obstacle and boundary value $w$. Then, by \eqref{corsuper} $w\le u_n \le w_j|_{\Omega_n}$, and letting $j\ra +\infty$ we deduce that $w=u_n$ is a supersolution on $\Omega_n$ for each $n$.\\
\end{proof}
It is easy to see that we can relax the assumption of local boundedness on $w$ if we assume a priori $w\in W^{1,p}_\loc(\Omega)$. Moreover with a simple trick we can prove that also local uniform convergence preserves the supersolution property, as in \cite[Theorem 3.78]{HKM}.
\begin{cor}
 Let $w_j$ be a sequence of supersolutions locally uniformly converging to $w$, then $w$ is a supersolution.
\end{cor}
\begin{proof}
 The trick is to transform local uniform convergence into monotone convergence. Fix any relatively compact $\Omega_0 \Subset \Omega$ and a subsequence of $w_j$ (denoted for convenience by the same symbol) with $\norm{w_j-w}_{L^\infty(\Omega_0)} \leq 2^{-j}$. The modified sequence of supersolutions
\begin{gather*}
\tilde w_j = w_j + \frac 3 2 \sum_{k=j}^\infty 2^{-k}= w_j+3 \times 2^{-j}  
\end{gather*}
is easily seen to be a monotonically decreasing sequence on $\Omega_0$. Since its limit is still $w$, we can conclude applying the previous Proposition.
\end{proof}

Now we prove that with locally \ho continuous supersolutions we can approximate every supersolution.

\begin{prop}\label{suxho}
 For every supersolution $w\in W^{1,p}_\loc(\Omega)$, there exists a sequence $w_j$ of locally \ho continuous supersolutions converging monotonically from below and in $W^{1,p}_\loc(\Omega)$ to $w$. The same statement is true for subsolutions with monotone convergence from above.
\end{prop}
\begin{proof}
Since every $w$ has a lower-semicontinuous representative, it can be assumed to be locally bounded from below, and since $w^{(m)}=\min\{w,m\}$ is a supersolution (for $m\geq 0$) and converges monotonically to $w$ as $m$ goes to infinity, we can assume without loss of generality that $w$ is also bounded above.

Moreover, lower semicontinuity implies that there exists a sequence of smooth functions $\phi_j$ converging monotonically from below to $w$. Let $\Omega_k$ be a smooth exhaustion of $\Omega$, and let $w_j^{(k)}$ be the unique solution to the obstacle problem on $\Omega_k$ with obstacle and boundary values $\phi_j$. By the regularity Theorem \ref{continuosta}, for a fixed compact set $\Omega_0$ there exists a positive $\beta$ and $C$ independent of $k$ such that
\begin{gather}
 \norm{w_j^{(k)}}_{C^{0,\beta}(\Omega_0)}\leq C\, .
\end{gather}
Then for $k\to \infty$, $w_j^{(k)}$ converges to a locally \ho continuous supersolution $w_j$. Since $\phi_j\uparrow w$, also $w_j\uparrow w$.
\end{proof}

\subsection{Proof of the main Theorem}
For the reader's convenience, we recall the precise statement of the main Theorem before its proof.
\begin{teo}
Let $M$ be a Riemannian manifold, and $\L$ be an $\L$-type operator. Then the following properties are equivalent:
\begin{itemize}
\item[\rm{(Li)}]  $u\in L^\infty(M) \ \wedge \ u\geq 0 \ \wedge \ \L(u)\geq 0 \ \Longrightarrow \ u \text{ is constant}$;
\item[\rm{(Ka)}] every triple $(K,\Omega,\epsilon)$ admits a \ka potential (see definition \ref{deph_ka}).
\end{itemize}
\end{teo}
\begin{proof}
First of all, note that in the Liouville property we can assume without loss of generality that $u$ is also a locally \ho continuous function (see Proposition \ref{suxho}).

(Ka)$\Rightarrow$(Li) is the easy implication. In order to prove it, let $u^\star=\esssup\{u\}<\infty$ and suppose by contradiction that $u$ is not constant. Then there exists a positive $\epsilon$ and a compact set $K$ such that $0\leq u|_K \leq u^\star - 2 \epsilon$.

For any relatively compact set $\Omega$, consider the $(K,\Omega,\epsilon)$ \ka potential $\K$, and let $\tilde \K = \K + u^\star-2\epsilon$. Note that $\tilde K$ is again a supersolution to $\L$ where defined. Since $\lim_{x\to \infty} \K(x)=\infty$ and $u$ is bounded, there exists a compact set $\Omega'$ such that $\tilde K(x)\geq u(x)$ for all $x\in \partial \Omega'$. Comparing the functions $\tilde K$ and $u$ on the open set $\Omega'\setminus K$, we notice immediately that $u\leq \tilde \K$, and in particular $u\leq u^\star-\epsilon$ on $\Omega$. Since $\Omega $ is arbitrary, we get to the contradiction $u^\star \leq u^\star -\epsilon$.

To prove the implication (Li)$\Rightarrow $(Ka) we adapt (and actually simplify) the proof of Theorem \ref{teo_khas_easy}. 
Fix a triple $(K,\Omega,\eps)$, and a smooth exhaustion $\{\Omega_j\}$ of $M$ with $\Omega\Subset \Omega_1$. We will build by induction an increasing sequence of
continuous functions $\{w_n\}$, $w_0 = 0$, such that:
\begin{enumerate}
 \item[(a)] $w_n|_{K}=0$, $w_n$ are continuous on $M$ and $\L (w_n) \le 0$ on $M\backslash K$,
 \item[(b)] for every $n$, $w_n\leq n$ on all of $M$ and $w_n=n$ in a large enough neighborhood of infinity denoted by $M\backslash C_n$,
 \item[(c)] $\norm{w_n}_{L^\infty(\Omega_n)} \le \norm{w_{n-1}}_{L^\infty(\Omega_n)}+ \frac{\eps}{2^n}$.
\end{enumerate}
Once this is proved, by $(c)$ the increasing sequence $\{w_n\}$ is locally uniformly convergent to a continuous exhaustion which, by Proposition \ref{convergence}, solves $L_{\oF} w\le 0$ on $M\setminus K$. Furthermore,
\begin{gather}
\norm{w}_{L^\infty(\Omega)} \leq \sum_{n=1}^{+\infty} \frac{\epsilon}{2^n} \leq \epsilon\, .
\end{gather}

In order to prove (a), (b) and (c) by induction, we define the sequence of continuous functions $h_j\in W^{1,p}_\loc(M)$ by
\begin{gather}
 h_j = \begin{cases}
        0 & \text{ on } K\\
        1 & \text{ on } \Omega_j^C \\
        \L(u)=0 & \text{ on } \Omega_j\setminus K
       \end{cases}
\end{gather}
Since $\Omega_j\setminus K$ has smooth boundary, there are no continuity issues for $h_j$. Observe that $h_j$ is a decreasing sequence, and so it has a nonnegative limit $h$. By the Pasting Lemma, this function is easily seen to be a bounded subsolution on all $M$, and thus it has to be zero. By Dini's theorem, $h_j$ converges locally uniformly to $h$.

We start the induction by setting $w_1 = h_j$, for $j$ large enough in order for property (c) to hold. Define $C_1=\Omega_j$, so that also $(b)$ holds. 

Suppose by induction that $w_n$ exists. For notational convenience, we will write $\bar w=w_n$ in the next paragraph. Consider the sequence of obstacle problems $\W_{\bar w + h_j}$ defined on $\Omega_{j+1}\setminus K$ with $j$ large enough such that $C_n \subset \Omega_{j}$. By Theorems \ref{continuosta} and \ref{teobjorn}, the solutions $s_j$ to these obstacle problems are continuous up the boundary of $\Omega_{j+1}\setminus K $. Thus we can easily extend $s_j$ to a function in $C^0(M)$ by setting it equal to $0$ on $K$ and equal to $n+1$ on $M\setminus \Omega_{j+1}$.

Since $s_j$ is easily seen to be monotone decreasing, we can define $\bar s\geq \bar w$ to be its limit. This limit is a continuous function, indeed if $\bar s(x)=\bar w(x)$, continuity follows from the continuity of $s_j$ and $\bar w$ \footnote{Since $s_j \searrow \bar s$, $\bar s$ is upper semicontinuous. Moreover, $\liminf_{y\to x} \bar s(y)\geq \lim_{y\to x} \bar w(y)=\bar w(x)=\bar s(x)$.}. 

If $\bar s(x) > \bar w(x)$, then for $j$ large enough and $\epsilon$ small enough
\begin{gather*}
 s_j(x)\geq \bar w + h_j + 2\epsilon\, .
\end{gather*}
Using Theorem \ref{continuosta}, we can find a uniform neighborhood $U$ of $x$ such that, for some positive $\alpha$ and $C$ independent of $j$
\begin{gather}
 \norm{s_j}_{C^{0,\alpha}(U)}\leq C\, .
\end{gather}
Thus $\bar s\in C^{0,\alpha/2}(U)$.

We use the Lioville property to prove that $\bar s \leq n$ everywhere. Suppose by contradiction that this is false, and so suppose that the open set
\begin{gather}
 V=\{\bar s > n\}
\end{gather}
is not empty. By Proposition \ref{esoluzione}, $\bar s$ satisfies $\L(u)=0$ on $V$, and by the Pasting Lemma \ref{pasting} we obtain that $\min\{\bar s,n\}-n$ is a nonnegative, nonconstant and bounded $\L$-subsolution, which is a contradiction.

In a similar way, $\bar s = \bar w$ everywhere. Indeed, on the open precompact set
\begin{gather}
 W=\{\bar s > \bar w\}\, ,
\end{gather}
$s$ is a solution and $w$ is a supersolution. Since $\bar s \leq \bar w$ on $\partial W$, this contradicts the comparison principle.

Summing up, we have shown that the sequence $s_j$ converges monotonically to $\bar w$, and the convergence is local uniform by Dini's theorem. Thus we can set $w_{n+1} = s_j$, where $j$ is large enough for property (c) to hold.

\end{proof}

 \section{Stokes' type theorems}\label{sec_J}

In this section, we present part of the results published in \cite{VV}. This article deals with some possible application of the \ka and Evans conditions introduced before.

We will present some new Stokes' type theorems on complete non-compact manifolds that extend, in different directions, previous work by Gaffney and Karp and also the so called Kelvin-Nevanlinna-Royden criterion for $p$-parabolicity. Applications to comparison and uniqueness results involving the $p$-Laplacian are deduced.

\subsection{Introduction}
In 1954, Gaffney \cite{G}, extended the famous Stokes' theorem to complete $m$-di\-men\-sio\-nal Riemannian manifolds $M$ by proving that, given a $C^1$ vector field $X$ on $M$, we have $\int_M\dive X=0$ provided $X\in L^1(M)$ and $\dive X\in L^1(M)$ (but in fact $(\dive X)_-=\max\{-\dive X, 0\}\in L^1(M)$ is enough).
This result was later extended by Karp \cite{K}, who showed that the assumption $X\in L^1(M)$ can be weakened to
\begin{gather}
\liminf_{R\to+\infty}\frac{1}{R}\int_{B_{2R}\setminus B_{R}} |X| dV_M=0\, .
\end{gather}
It is interesting to observe that completeness of a manifold is equivalent to $\infty$-parabolicity, i.e., $M$ is complete if and only if it is $\infty$-parabolic, so it is only natural to ask whether there is an equivalent of Gaffney's and Karp's result for a generic $p\in(1,\infty)$.
\index{Stokes' theorem}

A very useful characterization of $p$-parabolicity is the Kelvin-Nevanlinna-Royden criterion. In the linear setting $p=2$ it was proved in a paper by T. Lyons and D. Sullivan \cite{LS} \footnote{see also Pigola, Rigoli and Setti \cite[Theorem 7.27]{7i}}. The non-linear extension, due to V. Gol'dshtein and M. Troyanov \cite{GolTro}, states the following.
\begin{theorem}
\index{Kelvin-Nevanlinna-Royden criterion}
 Let $M$ be a complete noncompact Riemannian manifold. $M$ is $p$-hyperbolic if and only if there exists a vector field $X\in \Gamma(M)$ such that:
\begin{enumerate}
 \item[(a)] $\left\vert X\right\vert \in L^{p/(p-1)}\left(  M\right)  $,
 \item[(b)] $\operatorname{div}X\in L_{\text{\rm{loc}}}^{1}\left(  M\right) $ and $\left(  \operatorname{div}X\right)  _{-}\in L^{1}\left(  M\right)  $ (in the weak sense),
 \item[(c)] $0<\int_{M}\operatorname{div}X\leq+\infty$
\end{enumerate}
\end{theorem}
In particular this result shows that if $M$ is $p$-parabolic and $X$ is a vector field on $M$ satisfying (a) and (b), then $\int_M\dive X =0$, thus giving a $p$-parabolic analogue of the Gaffney result.

In this section we will prove similar results with weakened assumptions on the integrability of $X$. In some sense, we will obtain a generalization of Karp's theorem.

We will present two different ways to get this result. The first one, Theorem \ref{teo_ex}, is presented in subsection \ref{sec_evans} and relies on the existence of special exhaustion functions. It has a more theoretical taste and gives the desired $p$-parabolic analogue of Karp's theorem, at least on manifolds where it is easy to find an Evans potential.

The second one, Theorem \ref{th_gt}, is more suitable for explicit applications. In this Theorem the parabolicity assumption on $M$ is replaced by a control on the volume growth of geodesic balls in the manifold. In some sense, specified in Remark \ref{cutoff}, this result is optimal. 

In subsection \ref{appl} we use these techniques to generalize some results involving the $p$-Laplace operator comparison and uniqueness theorems on the $p$-harmonic representative in a homotopy class.

\subsection{Exhaustion functions and parabolicity}\label{sec_evans}

Given a continuous exhaustion function $f:M\to \R^+$ in $W^{1,p}_{\rm{loc}}(M)$, set by definition
\begin{gather}
C(r)=f^{-1}[0,2r)\setminus f^{-1}[0,r)\, .
\end{gather}
\begin{definition}
\index{EMP condition @$\mathcal{E}_{M,p}$ condition}
For fixed $p>1$ and $q>1$ such that $\frac 1 p + \frac 1 q =1$, we say that a vector field $X\in L^{q}_{\rm{loc}}(M)$ satisfies the condition \ref{cond_ex} if
\begin{equation}\label{cond_ex}\tag{$\mathcal{E}_{M,p}$}
 \liminf_{r\to \infty}\frac 1 r \abs{\int_{C(r)} \abs{\nabla{f}}^pdV_M}^{1/p}\abs{\int_{C(r)}\abs X ^{q} dV_M}^{1/q}=0 \, .
\end{equation}
\end{definition}

With this definition, we can state the first version of the generalized Karp theorem.
\begin{theorem}\label{teo_ex}
\index{Stokes' theorem!generalized Stokes' theorem}
Let $f:M\to \R^+$ be a continuous exhaustion function in $W^{1,p}_{\rm{loc}}(M)$. If $X$ is a $L^q_{\rm{loc}}(M)$ vector field with $(\dive X)_-\in L^1(M)$, $\dive(X)\in L^1_{\rm{loc}}(M)$ in the weak sense and $X$ satisfies the condition \ref{cond_ex}, then $\int_M \dive(X) dV_M=0$.
\end{theorem}
\begin{proof}
Note that $(\dive X)_-\in L^1(M)$ and $\dive(X)\in L^1_{\rm{loc}}(M)$ in the weak sense is the most general hypothesis under which $\int_M \dive X dV$ is well defined (possibly infinite). For $r>0$, consider the $W^{1,p}$ functions defined by
\begin{equation*}
 f_r(x) := \max\{\min\{2r-f(x),r\},0\}\, ,
\end{equation*}
i.e., $f_r$ is a function identically equal to $r$ on $D(r):= f^{-1}[0,r)$, with the support in $D(2r)$ and such that $\nabla {f_r}=-\chi_{C(r)}\nabla f $, where $\chi_{C(r)}$ is the characteristic function of $C(r)$. By dominated and monotone convergence we can write
\begin{equation*}
 \int_M \dive X dV_M=\lim_{r\to \infty}\frac 1 r \int_{D(2r)} f_r \dive X dV_M \, .
\end{equation*}
Since $f$ is an exhaustion function, $f_r$ has a compact support, by definition of weak divergence we get
\begin{align*}
\int_M \dive X dV_M &=\lim_{r\to\infty} \frac 1 r\int_{D(2r)}\ps{\nabla f_r}{X} dV_M\\ &\leq\liminf_{r\to \infty} \frac 1 r\ton{\int_{C(r)} \abs{\nabla f}^p dV_M}^{1/p}\ton{\int_{C(r)} \abs{X}^q dV_M}^{1/q}=0\, .
\end{align*}
This proves that $(\dive X)_+:=\dive X+(\dive X)_-\in L^1(M)$ and by exchanging $X$ with $-X$, the claim follows.
\end{proof}

Note that setting $p=\infty$ and $f(x)=r(x)$, one gets exactly the statement of Karp \cite{K}.

\begin{remark}
\rm{From the proof of Theorem \ref{teo_ex}, it is easy to see that condition \ref{cond_ex} can be generalized a little. In fact, if there exists a function $g:(0,\infty)\to(0,\infty)$ (without any regularity assumptions) such that $g(t)>t$ and
\begin{equation*}
\liminf_{r\to \infty}\frac 1 {g(r)-r} \abs{\int_{G(r)} \abs{\nabla{f}}^pdV_M}^{1/p}\abs{\int_{G(r)}\abs X ^q dV_M}^{1/q}=0\, ,
\end{equation*}
where $G(r)\equiv f^{-1}[0,g(r))\setminus f^{-1}[0,r]$, the conclusion of Theorem \ref{teo_ex} is still valid with the same proof, only needlessly complicated by an awkward notation.}
\end{remark}

The smaller the value of $\int_{C(r)} \abs{\nabla f}^p dV_M$ is, the more powerful the conclusion of the Theorem is, and since $p$-harmonic functions are in some sense minimizers of the $p$-Dirichlet integral, it is natural to look for such functions as candidates for the role of $f$. Of course, if $M$ is $p$-parabolic it does not admit any positive nonconstant $p$-harmonic function defined on all $M$. Anyway, since we are interested only in the behaviour at infinity of functions and vector fields involved (i.e., the behaviour in $C(r)$ for $r$ large enough), it would be enough to have a $p$-harmonic function $f$ which is defined outside a compact set (inside it could be given any value without changing the conclusions of the Theorem).

It is for these reasons that the Evans potentials are the natural candidates for $f$.

Using the estimates on the $p$-Dirichlet integral of the Evans potentials obtained in subsection \ref{sec_ev_e}, we can rephrase Theorem \ref{teo_ex} as follows.
\begin{theorem}
\index{Evans potential}
\index{Stokes' theorem!generalized Stokes' theorem}
Let $M$ be a $p$-parabolic Riemannian manifold. If $p\neq 2$, suppose also that $M$ is a model manifold, so that $M$ admits an Evans potential $\E$. If $X$ is a $L^q_{\rm{loc}}(M)$ vector field with $(\dive X)_-\in L^1(M)$, $\dive(X)\in L^1_{\rm{loc}}(M)$ in the weak sense and 
\begin{gather}
 \liminf_{r\to \infty}\frac{\int_{r\leq\E(x)\leq 2r}\abs{X}^q dV }{r} =0\, ,
\end{gather}
then $\int_M \dive(X) dV_M=0$.
\end{theorem}
This result is very similar (at least formally) to Karp's, except for the different exponents and for the presence of the Evans potential $E(\cdot)$ that plays the role of the geodesic distance $r(\cdot)$.

Even though this result is interesting from the theoretical point of view, it is of difficult application. Indeed, unless $M$ is a model manifold, in general there is no explicit characterization of the potential $\E$ which can help to estimate its level sets, and therefore the quantity $\int_{\{r\leq \E(x)\leq 2r\}}\abs X ^q dV$. For this reason, in the next section we adapt the technique used here to get more easily applicable results.

\subsection{Stokes' Theorem under volume growth assumptions}\label{volume}

In this section we use volume growth assumptions, rather than parabolicity, to prove a stronger version of Stokes' theorem. Recall that there is a strong link between parabolicity and volume growth.

Consider a Riemannian manifold $M$ with a pole $o$ and let $V(t)$ be the volume of the geodesic ball of radius $t$, and $A(t)$ be the surface of the same ball \footnote{the function $V$ is absolutely continuous, and $A$ is its (integrable) derivative defined a.e.. For the details see for example \cite[Proposition III.3.2]{chavel})}. Define $a_p(t):(0,+\infty)\to(0,+\infty)$ by
\begin{gather}
 a_p(t)= A(t)^{-1/(p-1)} \, .
\end{gather}
\begin{proposition}\label{prop_cfrcap}
 For a Riemannian manifold $M$, $a_p\not \in L^{1}(1,\infty)$ is a sufficient condition for the $p$-parabolicity of $M$. 
\end{proposition}
\begin{proof}
 This Proposition is an immediate corollary to \cite[Theorem 7.1]{gri} \footnote{which is easily seen to be valid also for all $p\in (1,\infty)$}. Indeed, the thesis follows from the capacity estimate
\begin{equation}\label{eq_cap_s}
 \cp (\bar B_{r_1},B_{r_2})\leq \ton{\int_{r_1}^{r_2} a_p(t) dt}^{1-p}\, .
\end{equation}
Note that a similar estimate is valid also with the function $V(t)$
\begin{equation}\label{eq_cap_v}
 \cp(\bar B_{r_1},B_{r_2})\leq 2^p\ton{\int_{r_1}^{r_2} \ton{\frac{T-r_1}{V(t)-V(r_1)}}^{1/(p-1)} dt}^{1-p}\, .
\end{equation}
\end{proof}
Recall that, if $M$ is a model manifold, then $a_p\not \in L^{1}(1,\infty)$ is also a necessary condition for $p$-parabolicity.

The function $a_p$ can be used to construct special cutoff functions with controlled $p$-Dirichlet integral. Using these cutoffs, with an argument similar to the one we used in the proof of Theorem \ref{teo_ex}, we get a more suitable and manageable condition on a vector field $X$ in order to guarantee that $\int_M \dive(X)dV_M=0$.

\begin{definition}
\index{AMP condition @$\mathcal{A}_{M,p}$ condition}
We say that a real function $f:M\to\mathbb R$ satisfies the condition \ref{cond_vol} on $M$ for some $p>1$ if
\begin{equation}\label{cond_vol}\tag{$\mathcal{A}_{M,p}$}
\liminf_{R\to\infty}\left(\int_{B_{2R}\setminus B_R}fdV_M\right)\left(\int_R^{2R} a_p(s) ds\right)^{-1}=0\, .
\end{equation}
\end{definition}

The next result gives the announced generalization under volume growth assumption of the Kelvin-Nevanlinna-Royden criterion.

\begin{theorem}\label{th_gt}
\index{Stokes' theorem!generalized Stokes' theorem}
Let $(M,\left\langle ,\right\rangle)$ be a non-compact Riemannian manifold. Let $X$ be a vector field on $M$ such that
\begin{gather}\label{liploc}
\operatorname{div}X\in L_{\text{\rm{loc}}}^{1}\left(  M\right)\, ,\\
\max\left(-\operatorname{div}X,0\right)  =\left(  \operatorname{div}X\right)  _{-}\in L^{1}\left(  M\right)\, .
\end{gather}
If $|X|^{p/(p-1)}$ satisfies the condition \ref{cond_vol} on $M$, then
\begin{gather*}
\int_M\operatorname{div}X dV_M=0\, .
\end{gather*}
\end{theorem}

In particular, if $X$ is a vector field such that $\abs X ^{p/(p-1)}$ satisfies the condition \ref{cond_vol}, $\operatorname{div}X$ is nonnegative and integrable on $M$, then we must necessarily conclude that
$\operatorname{div}X=0$ on $M$. As a matter of fact, even if $\operatorname{div}X\notin L_{\text{\rm{loc}}}^{1}\left(  M\right)  $, we can obtain a similar conclusion as shown in the next Proposition, inspired by \cite[Proposition 9]{HPV}.

\begin{proposition}\label{weak_KNR}
Let $(M,\left\langle ,\right\rangle)$ be a non-compact Riemannian manifold. Let $X$ be a vector field on $M$ such that
\begin{equation}\label{div_deb}
\operatorname{div}X\geq f
\end{equation}
in the sense of distributions for some $f\in L^1_{\text{\rm{loc}}}(M)$ with $f_-\in L^1(M)$. If $|X|^{p/(p-1)}$ satisfies condition \ref{cond_vol} on $M$ for some $p>1$, then
\begin{equation}\label{int_f}
\int_M f \leq 0\, .
\end{equation}
\end{proposition}

\begin{proof}
Fix $r_2>r_1>0$ to be chosen later. Define the functions $\hat\varphi=\hat\varphi_{r_1,r_2}:B_{r_2}\setminus B_{r_1}\to\mathbb R$ as
\begin{equation}\label{a1}
\hat\varphi(x):=\left(\int_{r_1}^{r_2}a_p(s)ds\right)^{-1}\int_{r(x)}^{r_2}a_p(s)ds
\end{equation}
and let $\varphi=\varphi_{r_1,r_2}:M\to\mathbb R$ be defined as
\begin{equation}\label{a2}
\varphi(x):=\begin{cases}
1 & r(x)<r_1\\
\hat\varphi(x) & r_1\leq r(x) \leq r_2\\
0 & r_2<r(x)
\end{cases}
\end{equation}
A straightforward calculation yields
\begin{gather*}
 \int_M|\nabla\varphi|^p dV_M=\int_{B_{r_2}\setminus B_{r_1}}|\nabla\hat\varphi|^p dV_M = \left(\int_{r_1}^{r_2}a_p(s)ds\right)^{1-p}\, .
\end{gather*}
By standard density results we can use $\varphi\in W^{1,p}_0(M)$ as a test function in the weak relation \eqref{div_deb}. Thus we obtain
\begin{align}\label{weak_ineq}
\int_{M}\varphi f dV_M &\leq (\operatorname{div}X,\varphi )\\
&=-\int_{M}\left\langle X,\nabla\varphi\right\rangle\nonumber\\
&\leq \left(\int_{\operatorname{supp}\left(\nabla\varphi\right)}|X|^{p/(p-1)}\right)^{(p-1)/p}\left(\int_M|\nabla\varphi|^p\right)^{1/p}\nonumber\\
&\leq \left\{\left(\int_{B_{r_2}\setminus B_{r_1}}|X|^{p/(p-1)}\right)\left(\int_{r_1}^{r_2}a_p(s)ds\right)^{-1}\right\}^{(p-1)/p}\, , \nonumber
\end{align}
where we have applied H\"older in the next-to-last inequality. Now, let $\left\{R_k\right\}_{k=1}^{\infty}$ be a sequence such that $R_k\to\infty$, which realizes the $\liminf$ in condition \eqref{cond_vol}. Up to passing to a subsequence, we can suppose $R_{k+1}\geq 2 R_k$. Hence, the sequence of cut-offs $\varphi_k:=\varphi_{R_k,2 R_k}$ converges monotonically to $1$ and applying monotone and dominated convergence, we have
\begin{align*}
\lim_{k\to\infty}\int_M \varphi_kf
= \lim_{k\to\infty}\int_M \varphi_k f_+ -\lim_{k\to\infty}\int_M\varphi_k f_- =\int_M f_+-\int_M f_- =\int_Mf\, .
\end{align*}
Taking limits as $k\to\infty$ in inequality \eqref{weak_ineq}, assumption \ref{cond_vol} finally gives
\[
\int_{M}f\leq 0\, .
\]
\end{proof}

\begin{proof}[\textsc{Proof of Theorem \ref{th_gt}}]
Choosing $f=\dive X$ in Proposition \ref{weak_KNR} we get $\int_M\dive X\leq 0$ and $(\dive X)_+\in L^1(M)$. Hence, we can repeat the proof replacing $X$ with $-X$.
\end{proof}

\begin{remark}\label{cutoff}\rm{
We point out that one could easily obtain results similar to Theorem \ref{th_gt} and Proposition \ref{weak_KNR} replacing $\varphi_{r_1,r_2}$ in the proofs with standard cut-off functions $0\leq\xi_{r_1,r_2}\leq1$ defined for any $\varepsilon>0$ in such a way that
\[
\xi_{r_1,r_2}\equiv 1\textrm{ on }B_{r_1}\, ,\qquad \xi_{r_1,r_2}\equiv 0\textrm{ on }M\setminus B_{r_2}\, ,\qquad |\nabla\xi_{r_1,r_2}|\leq\frac {1+\varepsilon} {r_2-r_1}\, .
\]
Nevertheless $\varphi_{r_1,r_2}$ gives better results than the standard cutoffs. For example, consider a $2$-dimensional model manifold with the Riemannian metric
\begin{gather*}
 ds^2 = dt^2 + g^2(t) d\theta^2\, ,
\end{gather*}
where $g(t)=e^{-t}$ outside a neighborhood of $0$. Then the $p$-energy of $\varphi_{r_1,r_2}$ and $\xi_{r_1,r_2}$ are respectively
\begin{gather*}
 \int_M \abs{\nabla \varphi_{r_1,r_2}}^p dv_M = \left(\int_{r_1}^{r_2}a_p(s)ds\right)^{1-p}= 2\pi (p-1)^{1-p}\ton{e^{{r_2}/(p-1)}-e^{{r_1}/({p-1})}}^{1-p},\\
\int_M \abs{\nabla \xi_{r_1,r_2}}^p dv_M \leq \ton{\frac{1+\epsilon}{r_2-r_1}}^p\int_{r_1}^{r_2} A(\partial B_s) ds =2\pi\ton{\frac{1+\epsilon}{r_2-r_1}}^p \ton{e^{-r_1}-e^{-r_2}}\, .
\end{gather*}
If we choose $r_2=2r_1\equiv 2r$ and let $r\to \infty$, we get
\begin{gather*}
  \int_M \abs{\nabla \varphi_{r,2r}}^p dv_M\sim c e^{-r} \ton{e^{{r}/({p-1})}-1}^{1-p}\sim c e^{-2r}\\
\int_M \abs{\nabla \xi_{r,2r}}^p dv_M \sim \frac{c'}{r^p} e^{-r}\ton{1-e^{-r}}\sim \frac{c'}{r^p}e^{-r}
\end{gather*}
for some positive constants $c,c'$. Using $\varphi_{r,2r}$ in the proof of Proposition \ref{weak_KNR} (in particular in inequality \eqref{weak_ineq}), we can conclude that $f=0$ provided
\begin{align*}
\lim_{r\to \infty} \left(\int_{B_{2r}\setminus B_r}|X|^{p/(p-1)}\right)^{(p-1)/{p}} e^{-2r}=0\, ,
\end{align*}
while using $\xi_{r,2r}$ we get a weaker result, i.e., $f=0$ provided
\begin{align*}
\lim_{r\to \infty} \left(\int_{B_{2r}\setminus B_r}|X|^{p/(p-1)}\right)^{{(p-1)}/{p}} \frac{1}{r^p}e^{-r}=0\, .
\end{align*}

One could ask if there exist even better cutoffs than the ones we chose. First, note that on model manifolds the cutoffs $\varphi_{r_1,r_2}$ are $p$-harmonic on $B_{r_2}\setminus B_{r_1}$, and so their $p$-energy is minimal. In general this is not true. Anyway, it turns out that the functions $\varphi_{r_1,r_2}$ are optimal at least in the class of radial functions, in fact they minimize the $p$-energy in this class, and this makes the condition \ref{cond_vol} radially sharp. To prove this fact, consider any radial cutoff $\psi:=\psi_{r_1,r_2}$ satisfying $\psi\equiv 1\textrm{ on }B_{r_1}, \ \psi\equiv 0\textrm{ on }M\setminus B_{r_2}$. By Jensen's inequality we have
\begin{gather*}
 \int_{} \abs{\nabla\varphi_{r_1,r_2}}^p dv_M = \left(\int_{r_1}^{r_2}a_p(s)ds\right)^{1-p}=c_\psi^{1-p}\left(\int_{r_1}^{r_2}\frac{a_p(s)}{\abs {\psi'(s)}}\frac{\abs {\psi'(s)}ds}{c_\psi}\right)^{1-p} \\
\leq c_\psi^{-p}\int_{r_1}^{r_2} \abs{\psi'(s)}^p A(\partial B_s) ds \leq \int_{} \abs{\nabla\psi}^p dv_M\, ,
\end{gather*}
where $\psi'$ is the radial derivative of $\psi$ and $c_\psi=\int_{r_1}^{r_2} \abs{\psi'(s)}ds\geq 1$.
}\end{remark}

\subsection{Applications}\label{appl}

Theorem \ref{th_gt}, and Proposition \ref{weak_KNR}, can be naturally applied to relax the assumptions on those theorems where the standard Kelvin-Nevanlinna-Royden criterion is used to deduce information on $p$-parabolic manifolds. In the following, we cite some of the possible application with a few results. Their proof is basically identical to the original results, only with Theorem \ref{th_gt} replacing the standard Stokes' theorem.

First, we present a global comparison result for the $p$-Laplacian of real valued functions. The original result assuming $p$-parabolicity appears in \cite[Theorem 1]{HPV}.

\begin{theorem}
\label{th_comparison}Let $\left(  M,\left\langle ,\right\rangle \right)  $ be
a connected, non-compact Riemannian manifold. Assume that
$u,v\in W^{1,p}_{\text{\rm{loc}}}(M)\cap C^0(M)$, $p>1$, satisfy%
\begin{gather*}
\Delta_{p}u\geq\Delta_{p}v\text{ weakly on }M\, ,
\end{gather*}
and that $\abs{\nabla u}^p$ and $\abs{\nabla v}^p$ satisfy the condition \ref{cond_vol} on $M$. Then, $u=v+A$ on $M$, for some constant $A\in\mathbb{R}$.
\end{theorem}

Choosing a constant function $v$, we immediately deduce the following result, which generalizes \cite[Corollary 3]{PRS-MathZ}.

\begin{corollary}\label{subharm}
Let $\left(  M,\left\langle ,\right\rangle \right)  $ be
a connected, non-compact Riemannian manifold. Assume that
$u\in W^{1,p}_{\text{\rm{loc}}}(M)\cap C^0(M)$, $p>1$, is a weak p-subharmonic function on $M$ such that $|\nabla u|^p$ satisfies the condition \ref{cond_vol} on $M$. Then $u$ is constant.
\end{corollary}

We also present an application related to harmonic maps. In \cite{SY-J}, Schoen and Yau considered the problem of uniqueness of the 2-harmonic representative with finite energy in a (free) homotopy class of maps from a complete manifold $M$ of finite volume to a complete manifold $N$ of non-positive sectional curvature. In particular, they obtained that if the sectional curvature of $N$ is negative then a given harmonic map $u$ is unique in its homotopy class, unless $u(M)$ is contained in a geodesic of $N$. Moreover, if $\operatorname{Sect}_N \leq 0$ and two homotopic harmonic maps $u$ and $v$ with finite energy are given, then $u$ and $v$ are homotopic through harmonic maps. In \cite{PRS-MathZ}, Pigola, Setti and Rigoli noticed that the assumption $V(M)<\infty$ can be replaced by asking $M$ to be $2$-parabolic. In Schoen and Yau's result, the finite energy of the maps is used in two fundamental steps of the proof:
\begin{enumerate}
 \item to prove that a particular subharmonic map of finite energy is constant;
 \item to construct the homotopy via harmonic maps.
\end{enumerate}
Using Corollary \ref{subharm} with $p=2$, we can deal with step (1) and thus obtain the following Theorem. If $\operatorname{Sect}_N\leq0$, weakening finite energy assumption in step (2) does not seem trivial to us, but we can still get a result for maps with fast $p$-energy decay without parabolicity assumption.

\begin{theorem}\label{SY}
\ifnum0\key=7 \index{harmonic map} \fi
Suppose $M$ and $N$ are complete manifolds.
\begin{enumerate}
\item[1)] Suppose $\operatorname{Sect}_N<0$. Let $u:M\to N$ be a harmonic map such that $|\nabla u|^2$ satisfies the condition $\mathcal{A}_{M,2}$ on $M$. Then there is no other harmonic map homotopic to $u$ satisfying the condition $\mathcal{A}_{M,2}$ unless $u(M)$ is contained in a geodesic of $N$.
\item[2)] Suppose $\operatorname{Sect}_N\leq0$. Let $u,v:M\to N$ be homotopic harmonic maps such that $|\nabla u|^2,|\nabla v|^2\in L^1(M)$ satisfy the condition $\mathcal{A}_{M,2}$ on $M$. Then there is a smooth one parameter family $u_t:M\to N$ for $t\in [0,1]$ of harmonic maps with $u_0=u$ and $u_1=v$. Moreover, for each $x\in M$, the curve $\{u_t(x);t\in [0,1]\}$ is a constant $($independent of $x)$ speed parametrization of a geodesic.
\end{enumerate}
\end{theorem}

More applications of the improved Stokes' theorem are available in \cite{VV}.

\begin{remark}\rm{
In the proof of Proposition \ref{prop_cfrcap}, we saw that the capacity of a condenser $(B_{r_1},B_{r_2})$ can be estimated from above using either the behaviour of $V(B_s)$ or the behaviour of $A(\partial B_s)$. This suggests that the condition \ref{cond_vol} should have an analogue in which $A(\partial B_s)$ is replaced by $V(B_s)$. This fact is useful since it is usually easier to verify and to handle volume growth assumptions than area growth conditions.}
\begin{definition}
\index{VMP condition @$\mathcal{V}_{M,p}$ condition}
We say that a real function $f:M\to\mathbb R$ satisfies the condition \ref{cond_vol2} on $M$ for some $p>1$ if there exists a function $g:(0,+\infty)\to(0,+\infty)$ such that
\begin{equation}\label{cond_vol2}\tag{$\mathcal{V}_{M,p}$}
\liminf_{R\to\infty}\left(\int_{B_{(R+g(R))}\setminus B_R}fdV_M\right)\left(\int_R^{R+g(R)} \ton{\frac{t}{V(t)}}^{1 /(p-1)} ds\right)^{-1}=0\, .
\end{equation}
\end{definition}
\rm{Indeed, it turns out that in every proposition stated in Section \ref{appl}, the condition \ref{cond_vol} can be replaced by the condition \ref{cond_vol2}, and the proofs remain substantially unchanged.}
\end{remark}

 \chapter{Estimates for the first Eigenvalue of the p-Laplacian}
 \ifnum0\key=7
  \thispagestyle{headings}
 \fi

\index{p-Laplacian @$p$-Laplacian}
Let $M$ be an $n$-dimensional compact Riemannian manifold with metric $\ps\cdot \cdot$ and diameter $d$. For a function $u\in W^{1,p}(M)$, we define its $p$-Laplacian as
\begin{gather}
  \Delta_p(u)\equiv \operatorname{div}(\abs{\nabla u}^{p-2} \nabla u)\ ,
\end{gather}
where the equality is in the weak $W^{1,p}(M)$ sense. We define $\lambda_{1,p}$ to be the smallest positive eigenvalue of this operator with Neumann boundary conditions, in particular $\lambda_{1,p}$ is the smallest positive real number such that there exists a nonzero $u\in W^{1,p}(M)$ satisfying 
\index{lambda 1 p @$\lambda_{1,p}$}
\begin{gather}\label{eq_plap}
\begin{cases}
 \Delta_p(u)= -\lambda_{1,p} \abs u ^{p-2} u \ \ \ \ \ &on \ M\\
 \ps{\nabla u}{\hat n}=0 \ \ \  \ \ \ \ &on \ \partial M
\end{cases}
\end{gather}
in the weak sense. Explicitly for every $\phi\in C^{\infty}(M)$
\begin{gather*}
 \int_M \abs{\nabla u}^{p-2}\ps{\nabla u}{\nabla \phi} dV =-\lambda_{1,p} \int_M\abs u ^{p-2} u \phi dV\, .
\end{gather*}
In this section we give some estimates on $\lambda_{1,p}$ assuming that the Ricci curvature of the manifold is bounded below by $(n-1)k$. We divide our study in three cases: $k>0$, $k=0$ and $k<0$.

\paragraph{Positive bound}In the first case, we recall the famous Lichnerowicz-Obata theorem valid for the case $p=2$. In \cite{lich} and \cite{obata}, the authors proved that, for a manifold without boundary, if $k=1$
\index{Lichnerowicz-Obata theorem}
\begin{gather}
 \lambda_{1,2}\geq n
\end{gather}
and that equality in this estimate forces the manifold to be the standard Riemannian $n$-dimensional sphere. In \cite{matei}, the author exploited the famous Levy-Gromov isoperimetric inequality to extend this theorem to all $p\in (1,\infty)$. In particular, she proved that
\index{Levy-Gromov isoperimetric inequality}
\begin{gather}
 \lambda_{1,p}\geq \lambda_{1,p}(S^n)
\end{gather}
and equality can be obtained only if $M=S^n$. In this chapter, we briefly recall the proof of this result and, using an improved version of the isoperimetric inequality, we also prove a rigidity Theorem which extends \cite[Theorem A]{wu}.

Note that if $p=2$ a stronger rigidity result is available for the eigenvalues of the Laplacian, however it is necessary to make some assumptions on the higher order eigenvalues $\lambda_{i,2}$ and their multiplicity. For generic $p$, the lack of linearity makes it very difficult to define higher order eigenvalues and study them.

Some rigidity results for $p=2$ are described, for example, in \cite{pete}. Later Aubry proved a sharp result in \cite{aubry}.

\paragraph{Zero or negative bound}
If $k\leq 0$ and $p\neq 2$, no sharp estimate for $\lambda_{1,p}$ was available before \cite{svelto,nava}. However, some lower bounds were already known, see for example \cite{matei,wang}.

To study the case $k\leq 0$, we follow the gradient comparison technique introduced by Li in \cite{li} (see also \cite{LY} and the book \cite{SYred}), and later developed by many other authors. In particular, we use the techniques described in \cite{new} and \cite{kro}, where the authors deal with the case where $p=2$. 

An essential tool in this part is the linearized $p$-Laplace operator. In fact, using this operator we prove a generalized $p$-Bochner formula which provides the key estimate in the proof of the main Theorem.

In the text, we study separately the case $k=0$ and the case $k<0$ since, as it is expected, the case $k=0$ is somewhat simpler to deal with. Moreover, for the case $k=0$, we also prove a characterization of the equality in the estimates.

\paragraph{Comparison of the two techniques}
Both the techniques used here, the one based on the isoperimetric inequality and the one based on the gradient comparison, are applicable to obtain estimates for any lower bound $(n-1)k\in \R$ for the Ricci curvature. However, as noted in \cite[section 7]{kro}, the isoperimetric technique does not yield the sharp estimate if $k\leq 0$. The reason seems to be that for different values of $v$ the isoperimetric function
\begin{gather}
 h(v)=\inf\cur{Area(\partial S), \ S\subseteq M, \ \Vol(S)=v, \ \ \Ric_M \geq (n-1)k, \ \ \operatorname{diam}(M)= d}
\end{gather}
\index{Poincaré© constant}
approaches its value for a different sequence of minimizing manifolds $M(v)$.

If $k=1$ and $d=\pi$, then $h(v)$ is the isoperimetric function of the unit sphere for all $v$, so in this case it is possible to recover the sharp result using this technique. Moreover, in this case the isoperimetric technique is also quicker than the gradient comparison.

\paragraph{Applications}As for the applications of this result, recall that the first eigenvalue of the $p$-Laplacian is related to the Poincar\'e constant, which is by definition
\begin{gather}\label{eq_cp}
 C_p=\inf\left\{\frac{\int_M \abs{\nabla u}^p d\V}{\int_M \abs u^p d\V} \ \text{ with }u\in M \ s.t. \ \ \int_M \abs u ^{p-2} u \ d\V=0\right\} \ .
\end{gather}
In particular by standard variational techniques one shows that $C_p=\lambda_p$, so a sharp estimate on the first eigenvalue is of course a sharp estimate on the Poincar\'e constant. Recall also that in the case of a manifold with boundary this equivalence holds if one assumes Neumann boundary conditions on the $p$-Laplacian.

It is worth mentioning that, even though with a different approach, in the case of Euclidean convex domains \cite{carlo,carlo2} independently obtained the same bounds for $C_p$ we prove in section \ref{sec_zero}, where we assume $k=0$.

Other applications (surprisingly also of practical interest) related to the $p$-Laplacian are discussed in \cite[Pag. 2]{W} and \cite{Diaz}.

\section{Technical tools}\label{sec_not}
In this section we gather some results that will be used throughout this chapter. In particular, we study the $p$-Laplace equation on a one dimensional manifold, we introduce the linearized $p$-Laplace operator, and use it to obtain a sort of $p$-Bochner formula. As we will see, this formula will play a central role in the estimates for $\lambda_{1,p}$. We start with some notational conventions.

\paragraph{Notation} Throughout this chapter, we fix $p>1$ \footnote{so we will often write $\lambda$ for $\lambda_{1,p}$}, and, following a standard convention, we define for any $w\in \R$
\begin{gather*}
 w^{(p-1)}\equiv \abs{w}^{p-2}w= \abs{w}^{p-1} \operatorname{sign}(w) \ .
\end{gather*}
Given a function $f:M\to \R$, $H_u$ denotes its Hessian where defined, and we set
\index{Hessian@Hessian of $f$ $H_f$}
\index{Hessian@Hessian of $f$ $H_f$!AU@$A_f$}
\begin{gather*}
 A_f\equiv \frac{\pst{\nabla f}{H_f }{\nabla f}}{\abs{\nabla f}^2} \ .
\end{gather*}
We use the convention $f_{ij}= \nabla_j \nabla_i f$ and the Einstein summation convention. We consider the Hessian as a $(2,0)$ or $(1,1)$ tensor, so for example
\begin{gather*}
 H_f(\nabla f,\nabla h)= f_{ij}f^i h^j = f_i^j f^i h_j = H_f(\nabla f) [\nabla h]\ .
\end{gather*}
$\abs{H_f}$ indicate the Hilbert-Schmidt norm of $H_f$, so that
\begin{gather*}
 \abs{H_f}^2= f_{ij}f^{ij} \ .
\end{gather*}

Unless otherwise specified, $u$ will denote an eigenfunction of the $p$-Laplacian relative to the eigenvalue $\lambda_{1,p}$. Note that the eigenvalue equation \eqref{eq_plap} is half-linear, meaning that if $u$ is an eigenfunction, then also $cu$ is for every $c\neq 0$. Moreover, by a simple application of the divergence theorem, one gets that $ \int_M \abs u ^{p-2} u = \lambda_{1,p}^{-1} \int_M \Delta_p (u) =0$. Thus our eigenfunction $u$ has to change sign on the manifold $M$, and so we can assume without loss of generality that
\begin{gather*}
 \min_{x\in M}\{u(x)\}=-1 \, , \quad 0<\max_{x\in M}\{u(x)\}\leq 1 \, .
\end{gather*}

\paragraph{Regularity}
In the following we use (sometimes implicitly) the regularity theorems valid for solutions of equation \eqref{eq_plap}. 
In general, the solution belongs to $W^{1,p}(M)\cap C^{1,\alpha}(M)$ for some $\alpha>0$, and elliptic regularity ensures that $u$ is a smooth function where $\nabla u \neq 0$ and $u\neq 0$. Moreover, for $p>2$ $u\in C^{3,\alpha}$ around the points where $\nabla u\neq 0$ and $u=0$, while for $1<p<2$ $u$ is only $C^{2,\alpha}$ around these points. The standard reference for these results is \cite{reg}, where the problem is studied in local coordinates.

\subsection{One dimensional p-Laplacian}\label{sec_1d}
The first nontrivial eigenfunction of the p-Laplacian is very easily found if $n=1$. In this case it is well-known that $M$ is either a circle or a segment, moreover equation \eqref{eq_plap} assumes the form
\begin{gather}\label{eq_plap1}
 (p-1)\abs{\dot u}^{p-2} \ddot u +\lambda u^{(p-1)}=0 \ .
\end{gather}
In order to study this eigenvalue problem, we define 
\begin{gather*}
 \pi_p = \int_{-1}^1 \frac{ds}{(1-s^p)^{1/p}} = \frac{2\pi}{p\sin(\pi/p)}\, ,
\end{gather*}
\index{pip@$\pi_p$}
and we define the function $\sinp(x)$ on $\qua{-\frac {\pi_p}2,\frac{3\pi_p}2}$ by
\index{sinp@$\sinp$}
\begin{gather*}
 \begin{cases}
x=\int_0^{\sin_p(x)} \frac{ds}{(1-s^p)^{1/p}} & \text{  if } x\in \qua{-\frac{\pi_p}{2},\frac{\pi_p}{2}}\\
\sin_p(\pi_p-x) & \text{  if } x\in \qua{\frac{\pi_p}{2},\frac{3 \pi_p}{2}}
 \end{cases}
\end{gather*}
and extend it on the whole real line as a periodic function of period $2\pi_p$. It is easy to check that for $p\neq 2$ this function is smooth around noncritical points, but only $C^{1,\alpha}(\R)$, where $\alpha=\min\{p-1,(p-1)^{-1}\}$. For a more detailed study of the $p$-sine, we refer the reader to \cite{dosly} and \cite[pag. 388]{pino2}.

Define the quantity
\begin{gather}
e(x)= \ton{\abs{\dot u }^{p} + \frac{\lambda \abs u^p}{p-1}}^{\frac 1 p} .
\end{gather}
If $u$ is a solution to \eqref{eq_plap1}, then $e$ is constant on the whole manifold, so by integration we see that all solutions of \eqref{eq_plap1} are of the form $A\sinp(\lambda^{1/p} x+B)$ for some real constants $A,B$. Due to this observation, our one-dimensional eigenvalue problem is easily solved.

Identify the circumference of length $2d$ with the real interval $[0,2d]$ with identified end-points. It is easily seen that the first eigenfunction on this manifold is, up to translations and dilatations, $u= \sin_p (k x)$, where $k= \frac {\pi_p}{d}$. Then by direct calculation we have
\begin{gather}
 \frac{\lambda}{p-1}= \ton {\frac {\pi_p}{d}}^p \ .
\end{gather}

The case with boundary (i.e., the one-dimensional segment) is completely analogous, so at least in the $n=1$ case the proof of Theorem \ref{teo_1} is quite straightforward.

\begin{rem}
\rm Note that in this easy case, the absolute values of the maximum and minimum of the eigenfunction always coincide and the distance between a maximum and a minimum is always $d=\frac{\pi_p}{\alpha}$. Note also that if we define $\cosp(x)\equiv \frac d {dx} \sinp(x)$, then the well known identity $\sin^2(x)+\cos^2(x)=1$ generalizes to $\abs{\sinp(x)}^p+\abs{\cosp(x)}^p=1$. 
\index{cosp@$\cosp$}
\end{rem}

\subsection{Linearized p-Laplacian and p-Bochner formula}
In this section we introduce the linearized $p$-Laplacian and study some of its properties.

\index{p-Laplacian @$p$-Laplacian!linearized p-Laplacian@linearized $p$-Laplacian}
First of all, we calculate the linearization of the $p$-Laplacian near a function $u$ in a naif way, i.e., we define
\begin{gather*}
 P_u(\eta)\equiv \left.\frac{d}{dt}\right\vert_{t=0} \Delta_p(u+t\eta)=\\
=\dive{(p-2) \abs{\nabla u}^{p-4}\ps{\nabla u}{\nabla \eta}\nabla u + \abs{\nabla u }^{p-2}\nabla \eta}=\\
= \abs{\nabla u}^{p-2}\Delta \eta +(p-2)\abs{\nabla u}^{p-4}\pst{\nabla u}{H_\eta}{\nabla u}+ (p-2)\Delta_p(u)\frac{\ps{\nabla u}{\nabla \eta}}{\abs{\nabla u }^2}+ \\
+ 2(p-2)\abs{\nabla u}^{p-4}\pst{\nabla u}{H_u}{\nabla \eta - \frac{\nabla u}{\abs{\nabla u}}\ps{\frac{\nabla u}{\abs{\nabla u}}}{\nabla \eta}} \ .
\end{gather*}
If $u$ is an eigenfunction of the $p$-Laplacian, this operator is defined pointwise only where the gradient of $u$ is non zero (and so $u$ is locally smooth) and it is easily proved that at these points it is strictly elliptic. For convenience, denote by $P^{II}_u$ the second order part of $P_u$, which is
\index{p-Laplacian @$p$-Laplacian!linearized p-Laplacian@linearized $p$-Laplacian!PU @$P_u$, $P_u^{II}$}
\begin{gather*}
 {P_u}^{II}(\eta) \equiv \abs{\nabla u}^{p-2}\Delta \eta+(p-2)\abs{\nabla u}^{p-4}\pst{\nabla u}{H_\eta}{\nabla u} \ ,
\end{gather*}
or equivalently
\begin{gather}\label{deph_pu}
 {P_u}^{II}(\eta) \equiv \qua{\abs{\nabla u}^{p-2}\delta_i^{j} +(p-2)\abs{\nabla u}^{p-4}\nabla_i u \nabla^j u }\nabla^{i}\nabla_j \eta\ . 
\end{gather}

Note that $P_u(u)=(p-1)\Delta_p(u)$ and $P^{II} _u(u)=\Delta_p(u)$.

The main property enjoyed by the linearized $p$-Laplacian is the following version of the celebrated Bochner formula.
\begin{prop}[p-\textsc{Bochner formula}]
\index{p-Bochner formula@$p$-Bochner formula}
 Given $x\in M$, a domain $U$ containing $x$, and a function $u\in C^2(U)$, if $\nabla u |_x\neq 0$ on $U$ we have
\begin{gather*}
 \frac{1}{p} P^{II}_u(\abs{\nabla u}^{p})=\\
 \abs{\nabla u}^{2(p-2)}\{\abs{\nabla u}^{2-p}\qua{\ps{\nabla \Delta_p u}{\nabla u}-(p-2)A_u\Delta_p u}+\\
+\abs{H_u}^2+p(p-2)A_u^2+ \operatorname{Ric}(\nabla u,\nabla u) \} \ .
\end{gather*}
In particular this equality holds if $u$ is an eigenfunction of the $p$-Laplacian and $\nabla u |_x \neq 0$ and $u(x)\neq 0$ (or if $\nabla u|_x \neq 0$ and $p\geq 2$).
\end{prop}
\begin{proof}
Just as in the usual Bochner formula, the main ingredients for this formula are the commutation rule for third derivatives and some computations.

First, we compute $\Delta (\abs{\nabla u}^p)$, and to make the calculation easier we consider a normal coordinate system centered at the point under consideration. Using the notation introduced in Section \ref{sec_not} we have
\begin{gather*}
\frac 1 p \Delta(\abs{\nabla u}^p) = \nabla^i\ton{ \abs{\nabla u}^{p-2} u_{ji} u^j}=\\
=\abs{\nabla u}^{p-2} \ton{ \frac {p-2}{\abs {\nabla u}^2}u^{is}u_s u_{ik}u^k+ u_{kii} u^k + u_{ik}u^{ik}} \ .
\end{gather*}
The commutation rule now allows us to interchange the indexes in the third derivatives. In particular remember that in a normal coordinate system we have
\begin{gather}\label{ref_a}
 u_{ij}=u_{ji}\ \ \ \ \ u_{ijk}-u_{ikj}=-R_{lijk}u^l\\
\notag u_{kii}=u_{iki}=u_{iik}+ \operatorname{Ric}_{ik}u^i \ .
\end{gather}
This shows that
\begin{gather}\label{ref_b}
\frac 1 p \Delta(\abs{\nabla u}^p) = \\
\notag =\abs{\nabla u}^{p-2} \ton{ \frac {p-2}{\abs {\nabla u}^2}\abs{H_u(\nabla u)}^2+ \ps{\nabla \Delta u}{\nabla u}+ \Ric(\nabla u, \nabla u)+\abs{H_u}^2} \ .
\end{gather}
In a similar fashion we have
\begin{gather*}
\frac 1 p \nabla_i \nabla_j \abs{\nabla u}^p= (p-2)\abs{\nabla u}^{p-4} u_{is}u^s u_{jk}u^k+ \abs{\nabla u}^{p-2}(u_{kij}u^k+ u_{ik}u_j^k) \ ,
\end{gather*}
which leads us to
\begin{gather*}
 \frac 1 p \frac {\nabla_i\nabla_j (\abs{\nabla u}^p) \nabla^i u \nabla^j u }{\abs {\nabla u}^2}=\\
= \abs{\nabla u}^{p-2}\ton{(p-2) A_u^2+ \frac{\nabla_i \nabla_j \nabla_k u \ \nabla^i u \nabla ^j u \nabla^k u}{\abs{\nabla u}^2}+ \frac{\abs{H_u(\nabla u)}^2}{\abs {\nabla u}^2}} \ . 
\end{gather*}
The last computation needed is
\begin{gather*}
 \ps{\nabla \Delta_p u}{\nabla u}= \nabla_i\qua{\abs{\nabla u}^{p-2}\ton{\Delta u + (p-2) A_u}} \nabla^i u=\\
=\ps{\nabla(\abs{\nabla u}^{p-2})}{\nabla u}\abs{\nabla u}^{2-p}\Delta_p u +\\ +\abs{\nabla u}^{p-2}\qua{\ps{\nabla \Delta u}{\nabla u}+(p-2) \ps{\nabla \abs{\nabla u}^{-2}}{\nabla u} \pst{\nabla u}{H_u}{\nabla u}}+\\
+(p-2) \abs{\nabla u}^{p-4}\qua{(\nabla H_u)(\nabla u,\nabla u,\nabla u)+2 \abs{H_u(\nabla u)}^2}=(p-2)A_u\Delta_p u+ \\
\abs{\nabla u}^{p-2}\cur{\ps{\nabla \Delta u}{\nabla u} +(p-2) \qua{-2A_u^2+ \frac{\nabla_i \nabla_j \nabla_k u \nabla^i u \nabla ^j u \nabla^k u}{\abs{\nabla u}^2}+ 2 \frac{\abs{H(\nabla u)}^2}{\abs{\nabla u}^2}}} \ .
\end{gather*}
Using the definition of $P^{II}_u$ given in \eqref{deph_pu}, the p-Bochner formula follows form a simple exercise of algebra.
\end{proof}

In the proof of the gradient comparison, we need to estimate $P^{II}_u(\abs{\nabla u}^p)$ from below.
If $\Ric\geq 0$ and $\Delta_p u =-\lambda u^{(p-1)}$, one could use the very rough estimate $\abs{H_u}^2\geq A_u^2$ to obtain
\begin{gather*}
\frac 1 p P^{II}_u(\abs{\nabla u}^{p}) \geq \\
\notag\geq (p-1)^2 \abs{\nabla u}^{2p-4}A_u^2+\lambda(p-2)\abs{\nabla u}^{p-2}u^{(p-1)}A_u - \lambda (p-1)\abs{\nabla u}^{p}\abs u ^{p-2} \ .  
\end{gather*}
This estimate is used implicitly in proof of \cite[Li and Yau, Theorem 1 p.110]{SYred} (where only the usual Laplacian is studied), and also in \cite{kn} and \cite{hui}.\\
A more refined estimate on $\abs{H_u}^2$ which works in the linear case is the following
\begin{gather*}
 \abs{H_u}^2\geq \frac{(\Delta u)^2} n + \frac{n} {n-1} \ton{\frac{\Delta u}{n} - A_u}^2 \ .
\end{gather*}
\index{curvature dimension inequality}
This estimate is the analogue of the curvature-dimension inequality and plays a key role in \cite{new} to prove the comparison with the one dimensional model. Note also that this estimate is the only point where the dimension of the manifold $n$ and the assumption on the Ricci curvature play their role. A very encouraging observation about the $p$-Bochner formula we just obtained is that the term $\abs{H_u}^2+p(p-2)A_u^2$ seems to be the right one to generalize this last estimate, indeed we can proven the following Lemma.
\begin{lemma}\label{lemma_n}
At a point where $u$ is $C^2$ and $\nabla u \neq 0$ we have
\begin{gather*}
 \abs{\nabla u}^{2p-4}\ton{\abs{H_u}^2 + p(p-2)A_u^2}\geq\\
\geq \frac{(\Delta_p u)^2}{n'} + \frac{n'}{n'-1}\ton{\frac{\Delta_p u}{n' } -(p-1)\abs{\nabla u}^{p-2}A_u}^2 \ .
\end{gather*}
Where $n'$ is any real number $n'\geq n$.
\end{lemma}
\begin{proof}
We will only prove the inequality with $n'=n$. Indeed, the general case follows easily from the estimate
\begin{gather}
\frac{x^2}{n} + \frac {n}{n-1}\ton{\frac x n - y}^2 -\ton{\frac{x^2}{n'} + \frac {n'}{n'-1}\ton{\frac x {n'} - y}^2 }=\\
=\ton{\frac 1 {n-1} - \frac 1 {n'-1}}(x-y)^2\geq 0 \ 
\end{gather}
valid for any $x,y\in \R$ and $n'\geq n$.

The proof consists only in some calculations that for simplicity can be carried out in a normal coordinate system for which $\abs{\nabla u}|_x=u_1(x)$. At $x$ we can write
\begin{gather*}
 \abs{\nabla u}^{2-p} \Delta_p(u)= \Delta u + (p-2) \frac{\pst{\nabla u}{H_u}{\nabla u}}{\abs{\nabla u}^2}=(p-1)u_{11}+ \sum_{j=2}^n u_{jj} \ .
\end{gather*}
By the standard inequality $\sum_{k=1}^{n-1} a_k^2 \geq \frac 1 {n-1} \ton{\sum_{k=1}^{n-1}a_k}^2$ we get
\begin{gather}\label{eq_Hn}
\notag \abs{H_u}^2+p(p-2)A_u^2 = (p-1)^2 u_{11}^2 + 2\sum_{j=1}^n u_{1j}^2 + \sum_{i,j=2}^n u_{ij}^2\geq\\
\geq (p-1)^2 u_{11}^2 + \frac 1 {n-1} \ton{\sum_{i=2}^n u_{ii}}^2 \ .
\end{gather}
On the other hand it is easily seen that
\begin{gather*}
 \frac{\ton{\abs{\nabla u}^{2-p} \Delta_p (u)}^2}{n} + \frac n {n-1} \ton{\frac{\abs{\nabla u}^{2-p}\Delta_p(u)}{n}- (p-1) A_u}^2= \\
= \frac 1 n \ton{(p-1) u_{11}+\sum_{i=2}^n u_{ii} }^2+ \frac n {n-1} \ton{-\frac{n-1}{n} (p-1) u_{11} + \frac 1 n \sum_{i=2}^n u_{ii}}^2=\\
= (p-1)^2 u_{11}^2 + \frac 1 {n-1} \ton{\sum_{i=2}^n u_{ii}}^2 \ .
\end{gather*}
This completes the proof.
\end{proof}

The following Corollary summarizes all the results of this subsection.
\begin{cor}\label{cor_est_pu}
 If $u$ is an eigenfunction relative to the eigenvalue $\lambda$, at a point where $\nabla u \neq 0$ and $u\neq 0$ (or $\nabla u\neq 0$ and $p\geq 2$) and for any $n'\geq n$ we can estimate
\begin{gather}\label{eq_pbochest}
 \frac{1}{p} P^{II}_u(\abs{\nabla u}^{p})=\\
\notag\geq  \abs{\nabla u}^{p-2}\qua{\ps{\nabla \Delta_p u}{\nabla u}-(p-2)A_u\Delta_p u}+\\
\notag+\frac{(\Delta_p u)^2}{n'} + \frac{n'}{n'-1}\ton{\frac{\Delta_p u}{n' } -(p-1)\abs{\nabla u}^{p-2}A_u}^2+ \abs{\nabla u}^{2(p-2)}\operatorname{Ric}(\nabla u,\nabla u)
\end{gather}
\end{cor}

\subsection{Warped products}\label{sec_warped}
\index{warped products}
Even though the computations are standard, for the sake of completeness we briefly recall some properties of warped products which we will use in dealing with the zero and negative lower bounds. A standard reference for this subject is \cite{petersen}, and similar computations in the more general setting of weighted Riemannian manifolds are carried out in \cite{milman}.

Given a strictly positive function $f:[a,b]\to \R^+$, define a Riemannian manifold $M$ by
\begin{gather*}
 M=[a,b]\times_f S^{n-1}\, ,
\end{gather*}
and, using the standard product coordinates on $M$, i.e, the coordinates given by $(t,x)$, $t\in [a,b]$, $x\in S^{n-1}$, define a metric on this manifold by
\begin{gather*}
 ds^2=g = dt^2 + f^2(t) g_{S^{n-1}}\, ,
\end{gather*}
where $g_{S^{n-1}}$ is the standard Riemannian metric on the $n-1$ dimensional sphere. If $X$ is any unit vector tangent to the sphere $\{t\}\times S^{n-1}$, we have
\begin{align*}
 \operatorname{Ric}(\partial t,\partial t)|_{t,x} &= -(n-1) \frac{\ddot f(t)}{f(t)}\, ,\\
 \operatorname{Ric}(X,X)|_{t,x} &= - \frac{\ddot f(t)}{f(t)} + (n-2) \frac{1-\dot f ^2(t)}{f^2(t)}\, .
\end{align*}
Note that $M$ is a manifold with boundary $\partial M = \{a,b\}\times S^{n-1}$, and the second fundamental form of $M$ at $\partial M$ with respect to $\partial t$ is
\begin{gather*}
 II^{\partial t}(X,Y)\equiv g(\nabla_X \partial t, Y)= \frac{\dot f (t)}{f(t)} g_{f^2(t) S^{n-1}} (X,Y)= \dot f (t) f(t) g_{S^{n-1}}(X,Y)
\end{gather*}

Note that if $f(a)=0$, then $M$ is still a smooth Riemannian if $\dot f(0)=1$ and all the even derivatives of $f$ are zero. In this case, $\partial M = \{b\} \times S^{n-1}$.

\section{Positive lower bound}\label{sec_pos}
In this section we study the case $\Ric\geq (n-1)k$ with $k>0$. Using a simple conformal transformation, we can assume for simplicity that $k=1$.

As mentioned in the introduction, it is possible to generalize Obata's theorem to any $p\in (1,\infty)$, and in fact this result has been proved in \cite{matei} using isoperimetric techniques. For the reader's convenience, we sketch the proof of this result; moreover we also state and prove a rigidity Theorem relative to the first eigenvalue. 

\subsection{Lichnerowicz-Obata theorem and rigidity}
Let $M$ be a compact $n$-dimensional Riemannian manifold without boundary, with Ricci curvature bounded from below by $n-1$ and with diameter $d$. We will denote by $\Omega$ any connected domain $\Omega \subset M$ with nonempty boundary, and by $\mu(\Omega)$ the first eigenvalue of the $p$-Laplacian with Dirichlet boundary conditions on $\partial \Omega$, i.e., the positive real number
\begin{gather}
 \mu(\Omega)=\inf\cur{\frac{\int_{\Omega} \abs {\nabla f}^p dV}{\int_{\Omega} \abs f ^pdV} \quad s.t. \quad  f\in W^{1,p}_0 (\Omega) \ \ f\neq 0}
\end{gather}
The main tool for the eigenvalue estimate is Levy-Gromov isoperimetric inequality proved in \cite{gro}. In order to prove also the rigidity Theorem, we use an improved version of this inequality proved by Croke in \cite{croke}. Using this inequality and a symmetrization technique, it is possible to compare $\mu(\Omega)$ with the first Dirichlet eigenvalue of geodesic balls on the sphere. Via standard techniques, it is then possible to obtain estimates on the first positive Neumann eigenvalue on the whole manifold $M$.

We start with the isoperimetric inequality.
\begin{lemma}\cite[Lemma p. 254]{croke}\label{gromov+}
 Let $M$ be as above. Then there exists a constant $C(n,d)\geq 1$ such that $C(n,\pi)=1$, $C(n,d)>1$ if $d<\pi$, and for any domain $\Omega\subset M$ with smooth boundary we have
\begin{gather}
 \frac{\text{Area}(\partial \Omega)}{\Vol(M)}\geq C(n,d) \frac {\text{Area}(\partial \Omega_0)}{\Vol(S^n)}\, ,
\end{gather}
where $S^n$ is the standard Riemannian sphere of radius $1$ and $\Omega_0$ is a geodesic ball on $S^n$ with $\Vol(\Omega_0)=\Vol(\Omega)$.
\end{lemma}

The following Theorem proves a comparison for the Dirichlet eigenvalue for subdomains.
\begin{theorem}
 Let $M$ be as above. Let $\Omega\subset M$ be a domain, and let $\Omega^\star$ be a geodesic ball in $S^n$ such that
\begin{gather}
 \frac{\Vol(\Omega^\star)}{\Vol(S^n)} =  \frac{\Vol(\Omega)}{\Vol(M)} 
\end{gather}
 Then
\begin{gather}
 \mu(\Omega)\geq C(n,D)^p\mu(\Omega^\star)
\end{gather}

\end{theorem}
\begin{proof}
 Recall the variational characterization of the first eigenvalue $\mu(\Omega)$:
\begin{gather}
 \mu(\Omega)=\inf\cur{\frac{\int_\Omega \abs{\nabla f}^p d\Vol_M}{\int_\Omega \abs{f}^p d\Vol(M)} \ : \ \ f\in C^\infty_0(\Omega), \ f\neq 0}\, .
\end{gather}
By a standard density argument, without loss of generality we can also add the condition that $f$ is a Morse function, so that $f^{-1}(t)$ is a smooth submanifold for almost all $t$. It is also evident that $f$ can be assumed to be nonnegative. Define $\Omega_t$ to be the superlevel sets of the function $f$. In symbols
\begin{gather}
\Omega_t = f^{-1} (t,\infty)
\end{gather}
Define also the symmetrized sets $\Omega^\star_t$ as concentric geodesic balls in $S^n$ with the property
\begin{gather}
 \Vol(\Omega^\star_t)=\beta^{-1} \Vol(\Omega_t)\, ,
\end{gather}
where $\beta= \frac{\Vol(M)}{\Vol(S^n)}$ and $\Omega^\star_0 =\Omega^\star$.\\
By Lemma \ref{gromov+}, $\Vol(\Omega_t)\geq C(n,D)\beta \Vol(\Omega^\star_t)$ for all $t$ which are not critical values for $f$, so for almost all $t$. Define $f^\star:\Omega_0\to \R^+$ by $f|_{\partial \Omega^\star _t} =t$, then it is easily seen that
\begin{gather}
 \int_\Omega \abs f ^p d\Vol = \beta\int_{\Omega^\star} \abs{f^\star}^p d\Vol\, .
\end{gather}
Using the coarea formula (see the proof of \cite[Theorem 2.1]{matei} for the details), it is easy to prove that
\begin{gather}
 \int_{\Omega} \abs {\nabla f}^p d\Vol \geq \beta C(n,d)^p  \int_{\Omega^\star} \abs {\nabla f^\star}^p d\Vol \, .
\end{gather}
This yields immediately the conclusion.
\end{proof}
Using to the properties of the first eigenfunction of the $p$-Laplacian on a closed manifold, the previous theorem leads immediately to the following corollary.
\begin{corollary}
Let $M$ be as above. Then
\begin{gather}
 \lambda_{1,p}(M) \geq C(n,d)^p \lambda_{1,p} (S^n)\geq \lambda_{1,p} (S^n)\, .
\end{gather}
\end{corollary}
\begin{proof}
 The proof relies on \cite[Lemma 3.1, 3.2]{matei}. The basic idea is that the first Neumann eigenfunction on $M$ has two nodal sets
\begin{gather}
 U_+=u^{-1}(0,\infty) \ \ \ \ U_-=u^{-1}(-\infty,0)\, .
\end{gather}
Using the minimizing properties of $u$ (see \eqref{eq_cp}), it is easy to prove that $\mu(U_-)=\mu(U_+)=\lambda(M)$ (see \cite[Lemma 3.2]{matei} for the details). Since the half sphere $S_n^+$ is a nodal set of the first Neumann eigenfunction on the sphere, it is evident that
\begin{gather}\label{eq_obata}
 \lambda(M)=\mu(U_+)\geq C(n,d)^p \mu(S_n^+) = C(n,d)^p\lambda(S^n)\, .
\end{gather}

\end{proof}

The extra factor $C(n,d)^p$ in inequality \eqref{eq_obata} allows us to get as a simple corollary a rigidity Theorem similar to that available if $p=2$. Indeed, if the difference $\lambda_{1,p}(M)-\lambda_{1,p}(S^n)$ is close to zero, then the previous Theorem proves that the diameter of $M$ is close to $\pi$. Since it is well-known that in this case also $\lambda_{1,2}(M)-\lambda_{1,2}(S^n)$ is close to zero, we can use the rigidity Theorems available when $p=2$ to prove the result for generic $p$. 
\begin{theorem}\label{th_main_eigen}
 Let $M$ be as above. Then there exists a function $\psi(n,x)$ which is zero for $x=0$ and strictly positive for positive $x$ such that
\begin{gather}
 d\geq \pi[1-\psi(n,\lambda_{1,p}(M)-\lambda_{1,p}(S^n))]
\end{gather}
\end{theorem}
\index{Lichnerowicz-Obata theorem!p Lichnerowicz-Obata theorem@$p$-Lichnerowicz-Obata theorem}
\begin{proof}
 We know that
\begin{gather}
 \lambda_{1,p}(M)-\lambda_{1,p}(S^n)\geq \ton{C(n,d)^p -1}\lambda_{1,p}(S^n)
\end{gather}
If we define $\psi(n,\cdot )^{-1}|_{d} = C(n,d)^p -1$, the properties of the function $\psi$ follow immediately from the properties of $C(n,d)$.
\end{proof}

Using the well-known results about rigidity in the $p=2$ case, it is possible to prove rigidity also for the generic $p$-case. We recall the standard notation
\begin{gather}
 d_{GH}(X,Y)
\end{gather}
to denote the Gromov-Hausdorff distance between the metric spaces $X$ and $Y$.
\begin{theorem}
 For any $p>1$, there exists a function $\tau_p(\epsilon)$ positive for $\epsilon>0$ with $\lim_{\epsilon\to 0}\tau_p(\epsilon)=0$ such that if
\begin{gather}
 \lambda_{1,p}(M)\leq \lambda_{1,p}(S^n)+\epsilon\, ,
\end{gather}
then there exists a compact geodesic space $X$ of Hausdorff dimension $n-1$ such that
\begin{gather}
 d_{GH} (M,\Sigma(X))\leq \tau(\epsilon)\, ,
\end{gather}
where $\Sigma(X)$ is the suspension $[0,\pi]\times_{\sin(x)} X$. 
\end{theorem}
\begin{proof}
 According to Theorem \ref{th_main_eigen}, if $\lambda_{1,p}(M)-\lambda_{1,p}(S^n)\leq \epsilon$, then there exists a positive $\delta$ such that $d\geq \pi-\delta$. Define $B(r)\subset S^n$ to be a geodesic ball of radius $r$ in the standard sphere. By Cheng's comparison theorem \cite[Theorem 2.1]{cheng}
\begin{gather}
 \lambda_{1,2}(M)\leq \mu_2(B(d/2))
\end{gather}
It is clear that $\mu(B(d/2))$ is a continuous function of $d$, and for $d\to \pi$ we have
\begin{gather}
 \lim_{d\to \pi} \mu_2(B(d/2)) = \mu_2(B(\pi/2))=\lambda_{1,2}(S^n)
\end{gather}
This proves that there exists a positive $\delta'$ such that
\begin{gather}
 \lambda_{1,2}(M)-\lambda_{1,2}(S^n)\leq \delta'
\end{gather}
Now we can conclude using the rigidity theorem available for $p=2$ (see for example \cite[Theorem A]{wu}).
\end{proof}

\begin{remark}\rm
 Recall that the standard Riemannian sphere $S^n$ can be written as the warped product $[0,\pi]\times_{\sin(x)} S^{n-1}$, so in some sense the previous theorem proves that if $\lambda_{1,p}(M)$ is close to $\lambda_{1,p}(S^n)$, then $M$ is close to $S^n$ at least in one ``direction'', so to speak.
\end{remark}

\section{Zero lower bound}\label{sec_zero}
Even though we use the same gradient comparison technique, we divide the case $k=0$ and $k<0$ because the first case is easier to study. Moreover, in the first case we are also able to prove a characterization of equality in the sharp estimate.

We start by discussing the case $k=0$, and in particular some of the available literature on the subject.

In the case where $p=2$, estimates on $\lambda_{1,2}$ have been intensively studied. In \cite{ZY} the sharp estimate
\begin{gather*}
 \lambda_2 \geq \frac{\pi^2}{d^2}
\end{gather*}
is obtained assuming that $M$ has nonnegative Ricci curvature.

The main tool used in this article is a gradient estimate for the function $u$. This technique was introduced by Li in \cite{li} (see also \cite{LY} and \cite{SYred}) to study the same problem. In particular, in \cite{LY} the authors were able to prove the following estimate.
\begin{lemma}
Let $M$ be a compact manifold with nonnegative Ricci curvature, and let $u:M\to [-1,1]$ solve $\Delta u = -\lambda_2 u$. Then the following estimate is valid where $u\neq \pm1$
\begin{gather*}
 \frac{\abs{\nabla u}^2}{1-u^2}\leq \lambda_2   \ .
\end{gather*}
Note that where $u=\pm 1$, $\nabla u=0$.
\end{lemma}
\begin{proof}[{Sketch of the proof}]
 The proof is based on a familiar argument. Consider the function $F\equiv \frac{\abs{\nabla u}^2}{1+\epsilon-u^2}$ on the manifold $M$. Necessarily $F$ attains a maximum, and at this point
\begin{gather}\label{eq_deltamax}
 \nabla F=0 \quad \text{and}\quad \Delta F\leq 0\, .
\end{gather}
From these two relations, one proves that $F\leq \lambda_2$.
\end{proof}
With this gradient estimate Li and Yau proved that
\begin{gather*}
 \lambda_2 \geq \frac{\pi^2}{4d^2} \ .
\end{gather*}
using a standard geodetic argument.  Since we are going to use a similar technique, we briefly sketch the proof of this estimate. Rescale the eigenfunction $u$ in such a way that $m=\min\{u\}=-1$ and $0<{u^\star}=\max\{u\}\leq 1$ and consider a unit speed minimizing geodesic $\gamma$ joining a minimum point $x_-$ and a maximum point $x_+$ for $u$, then a simple change of variables yields
\begin{gather*}
\frac{\pi}{2}=\int_{-1}^0\frac{du}{\sqrt{1-u^2}} < \int_{m}^{u^\star} \frac{du}{\sqrt{1-u^2}}  \leq \int_{\gamma}\frac{\abs{\nabla u}}{\sqrt{1-u^2}} dt \int_{\gamma} \frac{\abs{\nabla u}}{\sqrt{1-u^2}} dt \leq \sqrt{\lambda_2} d\, .
\end{gather*}
Note that the strict inequality in this chain forces this estimate to be non-sharp, inequality which arises from the fact that $\max\{u\}={u^\star}>0$. If in addition we suppose that ${u^\star}=1$, we can improve this estimate and get directly the sharp one. This suggests that it is important to consider the behaviour of the maximum of the eigenfunction to improve this partial result. In fact Li and Yau were able to sharpen their estimate by using the function $F\equiv \frac{\abs{\nabla u}^2}{(u-1)({u^\star}-u)}$ for their gradient estimate, which led them to prove that $\lambda_2\geq \frac{\pi^2}{2d^2}$.

J. Zhong and H. Yang obtained the sharp estimate using a barrier argument to improve further the gradient estimate (see \cite{ZY} and also \cite{SYred}).

Later on M. Chen and F. Wang in \cite{chen} and independently P. Kroger in \cite{kro} (see also \cite{kro2} for explicit bounds) with different techniques were able to estimate the first eigenvalue of the Laplacian by using a one dimensional model. Note that their work also applies to generic lower bounds for the Ricci curvature.
The main tool in \cite{chen-} is a variational formula, while \cite{kro} uses a gradient comparison technique.
\index{gradient comparison}
This second technique was also adapted by D. Bakry and Z. Qian in \cite{new} to obtain eigenvalue estimates for weighted Laplace operators with assumptions on the Bakry-Emery Ricci curvature. In this section we follow this latter technique based on the gradient comparison. Roughly speaking, the basic idea is to find the right function $w:\R\to \R$ such that $\abs{\nabla u}\leq \abs{\dot w}|_{w^{-1}(u)}$ on $M$.

For generic $p$, some estimates on the first eigenvalue of the $p$-Laplacian are already known in literature, in particular see \cite{hui} and \cite{kn}; \cite{take} presents different kind of estimates, and for a general review on the problem with a variational twist see \cite{le}. In \cite{hui} and \cite{kn} the general idea of the estimate is similar to the one described for the linear case. Indeed, the authors prove a gradient estimate via the maximum principle, replacing the usual Laplacian in \eqref{eq_deltamax} with linearized $p$-Laplace operator.

By estimating the function $F=\frac{\abs{\nabla u}^2}{1-u^2}$, \cite{kn} is able to prove that on a compact Riemannian manifold with $\Ric \geq 0$ and for $p\geq 2$
\begin{gather*}
 \lambda_p\geq \frac 1 {p-1} \ton{\frac \pi {4d} }^p   \ ,
\end{gather*}
while \cite{hui} uses $F=\frac{\abs{\nabla u}^p}{1-u^p}$ and assumes that the Ricci curvature is quasi-positive (i.e., $\Ric\geq 0$ on $M$ but with at least one point where $\Ric >0$), to prove that for $p>1$
\begin{gather*}
 \lambda_p \geq (p-1) \ton{\frac{\pi_p}{2d}}^p \ .
\end{gather*}

The estimate proved in this section is better than both these estimates and it is sharp. 
Indeed, using some technical lemmas needed to study the one dimensional model functions, we are able to state and prove the gradient comparison Theorem, and as a consequence also the following theorem on the spectral gap.
\begin{teo}\label{teo_1}
\index{gradient comparison}
 Let $M$ be a compact Riemannian manifold with nonnegative Ricci curvature, diameter $d$ and possibly with convex boundary. Let $\lambda_p$ be the first nontrivial (=nonzero) eigenvalue of the $p$-Laplacian (with Neumann boundary condition if necessary). Then the following sharp estimate holds
\begin{gather*}
 \frac{\lambda_p}{p-1} \geq \frac{\pi_p^p}{d^p} \ .
\end{gather*}
Moreover a necessary (but not sufficient) condition for equality to hold in this estimate is that $\max\{u\}=-\min\{u\}$.
\end{teo}

\noindent The characterization of the equality case is dealt with in the last subsection. In \cite{hang}, this characterization is proved in the case where $p=2$ to answer a problem raised by T. Sakai in \cite{saka}. Unfortunately, this proof relies on the properties of the Hessian of a 2-eigenfunction, which are not easily generalized for generic $p$.

\subsection{Gradient comparison}
In this subsection we prove a gradient comparison theorem that will be the essential tool to prove our main theorem. To complete the proof, we need some technical lemmas which, for the sake of clarity, are postponed to the next subsection. 
\begin{teo}[\textsc{Gradient comparison Theorem}]\label{grad_est}
 Let $M$ be an $n$-dimensional compact Riemannian manifold, possibly with $C^2$ convex boundary $\partial M$, let $u$ be a eigenfunction of the $p$-Laplacian relative to the eigenvalue $\lambda$, and let $w$ be a solution on $(0,\infty)$ of the one dimensional ODE
\begin{gather}\label{eq_ode1}
\begin{cases}
 \frac d {dt} \dot w ^{(p-1)} - T \dot w ^{(p-1)} +\lambda w^{(p-1)}=0\\
 w(a)=-1 \quad \quad \dot w (a)=0
\end{cases}
\end{gather}
where $T$ can be either $-\frac{n-1} t$ or $T=0$, and $a\geq0$. Let $b(a)>a$ be the first point such that $\dot w (b)=0$ (so that $\dot w >0$ on $(a,b)$). If $[\min(u),\max(u)]\subset [-1,w(b)=\max(w)]$, then for all $x\in M$
\begin{gather*}
\abs{\nabla u (x)}\leq \dot w|_{w^{-1}(u(x))} \ .
\end{gather*}
\end{teo}

\begin{rem}
 \rm{The differential equation \eqref{eq_ode1} and its solutions will be studied in the following section, in particular we will prove existence and continuous dependence on the parameters for any $a\geq 0$ and the oscillatory behaviour of the solutions. Moreover, the solution always belongs to the class $C^1(0,\infty)$. 

For the sake of simplicity, we use the following notational convention. For finite values of $a$, $w$ denotes the solution to the ODE \eqref{eq_ode1} with $T=-\frac{n-1} x$, while $a=\infty$ indicates the solution of the same ODE with $T=0$ and any $a$ as initial condition.
Observe that in this latter case all the solutions are invariant under translations, so the conclusions of the Theorem do not change if the initial point of the solution $w$ is changed.

With the necessary adaptations, the proof of this theorem is similar to the proof of \cite[Theorem 8]{new}.}
\end{rem}

\begin{proof}
Since the proof is almost the same, we prove the theorem on a manifold without boundary, and then point out in Lemma \ref{lemma_bou} where the proof is different if the boundary is not empty. 

First of all, in order to avoid problems at the boundary of $[a,b]$, we assume that $$[\min\{u\},\max\{u\}]\subset (-1,w(b))\, ,$$ so that we only have to study our one-dimensional model on compact subintervals of $(a,b)$. We can obtain this by multiplying $u$ by a constant $\xi<1$. If we let $\xi\to 1$, then the original statement is proved.

Consider the function defined on the manifold $M$
\begin{gather*}
 F\equiv \psi(u)\qua{\abs{\nabla u}^p -\phi(u)} \ ,
\end{gather*}
where $\psi:\R\to \R$ is a positive $C^2$ functions on $M$ which will be specified later, and $\phi(u(x))=\dot w^{p} |_{u(x)}$. We want to prove that $F\leq 0$ on all of $M$.

Note that we introduced the function $\psi$ in the definition of $F$ since it is not easy to prove that $\abs{\nabla u}^p-\phi(u)\leq 0$ directly.

Let $x_m$ be a point of maximum for $F$ on $M$. If $\nabla u|_{x_m}=0$, there is nothing to prove. If also $u(x_m)\neq 0$, then $u$ is smooth around $x_m$, but if $u(x_m)=0$, then $u$ has only $C^{2,\alpha}$ regularity around $x_m$ for $1<p<2$, and $C^{3,\alpha}$ if $p\geq 2$.

In the following, we will assume that $u$ is a $C^3$ function around $x_m$, and we will explain in Remark \ref{rem_reg} how to modify the proof if this is not the case (in particular, if $1<p<2$ and $u(x_m)=0$).

Since $P_u$ is an elliptic operator at $x_m$ we have
\begin{gather*}
 \nabla F|_{x_m}=0 \ \ \ \ P_u^{II} F|_{x_m}\leq 0 \ .
\end{gather*}
The first equation above implies
\begin{gather}\label{eq_nablaf}
\nabla \qua{\abs{\nabla u}^p - \phi(u)}= -\frac{\dot \psi F}{\psi^2} \nabla u \ ,\\
\notag \abs{\nabla u}^{p-2} A_u\equiv \abs{\nabla u}^{p-2} \frac{\pst{\nabla u}{H_u}{\nabla u}}{\abs{\nabla u}^2}= -\frac{1}{p} \ton{\frac{\dot\psi}{\psi^2} F - \dot \phi} \ .
\end{gather}
In order to study the second inequality, note that
\begin{gather*}
 \nabla_i \nabla_j F = \ddot \psi [\abs{\nabla u}^p-\phi(u)]\nabla_i u \nabla_j u + \dot \psi [\abs{\nabla u}^p-\phi(u)] \nabla_i\nabla_j u+\\
+ \dot \psi \nabla_j[\abs{\nabla u}^p-\phi(u)] \nabla _i u + \dot \psi \nabla_i[\abs{\nabla u}^p-\phi(u)] \nabla _j u+\\
+ \psi\qua{\nabla_i \nabla_j \abs{\nabla u}^p - \nabla_i \nabla_j \phi(u)} \ .
\end{gather*}
Using equation \eqref{eq_nablaf}, we have at $x_m$ 
\begin{gather*}
 \nabla_i \nabla_j F = \ddot \psi [\abs{\nabla u}^p-\phi(u)]\nabla_i u \nabla_j u + \dot \psi [\abs{\nabla u}^p-\phi(u)] \nabla_i\nabla_j u+\\
- 2\frac{\dot \psi^2 F}{\psi^2}\nabla_j u\nabla _i u + \psi\qua{\nabla_i \nabla_j \abs{\nabla u}^p - \nabla_i \nabla_j \phi(u)} \, .
\end{gather*}
By a straightforward calculation
\begin{gather*}
 \qua{\abs{\nabla u}^{p-2}\delta^{ij} +(p-2)\abs{\nabla u}^{p-4}\nabla^i u \nabla^j u }\nabla_i u \nabla_j u=(p-1)\abs{\nabla u}^p \, ,
\end{gather*}
and applying the definition of $P^{II}_u$ (see equation \eqref{deph_pu}), we get
\begin{gather*}
 0\geq P^{II}_u (F)=-2(p-1) \frac{\dot \psi^2}{\psi^2} F\ton{\frac{F}{\psi}+ \phi(u)}- F \frac{\dot \psi}{\psi} \lambda u^{(p-1)}+\\
+ (p-1) \frac{\ddot \psi}{\psi}F \ton{\frac F {\psi} + \phi(u)} + p \psi \ton{\frac 1 p P^{II}_u \abs{\nabla u }^p}-\psi P^{II}_u (\phi) \ .
\end{gather*}
Using Corollary \ref{cor_est_pu} we obtain the following relation valid at $x_m$
\begin{gather*}
a_1 F^2 +a_2 F +a_3\leq 0 \ ,
\end{gather*}
where
\begin{gather}\label{eq_c}
a_3=p\psi\qua{\frac{\lambda^2 u^{2p-2}}{n-1} + \frac{n+1}{n-1}\frac{p-1}{p}\lambda u^{(p-1)} \dot \phi +\right.\\
+ \left.\frac{n (p-1)^2}{p^2(n-1)} \dot \phi^2 -\lambda (p-1)\phi \abs u^{p-2}-\frac{p-1}p \phi \ddot \phi}\notag \ ,
\end{gather}
\begin{gather}\label{eq_b}
a_2= -(p-1)\frac{\dot \psi}{\psi} \lambda u^{(p-1)} \frac{n+1}{n-1}-\lambda p(p-1) \abs u^{p-2} +\\
- \frac{2n(p-1)^2}{p(n-1)}\frac{\dot \psi}{\psi} \dot \phi -(p-1) \ddot \phi +(p-1)\phi\ton{\frac{\ddot\psi}{\psi}-2 \frac{\dot\psi^2}{\psi^2}} \notag \ ,
\end{gather}
\begin{gather}\label{eq_a}
 a_1=\frac{p-1}{\psi}\qua{\frac{\ddot \psi}{\psi}+ \frac{\dot \psi^2}{\psi^2}\ton{\frac{n(p-1)}{p(n-1)}-2}} \ .
\end{gather}
Note that here both $\psi$ and $\phi$ are defined as functions of $u(x)$.\\
Now we want to have two smooth positive functions $\psi$ and $\phi$ such that $a_3=0$ and $a_1$ and $a_2$ are strictly positive everywhere, so that
\begin{gather*}
 F(a_1 F+a_2)\leq 0
\end{gather*}
and necessarily $F$ is nonpositive at its point of maximum, so it is nonpositive on the whole manifold $M$.
\paragraph{\textsc{Coefficient $a_3$}}
Since $a_3$ is a function of $u(x)$, we can eliminate this dependence and rewrite $a_3$ as $a_3\circ u^{-1}$ in the following way:
\begin{gather}\label{eq_cx}
a_3(s)=p\psi\qua{\frac{\lambda^2 \abs s^{2p-2}}{n-1} + \frac{n+1}{n-1}\frac{p-1}{p}\lambda s^{(p-1)} \dot \phi +\right.\\
+ \left.\frac{n (p-1)^2}{p^2(n-1)} \dot \phi^2 -\lambda (p-1)\phi \abs s^{p-2}-\frac{p-1}p \phi \ddot \phi}\notag \ ,
\end{gather}
where $s\in [-\xi,\xi\max\{u\}]$. Recall that $\phi=\dot w^p |_{w^{-1}(s)}$, so computing its derivatives it is important not to forget the derivative of $w^{-1}$, in particular
\begin{gather*}
 \dot \phi = p \abs{\dot w }^{p-2} \ddot w \ .
\end{gather*}
Remember that for a function in one dimension, the $p$-Laplacian is
\begin{gather*}
 \Delta_p w \equiv (p-1) \dot w^{p-2}\ddot w \, ,
\end{gather*}
so that we have
\begin{gather}\label{eq_phi}
 \frac{p-1}{p}\dot \phi = \Delta_p w \quad \text{and} \quad \frac{p-1} p \ddot \phi = \frac {d(\Delta_p w )} {dt} \frac{1}{\dot w} \ .
\end{gather}
With these substitutions, $a_3$ (or better $a_3\circ w:[w^{-1}(-\xi),w^{-1}(\xi \max\{u\})]\to \R$) can be written as
\begin{gather*}
 \frac{a_3}{p \psi}= \frac{n+1}{n-1} \lambda w^{(p-1)} \Delta_pw - \lambda(p-1) \abs w^{p-2} \dot w^p + \\
+\frac{\lambda^2\abs w^{2p-2}}{n-1}-\dot w^{p-1}\frac{d(\Delta_pw)}{dt}+ \frac{n}{n-1}(\Delta_pw)^2 \ .
\end{gather*}
Let $T$ be a solution to the ODE $\dot T=T^2/(n-1)$, i.e., either $T=0$ or $T=-\frac{n-1}{t}$ \footnote{note that from our point of view, there is no difference between $\frac{n-1}{t}$ and $\frac{n-1}{c+t}$, it is only a matter of shifting the variable $t$}. A simple calculation shows that we can rewrite the last equation as
\begin{gather*}
\frac{a_3}{p\psi}=\frac{1}{n-1}\ton{\Delta_p w - T \dot w^{(p-1)} +\lambda w^{(p-1)}}\ton{n\Delta_p w + T \dot w^{(p-1)}+\lambda w^{(p-1)}}+\\
- \dot w ^{(p-1)}\frac d {dx} \ton{\Delta_p w - T \dot w^{(p-1)} +\lambda w^{(p-1)}} \ .
\end{gather*}
This shows that if our one dimensional model satisfies the ordinary differential equation
\begin{gather}\label{eq_1d}
 \Delta_p w = T\dot w ^{(p-1)}-\lambda w^{(p-1)} \ ,
\end{gather}
then $a_3=0$. Note that intuitively equation \eqref{eq_1d} is a sort of damped (if $T\neq 0$) $p$-harmonic oscillator. Remember that we are interested only in the solution on an interval where $\dot w >0$.

\paragraph{\textsc{Coefficients $a_1$ and $a_2$}}\label{sec_psi}
To complete the proof, we only need to find a strictly positive $\psi \in C^2[-\xi,\xi\max\{u\}]$ such that both $a_1(u)$ and $a_2(u)$ are positive on all $M$. The proof is a bit technical, and relies on some properties of the model function $w$ that will be studied in the following section.

In order to find such a function, we use a technique similar to the one described in \cite[pag. 133-134]{new}. First of all, set by definition
\begin{gather}\label{ref_c}
X\equiv \lambda^{\frac{1}{p-1}} \frac {w(t)}{\dot w (t)}\ \ \ \ \  \psi(s)\equiv e^{\int h(s)} \ \ \  \ f(t)\equiv -h(w(t)) \dot w(t) \ ,
\end{gather}
so that
\begin{gather*}
 \dot f = - \dot h|_{w} \dot w ^2 - h|_{w} \ddot w = -\dot h|_{w} \dot w^2 - \frac {h|_{w}\dot w} {p-1} [T-X^{(p-1)}]=\\
=-\dot h|_{w} \dot w^2 + \frac {f} {p-1} [T-X^{(p-1)}] \ .
\end{gather*}
From equation \eqref{eq_a}, with our new definitions we have that
\begin{gather}\label{ref_d}
 \frac{a_1(w(t)) \psi(w(t))}{p-1}\dot w|_t ^2=
+\frac{f}{p-1}\qua{T-X^{(p-1)}} +f^2 \ton{\frac{p-n}{p(n-1)}}-\dot f\equiv \eta(f)-\dot f \ ,
\end{gather}
while if we use equations \eqref{eq_phi} and the differential equation \eqref{eq_1d} in equation \eqref{eq_b}, simple algebraic manipulations give
\begin{gather*}
\frac{a_2}{(p-1) \dot w^{p-2}} =-\frac{pT}{p-1} \ton{\frac{n}{n-1} T - X^{(p-1)}}+\\
-f^2 + f \qua{\ton{\frac{2n}{n-1} + \frac 1 {p-1}} T - \frac{p}{p-1} X^{(p-1)}} - \dot f\equiv\\
\equiv\beta(f) -\dot f \ .
\end{gather*}
The conclusion now follows from Lemma \ref{lemma_etabeta}.

\end{proof}

Analyzing the case with boundary, the only difference in the proof of the gradient comparison is that the point $x_m$ may lie in the boundary of $M$, and so it is not immediate to conclude $\nabla F|_{x_m}=0$. However, once this is proved it is evident that $P^{II}_u F|_{x_m}\leq 0$ and the rest of proof proceeds as before. In order to prove that $x_m$ is actually a stationary point for $F$, the (nonstrict) convexity of the boundary is crucial. In fact we have
\begin{lemma}\label{lemma_bou}
 Let $M$ be as in Theorem \ref{grad_est}, and let $\Delta_p$ be the $p$-Laplacian with Neumann boundary conditions. Then, using the notation introduced above, if $\nabla u|_{x_m}\neq 0$, 
\begin{gather}
 \nabla F|_{x_m}=0 \ .
\end{gather}
\end{lemma}
\begin{proof}
We can assume that $x_m\in \partial M$, otherwise there is nothing to prove. Let $\hat n$ be the outward normal derivative of $\partial M$.

Since $x_m$ is a point of maximum for $F$, we know that all the derivatives of $F$ along the boundary vanish, and that the normal derivative of $F$ is nonnegative
\begin{gather*}
 \ps{\nabla F}{\hat n}\geq 0 \ .
\end{gather*}
Neumann boundary conditions on $\Delta_p$ ensure that $\ps{\nabla u}{\hat n}=0$, and by direct calculation we have
\begin{gather*}
 \ps{\nabla F}{\hat n} = \qua{\ton{\abs{\nabla u}^p-\phi(u)}\dot \psi - \psi\dot \phi} \ps{\nabla u}{\hat n} +\\
+ p\psi(u)\abs{\nabla u}^{p-2} H_u(\nabla u,\hat n)=  p\psi(u)\abs{\nabla u}^{p-2} H_u(\nabla u,\hat n) \ .
\end{gather*}
Using the definition of second fundamental form $II(\cdot,\cdot)$, we can conclude
\begin{gather*}
 0\leq \ps{\nabla F}{\hat n} = p\psi(u)\abs{\nabla u}^{p-2} H_u(\nabla u,\hat n) = - p\psi(u)\abs{\nabla u}^{p-2} II(\nabla u,\nabla u)\leq 0 \ ,
\end{gather*}
and this proves the claim.
\end{proof}

According to Lemma \ref{lemma_n}, the fundamental estimate to prove the previous Theorem is valid also if we replace $n$ with any $n'\geq n$. So we can prove that
\begin{rem}\label{rem_m}
\rm The conclusions of Theorem \ref{grad_est} are still valid if we replace $n$ with any real $n'\geq n$.
\end{rem}
Note that while $n$ is the dimension of the Riemannian manifold under consideration, $n'$ does not represent any Riemannian entity.

\begin{rem}\label{rem_reg}
Before we end this section, we address the regularity issue in the gradient comparison theorem.
\end{rem}

 In the gradient comparison theorem, we assumed for simplicity $C^3$ regularity for $u$ around the point $x_m$. Since, as we have seen, we can assume without loss of generality that $\nabla u|_{x_m}\neq 0$, $C^3$ regularity is guaranteed if $p\geq 2$ or if $1<p<2$ and $u(x_m)\neq 0$.
 
 If we are in the case where $u(x_m)=0$ and $1<p<2$, $P^{II}_u(F)$ is not well defined. In particular, there are two terms diverging in this computation, one is $-\lambda \psi(u)\abs u ^{p-2} \abs{\nabla u}^{p}$ coming from $\psi(u)P^{II}_u (\abs{\nabla u}^2)$; the other is $-\psi(u) \ddot \phi(u) \abs{\nabla u}^p$, which comes from $-\psi(u) P^{II}_u (\phi(u))$.
 
 However, since $\nabla u|_{x}\neq 0$ in a neighborhood $U$ of $x_m$, we can still compute $P^{II}_u (F) $ on $U\setminus \{u=0\}$, which is an open dense set in $U$. If we assume that $\phi(u)=\dot w ^{p}|_{w^{-1}(u)}$, with $w$ satisfying \eqref{eq_1d}, it is easy to see that these two diverging terms precisely cancel each other in $U\setminus \{u=0\}$, and all the remaining terms in $P^{II}_u(F)$ are well-defined and continuous on $U$. Thus, even in this low-regularity case, the relation $P^{II}_u(F)|_{x_m}\leq 0$ is still valid, although the single terms $P^{II}_u(\abs{\nabla u}^p)$ and $P^{II}_u(\phi(u))$ are not well defined separately.

\subsection{One dimensional model}
This section contains the technical lemmas needed to study the properties of the solutions of the ODE
\begin{gather}\label{eq_1dm}
\begin{cases}
 \Delta_p w \equiv(p-1) \dot w^{p-2} \ddot w= T\dot w^{(p-1)} - \lambda w^{(p-1)}\\
w(a)=-1 \ \ \ \ \dot w(a)=0 
\end{cases}
\end{gather}
where either $T=-\frac{n-1}{t}$ or $T=0$. This second case has already been studied in Section \ref{sec_1d}, so we concentrate on the first one.\\
To underline that this equation is to be considered on the real interval $[0,\infty)$ and not on the manifold $M$, we denote by $t$ its independent variable. Notice that this ODE can be rewritten as
\begin{gather*}
 \frac d {dt} (t^{n-1} \dot w ^{(p-1)})+\lambda t^{n-1} w^{(p-1)}=0 \ ,
\end{gather*}
where $n\geq 2$ is the dimension of the manifold.

First of all we cite some known results on the solutions of this equation. 
\begin{teo}\label{teo_1dcont}
 If $a\geq0$, equation \eqref{eq_1dm} has always a unique solution in $C^{1}(0,\infty)$ with $\dot w^{(p-1)} \in C^1(0,\infty)$, moreover if $a=0$ the solution belongs to $C^{0}[0,\infty)$. The solution depends continuously on the parameters in the sense of local uniform convergence of $w$ and $\dot w$ in $(0,\infty)$. Moreover every solution is oscillatory, meaning that there exists a sequence $t_k \to \infty$ such that $w(t_k)=0$.
\index{oscillatory behaviour}
\end{teo}
\begin{proof}
Existence, uniqueness and continuity with respect to the initial data and parameters is proved for example in \cite[Theorem 3, pag 179]{W}, and its oscillatory behaviour has been proved in \cite[Theorem 3.2]{kus_osc}, or in \cite[Theorems 2.2.11 and 2.3.4(i)]{dosly} if $n>p$, while the case $n\leq p$ is treated for example in \cite[Theorem 2.1]{wong_osc} and \cite[Theorem 2.2.10]{dosly}. Note that all these reference deal with much more general equations than the one we are interested in.
\end{proof}

In the following we will be interested only in the restriction of the solution $w$ to the interval $[a,b(a)]$, where $b(a)>a$ is the first point where $\dot w(b)=0$. It is easily seen that $\dot w\geq 0$ on $[a,b(a)]$, with strict inequality in the interior of the interval. Let $t_0$ be the only point in $[a,b(a)]$ such that $w(t_0)=0$.

First of all, we state and prove a lemma needed to complete the gradient comparison. Fix $a$ and the corresponding solution $w$, and define for simplicity on $(a,b)$
\begin{gather*}
 X(t)\equiv {\lambda}^{\frac 1 {p-1}}\frac{w(t)}{\dot w(t)} \ \ \quad \ T(t)= -\frac{n-1}{t} \ .
\end{gather*}
By direct calculation
\begin{gather}\label{eq_dxp}
 \frac{d}{dt} X^{(p-1)} = (p-1) \lambda^{\frac{1}{p-1}}\abs X ^{p-2} - T X^{(p-1)} + \abs{X}^{2(p-1)}\\
\dot X = \lambda^{\frac{1}{p-1}} - \frac 1 {p-1} T X + \frac 1 {p-1} \abs{X}^p \ .
\end{gather}
\begin{lemma}\label{lemma_etabeta}
Let $\eta(s,t)$ and $\beta(s,t)$ be defined by
\begin{gather*}
 \eta(s,t)=\frac{s}{p-1}\qua{T-X^{(p-1)}} +s^2 \ton{\frac{p-n}{p(n-1)}} \ ,\\
\beta(s,t)=-\frac{pT}{p-1} \ton{\frac{n}{n-1} T - X^{(p-1)}}-s^2 +\\
+ s\qua{\ton{\frac{2n}{n-1} + \frac 1 {p-1}} T - \frac{p}{p-1} X^{(p-1)}} \ .
\end{gather*}
For every $\epsilon>0$, there exists a function $f:[a+\epsilon,b(a)-\epsilon]\to \R$ such that
\begin{gather}\label{eq_f}
 \dot f < \min\{\eta(f(t),t),\beta(f(t),t)\}
\end{gather}
\end{lemma}
\begin{proof}
We will prove that there exists a function $f:(a,b(a))\to \R$ which solves the ODE 
\begin{gather}\label{eq_fe}
 \begin{cases}
\dot f=\min\{\eta(f(t),t),\beta(f(t),t)\}\\
f(t_0)=\frac p {p-1} T_0 
 \end{cases}
\end{gather}
where we set $T_0=T(t_0)$. Then the Lemma follows by considering the solution to
\begin{gather}\label{ref_e}
\begin{cases}
\dot f_\eta=\min\{\eta(f_\eta(t),t),\beta(f_\eta(t),t)\}-\delta\\
f_\eta(t_0)=\frac p {p-1} T(t_0)\equiv  \frac p {p-1} T_0 
 \end{cases}
\end{gather}
By a standard comparison theorems for ODE, if $\delta>0$ is small enough, the solution $f_\eta$ is defined on $[a+\epsilon,b(a)-\epsilon]$ and satisfies inequality \eqref{eq_fe}.

Observe that by Peano's Theorem there always exists a solution to \eqref{eq_f} defined in a neighborhood of $t_0$. We show that this solution does not explode to infinity in the interior of $(a,b)$, while we allow the solution to diverge at the boundary of the interval. First of all note that for each $t\in (a,b(a))$
\begin{gather}
 \lim_{s\to \pm\infty} \min\{\eta(s,t),\beta(s,t)\}=-\infty \ .
\end{gather}
Then the solution $f$ is bounded from above in $(t_0,b)$ and bounded from below on $(a,t_0)$.

Set $\eta(f)(t)=\eta(f(t),t)$ and $\beta(f)(t)=\beta(f(t),t)$. A simple calculation shows that
\begin{gather}\label{eq_dQ}
 \eta(f)-\beta(f)= \frac{p-1}p\frac{n}{n-1}(f-y_1)(f-y_2) \ ,
\end{gather}
where
\begin{gather*}
 y_1\equiv \frac{p}{p-1}\ton{T-\frac{n-1}n X^{(p-1)}} \ \ \ \ \ y_2\equiv \frac{p}{p-1} T \ .
\end{gather*}

Now we prove that $f>y_1$ on $(t_0,b)$ and $f<y_1$ on $(a,t_0)$, and this completes the proof of the Lemma.
First we prove the inequality only in a neighborhood of $t_0$, i.e., we show that that there exists $\epsilon>0$ such that
\begin{gather}\label{eq_cfre}
 f(t)>y_1(t) \ \ \text{  for  } \ t_0<t<t_0+\epsilon \ ,\\
\notag f(t)<y_1(t) \ \ \text{  for  }\ t_0-\epsilon<t<t_0 \ .
\end{gather}
Indeed, using the ODE \eqref{eq_f}, at $t_0$ we have
\begin{gather*}
 \dot f (t_0) = \frac{p}{(p-1)(n-1)} T_0^2 \ ,
\end{gather*}
while, where defined,
\begin{gather*}
\dot y_1 = \frac p {p-1} \qua{\frac{T^2}{n-1} - \lambda^{\frac 1 {p-1}}\frac{(n-1)(p-1)}{n} \abs{X}^{p-2} +\right.\\
+\left. \frac{n-1}{n} T X^{(p-1)} - \frac{n-1}{n} \abs{X}^{2(p-1)}} \ .
\end{gather*}
If $p=2$, $\dot f(t_0) -\dot y(t_0)>0 $, and if $p<2$
\begin{gather*}
 \lim_{t\to t_0} \dot f(t) -\dot y(t)=+\infty \ .
\end{gather*}
Thus, if $p\leq 2$, it is easy to conclude that \eqref{eq_cfre} holds. Unfortunately, if $p>2$, $y_1\in C^1((a,b))$ but $\dot f(t_0) -\dot y(t_0)=0$.

However, by equation \eqref{eq_dQ}, $\eta(y_1)=\beta(y_1)=\min\{\eta(y_1),\beta(y_1)\}$. In particular
\begin{gather}\label{ref_f}
 \eta(y_1)-\frac{p}{(p-1)(n-1)} T^2=-\frac{p(2p-1)}{(p-1)^2 n} T X^{(p-1)}+\\
\notag+ \frac{p^2(n-1)}{(p-1)^2 n^2}\abs X^{2(p-1)}=c_1 X^{(p-1)} + o(X^{(p-1)}) \ ,
\end{gather}
while
\begin{gather*}
 \dot y_1 -\frac{p}{(p-1)(n-1)} T^2 = -c_2 \abs X^{p-2}+ O(X^{(p-1)}) \ ,
\end{gather*}
with $c_2>0$. If follows that in a neighborhood of $t_0$, $y_1$ solves the differential inequality
\begin{gather*}
\begin{cases}
 \dot y_1\leq \min\{\eta(y_1),\beta(y_1)\}\\
y_1(t_0)=\frac{p}{p-1} T_0 
\end{cases}
\end{gather*}
and, applying a standard comparison theorem for ODE \footnote{see for example \cite[Theorem 4.1 in Chapter 3]{hart}}, we can prove that the inequalities \eqref{eq_cfre} hold in a neighborhood of $t_0$.

To prove that they are valid on all $(a,b)$, suppose by contradiction that there exists some $t_1\in (a,t_0)$ such that $f(t_1)=y_1(t_1)$. The same argument works verbatim if $t_0< t_1 <b$.

Define $d(t)\equiv f(t)-y_1(t)$. By \eqref{eq_dQ}, $\dot f|_{t_1}=\eta(f(t_1),t_1)=\eta(y_1(t_1),t_1)$, which implies that
\begin{gather*}
 \dot d(t_1)= \frac{p(n-1)}{n} \lambda^{\frac 1 {p-1}}\abs X^{p-2}- \frac{p(n(p-1)+p)}{n(p-1)^2}T X^{(p-1)}+ \\
+\frac{(n-1)p(n(p-1)+p)}{n^2 (p-1)^2}\abs X^{2p-2}=\\
=p(n-1)\abs X^{p-2} \qua{\frac {\lambda^{\frac 1 {p-1}}} n+\frac{n(p-1)+p}{(p-1)^2 n^2}X\ton{X^{(p-1)}-\frac{n}{n-1}T}}\equiv \\
\equiv\frac{p(n-1)}{(p-1)^2 n^2}\abs X^{p-2}\kappa(t_1) \ ,
\end{gather*}
where we set
\begin{gather}\label{damnedH}
\kappa(t)\equiv n(p-1)^2\lambda^{\frac 1 {p-1}} + (n(p-1)+p)X\ton{X^{(p-1)}-\frac{n}{n-1}T} \ .
\end{gather}

We claim that $\kappa(t)$ is strictly positive, so that it is impossible for $d$ to be zero in a point different from $t_0$.

If $a>0$, it is evident that 
\begin{gather}\label{eq_liminf}
 \liminf_{t\to a} \kappa(t) >0 \quad \quad  \kappa(t_0)= n(p-1)^2\lambda^{\frac 1 {p-1}}>0 \quad \quad \liminf_{t\to a} \kappa(t) >0 \ .
\end{gather}
If $a=0$ the same conclusion holds by an approximation argument.

To show that $k(t)$ is positive everywhere, we argue by contradiction. Consider the first point $z\in (t_0,b)$ where $\kappa(z)=0$ (a similar argument works also if $z\in (a,t_0)$). At $z$ we have
\begin{gather*}
 X^{(p-1)}=-\frac{n(p-1)^2 \lambda^{\frac 1 {p-1}}}{(n(p-1)+p)X} + \frac{n}{n-1}T
\end{gather*}
and
\begin{gather*}
 \dot k = -\dot X \frac{n(p-1)^2 \lambda^{\frac 1 {p-1}}}{X}+ (n(p-1)+p)X\ton{\frac {d}{dt}X^{(p-1)}-\frac{nT^2}{(n-1)^2}} \ .
\end{gather*}
Using equation \eqref{eq_dxp} and some algebraic manipulations, we obtain
\begin{gather}\label{eq_kappa}
 \dot \kappa (z)= -\frac{n (-1 + p)^2 p^2}{(n (-1 + p) + p) X}\lambda^{\frac{2}{p-1}} \ .
\end{gather}

The right hand side has constant sign on $(t_0,b)$, and is never zero. For this reason, $z$ cannot be a minimum point for $k$, and by \eqref{eq_liminf} there exists a point $z'\in (z,b)$ such that $k(z')=0$ and $k(t)>0$ on $(z',b)$. Since $\dot k(z)$ and $\dot k(z')$ have the same sign, we have a contradiction.

\end{proof}

\subsection{Diameter comparison}
\index{diameter comparison}
As it will be clear later on, in order to obtain a sharp estimate on the first eigenvalue of the $p$-Laplacian we need to study the difference $\delta(a)=b(a)-a$ and find its minimum as a function of $a$. Note that if $T=0$, then the solution $w$ is invariant under translations and in particular $\delta(a)$ is constant and equal to $\frac{\pi_p}{\alpha}$, so we restrict our study to the case $T\neq 0$. For ease of notation, we extend the definition of $\delta$ by setting $\delta(\infty)=\frac{\pi_p}{\alpha}$.

In order to study the function $\delta(a)$, we introduce the \pf transformation (see \cite[section 1.1.3]{dosly} for a more detailed reference). Roughly speaking, the \pf transformation defines new variables $e$ and $\varphi$, which are the $p$-polar coordinates in the phase space of the solution $w$.
\index{Prufer@\pf transformation}
We set
\begin{gather}\label{ref_g}
  e(t)\equiv \ton{\dot w ^p + \alpha^p w^p}^{1/p}\ , \ \ \ \ \varphi(t)\equiv \operatorname{arctan_p} \ton{\frac{\alpha w}{\dot w}} \ .
\end{gather}
Recall that $\alpha= \ton{\frac{\lambda}{p-1}}^{1/p}$. It is immediate to see that
\begin{gather*}
 \alpha w = e \sinp (\varphi) \ , \ \ \dot w =e \cosp(\varphi) \ .
\end{gather*}
Differentiating, simplifying and using equation \eqref{eq_1dm}, we get the following differential equations for $\varphi$ and $e$
\begin{gather}\label{eq_pf}
 \dot \varphi = \alpha -\frac{T(t)}{p-1}\cosp^{p-1} (\varphi)\sinp(\varphi)=\alpha +\frac{n-1}{(p-1)t}\cosp^{p-1} (\varphi)\sinp(\varphi)\, , \\
\frac{\dot e}{e} =\frac {T(t)}{p-1}\cosp^{p} (\varphi)=-\frac{n-1}{(p-1)t}\cosp^{p} (\varphi)\notag \, .
\end{gather}
Rewritten in this form, it is quite straightforward to prove existence, uniqueness and continuous dependence of the solutions of the ODE \eqref{eq_1dm} at least f $a>0$. Moreover, note that the derivative of $\varphi$ is strictly positive.
Indeed, this is obviously true at the points $a, \ b(a)$ where $\dot w=0$ implies $\cosp(\varphi)=0$, while at the points where $\dot \varphi =0$ we have by substitution that $\ddot \varphi= \frac{\alpha} t $, which is always positive, so it is impossible for $\dot \varphi$ to vanish.
Moreover, a slight modification of this argument shows that $\dot \varphi$ is in fact bounded from below by $\frac{\alpha} n$.
Indeed, consider by contradiction the first point $\bar t$ in $[a,b]$ where $\dot \varphi(\bar t)=\frac{\alpha}{n} - \epsilon$. At this point we have
\begin{gather*}
 \ddot \varphi (\bar t)= \frac 1 {\bar t} \ton{-\frac{n-1}{(p-1)\bar t}\cosp^{p-1}(\varphi(\bar t))\sinp(\varphi(\bar t))+ \frac{n-1}{p-1}\ton{1-p\abs{\sinp(\varphi)}^p}\dot \varphi(\bar t)}=\\
 = \frac 1 {\bar t} \qua{\alpha - \dot \varphi(\bar t) + \frac{n-1}{p-1}\ton{1-p\abs{\sinp(\varphi)}^p}\dot \varphi(\bar t)}\geq \frac{1} {\bar t}\ton{\alpha - n \dot \varphi(\bar t)}\geq \frac n {\bar t} \epsilon \ .
\end{gather*}
Since $\dot \varphi(a)=\alpha$, it is evident that such a point cannot exist. This lower bound on $\dot \varphi$ proves directly the oscillatory behaviour of the solutions of ODE \eqref{eq_1dm}.

Note that, for every solution, $e$ is decreasing (strictly if $T\neq 0$), which means that the absolute value of local maxima and minima decreases as $t$ increases.

Now we are ready to prove the following lemma
\begin{lemma}\label{teo_delta_comp}
 For any $n>1$, the difference $\delta(a)$ is a continuous function on $[0,+\infty)$ and it is strictly greater than ${\pi_p}/\alpha$, which is the value of $\delta(\infty)$. Moreover, let $m(a)\equiv w(b(a))$, then for every $a\in [0,\infty)$
\begin{gather*}
 \lim_{a\to \infty} \delta(a)=\frac{\pi_p}{\alpha} =\delta(\infty) \ , \ \ \ \ \ m(a)<1 \quad \text{and} \quad \lim_{a\to \infty} m(a)=1 \ .
\end{gather*}
\end{lemma}
\begin{proof}
Continuity follows directly from Theorem \ref{teo_1dcont}. To prove the estimate, we rephrase the question using the \pf transformation. Consider the solution $\varphi$ of the initial value problem
\begin{gather*}
\begin{cases}
  \dot \varphi = \alpha + \frac{n-1}{(p-1)t}\cosp^{p-1} (\varphi)\sinp(\varphi)\\
\varphi(a)=-\frac{\pi_p}{2} 
\end{cases}
\end{gather*}
Then $b(a)$ is the only value $b>a$ such that $\varphi(b)=\frac{\pi_p}2$, which exists since $\dot \phi \geq \alpha/n$. Denote by $t_0\in (a,b)$ the only value where $\varphi(t_0)=0$. Since the function $\cosp^{p-1} (\varphi(t))\sinp(\varphi(t))$ is positive on $(t_0,b)$ and negative on $(a,t_0)$, it is easily seen that for $t\in [a,b]$
\begin{gather*}
 \frac{n-1}{(p-1)t}\cosp^{p-1} (\varphi(t))\sinp(\varphi(t))\leq \frac{n-1}{(p-1)t_0}\cosp^{p-1} (\varphi(t))\sinp(\varphi(t)) \ ,
\end{gather*}
so that $\varphi$ satisfies the differential inequality
\begin{gather}\label{eq_cfr}
0<  \dot \varphi \leq \alpha + \gamma\cosp^{p-1} (\varphi)\sinp(\varphi) \ ,
\end{gather}
where $\gamma=\frac{n-1}{(p-1)t_0}$. By a standard comparison theorem for ODE, $\varphi\leq \psi$ on $[a,b]$, where $\psi$ is the solution of the initial value problem
\begin{gather*}
\begin{cases}
  \dot \psi = \alpha + \gamma\cosp^{p-1} (\psi)\sinp(\psi)\\
\psi(a)=-\frac{\pi_p}{2} 
\end{cases}
\end{gather*}
By equation \eqref{eq_cfr}, it is evident that $\dot \psi>0$, moreover we can solve explicitly this ODE via separation of variables. Letting $c(a)$ be the first value $c>a$ such that $\psi(c)=\pi_p/2$, we have
\begin{gather*}
 c(a)-a=\int_{-\frac {\pi_p} 2}^{\frac {\pi_p} 2} \frac{d\psi}{\alpha + \gamma\cosp^{p-1} (\psi)\sinp(\psi)} \, .
\end{gather*}
Applying Jansen's inequality, and noting that $\cosp^{p-1}(\psi)\sinp(\psi)$ is an odd function, we obtain the estimate
\begin{gather}\label{eq_jan}
\frac{c(a)-a}{\pi_p}= \frac 1 {\pi_p} \int_{-\pi_p/2}^{\pi_p/2} \frac{d\psi}{\alpha+\gamma \cosp^{p-1} (\psi)\sinp(\psi)}\geq\\
\geq \qua{\frac 1 {\pi_p}\int_{-\pi_p/2}^{\pi_p/2} \ton{\alpha+\gamma \cosp^{p-1} (\psi)\sinp(\psi)}d\psi}^{-1} = \frac{1}{\alpha} \ . 
\end{gather}
Note that the inequality is strict if $\gamma\neq0$, or equivalently if $T\neq 0$.

Since $\varphi\leq \psi$, it is easily seen that $b(a)\geq c(a)$, and we can immediately conclude that $\delta(a) \geq \pi_p/\alpha$ with equality only if $a=\infty$.

The behaviour of $\delta(a)$ as $a$ goes to infinity is easier to study if we perform a translation of the $t$ axis, and study the equation
\begin{gather*}
\begin{cases}
  \dot \varphi = \alpha + \frac{n-1}{(p-1)(t+a)}\cosp^{p-1} (\varphi)\sinp(\varphi)\\
\varphi(0)=-\frac{\pi_p}{2} 
\end{cases}
\end{gather*}
Continuous dependence on the parameters of the equation allows us to conclude that if $a$ goes to infinity, then $\varphi$ tends to the affine function $\varphi_0(t)=-\frac{\pi_p}{2} + \alpha {t}$ in the local $C^1$ topology. This proves the first claim. As for the statements concerning $m(a)$, note that the inequality $m(a)<1$ follows directly from the fact that $\frac{\dot e }{e}<0$ if $T\neq 0$. Moreover we can see that $m(a)=\tilde w (\delta(a))$, where $\tilde w$ is the solution of
\begin{gather}
\begin{cases}
 \Delta_p \tilde w =(p-1) \dot {\tilde w}^{p-2} \ddot {\tilde w}= -\frac{n-1} {x+a} \dot {\tilde w}^{(p-1)} - \lambda \tilde  w^{(p-1)}\\
\tilde w(0)=-1 \ \ \ \ \dot {\tilde w}(0)=0 \ .
\end{cases}
\end{gather}
The function $\tilde w$ converges locally uniformly to $\sinp(\alpha t -\frac{\pi_p}{2})$ as $a$ goes to infinity, and since $\delta(a)$ is bounded from above, it is straightforward to see that $\lim_{a\to \infty} m(a)=1$.
\end{proof}

As an immediate consequence of the above Theorem, we have the following important
\begin{cor}\label{cor_delta}
 The function $\delta(a):[0,\infty]\to \R^+$ is continuous and
\begin{align}
 \delta(a)>\frac{\pi_p} \alpha \ \ \ \ \ &\text{ for  } a \in [0,\infty) \ ,\\
 \delta(a)=\frac{\pi_p} \alpha \ \ \ \ \ &\text{ for  } a =\infty \ . 
\end{align}
\end{cor}
Recall that $a=\infty$ if and only if $m(a)=1$, and also $\delta(a)=\frac{\pi_p} \alpha$ if and only if $m(a)=1$.

\subsection{Maxima of eigenfunctions and Volume estimates}\label{sec_max}
The next comparison theorem allows us to compare the maxima of eigenfunctions with the maxima of the model functions, so it is essential for the proof of the main Theorem. We begin with some definitions. Throughout this section, $u$ and $w$ are fixed and satisfy the hypothesis of Theorem \ref{grad_est}.
\begin{deph}
 Given $u$ and $w$ as in Theorem \ref{grad_est}, let $t_0\in (a,b)$ be the unique zero of $w$ and let $g\equiv w^{-1}\circ u$. We define the measure $m$ on $[a,b]$ by
\begin{gather*}
 m(A)\equiv \V(g^{-1}(A)) \ ,
\end{gather*}
where $\V$ is the Riemannian measure on $M$. Equivalently, for any bounded measurable $f:[a,b]\to \R$, we have
\begin{gather*}
 \int_a^b f(s) dm(s) = \int_M f(g(x))d\V(x) \ . 
\end{gather*}

\end{deph}
\begin{teo}\label{teo_comp}
 Let $u$ and $w$ be as above, and let
\begin{gather*}
 E(s)\equiv -\operatorname{exp}\ton{\lambda\int_{t_0}^s \frac{w^{(p-1)}}{\dot w^{(p-1)}} dt}\int_a ^s{w(r)}^{(p-1)} dm(r)\, .
\end{gather*}
Then $E(s)$ is increasing on $(a,t_0]$ and decreasing on $[t_0,b)$.
\end{teo}
Before the proof, we note that this theorem can be formulated in a more convenient way. Indeed, note that by definition
\begin{gather*}
 \int_{a}^s \pw(r)\ dm(r)= \int_{\{u\leq w(s) \}}u(x)^{(p-1)}\ d\V(x) \ .
\end{gather*}
Moreover, note that the function $w$ satisfies
\begin{gather*}
 \frac d {dt}(t^{n-1}\pdw)=-\lambda t^{n-1}\pw \ , \\
 -\lambda\frac{\pw}{\pdw}=
\frac d {dt} \log(t^{n-1}\pdw) \ ,
\end{gather*}
and therefore
\begin{gather*}
 -\lambda\int_a ^s \pw(t)t^{n-1}dt = s^{n-1}\pdw (s) \ , \\
\operatorname{exp}\ton{\lambda \int_{t_0}^s \frac{\pw}{\pdw} dt }=\frac{t_0^{n-1}\pdw(t_0)}{s^{n-1}\pdw(s)} \ .
\end{gather*}
Thus, the function $E(s)$ can be rewritten as
\begin{gather*}
 E(s)=C\frac{\int_a ^s\pw(t) \ dm(r)}{\int_a ^s \pw(t)\ t^{n-1}dt}=C\frac{\int_{\{u\leq w(s)\}}{u(x)}^{(p-1)}\ d\V(x)}{\int_a ^s {w(t)}^{(p-1)}\ t^{n-1}dt} \ ,
\end{gather*}
where $\lambda C^{-1}=t_0^{n-1}\pdw(t_0)$, and the previous Theorem can be restated as follows.
\begin{teo}\label{int_comp}
 Under the hypothesis of the previous Theorem, the ratio
\begin{gather*}
 E(s)=\frac{\int_a ^s\pw(r) \ dm(r)}{\int_a ^s \pw(t)\ t^{n-1}dt}=\frac{\int_{\{u\leq w(s)\}}{u(x)}^{(p-1)}\ d\V(x)}{\int_a ^s {w(t)}^{(p-1)}t^{n-1} dt}
\end{gather*}
is increasing on $[a,t_0]$ and decreasing on $[t_0,b]$.
\end{teo}
\begin{proof}[\textsc{Proof of Theorem \ref{teo_comp}}]
 Chose any smooth nonnegative function $H(s)$ with compact support in $(a,b)$, and define the function $G:[-1,w(b)]\to \R$ in such a way that
\begin{gather*}
 \frac d {dt} \qua{{G(w(t))}^{(p-1)}} =H(t) \quad \ \ G(-1)=0 \ .
\end{gather*}
It follows that
\begin{gather*}
 \pG(w(t))= \int_a ^t H(s) ds \, ,\quad (p-1) \abs{G(w(t))}^{p-2} \dot G(w(t)) \dot w (t) = H(t) \ .
\end{gather*}
Then choose a function $K$ such that $(t K(t))'=K(t)+t \dot K(t)=G(t)$. By the chain rule we obtain
\begin{gather*}
 \Delta_p (uK(u))=G^{(p-1)} (u) \Delta_p(u) + (p-1)\abs{G(u)}^{p-2} \dot G(u) \abs{\nabla u}^p \ .
\end{gather*}
Using the weak formulation of the divergence theorem, it is straightforward to verify that $$\int_M \Delta_p(uK(u))d \V =0\, ,$$ and so we get
\begin{gather*}
\frac{\lambda  }{p-1} \int_M \pu \pG(u) d\V(x)= \int_M \abs{G(u)}^{p-2} \dot G (u) \abs{\nabla u}^{p} d\V \ .
\end{gather*}
Since $\lambda>0$, applying the gradient comparison Theorem (Theorem \ref{grad_est}) we have
\begin{gather*}
\frac{\lambda  }{p-1} \int_M \pu \pG (u)d\V(x)\leq \int_M \abs{G(u)}^{p-2} \dot G (u) (\dot w \circ w^{-1}(u)) ^p d\V \ .
\end{gather*}
By definition of $d m $, the last inequality can be written as
\begin{gather*}
\frac{\lambda  }{p-1} \int_a^b \pw(s) \pG(w(s)) dm(s)\leq \\
\leq\int_a^b \abs{G(w(s))}^{p-2} \dot G (w(s)) (\dot w(s)) ^p dm(s) \ ,
\end{gather*}
and recalling the definition of $G$ we deduce that
\begin{gather*}
 \lambda \int_a^b \pw(s) \ton{\int_a^s H(t) dt}dm(s)= \lambda \int_a^b  \ton{\int_s^b \pw(t) dm(t)} H(s)ds \leq \\
\leq \int_a^b H(s) \pdw(s)\ dm(s) \ .
\end{gather*}
Since $\int_a^b \pw(t) dm(t)=0$, we can rewrite the last inequality as
\begin{gather*}
 \int_a^b  H(s) \qua{-\lambda\int_a^s \pw(t) dm(t)} ds \leq \int_a^b H(s) \pdw(s)\ dm(s) \ .
\end{gather*}
Define the function $A(s)\equiv -\int_a^s \pw(r) dm(r)$. Since the last inequality is valid for all smooth nonnegative function $H$ with compact support, then
\begin{gather*}
\pdw(s) dm(s)-\lambda A(s)ds \geq 0 
\end{gather*}
in the sense of distributions, and therefore the left hand side is a positive measure. In other words, the measure $\lambda Ads+ \frac{\pdw}{\pw} dA$ is nonpositive. If we multiply the last inequality by $\frac{\pw}{\pdw}$, and recall that $w\geq 0$ on $[t_0,b)$ and $w\leq 0$ on $(a,t_0]$, we conclude that the measure 
\begin{gather*}
 \lambda  \frac{\pw}{\pdw}Ads+  dA
\end{gather*}
is nonnegative on $(a,t_0]$ and nonpositive on $[t_0,b)$, or equivalently the function
\begin{gather*}
 E(s)= A(s) \operatorname{exp}\ton{\lambda\int_{t_0}^s \frac{\pw}{\pdw}(r) dr}
\end{gather*}
is increasing on $(a,t_0]$ and decreasing on $[t_0,a)$.
\end{proof}

Before stating the comparison principle for maxima of eigenfunctions, we need the following Lemma. The definitions are consistent with the ones in Theorem \ref{grad_est}.
\begin{lemma}\label{lemma_radepsilon}
 For $\epsilon$ sufficiently small, the set $u^{-1}[-1,-1+\epsilon)$ contains a ball of radius $r=r_\epsilon$, which is determined by
\begin{gather*}
r_\epsilon=w ^{-1}(-1+\epsilon)-a \ .
\end{gather*}
\begin{proof}
 This is a simple application of the gradient comparison principle (Theorem \ref{grad_est}). Let $x_0$ be a minimum point of $u$, i.e., $u(x_0)=-1$, and let $\bar x$ be another point in the manifold. Let $\gamma:[0,l]\to M$ be a unit speed minimizing geodesic joining $x_0$ to $\bar x$, and define $f(t)\equiv u(\gamma(t))$. It is easy to see that
\begin{gather}\label{ref_h}
\abs{\dot f (t)}=\abs{\ps{\nabla u |_{\gamma(t)}}{\dot \gamma(t)}}\leq \abs{\nabla u |_{\gamma(t)}}\leq \dot w |_{w^{-1}(f(t))}  \, ,
\end{gather}
and therefore 
\begin{gather*}
\frac d {dt} w^{-1}(f(t))\leq 1 \ ,
\end{gather*}
so that $a\leq w^{-1}(f(t))\leq a+t$, and since $\dot w$ is increasing in a neighborhood of $a$, we can deduce that
\[\dot w |_{w^{-1}f(t)}  \leq \dot w |_{a+t} \ .\]
By the absolute continuity of $u$ and $\gamma$, we can conclude that
\begin{gather*}
 \abs{f(t)+1}\leq\int_0^t  \dot w |_{a+s} ds = (w(a+ t)+1) \ .
\end{gather*}
This means that if $l=d(x_0,\bar x)< w ^{-1}(-1+\epsilon)-a$, then $u(\bar x)< -1+\epsilon$.
\end{proof}

\end{lemma}

And now we are ready to prove the comparison Theorem. 
\begin{teo}\label{teo_ccc}
If $u$ is an eigenfunction on $M$ such that $\min\{u\}=-1=u(x_0)$ and $\max\{u\}\leq m(0)=w(b(0))$, then for every $r>0$ sufficiently small, the volume of the ball centered at $x_0$ of radius $r$ is controlled by
\begin{gather*}
 \V(B(x_0,r))\leq c r^n \ .
\end{gather*}

\end{teo}
\begin{proof}
 Denote by $\nu$ the measure $t^{n-1}dt$ on $[0,\infty)$. For $k\leq -1/2^{p-1}$, applying Theorem \ref{int_comp} we can estimate
\begin{gather*}
 \V(\{u\leq k\})\leq -2 \int_{\{u\leq k\}}  u^{(p-1)} d\V \leq\\
\leq -2 C \int_{\{w\leq k\}} w^{(p-1)} d\nu\leq 2C\nu(\{w\leq k\}) \ .
\end{gather*}
If we set $k=-1+\epsilon$ for $\epsilon$ small enough, it follows from Lemma \ref{lemma_radepsilon} that there exist constants $C$ and $C'$ such that
\begin{gather*}
\V(B(x_0,r_\epsilon))\leq  \V(\{u\leq k\})\leq \\
\leq2C\nu(\{w\leq -1+\epsilon\})= 2C \nu([0,r_\epsilon])= C' r_\epsilon^n \ .
\end{gather*}

\end{proof}

\begin{cor}\label{cor_max}
With the hypothesis of the previous theorem, ${u^\star}=\max\{u\}\geq m(0)$.
\end{cor}
\begin{proof}suppose by contradiction that $\max\{u\}<m$. Then, by the continuous dependence of solutions of ODE \eqref{eq_1dm} on the parameters, there exists $n'>n$ ($n'\in \R$) such that $\max\{u\}\leq m_{n'}(0)$, i.e., there exists an $n'$ such that the solution $w'$ to the ode
\begin{gather*}
 \begin{cases}
 (p-1) \dot w'^{p-2} \ddot w' - \frac{n'-1}{t} \dot w'^{(p-1)} + \lambda w'^{(p-1)}=0\\
w'(0)=-1 \\ \dot w' (0)=0 
 \end{cases}
\end{gather*}
has a first maximum which is still greater than $\max\{u\}$. By Remark \ref{rem_m}, the gradient estimate $\abs{\nabla u}\leq \dot w'|_{w'^{-1}(u)}$ is still valid and so is also the volume comparison. But this contradicts the fact that the dimension of the manifold is $n$. In fact one would have that for small $\epsilon$ (which means for $r_\epsilon$ small)  $\V(B(x_0,r_\epsilon))\leq c r_\epsilon^{n'}$. Note that the argument applies even in the case where $M$ has a $C^2$ boundary.
\end{proof}

\subsection{Sharp estimate}
Now we are ready to state and prove the main theorem.
\begin{teo}\label{teo_main}
 Let $M$ be a compact Riemannian manifold with nonnegative Ricci curvature, diameter $d$ and possibly with convex boundary. Let $\lambda_p$ be the first nontrivial (=nonzero) eigenvalue of the $p$-Laplacian (with Neumann boundary condition if necessary). Then the following sharp estimate holds
\begin{gather*}
 \frac{\lambda_p}{p-1} \geq \frac{\pi_p^p}{d^p} \ .
\end{gather*}
Moreover a necessary (but not sufficient) condition for equality to hold is that $\max\{u\}=-\min\{u\}$.
\end{teo}
\begin{proof}
 First of all, we rescale $u$ in such a way that $\min\{u\}=-1$ and $0<\max\{u\}={u^\star}\leq 1$. Given a solution to the differential equation \eqref{eq_1dm}, let $m(a)\equiv w(b(a))$ the first maximum of $w$ after $a$. We know that this function is a continuous function on $[0,\infty)$, and
\begin{gather*}
 \lim_{a\to \infty} m(a) =1 \ .
\end{gather*}
By Corollary \ref{cor_max}, ${u^\star}\geq m(0)$. This means that for every eigenfunction $u$, there exists $a$ such that $m(a)={u^\star}$. If ${u^\star}=1$, then ${u^\star}=m(\infty)$.

We can rephrase this statement as follows: for any eigenfunction $u$, there exists a model function $w$ such that $\min\{u\}=\min\{w\}=-1$ and $0<\max\{u\}=\max\{w\}={u^\star}\leq 1$. Once this statement is proved, the eigenvalue estimate follows easily. In fact, consider a minimum point $x$ and a maximum point $y$ for the function $u$, and consider a unit speed minimizing geodesic (of length $l\leq d$) joining $x$ and $y$. Let $f(t)\equiv u(\gamma(t))$, and consider the subset $I$ of $[0,l]$ with $\dot f \geq0$. Then changing variables we get
 \begin{gather*}
  d\geq \int_0 ^l dt \geq \int_I dt \geq \int_{-1}^{u^\star} \frac{dy}{\dot f (f^{-1}(y))}\geq\int_{-1}^{u^\star} \frac{dy}{\dot w (w^{-1}(y))}=\\
=\int_a^{b(a)} 1 dt = \delta(a)\geq \frac{\pi_p}{\alpha} \ ,
 \end{gather*}
where the last inequality is proved in Corollary \ref{cor_delta}. This yields
\begin{gather*}
 \frac{\lambda}{p-1}\geq \frac{\pi_p^p}{d^p} \ .
\end{gather*}
Note that by Corollary \ref{cor_delta}, for any $a$, $\delta(a)\geq \frac{\pi_p}{\alpha}$ and equality holds only if $a=\infty$, i.e., only if $\max\{u\}=-\min\{u\}=\max\{w\}=-\min\{w\}$.
\end{proof}

\begin{rem}\label{rem_=}\rm{
 Note that $\max\{u\}=\max\{w\}$ is essential to get a sharp estimate, and it is the most difficult point to achieve. Analyzing the proof of the estimate in \cite{hui} with the tools developed in this article, it is easy to realize that in some sense the only model function used in \cite{hui} is $\sinp(\alpha x)$, which leads to $\phi(u)=\frac{\lambda}{p-1}(1-\abs u^p)$. Since the maximum of this model function is $1$, which in general is not equal to $\max\{u\}$, the last change of variables in the proof does not hold. Nevertheless $\max\{u\}>0$, and so one can estimate that
\begin{gather*}
 d\geq \int_{-1}^{u^\star} \frac{dy}{\dot w (w^{-1}(y))} > \int_{-1}^0 \frac{dy}{\dot w (w^{-1}(y))}= \int_{-1}^0 \frac{dy}{\alpha (1-y^p)^{1/p}}= \frac{\pi_p}{2 \alpha} \ ,
\end{gather*}
which leads to
\begin{gather*}
 \frac{\lambda}{p-1}>\ton{\frac{\pi_p}{2d}}^p \ .
\end{gather*}

}
\end{rem}

\subsection{Characterization of equality}
In this section we characterize the equality in the estimate just obtained, and prove that equality can be achieved only if $M$ is either a one dimensional circle or a segment.

In \cite{hang}, this characterization is proved for $p=2$ to answer a problem raised by T. Sakai in \cite{saka}. Unfortunately, this proof relies on the properties of the Hessian of an eigenfunction which are valid if $p=2$ and are not easily generalized for generic $p$.

Before we prove the characterization Theorem, we need the following Lemma, which is similar in spirit to \cite[Lemma 1]{hang}.
\begin{lemma}\label{lemma_=}
 Assume that all the assumptions of Theorem \ref{teo_main} are satisfied and assume also that equality holds in the sharp estimate. Then we can rescale $u$ in such a way that $-\min\{u\}=\max\{u\}=1$, and in this case $e^p=\abs{\nabla u}^p+\frac \lambda {p-1}\abs{u}^p$ is constant on the whole manifold $M$, in particular $e^p=\frac \lambda {p-1}$. Moreover, all integral curves of the vector field $X\equiv \frac{\nabla u}{\abs {\nabla u}}$ are minimizing geodesics on the open set $E\equiv\{\nabla u\neq 0\}=\{u\neq \pm1\}$ and for all geodesics $\gamma$, $\ps{\dot \gamma}{\frac{\nabla u}{\abs{\nabla u}}}$ is constant on each connected component of $\gamma^{-1}(E)$.
\end{lemma}
\begin{proof}
Using the model function $w(t)=\sinp(\alpha t)$ in the gradient comparison, we know that 
\begin{gather*}
 \abs{\nabla u}^p\leq \abs{\dot w }|_{w^{-1} (u)}^p= \frac{\lambda}{p-1} (1-\abs{u}^p) \ ,
\end{gather*}
so that $e^p\leq \frac{\lambda}{p-1}$ everywhere on $M$. Let $x$ and $y$ be a minimum and a maximum point of $u$ respectively, and let $\gamma$ be a unit speed minimizing geodesic joining $x$ and $y$. Define $f(t)\equiv u(\gamma(t))$. Following the proof of Theorem \ref{teo_main}, we know that
\begin{gather*}
 d\geq \int_{-1}^1 \frac{ds}{\dot f (f^{-1}(s))}\geq \int_{-1}^1 \frac{ds}{\dot w (w^{-1}(s))} = \frac{\pi_p}{\alpha} \ .
\end{gather*}
By the equality assumption, $\alpha d = \pi_p$, and this forces $\dot f(t) = \dot w |_{w^{-1} f(t)}$. So, up to a translation in the domain of definition, $f(t)=w(t)$, and $e^p|_{\gamma}=\frac \lambda{p-1}$ on the curve $\gamma$.

Now the statement of the Lemma is a consequence of the strong maximum principle (see for example \cite[Theorem 3.5 pag 34]{GT}). Indeed, consider the operator
\begin{gather*}
 L(\phi)=P_u(\phi) -\frac{(p-1)^2}{p \abs{\nabla u}^2}\ps{\nabla \ton{\abs{\nabla u}^p -\frac{\lambda}{p-1}u^{p}}}{\nabla \phi}+\\
+(p-2)^2\abs{\nabla u}^{p-4}\pst{\nabla u}{H_u}{\nabla \phi - \frac{\nabla u}{\abs{\nabla u}}\ps{\frac{\nabla u}{\abs{\nabla u}}}{\nabla \phi}} \ .
\end{gather*}
The second order part of this operator is $P^{II}_u$, so it is locally uniformly elliptic in the open set $E\equiv \{\nabla u \neq 0\}$, while the first order part (which plays no role in the maximum principle) is designed in such a way that $L(e^p)\geq 0$ everywhere. In fact, after some calculations we have that
\begin{gather}\label{eq_R}
 \frac {L(e^p)}{p\abs{\nabla u}^{2p-4}} = \ton{\abs{H_u}^2 - \frac{\abs{H_u(\nabla u)}^2}{\abs{\nabla u}^2}}+ \Ric(\nabla u,\nabla u)\geq 0 \ .
\end{gather}
Then by the maximum principle the set $\{e^p=\frac{\lambda}{p-1}\}$ is open and closed in $Z\equiv \{u\neq \pm1\}$, so it contains the connected component $Z_1$ containing $\gamma$.

Let $Z_2$ be any other connected component of $Z$ and choose $x_i \in Z_i$ with $u(x_i)=0$ for $i=1,2$. Let $\sigma$ a unit speed minimizing geodesic joining $x_1$ and $x_2$. Necessarily there exists $\bar t$ such that $\sigma(\bar t)\subset u^{-1}\{-1,1\}$, otherwise $Z_1=Z_2$. Without loss of generality, let $u(\bar t)=1$ and define $f(t)\equiv u(\sigma(t))$. Arguing as before we can conclude
\begin{gather*}
 d\geq \int_0^l dt= \int_0^{\bar t} dt + \int_{\bar t }^l dt \geq \int_{I_1}dt +\int_{I_2} dt \geq \\
\int_0^1 \frac{dy}{\dot f (f^{-1}(y))}-\int_{0}^{1} \frac{dy}{\dot f (f^{-1}(y))}\geq 2\int_{0}^1 \frac{dy}{\dot w (w^{-1}(y))} \geq \frac{\pi_p}{\alpha} \ ,
\end{gather*}
where $I_1\subset[0,\bar t]$ is the subset where $\dot f >0$ and $I_2\subset[\bar t ,l]$ is where $\dot f<0$. The equality assumption forces $\alpha d =\pi_p$ and so $\dot f(f^{-1}(t))=\dot w (w^{-1}(t))$ a.e. on $[0,\bar t]$ and $\dot f(f^{-1}(t))=-\dot w (w^{-1}(t))$ a.e. on $[\bar t,l]$, which implies that, up to a translation in the domain of definition,  $f(t)=w(t)=\sinp(\alpha t)$. This proves that for any connected component $Z_2$, there exists a point inside $Z_2$ where $e^p=\frac{\lambda}{p-1}$, and by the maximum principle $e^p=\frac{\lambda}{p-1}$ on all $Z$.

This also proves that $E=Z$. Moreover, for equality to hold in \eqref{eq_R}, $\Ric(\nabla u, \nabla u)$ has to be identically equal to zero on $Z$ and
\begin{gather}\label{eq_H}
 \abs{H_u}^2 =\frac{\abs{H_u(\nabla u)}^2}{\abs{\nabla u}^2} \ .
\end{gather}
Now the fact that $e^p$ is constant implies by differentiation that where $\nabla u\neq0$, i.e., on $Z$, we have
\begin{gather*}
 \abs{\nabla u}^{p-2} H_u(\nabla u) = - \frac{\lambda}{p-1}u^{p-1} \nabla u \ ,
\end{gather*}
and so
\begin{gather*}
 \abs{\nabla u }^{p-2} \pst{X}{H_u}{X}= - \frac{\lambda}{p-1}u^{p-1} \ .
\end{gather*}
This and equation \eqref{eq_H} imply that on $Z$
\begin{gather}\label{eq_HH}
\abs{\nabla u}^{p-2} H_u = - \frac{\lambda}{p-1} u^{p-1} X^\star \otimes X^\star \ .
\end{gather}
Now a simple calculation shows that
\begin{gather*}
 \nabla_X X = \frac 1 {\abs{\nabla u}}\nabla_{\nabla u} \frac{\nabla u }{\abs{\nabla u}}= \frac{1}{\abs{\nabla u}} (H_u (X)- H_u(X,X)X)=0 \ .
\end{gather*}
Which proves that integral curves of $X$ are geodesics. The minimizing property follows easily from $e^p\leq \frac{\lambda}{p-1}$.

As for the last statement, we have
\begin{gather*}
 \frac d {dt} \ps{\dot \gamma}{\frac{\nabla u}{\abs{\nabla u}}} = \frac{1}{\abs{\nabla u}} H_u(\dot \gamma,\dot \gamma)-\ps{\dot \gamma}{\nabla u} \frac{1}{\abs{\nabla u}^3} \pst{\nabla u}{H_u}{\dot \gamma}
\end{gather*}
and, by equation \eqref{eq_HH}, the right hand side is equal to $0$ where $\nabla u\neq 0$.

\end{proof}

After this proposition, we are ready to state and prove the characterization.
\begin{teo}\label{teo_=}
 Let $M$ be a compact Riemannian manifold with $\Ric \geq 0$ and diameter $d$ such that
\begin{gather}\label{eq_=}
 \frac{\lambda_p}{p-1} = \ton{\frac{\pi_p}{d}}^p \ .
\end{gather}
If $M$ has no boundary, then it is a one dimensional circle; if $M$ has boundary then it is a one dimensional segment.
\end{teo}
\begin{proof}

We prove the theorem studying the connected components of the set $N=u^{-1}(0)$, which, according to Lemma \ref{lemma_=}, is a regular submanifold. We divide our study in two cases, and we show that in both cases $M$ must be a one dimensional manifold (with or without boundary).
\paragraph{\textsc{Case 1, $N$ has more than one component}}Suppose that $N$ has more than one connected component. Let $x$ and $y$ be in two different components of $N$ and let $\gamma$ be a unit speed minimizing geodesic joining them.  Since $\ps{\dot \gamma}{\frac{\nabla u}{\abs{\nabla u}}}$ is constant on $E$, either $\gamma(t)=0$ for all $t$, which is impossible since by assumption $x$ and $y$ belong to two different components, or $\gamma$ must pass through a maximum or a minimum. Since the length of $\gamma$ is less than or equal to the diameter, we can conclude as in the previous Lemma that $\gamma(t)=\pm\sinp(\alpha t)$ on $[0,d]$. This in particular implies that $\ps{\dot \gamma}{\nabla u}=\pm\abs{\nabla u}$ at $t=0$, and since only two tangent vectors have this property, there can be only two points $y=\exp (x,\dot \gamma,d)$ (the exponential map from the point $x$ in the direction given by $\dot \gamma(0)$). Therefore the connected components of $N$ are discrete, and the manifold $M$ is one 
dimensional.

\paragraph{\textsc{Case 2, $N$ has only one component}} Now suppose that $N$ has only one connected component. Set $I=[-d/2,d/2]$ and define the function $h:N\times I \to M$ by
\begin{gather*}
 h(y,s)= \exp(y,X,s)\, ,
\end{gather*}
where $X=\nabla u/\abs{\nabla u}$ as before. We show that this function is a diffeomorphism and metric isometry.

First of all, let $\tilde h$ be the restriction of the map $h$ to $N\times I^\circ$ and note that if $\abs s < d/2$, by Lemma \ref{lemma_=}, $\tilde h(y,s)$ is the flux of the vector field $X$ emanating from $y$ evaluated at time $s$. Now it is easy to see that $u(h(y,s))=w(s)$ for all $s$. This is certainly true for all $y$ if $s=0$. For the other cases, fix $y$, let $\gamma(s)\equiv \tilde h(y,s)$ and $f(s)\equiv u(\tilde h(y,s))$. Then $f$ satisfies
\begin{gather*}
\dot \gamma = \frac{\nabla u}{\abs{\nabla u}} \ , \\
 \dot f = \ps{\nabla u}{\dot \gamma}= \abs {\nabla u}|_{\tilde h(y,s)} = \dot w |_{w^{-1} f(s)} \ .
\end{gather*}
Since $\dot w |_{w^{-1}(s)}$ is a smooth function, the solution of the above differential equation is unique and so $u(\tilde h(y,s))=f(s)=w(s)$ for all $y$ and $\abs s < d/2$. Note that by the continuity of $u$, we can also conclude that $u(h(y,s))=w(s)$ on all of $N\times I$.

The function $\tilde h$ is injective, in fact if $\tilde h(y,s)=\tilde h(z,t)$, then $w(s)=u(\tilde h(y,s))=u(\tilde h(z,t))=w(t)$ implies $s=t$. Moreover, since the flux of a vector at a fixed time is injective, also $y=z$.

Now we prove that $\tilde h$ is also a Riemannian isometry on its image. Let $\ps{\cdot}{\cdot}$ be the metric on $M$, $\ps{\cdot}{\cdot}_N$ the induced metric on $N$ and $\pps{\cdot}{\cdot}$ the product metric on $N\times I$. We want to show that $\pps{\cdot}{\cdot}=\tilde h^\star \ps{\cdot}{\cdot}$. The proof is similar in spirit to \cite[Lemma 9.7, step 7]{7i}. It is easily seen that
\begin{gather*}
 \tilde h^\star \ps{\cdot}{\cdot} (\partial s,\partial s)= \ps{\frac{\nabla u}{\abs{\nabla u}}}{\frac{\nabla u}{\abs{\nabla u}}}=1 \ ,
\end{gather*}
and for every $V\in T_{(y,s)} N$ we have
\begin{gather}\label{eq_perp}
\tilde  h^\star \psp (\partial s,V)=\ps{\frac{\nabla u}{\abs{\nabla u}}}{d\tilde h (V)}=0 \ .
\end{gather}
Indeed, since $\abs{\nabla u}$ is constant on all level sets of $u$, if $\sigma(t)$ is a curve with image in $N\times \{s\}$ and with $\sigma(0)=(y,s)$ and $\dot \sigma (0)=V$, then
\begin{gather*}
 \frac{d}{dt} \ton{\abs{\nabla u} \circ \tilde h} \sigma(t)=0 \ .
\end{gather*}
Recalling that $H_u=-\frac{\lambda}{p-1} u^{(p-1)} \abs{\nabla u}^{p-2} \frac{\nabla u}{\abs{\nabla u}}\otimes \frac{\nabla u}{\abs{\nabla u}}$, we also have
\begin{gather*}
 \frac{d}{dt} \abs{\nabla u} \circ \ton{\tilde h( \sigma(t))}= -\frac{\lambda}{p-1} u^{(p-1)} \abs{\nabla u}^{p-2} \ps{\frac{\nabla u}{\abs{\nabla u}}}{d\tilde h(V)} \ ,
\end{gather*}
and the claim follows.

Fix any $V,W \in TN$. For any $\abs t <d/2$ note that, by the properties of the Lie derivative, we have
\begin{gather*}
  \left.\frac d {ds}\right\vert_{s=t} (\tilde h^\star \psp)  (V,W)=\L _{\partial s} [d\tilde h \psp] (V,W)=[d\tilde h \L_{X}\psp](V,W)=\\
= \ps{\nabla_{d\tilde h(V)} \frac{\nabla u}{\abs{\nabla u}} }{d\tilde h(W)} + \ps{dh(V)}{\nabla_{d\tilde h(W)} \frac{\nabla u}{\abs{\nabla u}}} \ .
\end{gather*}
Since $\abs{\nabla u}$ is constant on the level sets, $d\tilde h(W) (\abs{\nabla u})=d\tilde h(V) (\abs{\nabla u})=0$, and therefore
\begin{gather*}
\left.\frac d {ds}\right\vert_{s=t} (h^\star \psp)  (V,W) = 2 H_u (d\tilde h(V),d\tilde h(W))=0 \ .
\end{gather*}
This implies that for every $y\in N$ fixed and any $V,W\in TN$, $\tilde h^\star|_{y,s} \psp (V,W)$ is constant on $(-d/2,d/2)$, and since $\tilde h$ is a Riemannian isometry by definition on the set $N\times \{0\}$, we have proved that $h^\star\psp = \ppsp$.

Now, $h$ is certainly a differentiable map being defined as an exponential map, and it is the unique differentiable extension of $\tilde h$.

Injectivity and surjectivity for $h$ are a little tricky to prove, in fact consider the length space $N\times I/\sim$, where $(y,s)\sim(z,t)$ if and only if $s=t=\pm d/2$, endowed with the length metric induced by $\tilde h$. It is still possible to define $h$ as the continuous extension of $\tilde h$, and $N\times I/\sim$ is a length space of diameter $d$, but evidently $h$ is not injective. This shows that injectivity of $h$ has to be linked to some Riemannian property of the manifold $M$.

\paragraph{\textbf{$h$ is surjective}}For any point $x\in M$ such that $u(x)\neq \pm1$ the flux of the vector field $X$ joins $x$ with a point on the surface $N$ and vice versa, so $h$ is surjective on the set $E$. The set of points $u^{-1}(1)$ (and in a similar way the set $u^{-1}(-1)$) has empty interior since $u$ is an eigenfunction with positive eigenvalue. Fix any $x\in u^{-1}(1)$. The estimate $\abs{\nabla u}^p+\frac{\lambda}{p-1} \abs u ^p\leq \frac{\lambda}{p-1}$ implies that any geodesic ball $B_\epsilon(x)$ contains a point $x_\epsilon$ with $w(\pi_p/2-\epsilon)<u(x_\epsilon)<1$. Let $y_\epsilon$ be the unique intersection between the flux of the vector field $X$ emanating from $x$ and the level set $N=u^{-1}(0)$, and consider the points $z_\epsilon = h(y_\epsilon,d/2)$. 

As we have seen previously, the curve $\gamma_\epsilon(t)=h(y_\epsilon,t)$ is a minimizing geodesic for $t\in [-d/2,d/2]$, so it is easy to see that $d(x_\epsilon,z_\epsilon)<\epsilon$. 
Let $\epsilon$ go to zero and take a convergent subsequence of $\{y_\epsilon\}$ with limit $y\in N$, then by continuity of the exponential map $h(y,d/2)=x$. Since $x$ was arbitrary, surjectivity is proved.

\paragraph{\textbf{$h$ is injective}}
Now we turn to the injectivity of $h$. Since $h$ is differentiable and its differential has determinant $1$ in $N\times I^\circ$, its determinant is $1$ everywhere and $h$ is a local diffeomorphism. By a similar density argument, it is also a local Riemannian isometry.

By the product structure on $N\times I$, we know that the parallel transport along a piecewise smooth curve $\sigma$ of the vector $X\equiv dh(\partial_s)$ is independent of $\sigma$. In particular if $\sigma$ is a loop, the parallel transport of $X$ along $\sigma$ is $\tau_\sigma(X)=X$.

Now consider two points $y,z\in N$ without any restriction on their mutual distance such that $h(y,d/2)=h(z,d/2)=x$. Let $\sigma$ be the curve obtained by gluing the geodesic $h(y,d/2-t)$ with any curve joining $x$ and $y$ in $M$ and with the geodesic $h(z,t)$. $\sigma$ is a loop around $x$ with
\begin{gather*}
 dh|_{(y,d/2)} \partial _s=X=\tau_\sigma(X)=dh|_{(z,d/2)} \partial _s \ .
\end{gather*}
Since by definition of $h$, $y=\exp(x,X,-d/2)$ and $z=\exp(x,\tau_\sigma X, -d/2)$, the equality $X=\tau_\sigma(X)$ implies $y=z$, and this proves the injectivity of $h$.

Now it is easily seen that $h$ is a metric isometry between $N\times I$ and $M$, which means that the diameter of $M$ is $d=\sqrt{d^2+diam(N)^2}$. Note that $diam(N)=0$ implies that $M$ is one dimensional (as in the case when $N$ has more than one connected component), and it is well-known that the only 1-dimensional connected compact manifolds are circles and segments.

As seen in Section \ref{sec_1d}, both these kind of manifolds realize equality in the sharp estimate for any diameter $d$, and so we have obtained our characterization.
\end{proof}

 \section{Negative lower bound}\label{sec_neg}
As in the case with zero lower bound on the Ricci curvature, the main tool used to prove the eigenvalue estimate is a gradient comparison with a one dimensional model function. However, there are some nontrivial differences between the two cases.

First of all, the one dimensional model is different and a little more complicated to study. In particular, the diameter comparison requires some additional care.

Even in this case, we will be able to obtain a gradient comparison theorem similar to \ref{grad_est} and its proof is going to be simpler than in the case with zero lower bound.

The volume comparison technique described in subsection \ref{sec_max} can be carried out verbatim also with generic lower bounds on Ricci, and thus we will obtain the sharp estimate for $\lambda_{1,p}$ in Theorem \ref{teo_main_proof}.

It is worth mentioning that some lower bounds for $\lambda_{1,p}$ have already been proved even in this case. In \cite{matei}, the author obtains lower and upper bounds on $\lambda_{1,p}$ as a function of Cheeger's isoperimetric constant \footnote{see in particular \cite[Theorems 4.1, 4.2 and 4.3]{matei}}. Among others, the following explicit (non sharp) lower bound has been proved by W. Lin-Feng in \cite{wang}.
\begin{theorem}\cite[Theorem 4]{wang}
 Let $M$ be an $n$-dimensional Riemannian manifold with $\Ric\geq -(n-1)$ and diameter $d$. Then there exist constants $C_1$ and $C_2$ depending on $n$ and $p$ such that
 \begin{equation}
  \lambda_{1,p} \geq C_1 \frac{ \operatorname{exp}[-(1+C_2 d)]}{d^p}\, .
 \end{equation}
\end{theorem}

\paragraph{Notation}
As before, we denote by $u$ a solution of
\begin{gather}
 \Delta_p(u)=-\lambda_{1,p} u^{(p-1)}
\end{gather}
with Neumann b. c. if necessary. We also rescale $u$ in such a way that
\begin{gather}
 \abs u \leq 1 \quad \quad \min\{u\}=-1 \quad \quad 0<\max\{u\}={u^\star}\leq 1\, .
\end{gather}

\subsection{One Dimensional Model}\label{sec_1d_neg}
Although similar to the one dimensional model studied in subsection \ref{sec_1d}, the model for the negative lower bound has many tricky technical points. Moreover, we were not able to find an analogue of the existence and uniqueness Theorem \ref{teo_1dcont} available in literature. Thus, even though the technique used to prove this theorem are quite standard, for the sake of completeness, we report the proof in details. 

First of all, fix $n$ and $k<0$, and for $i=1,2,3$ define the nonnegative functions $\tau_i$ on $I_i\subset \R$ by:
\begin{enumerate}
 \item $\tau_1=\operatorname{sinh}\ton{\sqrt{-k}\ t}$, defined  on $I_1=(0,\infty)$,
 \item $\tau_2=\operatorname{exp}\ton{\sqrt{-k}\ t}$ on $I_2=\R$ ,
 \item $\tau_3=\operatorname{cosh}\ton{\sqrt{-k}\ t}$ on $I_3=\R$ .
\end{enumerate}
and let $\mu_i= \tau_i^{n-1}$. Note that all the functions $\tau_i$ satisfy the ODE
\begin{gather*}
 \ddot \tau_i = -k \tau _i
\end{gather*}
The reason for these definitions is the following: fix any $0<\epsilon\leq 1$ and consider a manifold $M$ defined by the warped product
\begin{gather}
 M=[a,b]\times_{\epsilon \tau_i} S^{n-1}\, ,
\end{gather}
where $S^{n-1}$ is the standard $n-1$ dimensional sphere, with the metric given in product coordinates $(t,x)$ on $M$ by
\begin{gather*}
 g= dt^2 + \epsilon^2 t_i(t)^2 g_{S^{n-1}}\, ,
\end{gather*}
where $g_{S^{n-1}}$ is the standard metric on the sphere. By standard computations \footnote{see the subsection \ref{sec_warped}}, the Ricci tensor on $M$ is satisfies
\begin{gather*}
 \Ric(\partial t ,\partial t)=(n-1)k\\
 \Ric \geq (n-1)k\ g\, ,
\end{gather*}
and $\mu_i$ measures the volume of radial slices. Indeed,
\begin{gather}
 \Vol([c,d]\times X) =\epsilon^{n-1} \Vol(S^{n-1}) \int_{c}^d  \mu_i(t) dt\, .
\end{gather}
\begin{remark}
\rm Note that if we choose $\epsilon=1$, then $[0,d]\times_{\tau_1} S^{n-1}$ is the geodesic ball of radius $d$ in the hyperbolic space.
\end{remark}
Define also $T_i=-\frac{\dot \mu_i}{\mu_i}$, i.e.:
\begin{enumerate}
 \item $T_1=-(n-1)\sqrt{-k}\operatorname{cotanh}\ton{\sqrt{-k}\ t}$, defined on $I_1=(0,\infty)$ ,
 \item $T_2=-(n-1)\sqrt{-k}$, defined on $I_2=\R$ ,
 \item $T_3= -(n-1)\sqrt{-k}\operatorname{tanh}\ton{\sqrt{-k}\ t}$, defined on $I_3=\R$ .
\end{enumerate}
Note that all functions $T_i$ satisfy
\begin{gather}
 \dot T= \frac{T^2}{n-1} + (n-1)k\, .
\end{gather}

Now we are ready to introduce our one dimensional model functions.
\begin{definition}\label{deph_1dm}
 Fix $\lambda>0$. Define the function $w=w_{k,n,i,a}$ to be the solution to the initial value problem on $\R$
 \begin{gather}\label{eq_1dm_neg}
  \begin{cases}
   \frac{d}{dt} \dot w ^{(p-1)} - T_i \dot w ^{(p-1)} + \lambda w^{(p-1)} =0\\
   w(a)=-1 \quad \dot w (a)=0
  \end{cases}
 \end{gather}
where $a\in I_i$. Equivalently, $w_{k,n,i,a}$ are the solutions to
 \begin{gather}\label{eq_1dm_negmu}
  \begin{cases}
   \frac{d}{dt}\ton{\mu_i \dot w ^{(p-1)}} +\lambda \mu_i w^{(p-1)} =0\\
   w(a)=-1 \quad \dot w (a)=0
  \end{cases}
 \end{gather}
\end{definition}
\begin{remark}\label{rem_w}
\rm Define on $M=[a,b]\times_{\tau_i} S^{n-1}$ the function $u(t,x)=w(t)$. It is easy to realize that $u$ solves the eigenvalue equation $\Delta_p(u)+\lambda u^{(p-1)}=0$ on $M$ for any $X$. Moreover, if $\dot w(b)=0$, then $u$ has Neumann boundary conditions on such manifold.
\end{remark}
\begin{remark}\rm{
 Note that we could also define the functions:
\begin{enumerate}
 \item[\rm {4.}] $\tau_4=\operatorname{exp}\ton{-\sqrt{-k}\ t}$ on $I_4=\R$ ,
 \item[\rm {5.}] $\tau_5=\operatorname{sinh}\ton{-\sqrt{-k}\ t}$, defined  on $I_5=(-\infty,0)$,
\end{enumerate}
and similarly also $\mu_i$ and $T_i$ for $i=4,5$, and obtain similar properties to the functions $i=1,2,3$. However, it is easily seen that $\tau_5(x)=-\tau_1(-x)$, and similarly for $\tau_2$ and $\tau_4$. For this reason studying them adds no significant generality to the work.}
\end{remark}

First of all, we prove existence, uniqueness and continuous dependence on the parameters for the solutions of the IVP \eqref{eq_1dm_neg}. In order to do so, we use again the \pf transformation introduced in the previous section. For the reader's convenience, we briefly recall the definition.
\index{Prufer@\pf transformation}
\begin{definition}\label{deph_pf}
Fix some $w=w_{k,n,i,a}$. Define the functions $e=e_{k,n,i,a}\geq 0$ and $\phi=\phi_{k,n,i,a}$ by
\begin{gather}\label{eq_wpf}
 \alpha w = e \sinp (\phi) \ \ \ \dot w =e \cosp(\phi)\, ,
\end{gather}
or equivalently
\begin{gather*}
  e\equiv \ton{\dot w ^p + \alpha^p w^p}^{1/p}\ \ \ \ \ \phi\equiv \operatorname{arctan_p} \ton{\frac{\alpha w}{\dot w}}\, .
\end{gather*}
Note that the variable $\phi$ is well-defined up to $2\pi_p$ translations. 
\end{definition}

Let $w_{k,n,i,a}$ satisfy \eqref{eq_1dm_neg}. Differentiating, substituting and using equation \eqref{eq_1dm_neg} we get that $\phi=\phi_{k,n,i,a}$ and $e=e_{k,n,i,a}$ satisfy the following first order IVPs
\begin{gather}\label{eq_pf_neg}
 \begin{cases}
  \dot \phi = \alpha - \frac{T_i}{p-1}\cosp^{p-1} (\phi)\sinp(\phi) \\
  \phi(a)=-\frac{\pi_p}2  
 \end{cases}\\
 \begin{cases}\label{eq_pf_nege}
  \frac d {dt} \log(e) = \frac{\dot e}{e} =\frac{T_i}{(p-1)}\cosp^{p} (\phi)\\
  e(a)=\alpha
 \end{cases}
\end{gather}

Since both $\sinp$ and $(p-1)^{-1}\cosp^{p-1}$ are Lipschitz functions with Lipschitz constant $1$, it is easy to apply Cauchy's theorem and prove existence, uniqueness and continuous dependence on the parameters. Indeed, we have the following.
\begin{prop}\label{prop_exun1}
 If $T=T_2,T_3$, for any $a\in \R$ there exists a unique solution to \eqref{eq_1dm_neg} defined on all $\R$. The solution $w$ is of class $C^1(\R)$ with $\dot w^{(p-1)}\in C^1(\R)$ as well. Moreover, the solution depends continuously on the parameters $a$, $n\in \R$ and $k<0$ in the sense of local uniform convergence of $w$ and $\dot w$ in $\R$.
\end{prop}
Note that if $T=T_2$, then the model is also translation invariant.

A similar existence and uniqueness result is valid also if $T=T_1$ as long as $a>0$, while the boundary case deserves some more attention. In order to study it, note that a function $w\in C^1(I_1)$ is a solution to \eqref{eq_1dm_negmu} if and only if
\begin{gather}\label{eq_int}
 w(r)=w(a)  +\int_a^r  \ton{\frac{\mu_1(a)}{\mu_1(t)} \dot w^{(p-1)}(a)-\lambda\int_a^t \frac{\mu_1(s)}{\mu_1(t)} w^{(p-1)}(s) ds}^{\frac 1 {(p-1)}} dt\, .
\end{gather}
If $a=0$, necessary conditions for the solution to exists are $\dot w(0)=0$ and $w\in C^1[0,\infty)$. 

Define the continuous real function
\begin{gather}
 h(s)= \begin{cases}
        s & \text{ if } -2\leq  s \leq -\frac 1 2\\
	-2 & \text{ if } s\leq -2 \\
	- \frac 1 2 & \text{ if } s\geq - \frac 1 2
       \end{cases}
\end{gather}
and the mappings $L,\tilde L:C[0,1]\to C[0,1]$ by
\begin{align}
 L(w)(r)& = -1  - \int_0^r  \ton{\lambda\int_a^t \frac{\mu_1(s)}{\mu_1(t)} w^{(p-1)}(s) ds}^{\frac 1 {(p-1)}} dt\, , \\
 \tilde L(w)(r)&=-1  - \int_0^r  \ton{\lambda\int_a^t \frac{\mu_1(s)}{\mu_1(t)} h(w)^{(p-1)}(s) ds}^{\frac 1 {(p-1)}} \, .
\end{align}
It is evident that if $w$ is a fixed point for $L$, then it is a solution to \eqref{eq_1dm_negmu}. In a similar way, if $w$ is a fixed point of $\tilde L$, then it is a solution to \eqref{eq_1dm_negmu} as long as $-2\leq w(t)\leq -\frac 1 2$, so at least in a right neighborhood of zero.

Here we prove that, when restricted to an appropriate domain $[0,c]$, either $L$ or $\tilde L$ is a contraction on $C^0[0,c]$. Then we can conclude existence and uniqueness of the solution using the Banach fixed-point theorem. We divide the proof in two cases: $1<p<2$ and $p\geq 2$.

If $p\geq 2$, we can use the Minkowski inequality with exponent $p-1$. This allows us to estimate
\begin{gather}
 L(w)(r)- L(v)(r)= \\
\notag=\int_0^r  \qua{\ton{\lambda\int_a^t \frac{\mu_1(s)}{\mu_1(t)} v^{(p-1)}(s) ds}^{\frac 1 {(p-1)}}-\ton{\lambda\int_a^t \frac{\mu_1(s)}{\mu_1(t)} w^{(p-1)}(s) ds}^{\frac 1 {(p-1)}}} dt\leq \\
\leq \notag \int_{0}^r \ton{\lambda\int_a^t \frac{\mu_1(s)}{\mu_1(t)} \abs{w-v}^{p-1}(s) ds}^{\frac 1 {p-1}} dt \leq \lambda^{\frac 1 {p-1}} \int_{0}^r dt\norm{w-v}_{L^{p-1}[0,r]}\, ,
\end{gather}
and so
\begin{gather}
 \norm{Lw-Lv}_{C^0[0,c]}\leq \lambda^{\frac 1 {p-1}} c^{\frac p {p-1}} \norm{w-v}_{C^0[0,c]}\, .
\end{gather}
So for $0<c<\lambda^{-\frac 1 {p}}$, $L$ is a contraction on $C^0[0,c]$.

If $1<p<2$, note that if $w\in C^0[0,1]$, then $\tilde L(w)$ is $C^1$, $\tilde L(w)(0)=-1$ and
\begin{gather}
\abs{ \left.\frac d {dt} \tilde Lw \right \vert _{t=r}}^{p-1} \leq \lambda \int_0^r 2^{p-1}\leq 2^{p-1}\lambda r\, .
\end{gather}
Since for $1<p<2$, and for any $a,b\in \R$
\begin{gather}
 a^{(p-1)}-b^{(p-1)}=\int_{b}^a (p-1)\abs s ^{p-2} ds\, ,
\end{gather}
we can estimate
\begin{gather}
 \abs{a^{(p-1)}-b^{(p-1)}}\geq (p-1)\max\{\abs a, \abs b\}^{p-2}\abs{b-a}  \, ,
\end{gather}
and, if $-2\leq a\leq b\leq -\frac 1 2$
\begin{gather}
 \abs{a^{(p-1)}-b^{(p-1)}}\leq 2^{2-p}(p-1)\abs{b-a}  \, .
\end{gather}
Using these estimates, we have
\begin{gather}
 \abs{\ton{\frac {d}{dr} \tilde L w (r)}^{(p-1)} - \ton{\frac {d}{dr} \tilde L v (r)}^{(p-1)}}\leq \lambda \int_0^r \frac {\mu_1(s)}{\mu_1(r)} \abs{h(w)^{(p-1)} - h(v)^{(p-1)}}ds\leq\\
\leq 2^{2-p}\lambda (p-1) r \norm{w-v}_{C^0[0,r]} \, ,
\end{gather}
and also
\begin{gather}
 \abs{\ton{\frac {d}{dr} \tilde L w (r)}^{(p-1)} - \ton{\frac {d}{dr} \tilde L v (r)}^{(p-1)}}\geq (p-1) 2^{p-2}\ton{\lambda r}^{\frac{p-2}{p-1}} \abs{\frac {d}{dr} \tilde L w (r) -\frac {d}{dr} \tilde L v (r)}\, .
\end{gather}
These equations allow us to bound the $C^0$ norm of $\tilde Lw - \tilde Lv$. Namely
\begin{gather}
\norm{\tilde Lw-\tilde Lv}_{C^0[0,c]} \leq 2^{4-2p}(\lambda c)^{\frac{1}{p-1}}c \norm{w-v}_{C^0[0,c]} \, ,
\end{gather}
and, for $c$ small enough, $\tilde L$ is a contraction on $C^0[0,c]$. 

Once existence and uniqueness are established on $[0,c]$, using the \pf transformation it is possible to prove existence and uniqueness on all $\R$.

With simple modifications to the above estimates, it is easy to prove also continuous dependence of the solution $w$ on $a\geq 0$, $n\geq 1$, $\lambda >0$ and $k<0$ in the sense of $C^1$ uniform convergence on compact subsets of $(0,\infty)$.

We summarize the results obtained here in the following Proposition, which is the analogue of Proposition \ref{prop_exun1} for the case where $T=T_1$.
\begin{prop}\label{prop_exun2}
 If $T=T_1$, for any $a>0$ there exists a unique solution to \eqref{eq_1dm_negmu} defined (at least) on $(0,\infty)$. The solution $w$ is of class $C^1((0,\infty))$ with $\dot w^{(p-1)}\in C^1((0,\infty))$ as well.  
 
 If $a=0$, the solution is unique and belongs to $C^1([0,\infty))$. 
 Moreover, the solution depends continuously on the parameters $a\geq 0$, $\lambda>0$, $n\geq 1$ and $k<0$ in the sense of local uniform convergence of $w$ and $\dot w$ in $(0,\infty)$.
\end{prop}
\subsection{Gradient Comparison}
\index{gradient comparison}
The following gradient comparison Theorem is similar to Theorem \ref{grad_est}. However, in this case we use a slightly simpler technique for the proof. This technique is based on a contradiction argument similar to the one used to prove \cite[Theorem 1]{kro}, and a careful use of the maximum principle. In particular, by ellipticity $P_u(f)|_x\leq 0$ if $x$ is an internal maximum point of a function $f$ which is smooth in a neighborhood of $x$.
\begin{teo}[\textsc{Gradient comparison Theorem}]\label{grad_est_neg}
 Let $M$ be a compact $n$-dimensional Riemannian manifold with Ricci curvature bounded from below by $(n-1)k<0$, and, possibly, with a $C^2$ convex boundary. Let $u$ be a solution of
 \begin{gather}
  \Delta_p (u) = -\lambda u^{(p-1)}
 \end{gather}
rescaled in such a way that $-1=\min\{u\} <0<\max\{u\}\leq 1$. Let $w$ be a solution of the one dimensional initial value problem
\begin{gather}
\begin{cases}
 \frac d {dt}\dot w ^{(p-1)} - T \dot w ^{(p-1)} +\lambda w^{(p-1)}=0\\
 w(a)=-1 \quad \dot w(a)=0
\end{cases}
\end{gather}
where $\dot T= \frac{T^2}{n-1} + (n-1)k$. Consider an interval $[a,b]$ in which $\dot w \geq 0$. If $$[\min(u),\max(u)]\subset [-1,w(b)]\, ,$$ then
\begin{gather*}
\abs{\nabla u (x)}\leq \dot w(w^{-1}(u(x)))
\end{gather*}
for all $x\in M$.
\end{teo}


\begin{proof}
Suppose for the moment that $\partial M$ is empty; the modification needed for the general case will be discussed in Remark \ref{rem_bou}.

In order to avoid problems at the boundary of $[a,b]$, we assume that 
\begin{gather*}
[\min\{u\},\max\{u\}]\subset (-1,w(b))\, , 
\end{gather*}
so that we have to study our one dimensional model only on compact subintervals of $(a,b)$. Since $\max\{u\}>0$, we can obtain this by multiplying $u$ by a positive constant $\xi<1$. If we let $\xi\to 1$, then the original statement is proved.

Using the notation introduced in subsection \ref{sec_1d_neg}, we define the family of functions on $M$
\begin{gather*}
 F_{z,c}\equiv \abs{\nabla u}^p - (c\dot w_{z,n,i,a})^p|_{(cw_{z,n,i,a})^{-1} u(x)}
\end{gather*}
for $c\geq 1$ \footnote{note that, since $\min\{u\}=-\xi\simeq -1$, $F$ is not well defined for all $c<1$}. Since $w_{z,n,i,a}$ depends continuously in the $C^1$ sense on $z$, these functions are well-defined and continuous on $M$ if $z$ is sufficiently close to $k$.

In the following, we consider $i$, $a$, $\lambda$ and $n$ to be fixed parameters, while we will need to let $z$ vary in a neighborhood of $k$.

Using a contradiction argument, we prove that for every $\epsilon>0$ sufficiently small, $F_{k-\epsilon,1}\leq 0$ on all of $M$. Define $\bar F_{k-\epsilon,c}=\max\{F_{k-\epsilon,c}(x), \ x\in M\}$, and suppose by contradiction that $\bar F_{k-\epsilon,1}>0$. Since
\begin{gather}
 \lim_{c\to \infty} \bar F_{k-\epsilon,c} =-\infty\, ,
\end{gather}
there exists a $\bar c \geq 1$ such that $\bar F_{k-\epsilon,\bar c}=F_{k-\epsilon, \bar c}(\bar x)=0$. It is clear that, at $\bar x$, $\abs{\nabla u }> 0$.

Hereafter, we will assume that $u$ is a $C^3$ function in a neighborhood of $\bar x$, so that $F$ will be a $C^2$ function in a neighborhood of this point. This is certainly the case if $u(\bar x)\neq 0$, or if $p\geq 2$. If $1<p<2$ and $u(\bar x)=0$, then $u$ has only $C^{2,\alpha}$ regularity around $\bar x$. However, this regularity issue is easily solved, as we will see in Remak \ref{rem_reg_neg}.

Since we are assuming $\partial M=\emptyset$, $\bar x$ is internal maximum point, and thus
\begin{gather}\label{eq_max1}
 \nabla F_{ k-\epsilon, \bar c}(\bar x)=0 \\
 P^{II}_u (F_{ k-\epsilon, \bar c})(\bar x)\leq 0\label{eq_max2}
\end{gather}
Simple algebraic manipulations applied to equation \eqref{eq_max1} yield to the following relations valid at $\bar x$
\begin{align*}
 p\abs{\nabla u}^{p-2} H_u\nabla u & = \frac p {p-1} \Delta_p (\bar cw_{k-\epsilon,n,a}) \nabla u\, ,\\
\abs{\nabla u}^{p-2} A_u = \abs{\nabla u}^{p-2} \frac{\pst{\nabla u}{H_u}{\nabla u}}{\abs{\nabla u}^2}&= \frac{1}{p-1}\Delta_p (\bar cw_{k-\epsilon,n,a})\, .
\end{align*}
Substituting inequality \eqref{eq_pbochest} into \eqref{eq_max2}, we have that, at $\bar x$
\begin{gather}
0\geq\frac 1 p P^{II}_u(F_{ k-\epsilon, \bar c})\geq -\lambda(p-1) u^{p-2} \abs{\nabla u}^p + (p-2)\lambda u^{p-1} \abs{\nabla u}^{p-2}A_u+\notag \\
+\frac{\lambda^2 u^{2p-2}}{n} + \frac n {n-1}\ton{\frac{\lambda u^{p-1}}{n} - (p-1)\abs{\nabla u}^{p-2} A_u}^2 + (n-1)k\abs{\nabla u}^{2p-2}+\notag \\
+\frac{\lambda u^{p-1}}{p-1} \Delta_p (\bar c w)|_{(\bar cw)^{-1} (u)} - \abs{\nabla u}^p \frac{1}{\bar c\dot w} \frac {d(\Delta_p w)} {dt}|_{(\bar cw)^{-1} (u)} \label{eq_last}\, .
\end{gather}
Since at $\bar x$, $\abs{\nabla u }^p = (c\dot w)^p |_{(\bar cw)^{-1} (u)}$, we obtain that, at $\bar t=(\bar c w )^{-1} (u(\bar x))$,
\begin{gather*}
0\geq-\lambda(p-1) (\bar c w)^{p-2} (\bar c\dot w )^p+ \frac{p-2}{p-1}\lambda (\bar c w)^{p-1}\Delta_p(\bar c w)+\\
+\frac{\lambda^2 (\bar c w)^{2p-2}}{n} + \frac n {n-1}\ton{\frac{\lambda (\bar c w)^{p-1}}{n} - \Delta_p (\bar c w)}^2 + (n-1)k(\bar c \dot w)^{2p-2}+\\
+\frac{\lambda (\bar c w)^{p-1}}{p-1} \Delta_p (\bar cw)- (\bar c \dot w )^p\frac{1}{\bar c\dot w} \frac {d(\Delta_p w)} {dt}\, .
\end{gather*}
By direct calculation, the ODE \ref{eq_1dm_neg} satisfied by $\bar cw$ shows that this inequality is equivalent to
\begin{gather*}
(n-1)(k-z)(\bar c \dot w)^{2p-2}|_{\bar t}=(n-1)\epsilon (\bar c \dot w)^{2p-2}|_{\bar t}\leq 0\, ,
\end{gather*}
which is a contradiction. 
\end{proof}

\begin{remark}\label{rem_bou}
\rm{In the case where the boundary is not empty, the only difference in the proof of the gradient comparison is that the point $\bar x$ may lie in the boundary of $M$, and so it is not immediate to obtain equation \eqref{eq_max1}. However, once this equation is proved, it is evident that $P^{II}_u F|_{\bar x}\leq 0$ and the rest of proof proceeds as before. In order to prove that $\bar x$ is actually a stationary point for $F$, the (nonstrict) convexity of the boundary is crucial, as it is easily seen in the following Lemma \footnote{A similar technique has been used in the proof of \cite[Theorem 8]{new}}.}
\end{remark}
\begin{lemma}
If $\partial M $ is non empty and convex, the equation
\begin{gather}
 \nabla F_{k-\epsilon, \bar c}|_{\bar x}=0 
\end{gather}
remains valid even if $\bar x \in \partial M$ .
\end{lemma}
\begin{proof}
Let $\hat n$ be the outward normal derivative of $\partial M$.

Since $\bar x$ is a point of maximum for $F_{k-\epsilon, \bar c}$, we know that all the derivatives of $F$ along the boundary vanish, and that the normal derivative of $F$ is nonnegative
\begin{gather*}
 \ps{\nabla F}{\hat n}\geq 0 \ .
\end{gather*}
Neumann boundary conditions on $\Delta_p$ ensure that $\ps{\nabla u}{\hat n}=0$. Define for simplicity $\phi(x)=(\bar c \dot w )^p|_{(\bar c w)^{-1}(x)}$. By direct calculation we have
\begin{gather*}
 \ps{\nabla F}{\hat n} = -\dot \phi|_{u(\bar x)} \ps{\nabla u}{\hat n} + p\abs{\nabla u}^{p-2} H_u(\nabla u,\hat n)=  p\abs{\nabla u}^{p-2} H_u(\nabla u,\hat n) \ .
\end{gather*}
Using the definition of second fundamental form $II(\cdot,\cdot)$ and the convexity of $\partial M$, we can conclude that
\begin{gather*}
 0\leq \ps{\nabla F}{\hat n} = p\abs{\nabla u}^{p-2} H_u(\nabla u,\hat n) = - p\abs{\nabla u}^{p-2} II(\nabla u,\nabla u)\leq 0 \, .
\end{gather*}
\end{proof}

\begin{rem}\label{rem_m_neg}
\rm Since Corollary \ref{cor_est_pu} is valid for all real $n'\geq n$, the gradient comparison remains valid if we use model equations with ``dimension'' $n'$, i.e., if we assume $\dot T = \frac{T^2}{n'-1} +(n'-1)k$.
\end{rem}

\begin{remark}\label{rem_reg_neg}
\rm As mentioned before, in case $1<p<2$ and $u(x)=0$, we have a regularity issue to address in the proof of the gradient comparison theorem. Indeed, in this case $F$ is only a $C^{1,\alpha}$ function and Equation \eqref{eq_max2} is not well-defined since there are two diverging terms in this equation. As it can be seen from \eqref{eq_last}, these terms are
\begin{gather*}
-\lambda(p-1) \abs u ^{p-2} \abs{\nabla u}^p\quad \text{ and } \quad -\abs{\nabla u}^p \frac{1}{\bar c\dot w} \left.\frac {d(-\lambda (\bar c w)^{p-1})} {dt}\right\vert_{(\bar cw)^{-1} (u)} \, .
\end{gather*}
 
 However, since $\nabla u(\bar x)\neq 0$, there exists an open neighborhood $U$ of $\bar x$ such that $U\setminus\{u=0\}$ is open and dense in $U$. On this set, it is easy to see that these two terms exactly cancel each other, and all the other terms in $P^{II}_u (F)$ are well-defined and continuous on $U$. Thus Equation \eqref{eq_max2} is valid even in this low-regularity context.
\end{remark}

It is not difficult to adapt the proof of the previous Theorem in order to compare different functions $w_{k,n,i,a}$. In particular, we can state the following.
\begin{teo}\label{grad_est_negw}
For $j=1,2$ let $w_j=w_{k,n,i_j,a_j}$ be solutions to the one dimensional IVP \eqref{eq_1dm_neg} and let $b_j<\infty$ be the first point $b_j>a_j$ such that $\dot w_j(b_j)=0$. If
\begin{gather}
 w_1[a_1,b_1]\subset w_2[a_2,b_2]\, ,
\end{gather}
then we have the following comparison for the derivatives
\begin{gather}
 \abs{\dot w_1}|_t \leq \dot w_2|_{w_2^{-1}(w_1(t))}\, ,
\end{gather}
or equivalently
\begin{gather}
 \abs{\dot w_1}|_{w_1^{-1}(s)}\leq \abs{\dot w_2}|_{w_2^{-1}(s)}\, .
\end{gather}
\end{teo}
\begin{proof}
 This Theorem can be proved directly using a method similar to the one described in the proof of Theorem \ref{grad_est_neg}. Another method is to define on $M=[a_1,b_1]\times_{\tau_i} S^{n-1}$ the function $u(t,x)=w_1(t)$, and use directly Theorem \ref{grad_est_neg} to get the conclusion. Note that $M$ might have nonconvex boundary in this case, but since $u(t,x)$ depends only on $t$, it is easy to find a replacement for Remark \ref{rem_bou}.
\end{proof}

\subsection{Fine properties of the one dimensional model}\label{sec_1drev}
In this subsection we study some fine properties of our one dimensional model. In particular, we study the oscillatory behaviour of the functions $w$ depending on $\lambda, \, i$ and $a$. 
\index{oscillatory behaviour}
Throughout this section, $n$ and $k$ are fixed, and as usual we set $\alpha= \ton{\frac \lambda {p-1}}^{\frac 1 p}$.

First of all, it is easy to see that in the model $i=3$ there always exists an odd solution $w_{3,-\bar a}$ which has maximum and minimum equal to $1$ in absolute value.
\begin{prop}\label{prop_barasym}
 Fix $\alpha>0$, $n\geq 1$ and $k<0$. Then there always exists a unique $\bar a>0$ such that the solution $w_{3,-\bar a}$ to the IVP \eqref{eq_1dm_neg} (with $T=T_3$) is odd. In particular, $w_{3,-\bar a}$ restricted to $[-\bar a,\bar a]$ has nonnegative derivative and has maximum equal to $1$.
\end{prop}
\begin{proof}
 We use the \pf transformation to prove this Theorem. For the sake of simplicity, here we write $\phi$ for $\phi_{i,a}$. Consider the IVP
 \begin{gather}
  \begin{cases}
     \dot \phi = \alpha - \frac{T_3}{p-1}\cosp^{p-1}(\phi)\sinp(\phi)\\
     \phi(0)=0
  \end{cases}
 \end{gather}
Recall that $T_3(0)=0$, $T_3$ is odd and it is negative on $(0,\infty)$. By uniqueness of the solution, $\phi$ is also odd. Moreover, it is easily seen that, as long as $\phi\in [-\pi_p/2,\pi_p/2]$, $\dot \phi \geq \alpha$.

This implies that there exists a $-\bar a \in [-\pi_p/(2\alpha),0]$ such that $\phi(-\bar a)=-\pi_p/2$. It is also easy to see that the corresponding solution $e(t)$ to equation \eqref{eq_pf_nege} is even, no matter what $e(0)$ is.

Thus we have proved all the properties we were seeking for $w_{3,-\bar a}$.
\end{proof}

This Proposition proves that we can always use the gradient comparison Theorem with $w_{3,-\bar a}$ as a model function. However, as we will see in the following section, to get a sharp estimate on the eigenvalue we will need a model function $w$ such that $\min\{w\}=\min\{u\}=-1$ \textit{and} $\max\{w\}=\max\{u\}={u^\star}$.

In order to prove that such a model function always exists, we study in more detail the one dimensional model. We start by recalling some definitions given in the previous section.
\begin{definition}
 Given the model function $w_{i,a}$, we define $b(i,a)$ to be the first value $b>a$ such that $\dot w_{i,a}(b)=0$, and set $b(a)=\infty$ if such value doesn't exist. Equivalently, $b(i,a)$ is the first value $b>a$ such that $\phi_{i,a}(b)=\frac {\pi_p}{2}$. 
 
We also define the diameter of the model function as 
\begin{gather}
\delta(i,a)=b(i,a)-a 
\end{gather}
and the maximum of the model function
\begin{gather}
 m(i,a)= w_{i,a}(b(i,a))= \alpha^{-1} e_{i,a} (b(i,a))\, .
\end{gather}
\end{definition}
\begin{remark}\rm
 It is evident that, when $b(i,a)<\infty$, the range of $w$ on $[a,b]$ is $[-1,m]$. More precisely
\begin{gather}
 w_{i,a}[a,b(i,a)]= [-1,m(i,a)]
\end{gather}
If $b(i,a)=\infty$, then $w_{i,a}[a,b(i,a))=[-1,m(i,a))$. In this case, we will see that $m(i,a)=0$.\\
An immediate consequence of Proposition \ref{prop_barasym} is that there always exists some $\bar a>0$ such that $b(3,-\bar a)<\infty$.
\end{remark}

In the following, we study the oscillatory behaviour of the differential equation, i.e., we try to get conditions under which the solution $\phi_{i,a}$ always tends to infinity for $t\to \infty$. We will find a limiting value $\bar \alpha=\bar \alpha(k,n)$ (defined in equation \eqref{eq_baralpha}) such that for $\alpha>\bar \alpha$ all the functions $w_{i,a}$ have an oscillatory behaviour, while for $\alpha\leq \bar \alpha$, $\phi_{i,a}$ has always a finite limit at infinity. We begin by studying the translation invariant model $T=T_2$.

\begin{proposition}
 Consider the model $T_2$. For $\alpha>\bar \alpha$ the solution to
\begin{equation}
 \begin{cases}
  \dot \phi = \alpha + \frac{(n-1)\sqrt{-k}}{p-1} \cosp^{(p-1)}(\phi)\sinp(\phi)\\
  \phi(0)=-\frac{\pi_p}{2}
 \end{cases}
\end{equation}
satisfies
\begin{gather}
 \lim_{t\to \infty} \phi(t)=\infty\, ,
\end{gather}
and in particular $\delta(2,0)<\infty$. While for $\alpha\leq \bar \alpha$
\begin{gather}
-\frac{\pi_p} 2 < \lim_{t\to \infty} \phi(t)<0\, ,
\end{gather}
and $\delta(2,0)=\infty$.
\end{proposition}
\begin{proof}
 This problem is a sort of damped p-harmonic oscillator. The value $\bar \alpha$ is the critical value for the damping effect. The proof can be carried out in detail following the techniques used in the next Proposition.
\end{proof}

Even the behaviour of the solutions to the models $T_1$ and $T_3$ change in a similar way according to whether $\alpha>\bar \alpha$ or not. We first describe what happens to the symmetric model $T_3$.
\begin{proposition}\label{prop_alpha}
There exists a limiting value $\bar \alpha>0$ such that for $\alpha>\bar \alpha$ the solution $w_{3,a}$ has an oscillatory behaviour and $\delta(3,a)<\infty$ for every $a\in\R$; while, for $\alpha<\bar \alpha$, we have
\begin{gather}
 \lim_{t\to \infty} \phi_{3,a}(t)<\infty
\end{gather}
for every $a\in \R$. Equivalently for $a$ large enough
\begin{gather}
 -\frac{\pi_p} 2 <\lim_{t\to \infty} \phi_{3,a}(t)<0
\end{gather}
and $\delta(3,a)=\infty$. For $\alpha = \bar \alpha$, we have
\begin{gather}
 \lim_{a\to \infty} \delta(3,a) = \infty\, .
\end{gather}

\end{proposition}
\begin{proof}
We study the IVP
\begin{gather}
 \begin{cases}
  \dot \phi = \alpha - \frac{T_3}{p-1} \cosp^{p-1}(\phi)\sinp(\phi)\\
  \phi(a)=-\frac{\pi_p}{2}
 \end{cases}
\end{gather}
We will only prove the claims concerning $\delta(3,a)$, and restrict ourselves to the case $a\geq -\bar a$. The other claims can be proved using a similar argument. 

Note that if there exists some $\bar t$ such that $\phi(\bar t)=0$, then necessarily $\bar t\geq 0$ and, for $s\in[\bar t, \phi^{-1}(\pi_p/2)]$, $\dot \phi(s)\geq \alpha$. So in this case $b-a<\bar t+\pi_p/(2\alpha)<\infty$.

Thus $b(3,a)$ (or equivalently $\delta(3,a)$) can be infinite only if $\phi(t)<0$ for all $t$. Note that either $\dot \phi(t) >0$ for all $t\geq a$, or $\dot \phi(t)<0$ for $t$ large enough. In fact, at the points where $\dot \phi=0$ we have
\begin{gather}
 \ddot \phi = -\frac{\dot T_3}{T_3} \alpha\, .
\end{gather}
Since $\dot T_3<0$, and $T_3(t)<0$ for all $t>0$, once $\dot \phi$ attains a negative value it can never become positive again. So $\phi$ has always a limit at infinity, finite or otherwise.

By simple considerations on the ODE, if $b(3,a)=\infty$, the limit $\psi = \lim_{t\to \infty} \phi(t)$ has to satisfy
\begin{gather}\label{eq_psi}
 F(\psi)\equiv\alpha- \frac{-(n-1)\sqrt{-k} }{p-1} \cosp^{(p-1)}(\psi)\sinp(\psi)=0\, .
\end{gather}
Notice that the function $\cosp^{(p-1)}(\psi)\sinp(\psi)$ is negative and has a single minimum $-l$ on the interval $[-\pi_p/2,0]$. Set
\begin{gather}\label{eq_baralpha}
 \bar \alpha = \frac{(n-1)l\sqrt{-k}}{(p-1)}
\end{gather}
so that for $\alpha >\bar \alpha$, $F(\psi)$ has a positive minimum, for $\alpha=\bar \alpha$, its minimum is zero, and for $0<\alpha <\bar \alpha$, its minimum is negative.

\paragraph{\textbf{Case 1: $\alpha=\bar \alpha$}}
Before turning to the model $T_3$, in this case we briefly discuss what happens in the model $T_2$ \footnote{recall that this model is translation invariant, so we only need to study the solution $\phi_{2,0}=\phi_2$}, in particular we study the function
\begin{gather}
 \begin{cases}
  \dot \phi_{2}= \alpha + \frac{(n-1)\sqrt{-k}}{p-1}\cosp^{(p-1)}(\phi_2)\sinp(\phi_2)=F(\phi_2)\\
  \phi_2(0)=-\frac{\pi_p}2
 \end{cases}
\end{gather}
Since $\alpha=\bar \alpha$, it is easy to see that $\dot \phi_2\geq0$ everywhere. Let $\psi\in\ton{-\pi_p/2,\pi_p/2}$ be the only solution of $F(\psi)=0$. Since $\psi(t)\equiv \psi$ satisfies the differential equation $\dot \psi = F(\psi)$, by uniqueness $\phi_2\leq \psi$ everywhere, and thus the function $\phi_2$ is strictly increasing and has a finite limit at infinity
\begin{gather}
 \lim_{t\to \infty} \phi_2(t)=\psi\, .
\end{gather}
%
Keeping this information in mind, we return to the model $T_3$. In some sense, the bigger is $a$, the closer the function $T_3(a+t)$ is to the constant function $T_2$. Consider in fact the solution $\phi_{3,a}(t)$, and for convenience translate the independent variable by $t\to t-a$. The function $\tau \phi_{3,a} (t) = \phi_{3-a} (t+a)$ solves
\begin{gather}
  \begin{cases}
  \tau \dot \phi_{3,a}= \alpha- \frac{T_3(a+t)}{p-1}\cosp^{(p-1)}(\phi)\sinp(\phi)\\
  \tau \phi_{3,a}(0)=-\frac{-\pi_p}2
 \end{cases}
\end{gather}
Since $T_3(a+t)$ converges in $C^1([0,\infty))$ to $T_2=-(n-1)\sqrt{-k}$, we have that
\begin{gather}
 \lim_{a\to \infty} \tau \phi_{3,a} = \phi_2
\end{gather}
in the sense of local $C^1$ convergence on $[0,\infty)$. This implies immediately that
\begin{gather}
\lim_{a\to \infty} \delta(3,a)= \lim_{a\to \infty} b(3,a)-a =b(2,0)=\infty\, .
\end{gather}

\paragraph{\textbf{Case 2: }}
if $0<\alpha <\bar \alpha$, there are two solutions $-\pi_p/2<\psi_1< \psi_2 <0$ to equation \eqref{eq_psi}. Take $\epsilon>0$ small enough such that $\psi_2-\epsilon>\psi_1$. Thus there exists a $\epsilon'>0$ such that
\begin{gather}
 \frac{d}{dt}(\psi_2-\epsilon)=0>\alpha+ \frac{(n-1)\sqrt{-k}-\epsilon'}{(p-1)}\cosp^{(p-1)}(\psi_2-\epsilon)\sinp(\psi_2-\epsilon)>\\
\notag >\alpha+ \frac{(n-1)\sqrt{-k}}{(p-1)}\cosp^{(p-1)}(\psi_2-\epsilon)\sinp(\psi_2-\epsilon)\, .
\end{gather}
Since $\lim_{t\to \infty} T_3(t)=-(n-1)\sqrt{-k}$, there exists an $A>>1$ such that $T_3(t)\leq -(n-1)\sqrt{-k}+\epsilon'$ for $t\geq A$. 

Choose $a>A$, and observe that, as long as $\phi_{3,a}(t)<0$,
\begin{gather}
 \begin{cases}
  \dot \phi_{3,a}= \alpha- \frac{T}{p-1}\cosp^{(p-1)}(\phi)\sinp(\phi)\leq \alpha + \frac{-\epsilon'+(n-1)\sqrt{-k} }{p-1} \cosp^{(p-1)}(\phi)\sinp(\phi)\\
  \phi_{3,a}(a)=-\frac{-\pi_p}2
 \end{cases}
\end{gather}
Then, by a standard comparison theorem for ODE, $\phi_{3,a}\leq \psi_2-\epsilon$ always.

In particular, $\lim_{t\to \infty}\phi_{3,a}(t)=\psi_1$, and using equation \eqref{eq_pf_nege} we also have
\begin{gather}\label{eq_0lim}
 \lim_{t\to \infty} e_{3,a}= \lim_{t\to \infty} w_{3,a}=0\, .
\end{gather}
It is also clear that, if $a>A$, $\delta(3,a)=\infty$.

\paragraph{\textbf{Case 3}}

If $\alpha > \bar \alpha$, then $F$ has a minimum $\geq \epsilon>0$, and so $\dot \phi_{i,a}\geq \epsilon$ for $i=2,3$ and all $a\in \R$. Thus with a simple estimate we obtain for $i=2,3$
\begin{gather}
 b(i,a)-a\leq \frac{\pi_p}{\epsilon} \, .
\end{gather}
Moreover, as in case $1$, it is easy to see that $\phi_{3,a}(t-a)$ converges locally uniformly in $C^1$ to $\phi_{2,0}$, and so, in particular,
\begin{gather}
 \lim_{a\to \infty} \delta(3,a) = b(2,0)\, .
\end{gather}
\end{proof}

As for the model $T_1$, a similar argument leads to the following proposition.
\begin{proposition}
 Consider the model $T_1$. If $\alpha>\bar \alpha$, then the solutions have an oscillatory behavior with $\phi_{1,a}(\infty)=\infty$ and $\delta(1,a)<\infty$ for all $a\in [0,\infty)$. If $\alpha\leq \bar \alpha$, then $\phi_{1,a}$ has a finite limit at infinity and $\delta(1,a)=\infty$ for all $a\in [0,\infty)$. 
\end{proposition}

Now we turn our attention to the maximum $m(i,a)$ of the model functions $w_{i,a}$. Our objective is to show that for every possible $0<{u^\star}\leq 1$, there exists a model such that $m(i,a)={u^\star}$. This is immediately seen to be true if $\alpha\leq \bar \alpha$. Indeed, in this case we have
\begin{proposition}\label{prop_kasmall}
 Let $\alpha\leq \bar \alpha$. Then for each $0<{u^\star}\leq 1$, there exists an $a\in[-\bar a, \infty)$ such that $m(3,a)={u^\star}$.
\end{proposition}
\begin{proof}
 Proposition \ref{prop_barasym} shows that this is true for ${u^\star}=1$. For the other values, we know that if $\alpha\leq \bar \alpha$
\begin{gather}
 \lim_{a\to \infty}\delta(3,a)=\infty\, .
\end{gather}
By equation \eqref{eq_0lim} (or a similar argument for $\alpha=\bar \alpha$) and using continuous dependence on the parameters of the solution of our ODE, it is easy to see that
\begin{gather}
 \lim_{a\to a^\star} m(3,a)= 0 \, ,
\end{gather}
where $a^\star $ is the first value for which $\delta(3,a^\star)=\infty$ (which may be infinite if $\alpha=\bar \alpha$).

Since $m(3,a)$ is a continuous function and $m(3,-\bar a)=1$, we have proved the proposition.
\end{proof}

The case $\alpha>\bar \alpha$ requires more care. First of all, we prove that the function $m(i,a)$ is invertible.
\begin{proposition}
If $m(i,a)=m(i,s)>0$, then $w_{i,a}$ is a translation of $w_{i,s}$. In particular, if $i\neq 2$, $a=s$.
\end{proposition}
\begin{proof}
 Note that if $i=2$, the model is translation invariant and the proposition is trivially true. In the other cases, the proof follows from an application of Theorem \ref{grad_est_negw}. Since $m(i,a)=m(i,s)>0$, we know that $b(i,a)$ and $b(i,s)$ are both finite. So by Theorem \ref{grad_est_negw}
\begin{gather}
 \dot w_{i,a}|_{w_{i,a}^{-1}} = \dot w_{i,s}|_{w_{i,s}^{-1}}\, .
\end{gather}
By the uniqueness of the solutions of the IVP \eqref{eq_1dm_neg}, we have that $w_{i,a}(t)=w_{i,s}(t+t_0)$, which, if $i\neq 2$, is possible only if $a=s$.
\end{proof}

If $\alpha>\bar \alpha$, then $m(2,a)$ is well-defined, positive, strictly smaller than $1$ and independent of $a$. We define $m_2=m(2,a)$.
\begin{proposition}\label{p_1}
 If $\alpha>\bar \alpha$, then $m(3,a)$ is a decreasing function of $a$, and
\begin{gather}
 \lim_{a\to \infty} m(3,a)= m_2\, .
\end{gather}
\end{proposition}
\begin{proof}
 This proposition is an easy consequence of the convergence property described in the proof of Proposition \ref{prop_alpha}. Since $m(3,a)$ is continuous, well defined on the whole real line and invertible, it has to be decreasing.
\end{proof}

We have just proved that, for $a\to \infty$, $m(3,a)$ decreases to $m_2$. With a similar technique, we can show that, for $a\to \infty$, $m_{1,a}$ increases to $m_2$.
\begin{proposition}\label{p_2}
 If $\alpha>\bar \alpha$, then for all $a\geq 0$, $b(1,a)<\infty$. Moreover, $m(1,a)$ is an increasing function on $[0,\infty)$ such that
\begin{gather}
 m(1,0)=m_0>0 \quad \text{and} \quad \lim_{a\to \infty} m(1,a)=m_2\, .
\end{gather}
\end{proposition}
%
%

Using the continuity of $m(i,a)$ with respect to $a$, we deduce, as a corollary, the following proposition.
\begin{proposition}\label{p3}
 For $\alpha>\bar \alpha$ and for any $u^\star \in [m(1,0),1]$, there exists some $a\in \R$ and $i\in \{1,2,3\}$ such that $m(i,a)=u^\star$.
\end{proposition}

In the next section we address the following question: is it possible that $u^\star<m(1,0)$? Using a volume comparison theorem, we will see that the answer is no. Thus there always exists a model function $w_{i,a}$ that fits perfectly the eigenfunction $u$.

\subsection{Maxima of eigenfunctions and volume comparison}
Using exactly the same steps as in section \ref{sec_max}, it is easy to prove a volume comparison theorem just like \ref{teo_comp}. In a similar fashion, it is possible to generalize also Lemma \ref{lemma_radepsilon} and Theorem \ref{teo_ccc}. Thus we obtain the following lower bound for maxima if $\alpha>\bar \alpha$.

\begin{proposition}
 Let $u:M\to \R$ be an eigenfunction of the $p$-Laplacian, and suppose that $\alpha>\bar \alpha$. Then $\max\{u\}=u^\star\geq m(1,0)>0$.
\end{proposition}

As a corollary of this Proposition and Proposition \ref{p3}, we get the following
\begin{corollary}\label{cor_maxfit}
 Let $u:M\to \R$ be an eigenfunction of the $p$-Laplacian. Then there always exists $i\in \{1,2,3\}$ and $a\in I_i$ such that $u^\star=\max\{u\}=m(i,a)$. This means that there always exists a solution of equation \eqref{eq_1dm_neg} relative to the models $1,2$ or $3$ such that
\begin{gather}
 u(M)= [-1,w(b)]\, .
\end{gather}
\end{corollary}

\subsection{Diameter comparison}\label{sec_diam_neg}
\index{diameter comparison}
In this subsection we study the diameter $\delta(i,a)=b(i,a)-a$ as a function of $i$ and $a$, for fixed $n, \, k$ and $\lambda$. In particular, we are interested in characterizing the minimum possible value for the diameter.

\begin{definition}\label{deph_delta}
 For fixed $n, \, k$ and $\lambda$, define $\bar \delta$ by
\begin{gather}
 \bar \delta =\min \{\delta(i,a), \ \ i=1,2,3,\ a\in I_i\}
\end{gather}

\end{definition}

Using a comparison technique similar to the one in Theorem \ref{teo_delta_comp}, we obtain a simple lower bound on $\delta(i,a)$ for $i=1,2$.
\begin{proposition}
 For $i=1,2$ and for any $a\in I_i$, $\delta(i,a)>\frac{\pi_p}{\alpha}$.
\end{proposition}
\begin{proof}
We can rephrase the estimate in the following way: consider the solution $\phi(i,a)$ of the initial value problem
\begin{gather*}
\begin{cases}
  \dot \phi = \alpha - \frac{T_i}{(p-1)}\cosp^{p-1} (\phi)\sinp(\phi)\\
\phi(a)=-\frac{\pi_p}{2} 
\end{cases}
\end{gather*}
Then $b(i,a)$ is the first value $b>a$ such that $\phi(i,b)=\frac{\pi_p}2$, and $\delta(i,a)=b(i,a)-a$.

We start by studying the translation invariant model $T_2$. In this case, using separation of variables, we can find the solution $\phi(2,0)$ in an implicit form. Indeed, if $\alpha\leq \bar \alpha$, then we have already shown in Proposition \ref{prop_alpha} that $\delta(i,a)=\infty$. If $\alpha>\bar \alpha$, we have $\dot \phi >0$ and
\begin{gather*}
 \delta(2,0)=b(2,0)-0=\int_{-\frac {\pi_p} 2}^{\frac {\pi_p} 2} \frac{d\psi}{\alpha + \gamma\cosp^{p-1} (\psi)\sinp(\psi)} \, ,
\end{gather*}
where $\gamma=-\frac{T_2}{p-1}$ is a nonzero constant. Since $\cosp^{p-1}(\psi)\sinp(\psi)$ is an odd function, by Jensen's inequality we can estimate
\begin{gather}\label{eq_jan_neg}
\frac{\delta(2,0)}{\pi_p}= \frac 1 {\pi_p} \int_{-\pi_p/2}^{\pi_p/2} \frac{d\psi}{\alpha+\gamma \cosp^{p-1} (\psi)\sinp(\psi)}>\\
>\qua{\frac 1 {\pi_p}\int_{-\pi_p/2}^{\pi_p/2} \ton{\alpha+\gamma \cosp^{p-1} (\psi)\sinp(\psi)}d\psi}^{-1} = \frac{1}{\alpha} \ . 
\end{gather}
Note that, since $T_2\neq 0$, this inequality is strict.

If $i=1$ and $\alpha\leq \bar \alpha$, we still have $\delta(1,a)=\infty$ for every $a\geq 0$. On the other hand, if $\alpha >\bar \alpha$ we can use the fact that $\dot T_1>0$ to compare the solution $\phi(1,a)$ with a function easier to study. Let $t_0$ be the only value of time for which $\phi(1,a)(t_0)=0$. Then it is easily seen that
\begin{gather}
 \dot \phi(1,a)= \alpha - \frac{T_1}{p-1} \cosp^{(p-1)}(\phi)\sinp(\phi) \leq \alpha-\frac{T_1(t_0)}{p-1} \cosp^{(p-1)}(\phi)\sinp(\phi)\, .
\end{gather}
Define $\gamma=\frac{T_1(t_0)}{p-1}$. Using a standard comparison theorem for ODE, we know that, for $t>a$,
\begin{gather}
 \phi(1,a)(t)< \psi(t)\, ,
\end{gather}
where $\psi$ is the solution to the IVP
\begin{gather}
 \begin{cases}
  \dot \psi = \alpha-\gamma \cosp^{(p-1)}(\psi)\sinp(\psi)\\
  \psi(a)=-\frac{\pi_p}2
 \end{cases}
\end{gather}
If we define $c(a)$ to be the first value of time $c>a$ such that $\psi(c)=\pi_p/2$, we have $b(2,a)\geq c(a)$. Using separation of variables and Jensen's inequality as above, it is easy to conclude that
\begin{gather}
 \delta(1,a)>c(a)-a>\frac{\pi_p}{\alpha}\, .
\end{gather}
\end{proof}

\begin{remark}
 \rm For the odd solution $\phi_{3,-\bar a}$, it is easy to see that $\dot \phi \geq \alpha$ on $[-\bar a,\bar a]$ with strict inequality on $(-\bar a, 0)\cup(0,\bar a)$. For this reason,
 \begin{gather*}
  \delta(3,-\bar a)<\frac{\pi_p}{\alpha}\, ,
 \end{gather*}
and so $\bar \delta$ is attained for $i=3$.
\end{remark}

In the following proposition we prove that $\bar \delta = \delta(3,-\bar a)$, and for all $a\neq \bar a$ the strict inequality $\delta(3,a)>\bar \delta$ holds.

\begin{proposition}
 For all $a\in I_3=\R$:
 \begin{gather*}
  \delta(3,a)\geq \delta (3,-\bar a) = 2\bar a = \bar \delta\, ,
 \end{gather*}
with strict inequality if $a\neq -\bar a$.
\end{proposition}
\begin{proof}
 The proof is based on the symmetries and the convexity properties of the function $T_3$. Fix any $a<\bar a$ (with a similar argument it is possible to deal with the case $a>\bar a$), and set
 \begin{gather*}
 \psi_+(t) = \phi_{3,-a}(t) \, ,\quad \quad  \varphi (t)= \phi_{3,-\bar a}(t) \, ,\quad \quad \psi_-(t) = -\psi_+(-t)\, .
 \end{gather*}
We study these functions only when their range is in $[-\pi_p/2,\pi_p/2]$, and since we can assume that $b(3,-a)<\infty$, we know that $\dot \psi_{\pm} >0$ on this set (see the proof of Proposition \ref{prop_alpha}). Using the symmetries of the IVP \eqref{eq_pf_neg}, it is easily seen that the function $\varphi$ is an odd function, that $b(3,-a)>\bar a$ and that $\psi_-$ is still a solution to \eqref{eq_pf_neg}. In particular
\begin{gather*}
 \psi_- (t) = \phi_{3,-b(3,-a)}(t)\, .
\end{gather*}
Note that by comparison, we always have $\psi_-(t) > \varphi(t) > \psi_+(t)$.

Since all functions have positive derivative, we can study their inverses
\begin{gather*}
 h= \psi_-^{-1}\, , \quad \quad s = \varphi^{-1}\, , \quad \quad g = \psi_+ ^{-1}\, .
\end{gather*}

Set for simplicity
\begin{gather*}
 f(\phi)\equiv \frac 1 {p-1} \cosp^{(p-1)}(\phi)\sinp(\phi)
\end{gather*}
and note that on $[-\pi_p/2,\pi_p/2]$, $f(\phi)$ is odd and has the same sign as $\phi$. 
The function defined by
\begin{gather*}
 m(\phi) = \frac{1}{2} \ton{h(\phi) + g(\phi)}
\end{gather*}
is an odd function such that $m(0)=0$ and 
\begin{gather*}
 m(\pi_p/2) = \frac{1}{2} \ton{h(\pi_p/2)+g(\pi_p/2)} = \frac {1}{2} \ton{b(3,-a) -(-a)  } = \frac {1}{2} \delta (3,-a)\, ,
\end{gather*}
thus the claim of the proposition is equivalent to $m(\pi_p/2) >\bar a$.

By symmetry, we restrict our study to the set $\phi \geq 0$, or equivalently $m \geq 0$. Note that $m$ satisfies the following ODE
\begin{gather*}
 2 \frac{dm}{d\phi} =  \frac{1}{\alpha -  T_3(g) f(\phi)} +  \frac{1}{\alpha -  T_3(h) f(\phi)} \, .
\end{gather*}
Fix some $\alpha,\beta\in \R^+ $ and consider the function
\begin{gather}
 z(t) = \frac{1}{\alpha - \beta T_3(t)}\, .
\end{gather}
Its second derivative is
\begin{gather*}
 \ddot z = \frac{2\beta^2 \dot T_3^2}{(\alpha- \beta T_3)^3 } + \frac{\beta \ddot T_3}{(\alpha- \beta T_3)^2}\, .
\end{gather*}
So, if $\beta\geq 0$ and $t\geq 0$, $z$ is a convex function. In particular this implies that
\begin{gather}\label{eq_conv}
 \frac{dm}{d\phi} \geq \frac{1}{\alpha - T_3(m) f(\phi)} 
\end{gather}
for all those values of $\phi$ such that both $g$ and $h$ are nonnegative. However, by symmetry, it is easily seen that this inequality also holds when one of the two functions attains a negative value. Indeed, if $h<0$, we have that
\begin{gather*}
 \frac 1 2 \qua{\frac{1}{\alpha-\beta T_3(h)} + \frac{1}{\alpha-\beta T_3(-h)}} \geq \frac 1 \alpha  = \frac{1}{\alpha-\beta T_3\ton{\frac{h-h}{2}}}\, .
\end{gather*}
Since $z(0)\leq [z(h)+z(-h)]/2$ and $z$ is convex on $[0,g]$, then $z((h+g)/2) \leq [z(h)+z(g)]/2$, so inequality \eqref{eq_conv} follows. Moreover, note that if $\beta > 0$ (i.e. if $\phi \in (0,\pi_p)$) and if $g\neq h$, the inequality is strict.

Using a standard comparison for ODE, we conclude that $m(\phi)\geq s(\phi)$ on $[0,\pi_p/2]$ and in particular
\begin{gather}
m\ton{\pi_p/2} > s\ton{\pi_p/2} =\bar a\, ,
\end{gather}
and the claim follows immediately.
\end{proof}


In the last part of this subsection, we study $\bar \delta$ as a function of $\lambda$, having fixed $n$ and $k$. By virtue of the previous proposition, it is easily seen that $\bar \delta(\lambda)$ is a strictly decreasing function, and therefore invertible. In particular, we can define its inverse $\bar \lambda (\delta)$, and characterize it in the following equivalent ways.
\begin{proposition}\label{prop_barlambda}
 For fixed $n$, $k<0$ and $p>1$, and for $\delta>0$, $\bar \lambda(n,k,\delta)$ is the first positive Neumann eigenvalue on $[-\delta/2,\delta/2]$ relative to the operator
 \begin{gather*}
  \frac{d}{dt}\ton{\dot w^{(p-1)}} + (n-1)\sqrt{-k}\tanh\ton{\sqrt{-k} t}\dot w ^{(p-1)} +\lambda w^{(p-1)}\, ,
 \end{gather*}
or equivalently it is the unique value of $\lambda$ such that the solution to
\begin{gather*}
 \begin{cases}
  \dot \phi = \ton{\frac{\lambda}{p-1}}^{\frac{1}{p}} + \frac{(n-1)\sqrt{-k}}{p-1}\tanh\ton{\sqrt{-k} t}  \cosp^{(p-1)}(\phi) \sinp(\phi)\\
  \phi(0)=0
 \end{cases}
\end{gather*}
satisfies $\phi(\delta/2) = \pi_p/2$.
\end{proposition}
\begin{remark}
 \rm It is easily seen that the function $\bar \lambda (n,k,\delta)$ is a continuous function of the parameters. Moreover it has the following monotonicity properties
 \begin{gather*}
   \delta_1 \leq \delta_2 \ \ \ \text{and} \ \ \ n_1\geq n_2 \ \ \ \text{and} \ \ \ k_1\geq k_2 \Longrightarrow \bar \lambda (n_1,k_1,\delta_1)\geq \bar \lambda (n_2,k_2,\delta_2)\, .
 \end{gather*}

\end{remark}

%

\subsection{Sharp estimate}
Now we are ready to state and prove the main Theorem on the spectral gap.
\begin{teo}\label{teo_main_proof}
 Let $M$ be a compact $n$-dimensional Riemannian manifold with Ricci curvature bounded from below by $k<0$, diameter $d$ and possibly with convex $C^2$ boundary. Let $\lambda_{1,p}$ be the first positive eigenvalue of the $p$-Laplacian (with Neumann boundary condition if necessary). Then 
\begin{gather*}
 \lambda_{1,p}\geq \bar \lambda(n,k,d) \ ,
\end{gather*}
where $\bar \lambda$ is the function defined in Proposition \ref{prop_barlambda}.
\end{teo}
\begin{proof}
 First of all, we rescale $u$ in such a way that $\min\{u\}=-1$ and $0<u^\star=\max\{u\}\leq 1$. By Corollary \ref{cor_maxfit}, we can find a solution $w_{k,n,i,a}$ such that $\max\{u\}=\max\{w \ \text{ on }\ [a,b(a)]\}=m(k,n,i,a)$.

Consider a minimum point $x$ and a maximum point $y$ for the function $u$, and consider a unit speed minimizing geodesic (of length $l\leq d$) joining $x$ and $y$. Let $f(t)\equiv u(\gamma(t))$, and define $I\subset [0,l]$ by $I= \dot f^{-1} (0,\infty)$. Then, by a simple change of variables, we get
 \begin{gather*}
  d\geq \int_0 ^l dt \geq \int_I dt \geq \int_{-1}^{u^\star} \frac{dy}{\dot f (f^{-1}(y))}\geq\int_{-1}^{u^\star} \frac{dy}{\dot w (w^{-1}(y))}=\\
=\int_a^{b(a)} 1 dt = \delta(k,n,i,a)\geq \bar \delta(n,k,\lambda)\ ,
 \end{gather*}
where the last inequality follows directly from the definition \ref{deph_delta}. This and Proposition \ref{prop_barlambda} yield immediately to the estimate.

Sharpness can be proved in the following way. Fix $n$, $k$ and $d$, and consider the family of manifolds $M_i$ defined by the warped product
\begin{gather}
 M_i = [-d/2,d/2]\times_{i^{-1} \tau_3} S^{n-1}\, ,
\end{gather}
where $S^{n-1}$ is the standard $n$-dimensional Riemannian sphere of radius $1$. It is easy to see that the diameter of this manifold satisfies
\begin{gather*}
 d<d(M_i)\leq \sqrt{d^2 + i^{-2}\pi^2 \tau_3(d/2)^2 }\, ,
\end{gather*}
and so it converges to $d$ as $i$ converges to infinity. Moreover, using standard computations it is easy to see that the Ricci curvature of $M_i$ is bounded below by $(n-1)k$ and that the boundary $\partial M_i = \{a,b\} \times S^{n-1}$ of the manifold is geodesically convex (see subsection \ref{sec_warped}).

As mentioned in Remark \ref{rem_w}, the function $u(t,x)= w_{3,d/2}(t) $ is a Neumann eigenfunction of the $p$-Laplace operator relative to the eigenvalue $\bar \lambda$. Since the function $\bar \lambda (n,k,d)$ is continuous with respect to $d$, sharpness follows easily.
\end{proof}

 \chapter{Estimates on the Critical sets of Harmonic functions and solutions to Elliptic PDEs}
 \ifnum0\key=7
  \thispagestyle{headings}
 \fi

 In this chapter, we discuss the results obtained in \cite{chnava}. In particular, we study solutions $u$ to second order linear elliptic equations on manifolds and subsets of $\dR^n$ with both Lipschitz and smooth coefficients. We introduce new quantitative stratification techniques in this context, based on those first introduced in \cite{ChNa1,ChNa2}, which will allow to obtain new estimates and control on the critical set
\begin{align}
\Cr(u)\equiv\{x\, s.t. \ \abs{\nabla u}=0\}\, ,
\end{align}
and on its tubular neighborhood of radius $r$, $\T_r(\Cr(u))$. The results are new even for harmonic functions on $\dR^n$, but many of them hold under only a Lipschitz constraint on the coefficients (which is sharp, in that the results are false under only a H\"older assumption). 
\index{critical set}\index{elliptic equation}

Because the techniques are local and do not depend on the ambient space, we will restrict our study to functions defined on the unit ball $B_1(0)\subseteq \dR^n$. When needed, we will mention what modifications are required to study more general situations.
\paragraph{Notation}
In this chapter, we denote by $\Ha^{k}(A)$ the $k$-dimensional Hausdorff measure of the set $A$, while $\operatorname{dim}_{Haus}(A)$ and $\operatorname{dim}_{Min}(A)$ will denote respectively the Hausdorff and Minkowski dimension of the set $A$.
\index{Hausdorff measure} \index{Hausdorff measure!Ha@$\Ha^{k}$} \index{Hausdorff measure!Hausdorff dimension} \index{Minkowski size!Minkowski dimension}

\paragraph{}

\section{Introduction}
We are specifically interested in equations of the form
\begin{gather}\label{eq_Lu}
\L(u)=\partial_i(a^{ij}(x)\partial_j u) + b^i(x) \partial _i u=0\, ,
\end{gather}
and
\begin{gather}\label{eq_Lu2}
\L(u)=\partial_i(a^{ij}(x)\partial_j u) + b^i(x) \partial _i u + c(x)u=0\, .
\end{gather}
We require that the coefficients $a^{ij}$ are uniformly elliptic and uniformly Lipschitz, and that $b^i$, $c$ are measurable and bounded. That is, we assume that there exists a $\lambda\geq 1$ such that for all $x$
\begin{align}\label{e:coefficient_estimates}
(1+\lambda)^{-1}\delta^{ij}\leq a^{ij}(x)\leq (1+\lambda)\delta^{ij}\, , \, \, \text{Lip}(a^{ij})\leq \lambda\, , \, \, \max\cur{\abs {b^i(x)},\abs {c(x)}} \leq \lambda\, .
\end{align}
\begin{remark}
\rm Given these conditions, it is easy to see that all equations of the form
\begin{gather}\label{eq_Lu_plis}
\L(u)=a^{ij}(x)\partial_i\partial_j u + \tilde b^i(x) \partial _i u =0 
\end{gather}
can be written in the form \eqref{eq_Lu} with $b^i(x) = \tilde b^i(x) - \partial_j \ton{a^{ij}(x)}$.
\end{remark}

We will denote by $u$ a weak solution to \eqref{eq_Lu} in the sense of $W^{1,2}(B_1)$ \footnote{with respect to the standard Lebesgue measure}. Note that Lipschitz continuity of the coefficients is minimal if we want effective bounds on the critical set $\Cr(u)$. Indeed, A. Pli\'s in \cite{plis} found counterexamples to the unique continuation principle
\index{unique continuation principle}
for solutions of elliptic equations similar to \eqref{eq_Lu_plis} where the coefficients $a^{ij}$ are H\"older continuous with any exponent strictly smaller than $1$. No reasonable estimates for $\Cr(u)$ are possible in such a situation.

While giving some informal statements of our results, we give a brief review of what is known.

\paragraph{Hausdorff estimates} For simplicity sake we begin by discussing harmonic functions on $B_1(0)$. Using the fact that $u$ is analytic, it is clear that $\Ha^{n-2}(\Cr(u)\cap B_{1/2})<\infty$, unless $u$ is a constant. Quantitatively, a way to measure the {\it nonconstant} behavior of $u$ on a ball $B_r(x)$ is Almgren's frequency (or normalized frequency)
\index{frequency function!Almgren's frequency function}
\begin{align}\label{d:frequency}
N^u(x,r)\equiv\frac{r \int_{B_r(x)} \abs{\nabla u}^2 dV}{\int_{\partial B_r(x)} u^2 dS}\, ,\quad \bar N^u(x,r)\equiv\frac{r\int_{B_r(x)} |\nabla u|^2 dV}{\int_{\partial B_r(x)} [u-u(x)]^2 dS}\, .
\end{align}
Even in the case of a harmonic function, one would maybe like an estimate of the form
\begin{gather}\label{eq_hacr}
 \Ha^{n-2}(\Cr(u)\cap B_{1/2})\leq C(n,\bar N^u(0,1))\, .
\end{gather}
In other words, if $u$ is bounded away from being a constant then the critical set can only be so large in the $n-2$-Hausdorff sense. 

In \cite{hoste}, the authors deal with equations of the form \eqref{eq_Lu} with smooth coefficients, and prove the ineffective estimate
\begin{gather*}
 \Ha^{n-2}(\Cr(u)\cap B_{1/2}) < \infty
\end{gather*}
In \cite{hanhardtlin}, the authors prove effective estimates on the singular set, i.e., $\Ha^{n-2}(\Cr(u)\cap B_{1/2}\cap \{u=0\})< C(n,\bar N^u(0,1))$. By some simple adaptation, the estimate \eqref{eq_hacr} has been proved in \cite{HLrank} for harmonic functions and solutions to \eqref{eq_Lu} with some additional assumption on the coefficients and their derivatives. 

We provide a new proof of these estimates as a corollary of our Minkowski estimates and the $\epsilon$-regularity theorem \cite[Lemma 3.2]{hanhardtlin}.

In section \ref{sec_sing}, we will show that the techniques work even for solutions of \eqref{eq_Lu2}, however in this case it is necessary to further restrict the estimate to the zero level set, that is, $\Ha^{n-2}(\Cr(u)\cap B_{1/2}\cap \{u=0\})\leq C(n,\bar N^u(0,1))$.

\paragraph{Minkowski estimates}
Minkowski estimates are the real contribution of this work. As we will see, we will not need any additional assumptions on the coefficients in order to obtain these estimates.

It is known, see \cite{lin}, that $\Cr(u)\cap B_{1/2}$ has Hausdorff dimension $\dim_{Haus}(\Cr(u)\cap B_{1/2})=n-2$. We are not able to improve this to an effective finiteness in this context, but we do make advances in two directions. First we do show effective Minkowski estimates of the form
\begin{align}
\Vol(\T_r(\Cr(u))\cap B_{1/2})\leq C(n,\bar N^u(0,1),\epsilon)r^{2-\epsilon}\, .
\end{align}
Among other things this improves $\dim_{Haus}\Cr(u)\leq n-2$ to $\dim_{Min}\Cr(u)\leq n-2$. That is, the Minkowski dimension of the critical set is at most $n-2$. More importantly, this gives effective estimates for the volume of tubes around the critical set, so that even if $\Ha^{n-2}(\Cr(u))=\infty$ in the Lipschitz case, we have a very definite effective control on the size of the critical set regardless.

More generally, we prove effective Minkowski estimates not just for the critical set, but our primary contribution is the introduction and the analysis of a quantitative stratification.
The standard stratification of $u$, based on tangential behavior of $u$, separates points of $u$ based on the leading order polynomial of the Taylor expansion of $u-u(x)$. The stratification is based not on the degree of this polynomial, but on the number of symmetries it has.
\index{symmetry!standard symmetry}\index{stratification!standard stratification}
More specifically, $\cS^k$ consists of those points $x$ such that the leading order polynomial $P(y)$ of $u(y)-u(x)$ is a function of at least $n-k$ variables. For instance, if $u$ has vanishing gradient at $x$, then the leading order polynomial has degree at least two, and being a harmonic function it must depend at least on two variables, so $x\in \cS^{n-2}$.

In a manner similar to \cite{ChNa1,ChNa2}, we will generalize the standard stratification to a quantitative stratification.
\index{stratification!standard stratification}\index{stratification!quantitative stratification}
Very roughly, for a fixed $r,\eta>0$ this stratification will separate points $x$ based on the degrees of $\eta$-{\it almost} symmetry
\index{symmetry!almost symmetry}\index{stratification!quantitative stratification}
of an approximate leading order polynomials of $u-u(x)$ at scales  $\geq r$. The key point is that we will prove strong Minkowski estimates on the quantitative stratification, as opposed to the weaker Hausdorff estimates on the standard stratification. As in \cite{ChNa1,ChNa2}, these estimates require new blow up techniques distinct from the standard dimension reduction arguments, and will work under only Lipschitz constraints on the coefficients. The key ideas involved are that of a frequency decomposition and cone splitting. In short, cone-splitting is the general principle that nearby symmetries interact to create higher order symmetries. 
Thus, the frequency decomposition 
\index{frequency function!frequency decomposition}
will decompose the space $B_1(0)$ based on which scales $u$ looks almost $0$-symmetric. On each such piece of the decomposition nearby points automatically either force higher order symmetries or a good covering of the space, and thus the estimates can be proved easily on each piece of the decomposition. The final theorem is obtained by then noting that there are far fewer pieces of the decomposition than seem possible {\it apriori}. 

The $(n-2)$-Hausdorff estimate on the critical sets of solutions of (\ref{eq_Lu}) with smooth coefficients will be gotten by combining the estimates on the quantitative stratification with an $\epsilon$-regularity type theorem from \cite{hanhardtlin}.
\index{epsilon reg@$\epsilon$-regularity theorem}

\subsection{The Main Estimates on the Critical Set}\label{ss:mainresults}

The primary goal of this section is to study the critical sets of solutions of (\ref{eq_Lu}). Before stating the theorems let us quickly recall the notion of Hausdorff and Minkowski measure. Recall that, while the Hausdorff dimension of a set can be small while still being dense (or arbitrarily dense), Minkowski estimates bound not only the set in question, but the tubular neighborhood of that set, providing a much more analytically effective notion of {\it size}. Precisely, given a set $A\subseteq \dR^n$ its $k$-dimensional Hausdorff measure is defined by
\begin{align}
\Ha^k(A)\equiv \lim_{r\to 0}\, \inf \cur{\sum_{i}\omega_k r_i^{k} \ \ s.t. \ \ A\subseteq\cup B_{r_i}(x_i):r_i\leq r}\, . 
\end{align}
Hence, the Hausdorff measure is obtained by finding coverings of $A$ by balls of arbitrarily small size.
\index{Hausdorff measure}
On the other hand, the Minkowski size is defined by
\begin{align}
\Mi^k(A)\equiv \lim_{r\to 0}\, \inf \cur{\sum_{i}\omega_k r^{k}\ \ s.t. \ \ A\subseteq\cup B_{r}(x_i)}\, .
\end{align}
Hence, the Minkowski size of $A$ is obtained by covering $A$ with balls of a fixed size, that is, by controlling the volume of tubular neighborhoods of $A$.
\index{Minkowski size}
The Hausdorff and Minkowski dimensions are then defined as the smallest numbers $k$ such that $\Ha^{k'}(A)=0$ or $\Mi^{k'}(A)=0$, respectively, for all $k'>k$. As a simple example note that the Hausdorff dimension of the rational points in $B_1(0)$ is $0$, while the Minkowski dimension is $n$.
\index{Hausdorff measure!Hausdorff dimension} \index{Minkowski size!Minkowski dimension}

Given these definitions, we may state the main result of this chapter.
\begin{teo}\label{t:crit_lip}
Let $u:B_1(0)\to\dR$ satisfy (\ref{eq_Lu}) and (\ref{e:coefficient_estimates}) weakly with
\begin{gather}
 \frac{\int_{B_1} \abs{\nabla u}^2 dV}{\int_{\partial B_1(0)} [u-u(0)]^2 dS}\leq \Lambda\, .
\end{gather}
Then for every $\eta>0$ we have that
\begin{align}
\Vol(\T_r(\Cr(u))\cap B_{1/2}(0))\leq C(n,\lambda,\Lambda,\eta)r^{2-\eta}\, .
\end{align}
\end{teo} 
\begin{remark}\rm
This immediately gives us the weaker estimate $\dim_{Min}\Cr(u)\leq n-2$.
\end{remark}
\begin{remark}\rm
We note that a version of the theorem still holds for solutions $u$ of (\ref{eq_Lu2}) if we restrict the estimate to the zero level set of $u$. Indeed, in this case
\[\Vol(\T_r(\Cr(u))\cap B_{1/2}(0)\cap\{u=0\})\leq C(n,\lambda,\Lambda,\eta)r^{2-\eta}\, .
 \]

\end{remark}
\begin{remark}\rm
On a Riemannian manifold the constant $C$ should also depend on the sectional curvature of $M$ and the volume of $B_1$. In this case one can use local coordinates to immediately deduce the theorem for manifolds from the Euclidean version. The estimates \eqref{e:coefficient_estimates} involve the Riemannian structure of $M$, and $a^{ij}$ and $b^i$ are now tensors on $M$ and $\partial$ is the covariant derivative on $M$.
\end{remark}

If we make additional assumptions on the regularity on the coefficients in (\ref{eq_Lu}) then we can do better. The next Theorem, which is our main application to solutions of (\ref{eq_Lu}) with smooth coefficients, will be proved by combining Theorem \ref{t:crit_lip} with the important $\epsilon$-regularity theorem \cite[Lemma 3.2]{hanhardtlin}.
\begin{teo}\label{th_Ha}
Let $u:B_1\to \R$ be a solution to equation \eqref{eq_Lu} with \eqref{e:coefficient_estimates} such that
\begin{gather}
 \frac{\int_{B_1} \abs{\nabla u}^2 dV}{\int_{\partial B_1} [u-u(0)]^2 dS} \leq \Lambda\, .
\end{gather}
Then there exists a positive integer $Q=Q(n,\lambda,\Lambda)$ such that if the coefficients of the equation satisfy
\begin{gather}
 \sum_{ij}\norm{a^{ij}}_{C^{Q}(B_1)} + \sum_i \norm{b^i}_{C^{Q}(B_1)}\leq L^+\, ,
\end{gather}
then there exist positive constants $C(n,\lambda,L^+,\Lambda)$ such that
\begin{gather}
 \Ha^{n-2}(\Cr(u)\cap B_{1/2}(0))\leq C(n,\lambda,L^+,\Lambda)\, .
\end{gather}
\end{teo}
\begin{remark}\rm
As for the previous result, also this one holds, with the necessary modifications, also on a Riemannian manifold.
\end{remark}
\begin{remark}\rm
Although with a slightly different proof, this result has already been obtained in \cite{HLrank}.
\end{remark}

For the sake of clarity, we will at first restrict our study to harmonic functions. Apart from some technical details, all the ideas needed for the proof of the general case are already present in this case. We will then turn our attention to the generic elliptic case, pointing out the differences between the two situations.

\subsection{Notation}
Recall the definition of order of vanishing and, similarly, of order of vanishing of the derivatives.
\begin{definition}\label{def_sing}
\index{order of vanishing}
 Given a sufficiently smooth function $f:B_1\subset \R^n\to \R$, the order of vanishing of $u$ at $p$ is defined as
\begin{gather}
\O(u,p)=\min\{N\geq 0\ \ s.t. \ \ \partial^m u|_p =0 \ \ \forall 0\leq\abs m \leq N-1  \ \ \text{and} \ \ \exists \abs m =N \ \ s.t. \ \ \partial^m u|_p \neq 0  \}\, ,
\end{gather}
 while the order of vanishing of derivatives is
\begin{gather}
\O'(u,p)=\min\{N\geq 1 \ \ s.t. \ \ \partial^m u|_p =0 \ \ \forall 0<\abs m \leq N-1  \ \ \text{and} \ \ \exists \abs m =N \ \ s.t. \ \ \partial^m u|_p \neq 0  \}\, .
\end{gather}
\end{definition}

It is evident that $\O'(u,p)=\O(u-u(p),p)$.
\begin{definition}\label{deph_leading}
 Given a nonconstant harmonic function $u:B_r(x)\to \R$ with Taylor expansion
\begin{gather}
 u(x+y)=\sum_{\alpha \geq 0} u_{\alpha}(x) y^\alpha=\sum_{\alpha \geq 0} \partial _\alpha u|_x y^\alpha
\end{gather}
we define its leading polynomial at $x$ to be
\begin{gather}
 P_x(y)=\sum_{\alpha=\O(u,x)}u_\alpha(x) y^\alpha\, .
\end{gather}
It is easily seen that this polynomial is a homogeneous harmonic polynomial. Moreover
\begin{gather}
 \frac{P_x(y)}{\ton {\fint_{\partial B_1} P_x(y)^2}^{1/2}}=\lim_{r\to 0} \frac{u(x+ry)}{\ton {\fint_{\partial B_1} u(x+ry)^2 \ dy}^{1/2}}\, ,
\end{gather}
where the limit is taken in the uniform norm on $B_1(0)$. Note that, by elliptic regularity, we could equivalently choose a wide variety of higher order-norms. Note that
\begin{gather}
 P_x(y)=\sum_{\alpha=\O'(u,x)}u_\alpha(x)y^\alpha
\end{gather}
is the leading polynomial of $u-u(x)$.
\index{leading polynomial}
\end{definition}

We will use the following notation for the average of an integral.
\begin{definition}
 Let $(\Omega,\mu)$ be a measure space, and let $u:\Omega\to \R$ be a measurable function. We define
\begin{gather}
 \fint_{\Omega} u\ d\mu= \frac{\int_{\Omega} u\ d\mu}{\mu(\Omega)}
\end{gather}
In particular, if $d\mu=dy$ is the standard Lebesgue measure on $\R^n$ and $B_r(x)$ is the Euclidean ball of radius $r$ and center $x$ we have
\begin{gather}
 \fint_{B_r(x)} u \ dy=\frac{\int_{B_r(x)} u \ dy}{\omega_n r^n }\\
 \fint_{\partial B_r(x)} u \ dy=\frac{\int_{\partial B_r(x)} u \ dy}{n\omega_n r^{n-1} }
\end{gather}
where $\omega_n$ is the volume of the unit ball in $\R^n$.
\end{definition}
To have both integrals well defined, one can appeal to the theory of traces for Sobolev spaces (for a complete reference, see \cite[Section 5.5]{evans}).
\begin{prop}
 Let $u\in W^{1,2}(B_r(0))$. Then the following integral is well-defined and depends continuously on the $W^{1,2}(B_r)$ norm of $u$:
\begin{gather}\label{eq_pint}
 \int_{\partial B_r} u^2 dS\, .
\end{gather}
Moreover, let $\vec v$ be the vector field $\vec v=(x_1,\cdots,x_n)=r\partial_r$. Then by the weak formulation of the divergence theorem
\begin{gather}
 \int_{\partial B_r} u^2 dS= r^{-1}\int_{\partial B_r} u^2\ps{\vec v}{r^{-1}\vec v} dS= r^{-1}\int_{B_r} \qua{2u\ps{\nabla u}{\vec v} + n u^2} dV\, .
\end{gather}
This in particular implies that the integral in \eqref{eq_pint} depends continuously on $u$ also with respect to the weak $W^{1,2}$ topology.
\end{prop}
\begin{definition}
 Given a set $A$, $\T_r(A)$ will denote its tubular neighborhood of radius $r$, i.e.,
 \begin{gather}
  \T_r(A)= \bigcup_{x\in A} B_r(x)\, .
 \end{gather}
\index{tubular neighborhood}

\end{definition}

\section{Harmonic functions}\label{sec_harmonic}
In this section we prove the $n-2+\epsilon$ Minkowski and $n-2$ Hausdorff uniform estimates for the critical sets of harmonic functions.

Before introducing the quantitative stratification, which will play a central role in the estimates, we recall some basic results about harmonic polynomials and Almgren's frequency function.

\subsection{Homogeneous polynomials and harmonic functions}
\index{homogeneous polynomial}
We start by discussing some properties of homogeneous polynomials and harmonic functions that will be of use later on in the text. 
The following result is completely similar to the one obtained in the proof of \cite[Theorem 4.1.3 p. 67]{hanlin}, and we refer to it for the proof.
\begin{teo}\label{thm_sisplit}
 Let $P$ be a homogeneous polynomial of degree $d\geq 2$. The set
\begin{gather}
 S_{d}(P)=\{x\in \R^n \ s.t. \ \O'(P,x)=d\}
\end{gather}
is a subspace of $\R^n$. Moreover, for every point $x\in S_{d}$, $P(x)=0$. Let $V=S_d(P)^\perp$, then for all $x\in \R^n$
\begin{gather}
P(x)=P(\Pi_V(x)) \, ,
\end{gather}
where $\Pi_V$ is the projection onto the subspace $V$. 
\end{teo}
This means that if $P$ is a homogeneous polynomial in $\R^n$ with degree $d$ and which has a ``null space'' $S_d(P)$ of dimension $k$, then, after a suitable change in the coordinates, it is actually a polynomial in only $n-k$ variables.

Note that the results in the Theorem do not apply to linear functions. In fact, if $d=1$, then evidently $S_1(P)=\R^n$.

It is interesting to recall what it means for a harmonic function to be homogeneous. By definition a function $u:B_1(0)\to \R$ is homogeneous on $A_{r_1,r_2}=B_{r_2}\setminus \overline B_{r_1}$ if, for all $r_1<\abs x < r_2$
\begin{gather}
\ps{\nabla u|_x}{\hat r} = h(\abs x) u(x)\, ,
\end{gather}
where $\hat r =\abs x ^{-1} \vec x$ is the unit radial vector (not defined at the origin). Integrating this equation, we have that $u(r,\theta)=f(r)\phi(\theta)$ on $A_{r_1,r_2}$. Such a splitting of the radial and spherical variables is most useful, especially when the function $u$ is harmonic. Indeed, using the theory of spherical harmonics and the unique continuation principle \footnote{for a more detailed reference, see for example \cite[Section 1, pp 5-6]{colding_hp}}, it is straightforward to see that
\begin{lemma}\label{lemma_harrad}
 Let $u$ be a harmonic function in $B_1(0)$ such that, for some $0\leq r_1<r_2\leq 1$, $u$ is homogeneous on $A_{r_1,r_2}$. Then $u$ is a homogeneous harmonic polynomial.
\end{lemma}

\subsection{Almgren's frequency function and rescaled frequency}\label{sec_freq}
\index{frequency function!Almgren's frequency function}\index{frequency function!rescaled frequency function}
One of the main and most versatile tools available for harmonic functions is the so-called Almgren's frequency (or simply frequency) of $u$, which is given by the next definition.
\begin{definition}\label{deph_N}
 Given a nonzero harmonic function $u:B_1(0)\to \R$, we define
\begin{gather}
 D^u(x,r)= \int_{B_r(x)} \abs{\nabla u}^ 2 dV\, , \ \ \quad \ \ D_m^u(x,r)= \fint_{B_r(x)} \abs{\nabla u}^ 2 dV\, ,\\
 H^u(x,r)= \int_{\partial B_r(x)} u^2 dS\, ,\ \ \ \ \quad \ H^u_m(x,r)= \fint_{\partial B_r(x)} u^2 dS\, ,\\
 N^u(x,r) = \frac{r \int_{B_r(x)} \abs{\nabla u}^ 2 dV}{\int_{\partial B_r(x)} u^2dS}=\frac{rD^u(x,r)}{H^u(x,r)}= \frac{r^2D_m^u(x,r)}{nH_m^u(x,r)}\, .
\end{gather}
\end{definition}
We will drop the superscript $u$ and the variable $x$ when there is no risk of confusion.

From this definition, it is straightforward to see that $N$ is invariant under rescaling of $u$ and blow-ups. In particular
\begin{lemma}
 Let $u$ be a nonzero harmonic function in $B_1(0)$, and let $w(x)=c\ u(k x+\bar x)$ for some constants $c,k \in \R^+$. Then
\begin{gather}
 N^u(\bar x,r)=N^w (0, k^{-1}r)
\end{gather}
for all $r$ such that either side makes sense.
\end{lemma}

If $u$ is a harmonic polynomial of degree $d$ homogeneous with respect to the origin, by direct calculation it is easy to see that $N(0,r)=d$ for all $r$. Moreover, using the Taylor expansion of $u$ around $x$, we get the following link between the frequency and the vanishing order of $u$:
\begin{gather}\label{eq_corr}
 N(x,0)=\lim_{r\to 0} N(x,r)=\O(u,x)\, ,
\end{gather}
which means that $N(x,0)$ is the degree of the leading polynomial of $u$ at $x$.

The most important property of this function is its monotonicity with respect to $r$, namely
\begin{teo}\label{th_Nmon}
 Let $u$ be a nonzero harmonic function on $B_1(0)\subset \R^n$, then $N(r)$ is monotone nondecreasing in $(0,1)$. Moreover, if for some $0\leq r_1<r_2$, $N(0,r_1)=N(0,r_2)$, then for $r_1<\abs {x} <r_2$
\begin{gather}
 \left.\frac{\partial{u}}{\partial {x}}\right\vert_{ x} = h(\abs x ) u(x)\, .
\end{gather}
\end{teo}
With obvious modifications, this theorem holds for $N(x,r)$ for any $x$.
\begin{proof}
The proof of this theorem is quite standard. It relies on some integration by parts technique, or it could also be proved with the first variational formula for harmonic functions, see for example \cite[Section 2.2]{hanlin} or \cite[Lemma 1.19]{cm1}. In the proof of Theorem \ref{th_Nellmon}, we will carry out a more general computation, so here we omit the details of the proof.
\end{proof}

Using Lemma \ref{lemma_harrad} we can prove the following.
\begin{corollary}\label{cor_hom}
Let $u$ be a harmonic function on $B_1\subset \R^n$. If there exist $0\leq r_1<r_2<R$ such that $N(0,r_1)=N(0,r_2)$, then $u$ is a homogeneous harmonic polynomial of degree $d$ and $N(0,r)=d$ for all $r$.
\end{corollary}

Very important consequences of the monotonicity of $N$ are the following doubling conditions, which, as a byproduct, imply the unique continuation principle.
\begin{lemma}\label{lemma_2}\cite[Corollary 2.2.6]{hanlin}
\index{doubling conditions}
Let $u$ be a nonzero harmonic function on $B_1\subset \R^n$. Then for every $0<r_1<r_2\leq 1$ we have
\begin{gather}
 \frac{H(r_2)}{r_2^{n-1}}= \frac{H(r_1)}{r_1^{n-1}}\exp\ton{2\int_{r_1}^{r_2} \frac{N(s)}{s} ds}\, ,
\end{gather}
or equivalently
\begin{gather}\label{eq_corr2}
 H_m(r_2)= H_m(r_1) \exp\ton{2\int_{r_1}^{r_2} \frac{N(s)}{s} ds}\, .
\end{gather}
\end{lemma}
\begin{proof}
 The proof follows easily from
 \begin{gather}
  \frac{d}{dr} \log\ton{H_m(r)} = 2\frac{N(r)}{r}\, .
 \end{gather}

\end{proof}
Note that, by monotonicity, it is immediate to estimate $N(s)\leq N(\max\{r_1,r_2\})$ and so
\begin{gather}
 \ton{\frac {r_2} {r_1} }^{2N(r_1)} \frac{H(r_1)}{{r_1}^{n-1}}\leq \frac{H(r_2)}{r_2^{n-1}}\leq \ton{\frac {r_2}{r_1} }^{2N(r_2)} \frac{H(r_1)}{{r_1}^{n-1}}\, .
\end{gather}

An almost immediate corollary is the following.
\begin{corollary}\label{cor_2}
 Under the hypothesis of the previous Lemma we can estimate that for $0<r\leq 1$
\begin{gather}
 \frac{r H(r)}{n+2N(r)} \leq \int_{B_r} u^2dV \leq \frac{rH(r)}{n}\, .
\end{gather}
\end{corollary}

Monotonicity of $N$ implies that an upper bound on $N(0,1)$ is also an upper bound on $N(0,r)$ for all $r\in (0,1)$. A similar statement is valid if we change the base point $x$. The following property is crucial because it gives us control on the order of vanishing of the function $u$ at each point $B_1(0)$.
\begin{lemma}\label{lemma_strong2}\cite[Theorem 2.2.8]{hanlin}
Let $u$ be a nonconstant harmonic function in $B_1\subset \R^n$. Then for any $r\in(0,1)$, there exists a constant $N_0=N_0(r)<<1$ such that if $N(0,1)\leq N_0$ then $u$ does not vanish in $B_r$ \footnote{so that $\lim_{r\to 0}N(x,r)=0$ for all $x\in B_r$}, and if $N(0,1)>N_0$ then for all $\tau<1$ and $x\in B_r$
\begin{gather}
 N\ton{x,\tau (1-R)}\leq C (n,R,\tau,N(0,1))\, .
\end{gather}
\end{lemma}

Note that, as a corollary, we get the following.
\begin{corollary}
 Let $u$ be as above with $u(0)=0$ and $N(0,1)\leq \Lambda$. Then there exists a constant $C(n,r,\Lambda)$ such that for all $x\in B_r$
\begin{gather}
 N\ton{x,\frac 2 3 (1-r)}\leq C(n,r,\Lambda)\, .
\end{gather}
\end{corollary}
This is easily proved since, by monotonicity and \eqref{eq_corr}, $N(0,r)\geq 1$ for all $r$ if $u(0)=0$.

In order to study properly the critical set of a harmonic function, it seems reasonable to change slightly the definition of frequency in order to make it invariant also under translation of the function $u$. For this reason, in this work we introduce and study the modified frequency function.
\begin{definition}
 Given a nonconstant $u:B_1(0)\to \R$, we define
\begin{gather}
  \bar H^u(x,r)= \int_{\partial B_r(x)} (u-u(x))^2dS\, , \ \ \ \ \quad \quad \ \bar H^u_m(x,r)= \fint_{\partial B_r(x)} (u-u(x))^2dS\, ,\\
 \bar N^u(x,r) = \frac{r \int_{B_r(x)} \abs{\nabla u}^ 2dV}{\int_{\partial B_r(x)} (u-u(x))^2dS}=\frac{rD^u(x,r)}{\bar H^u(x,r)}= \frac{r^2D_m^u(x,r)}{n\bar H_m^u(x,r)}=N^{u-u(x)}(x,r) \, .
\end{gather}
\end{definition}
Note that, by the unique continuation principle, all the quantities introduced above are well-defined at all scales if $u$ is harmonic and non constant.

Using the mean-value theorem for harmonic functions, we can rewrite the rescaled frequency as
\begin{gather}
 \bar N^u(x,r)= \frac {r^2} n \frac{\fint_{B_r(x)} \abs{\nabla u}^ 2dV}{\qua{\ton{\fint_{\partial B_r(x)} u^2dS}-u(x)^2}} \, .
\end{gather}

Since the Laplace equation is invariant under addition by a constant, it is clear that some of the properties on $N$ carry over to $\bar N$ immediately. In particular, $\bar N(x,r)$ is still monotone nondecreasing with respect to $r$, and, if for some $0\leq r_1<r_2\leq 1$, $\bar N(x,r_1)=\bar N(x,r_2)$, then $u-u(x)$ is a homogeneous harmonic polynomial centered at $x$. Moreover, it is easily seen that
\begin{gather}
 \bar N (x,0)=\lim_{r\to 0} \bar N(x,r) = \O'(u,x)\geq 1\, .
\end{gather}

Note that, as the usual frequency $N$, $\bar N$ is invariant for blow-ups and rescaling of the function $u$. Moreover it is also invariant under translation of $u$, so that we have the following lemma.
\begin{lemma}\label{lemma_Ninv}
 Let $u$ be a nonconstant harmonic function in $B_1(0)$, and let $w(x)=\alpha(u(k x+\bar x) + \beta)$ for some constants $\alpha,\beta,k \in \R$, $k\neq 0$. Then for all $0\leq r\leq 1$ we have
\begin{gather}
 \bar N_{u}(\bar x,r)=\bar N_w (0, k^{-1}r)
\end{gather}
\end{lemma}

However, some properties of $N$ are not immediately inherited by $\bar N$. In particular, it is not trivial to notice that a bound similar to the one given in Lemma \ref{lemma_strong2} is available also for the rescaled frequency. 
\begin{lemma}\label{lemma_pre-strong2bar}
Let $u$ be a nonconstant harmonic function in $B_1(0)\subset \R^n$. Then there exists a function $C(N,n )$ such that for all $x\in B_{1/4}(0)$ we have
\begin{gather}
 \bar N(x,1/2) \leq  C(\bar N(0,1),n)
\end{gather}
\end{lemma}
Note that by monotonicity we can extend this estimate to all $r $ smaller than $1/2$.
\begin{proof}
We can assume without loss of generality $u(0)=0$ and replace $\bar N (0,r)=N(0,r)\geq 1$ for all $r\leq 1$. 

The first step of the proof follows from an estimate of $u(x)^2$ with respect to $H_m(x)$. By continuity of $u$, it is evident that
\begin{gather}
 u(x)^2=\lim_{r\to 0} H_m (x,r)\, .
\end{gather}
The fact that $u$ is not constant and \eqref{eq_corr2} force $H_m(x,r)$ to be strictly increasing with respect to $r$, and the growth can be controlled in a quantitative way. Indeed, using the doubling conditions in Lemma \ref{lemma_2} we get
\begin{gather}
 H_m(x,1/2)\geq  H_m(x,1/3)(3/2)^{2N(x,1/3)}\geq  u(x)^2(3/2)^{2N(x,1/3)}\, ,\\
\notag  u(x)^2 \leq H_m(x,1/2) (3/2)^{-2N(x,1/3)}\, .
\end{gather}
Then by simple algebra we have
\begin{gather}
 \bar N(x,1/2) = \frac {(1/2)^2} n \frac{\fint_{B_{1/2}(x)} \abs{\nabla u}^ 2dV}{\qua{\ton{\fint_{\partial B_{1/2}(x)} (u)^2dS}-u(x)^2}}\leq N(x,1/2) \ton{1-(3/2)^{-2N(x,1/3)}}^{-1}\, .
\end{gather}
We know from Lemma \ref{lemma_strong2} that $N(x,r)$ is bounded from above by $c(n,N(0,1))$. Now we derive a bound from below which will allow us to conclude the estimate.

The hypothesis imply that $B_{12^{-1}}(0)\subset B_{1/3}(x)$. Thus, by a repeated application of Corollary \ref{cor_2}, we have \footnote{we omit for simplicity the symbols $dV$, $dS$}
\begin{gather*}
 N(x,1/3)=\frac{(1/3) \int_{B_{1/3}(x)} \abs{\nabla u}^ 2} {\int_{\partial B_{1/3}(x)} u^2} \geq \frac{(1/3)^2 \int_{B_{12^{-1}}(0)} \abs{\nabla u}^ 2}{[n+2N(x,1/3)] \int_{B_{1/3}(x)}u^2 }\geq \frac{(1/3)^2 \int_{B_{12^{-1}}(0)} \abs{\nabla u}^ 2}{[n+2N(x,1/3)] \int_{B_{1}(0)}u^2 }\geq\\
\notag \geq n\frac{(1/3)^2}{1}\frac{ \int_{B_{12^{-1}}(0)} \abs{\nabla u}^ 2}{[n+2N(x,1/3)] \int_{\partial B_{1}(0)}u^2 }\geq  n\frac{(1/3)^2}{ 1} \ton{\frac{1}{12}}^{n-1}\frac{ 12^{-2N(0,1)}\int_{  B_{12^{-1}}(0)} \abs{\nabla u}^ 2}{[n+2N(x,1/3)] \int_{\partial B_{12^{-1}}(0)}u^2 } = \\
\notag = n\frac{(1/3)^2}{ 1} \ton{\frac{1}{12}}^{n-2}\frac{ 12^{-2N(0,1)} N(0,12^{-1})}{[n+2N(x,1/3)] } \geq n\frac{(1/3)^2}{ 1} \ton{\frac{1}{12}}^{n-2}\frac{ 12^{-2N(0,1)}}{[n+2C(n,N(0,1))] }\geq C(n,\Lambda)>0\, .
\end{gather*}
\end{proof}

The assumption $x\in B_{1/4}(0)$ is a little awkward, although necessary for the proof. In order to dispense with it, we use a compactness argument.

\begin{lemma}\label{lemma_LcontrolsN}
 Let $u$ be a nonconstant harmonic function in $B_1(0)$ with $\bar N(0,1)\leq \Lambda$. Then for every $\tau<1$ and $x\in B_\tau(0)$, there exists $C(x,\tau,\Lambda)$ such that
\begin{gather}
 \bar N \ton{x,\frac{1-\tau} 2 }\leq C(x,\tau,\Lambda)\, .
\end{gather}
\end{lemma}
\begin{proof}
 Fix $x\in B_\tau(0)$, and suppose by contradiction that there exists a sequence of harmonic functions $u_i$ such that $\bar N_{u_i}(0,1)\leq \Lambda$ and $\bar N_{u_i}(x,(1-\tau)/2)\geq i$. Using the symmetries of $\bar N$, we can assume that $u_i(0)=0$ and $\int_{\partial B_1} u_i^2 =1$. This implies that
\begin{gather}
 \int_{B_1} \abs{\nabla u_i}^2 \leq \Lambda\, ,\\
 \int_{B_1} u_i^2dV\leq \frac 1 n \int_{\partial B_1} u_i^2dS \leq \frac 1 n\, .
\end{gather}
So there exists a subsequence converging to some harmonic function $u$ in the weak $W^{1,2}(B_1)$ topology, and by elliptic estimates the convergence is also $C^1$ uniform on compact subsets of $B_1(0)$, and, in particular, on $\overline B_{(1-\tau)/2}(x)$. Since $u$ cannot be constant,
\begin{gather}
 \int_{\partial B_{(1-\tau)/2}(x)} [u-u(x)]^2 dS \geq \epsilon>0\, ,
\end{gather}
and this implies that $\int_{B_{(1-\tau)/2}(x)}\abs{\nabla u_i}^2 dV$ tends to infinity, which is impossible. So the lemma is proved.
\end{proof}
Now we are ready to prove the following lemma, which will play a crucial role in the proof of the main theorems. It is a generalization of Lemma \ref{lemma_strong2}.
\begin{lemma}\label{lemma_strong2bar}
 Let $\tau<1$ and $\Lambda >0$. For every (nonconstant) harmonic function $u$ in $B_1(0)$ with $\bar N(0,1)\leq \Lambda$ and for every $x\in B_{\tau}(0)$, there exists a constant $\bar C(\Lambda,\tau)$ such that $\bar N (x,(1-\tau)/2) \leq \bar C$.
\end{lemma}
\begin{proof}
By the previous lemma, there exists $C(\Lambda,\tau,x)$ such that $\bar N (x,(1-\tau)/2) \leq \tilde \Lambda (\Lambda,\tau,x)$.

By Lemma \ref{lemma_pre-strong2bar} we have that for every $y\in B_{(1-\tau)/8}(x)$, there exists a constant $\tilde C(\Lambda,\tau,x)$ such that
\begin{gather}
 \bar N \ton{y,\frac {1-\tau}{4}}\leq \tilde C(\Lambda,\tau,x)\, .
\end{gather}
Now by compactness cover $\overline B_{\tau}$ with a finite number of balls $B_{(1-\tau)/8}(x_i)$, and set
\begin{gather}
 \bar C(\Lambda,\tau)= \max_i \{\tilde C(\Lambda,\tau,x_i)\}\, ,
\end{gather}
 then the lemma is proved.
\end{proof}

\subsection{Standard symmetry}\label{sec_homsym}
\index{symmetry!standard symmetry}
In this section, we define $k$-symmetric and almost $k$-symmetric harmonic functions, and study the relation between almost symmetric functions and the behaviour of the rescaled frequency.

We begin by collecting some of the results of the previous sections in the following definition.
\begin{definition}\label{deph_hom}
 Given a harmonic function $u:B_1(0)\to \R$, we say that $u$ is $(0,x)$-symmetric if $\bar N(x,r)$ is constant on a nonempty interval $(r_1,r_2)$. By Corollary \ref{cor_hom}, this is equivalent to $u-u(x)$ being a homogeneous harmonic polynomial centered at $x$ of degree $\bar N (0,r_1)$. 
\end{definition}
Now we are ready to define $k$-symmetric harmonic functions.
\begin{definition}
 Given a harmonic function $u:B_1(0)\to \R$, we say that $u$ is $(k,x)$-symmetric if $\bar N(x,r)=d$ is constant (equivalently everywhere or for $r\in (r_1,r_2)$) and there exists a hyperplane $V$ of dimension $k$ such that
\begin{gather}
 u(y+z)=u(y)
\end{gather}
for all $y\in \R^n$, $z\in V$.
\end{definition}
It is easy to see that if $u$ is $(k,x)$-symmetric with respect to $V$, then $u-u(x)$ is a homogeneous harmonic polynomial $P$ centered at $x$ with degree $d=\bar N (x,r)$. Using the definitions given in Theorem \ref{thm_sisplit}, we can also deduce that $x+V\subset \Si_{d}(P)$.
 
The converse is also true in some sense, but we have to distinguish between two cases.
\begin{teo}
 Let $P$ be a homogeneous harmonic polynomial of degree $d$. If $d=1$, then $P$ is a linear function $u(x)=\ps{\vec P}{x}+c$, and, for any $x\in \R^n$, $P$ is $(n-1,x)$-symmetric with respect $\vec P ^\perp$. 

If $d\geq 2$ and $\operatorname{dim}(\Si_{d}(P))=k$, then for every $x\in \Si_d(P)$, $P$ is $(k,x)$-symmetric with respect to $V=\Si_d(P)$.
\end{teo}
\begin{proof}
 The proof of the case $d=1$ is immediate. The case where $d\geq 2$ is a simple application of Theorem \ref{thm_sisplit}.
\end{proof}

As mentioned in the introduction, symmetric points close to each other force higher order symmetries. In particular, we have the following theorem.
\begin{teo}\label{th_prepolyspl}
 Let $d\geq 2$ and let $u$ be a $(k,x)$-symmetric harmonic polynomial of degree $d$ with respect to $V$. Suppose that $u$ is also $(0,y)$-symmetric with $y\not \in V+x$. Then $u$ is $(k+1,x)$ symmetric with respect to $V\oplus(y-x)=\operatorname{span}(y-x,V)$
\end{teo}
\begin{proof}
 Assume without loss of generality, $x=0$ and $u(x)=0$.\\
Since $u$ is a homogeneous harmonic polynomial with respect to $0$ and also with respect to $y$, then $y\in \Si_{d} (u)$. We also know that $V\subset \Si_{d}(u)$, and since $\Si_{d}(u)$ is a vector space, also $V\oplus y\subset \Si_{d}(u)$. Now, by Theorem \ref{thm_sisplit}, $u$ is $(\operatorname{dim}(\Si_{d}(u)),0)$-symmetric.
\end{proof}

As it is easily seen, the assumption $d\geq 2$ is essential in the statement of the previous theorem. However, it is possible to replace it with a more convenient assumption on the symmetries of the function $u$. Indeed, since we are dealing with harmonic polynomials, such a function has degree $1$ if and only if it has $n-1$-symmetries, i.e., it depends on only $1$ variable.
\begin{lemma}\label{lemma_d1kn-1}
 Let $u$ be a harmonic function. Then the following are equivalent:
\begin{enumerate}
 \item $u$ is a polynomial of degree $1$,
 \item $u$ is $(n-1,x)$-symmetric with respect to some/all $x\in \R^n$,
 \item $u(x)=u(\Pi_V (x))$ for some $1$-dimensional subspace $V\subset \R^n$,
 \item for some/every $x\in \R^n$ and some/every $r> 0$, $\bar N^u(x,r)=1$.
\end{enumerate}
\end{lemma}

This remark allow us to restate Theorem \ref{th_prepolyspl} in the following way which will be very useful in the following.
\begin{teo}\label{th_polyspl}
 Let $u$ be $(k,x)$-symmetric with respect to $V$, but NOT $(n-1,x)$- symmetric and suppose that $u$ is also $(0,y)$-symmetric with $y\not \in x+V$. Then $u$ is $(k+1,x)$ symmetric with respect to $V\oplus(y-x)=\operatorname{span}(y-x,V)$.
\end{teo}

\subsection{Almost symmetry}\label{sec_almsym}
\index{symmetry!almost symmetry}
The concepts of symmetry just defined are obviously too restrictive on the function $u$. Here we try to restate the previous results turning equalities into almost equalities, and symmetries into almost symmetries, in order to make this concept applicable in some sense to every harmonic function, not just homogeneous polynomials.

Moreover, we exploit the results related to the frequency function to measure in a quantitative way how much a harmonic function looks almost symmetric at a certain scale.
\begin{definition}\label{deph_Tharm}
\index{TUR@$T^u_{0,r}$}
 For $r>0$, given a nonconstant harmonic function $u:B_{r}(x)\to \R$ define a harmonic function $T_{x,r}^u:B_1(0)\to \R$ by
\begin{gather}
 T_{x,r}^u(y) = \frac{u(x+ry)-u(x)}{\ton{\fint_{\partial B_1(0)} [u(x+ry)-u(x)]^2 dS(y) }^{1/2}}\, .
\end{gather}
We will omit the superscript $u$ and the subscript $x$ when there is no risk of confusion.
\end{definition}

As described in the definition of leading polynomial (definition \ref{deph_leading}), the limit of this family of functions as $r$ goes to zero is the normalized leading polynomial of $u-u(x)$. We will set
\begin{gather}
 T_{x,0}^u(y)=T_x^u(y)=\lim_{r\to 0} T_{x,r}^u(y)\, .
\end{gather}
By elliptic estimates, this limit can be taken in the $W^{1,2}(B_1)$ topology, or equivalently in the $L^2(\partial B_1)$ topology or in any $C^m$ uniform topology on compact subsets of $B_1$.

An immediate and useful observation is that (nonzero) homogeneous polynomials are fixed points for $T$, in the sense that
\begin{gather}
 T^P_{0,r}(y) =\frac{P(y)}{\ton{\fint_{\partial B_1}P(y)^2dS}^{1/2}}
\end{gather}
for all $P$ and $r$.

In the previous subsection we gave the definition of homogeneity and symmetry. Since for all harmonic functions $u$ and for all points $x\in B_1(0)$, $T^{u}_{x,r}\to T^{u}_{x,0}$, in some sense all harmonic functions are close to be homogeneous after a suitable blowup. This suggests the following quantitative definitions of symmetry.
\begin{definition}\label{deph_hom2}
 For $r>0$, given a nonconstant harmonic function $u:B_{r}(x)\to \R$, we say that $u$ is $(\epsilon,r,k,x)$-symmetric if there exists a homogeneous harmonic polynomial $P$ symmetric with respect to some $k$-dimensional subspace such that
\begin{gather}
\notag \fint_{\partial B_1(0)} P(y)^2 dS=1\, ,\\
 \fint_{\partial B_1(0)} \abs{T_{x,r}(y)-P(y)}^2 dS<\epsilon \label{eq_asdf}\, .
\end{gather}
\end{definition}
\begin{remark}\rm
 Note that, since $T_{x,r}-P$ is harmonic, by monotonicity of $\bar N$ we have that the $L^2(B_1)$-norm is controlled by the $L^2(\partial B_1)$ norm, and using elliptic estimates \eqref{eq_asdf} gives a control on the $C^m$ norms on compact subsets of $B_1$ for all $m$.
\end{remark}
\begin{remark}\rm
 It is important to recall the relation given by Lemma \ref{lemma_Ninv} between the frequency of $u$ and the frequency of $T^u_{x,r}$. Indeed, we have
\begin{gather}
 \bar N^u(x,r)=\bar N ^{T^u_{x,r}}(0,1)\, .
\end{gather}

\end{remark}

Note that if a harmonic function $u$ is homogeneous in some ball $B_r$, then it is homogeneous on the whole $\R^n$. This suggests that a similar property should also hold for ``almost homogeneity'' in the following sense.
\begin{prop}\label{prop_scalehom}
 For every $\eta>0$ and $0< r\leq 1$, there exists $\epsilon(\eta,\Lambda,r)$ such that if $u$ is a nonconstant harmonic function in $B_1(0)$ with $\bar N(0,1)\leq \Lambda$ and which is $(\epsilon,r,k,0)$-symmetric, then $u$ is also $(\eta,1,k,0)$-symmetric.
\end{prop}
\begin{proof}
 Although it is possible to prove this proposition in a more direct way, we sketch a short proof based on a simple compactness argument that will be used over and over in the following. Without loss of generality, we assume $u(0)=0$.

Consider by contradiction a sequence of nonconstant harmonic functions $u_i:B_1(0)\to \R$ which are $(i^{-1},r,k,0)$-symmetric but such that for every $k$-symmetric harmonic polynomial $P$ normalized on $\partial B_1$
\begin{gather}\label{eq_contr}
 \fint_{\partial B_1} \abs{T_{0,1}^i-P}^2 dS\geq \eta\, .
\end{gather}

Suppose without loss of generality that $\fint_{\partial B_1} u_i^2=1$. Since $N^{u_i}(0,1)\leq \Lambda$, we can extract a subsequence which converge weakly in $W^{1,2}(B_1)$ to some harmonic $u$. At the same time, consider some approximating homogeneous harmonic polynomials $P_i$ of degree $d_i\leq \Lambda$ and $k$-symmetric such that
\begin{gather}
 \fint_{\partial B_1} \abs{P_i}^2 dS=1\quad \text{and} \quad  \fint_{\partial B_1} \abs{T_{0,r}^i-P_i}^2dS \leq i^{-1}\, .
\end{gather}
Passing to a subsequence if necessary, $P_i$ converge in the smooth topology of $\R^n$ to a homogeneous polynomial $P$ of degree $d\leq \Lambda$ and $k$-symmetric. By uniqueness of the limit,
\begin{gather*}
\lim_{i\to \infty} T^i_{0,r}=T^u_{0,r}=P \, , 
\end{gather*}
so that $u=\alpha P$ on $B_r(0)$, where $\alpha=\lim_{i\to \infty} \ton{\fint_{\partial B_r} u_i^2dS}^{1/2}$. Using the unique continuation property for $u$, or the doubling conditions in Lemma \ref{lemma_2}, we have that $\alpha\neq 0$, so that $u$ is nonzero and proportional to a homogeneous harmonic polynomial with $k$ symmetries on the whole of $B_1(0)$. Since $u_i$ converges to $u$, this contradicts \eqref{eq_contr}.
\end{proof}

Here we present two theorems that describe a quantitative link between the frequency $\bar N$ of $u$ and the quantitative homogeneity. The following, which is inspired by \cite[Theorem 3.3]{ChNa2}, is a quantitative version of definition \ref{deph_hom}.
\begin{teo}\label{th_N2hom}
 Fix $\eta>0$ and $0\leq \gamma <1$. There exists $\epsilon=\epsilon(n,\Lambda,\eta,\gamma)$ such that if $u$ is a nonconstant harmonic function in $B_1\subset \R^n$, with frequency function bounded by $\bar N(0,1)\leq \Lambda$ and with
\begin{gather}
\bar N(0,1)-\bar N(0,\gamma)< \epsilon\, ,
\end{gather}
then $u$ is $(\eta,1,0,x)$-symmetric.
\end{teo}
\begin{proof}
The proof of this theorem follows easily from a contradiction/compactness argument similar to the one in the previous proposition. 
\end{proof}

\begin{remark}\rm
Note that, using the invariance of $\bar N$ under blow-ups, a similar statement is valid at any scale $r>0$. In particular, if $\bar N(0,r)\leq \Lambda$ and $\bar N(0,r)-\bar N(0,\gamma r)< \epsilon$, then $u$ is $(\eta,1,0,x)$-symmetric.
\end{remark}

The following cone-splitting theorem is in some sense a quantitative version of Theorem \ref{th_polyspl}. It describes how nearby almost-symmetries interact with each other to create higher order symmetries.

\begin{teo}[Cone splitting Theorem]\label{th_qspl}
\index{cone splitting theorem}
Let $k\leq n-2$. For every $\epsilon,\tau>0$ and $r>1$, there exists a $\delta(\epsilon,\tau,\Lambda,n,r)$ such that if $u$ is a (nonconstant) harmonic function on $B_{r}(0)$ with $\bar N(0,r)\leq \Lambda$ and
\begin{enumerate}
 \item $u$ is $(\delta,1,k,0)$-symmetric with respect to some $k$-dimensional $V$,
 \item there exists $y\in B_{1}(0)\setminus \T_{\tau}(V)$ such that $u$ is $(\delta,1,0,y)$-symmetric.
\end{enumerate}
Then $u$ is $(\epsilon,r,k+1,0)$-symmetric.
\end{teo}
\begin{proof}
We only sketch this proof, since it is based on the usual contradiction/compactness argument. 

Let $u_i$ converge to $u$ weakly in $W^{1,2}(B_r(0))$, and let $y_i \to y$. Hypothesis (1) is used to prove that $u=P$ for some nonzero homogeneous harmonic polynomial $k$-symmetric with respect to some $V$, and by hypothesis (2) $u$ is also a nonzero harmonic polynomial homogeneous with respect to $y\not \in V$.

If $P$ is of degree $1$, then it already has $n-1$ symmetries. Otherwise, by Theorem \ref{thm_sisplit}, we have that the polynomial $P$ is invariant with respect to the subspace $\operatorname{span}(V,y)$, which is $k+1$ dimensional. Since $u_i$ converges to $P$ with respect to the weak topology of $B_{r} (0)$, this completes the proof.
\end{proof}

An almost immediate corollary of this theorem is the following.
\begin{corollary}\label{cor_symsum}
 Fix $\eta,\tau>0$, $k\leq n-2$ and $r\geq 1$. There exists a $\epsilon(\eta,\tau,r,\Lambda,n)$ such that if $u$ is a (nonconstant) harmonic function on $B_{r}(0)$ with $\bar N(0,r)\leq \Lambda$ and
\begin{enumerate}
 \item $u$ is $(\epsilon,1,0,0)$-symmetric,
 \item for every subspace $V$ of dimension $\leq k$, there exists $y\in B_{1}(0)\setminus \T_{\tau}V$ such that $u$ is $(\epsilon,1,0,y)$-symmetric.
\end{enumerate}
then $u$ is $(\eta,r,k+1,0)$-symmetric.
\end{corollary}
\begin{proof}
Suppose by induction that $u$ is $(\epsilon_{(m)},R,m,0)$-symmetric for some $m\leq k$ with respect to a $m$-dimensional subspace $V_{m}$ \footnote{for $m=0$, this is true by Proposition \ref{prop_scalehom}}. By hypothesis, there exists a point $y\in B_{1}(0)\setminus T_\tau V_m$ such that $u$ is $(\epsilon,1,0,y)$-symmetric. Then the previous Theorem guarantees that $u$ is also $(\epsilon_{(m+1)},R,m+1, 0)$-symmetric, and that it is possible to find a sequence of $\epsilon_{(m)}$ such that $\epsilon_{k+1}=\eta$ and $\epsilon_{(0)}=\epsilon>0$.

\end{proof}

The following Proposition gives us a link between the quantitative stratification and the critical set. As seen in the previous subsection, $n-1$-symmetric harmonic polynomial are linear functions, and so they do not have critical points. Here we prove that if $u$ is close enough to an $n-1$ symmetric polynomial at a certain scale, then also in this case $u$ does not have critical points at that scale.

\begin{prop}\label{prop_qcrit}
 There exists $\epsilon(\Lambda,n)>0$ such that if $u$ is $(\epsilon,r,n-1,x)$-symmetric with $\bar N(x,r)\leq \Lambda$, then $u$ does not have critical points in $B_{r/2}(x)$.
\end{prop}
\begin{proof}
 Let $u_i$ be $(i^{-1},r,n-1,x)$-symmetric. Then $T^{u_i}_{x,r}$ is a sequence of harmonic functions converging in $W^{1,2}(B_1(0))$ to a linear function $L$ normalized with $\fint_{\partial B_1} \abs L ^2 dS =1$. By elliptic estimates, the convergence is also in $C^1(B_{1/2}(0))$, and this proves the claim.
\end{proof}

\subsection{Standard and Quantitative stratification of the critical set}\label{sec_qs}
\index{stratification!standard stratification}\index{stratification!quantitative stratification}
In this subsection we define the standard and quantitative stratification for the critical set of a harmonic function $u$.

The standard stratification is based on the symmetries of the leading polynomial $T^{u}_{x,0}$ at each point $x$. In particular, we define the standard strata $\cS^k$ by
\begin{gather}
 \cS^k(u)=\{x\in B_1(0) \ \ s.t. \ \ T^{u}_{x,0} \ \ \text{is not } k+1 \ \text{symmetric}\}\, .
\end{gather}
It is evident that $\cS^k(u)\subset \cS^{k+1}(u)$. Moreover, since the only $n-1$-symmetric harmonic functions are linear polynomials, it is easy to see that $\cS^{n-2}(u)=\Cr(u)$. Note also that, if $u$ is not constant, then the leading polynomial of $u-u(x)$ cannot be zero, or equivalently $n$-symmetric. Thus $\cS^{n-1}(u)$ covers the whole domain of $u$.

Just by using the definition, it is possible (although non completely trivial) to show that $\cS^k(u)$ is locally contained in a $k$-dimensional graph, and so has locally finite $k$-dimensional Hausdorff measure \footnote{for the details, see for example \cite[Theorem 4.1.3]{hanlin}}. 

In order to make such a statement effective, we introduce a new quantitative stratification on the critical set, which is based on how close the function $u$ is to a $k$-symmetric homogeneous harmonic polynomial \textit{at a certain scale}.
\begin{definition}\label{deph_effstra}
\index{skn@$\cS^k_{\eta,r}$}
 Let $u$ be a harmonic function in $B_1(0)$. The $(k,\eta,r)$ effective singular stratum is defined as
 \rm{\begin{gather}
 \cS^k_{\eta,r}(u)=\cur{x\in B_1(0) \ \text{ s.t. } \forall r\leq s \leq 1, u \text{ is not } (\eta,s,k,x)-\text{symmetric }  }\, .
\end{gather}}
\end{definition}

\begin{remark}\rm
 Given the definition, it is easy to prove the following inclusions.
\begin{align}
\cS^k_{\eta,r}\subseteq \cS^{k'}_{\eta',r'} \text{ if } (k'\leq k, \eta'\leq\eta, r\leq r')\, .
\end{align}
Moreover we can recover the standard stratification by
\begin{align}
\cS^k = \bigcup_{\eta}\bigcap_{r} \cS^k_{\eta,r}\, .
\end{align}
\end{remark}
If $u$ is harmonic, there is a very strong quantitative relation between the effective stratum and the critical set that follows from Proposition \ref{prop_qcrit}. This relation is the key that will allow us to turn an effective estimate on $\cS^{n-2}_{\eta,r}$ into an estimate on $\T_r(\Cr(u))$.
\begin{prop}\label{prop_ereg}
Given a harmonic function $u:B_1\to \R$ with $\bar N_u(0,1)\leq \Lambda$, there exists $\eta=\eta(\Lambda)>0$ such that
\begin{gather}
 \T_{r/2}\Cr(u)=\bigcup_{x\in \Cr(u)} B_{r/2}(x)\subset \cS^{n-2}_{\eta,r}(u)
\end{gather}
\end{prop}
\begin{proof}
 Let $x\not \in \cS^{n-2}_{\eta,r}(u)$. Then, by definition, $u$ is $(\eta,s,n-1,0)$-symmetric for some $s\geq r$. Proposition \ref{prop_qcrit} concludes the proof.
\end{proof}

\subsection{Volume estimates on singular strata}
Given the definitions above, we are ready to prove the main volume estimates on the singular strata $\cS^{k}_{\eta,r}$.
\begin{teo}\label{th_main}
 Let $u$ be a harmonic function in $B_1(0)\subset \R^n$ with rescaled frequency $\bar N(0,1)\leq \Lambda$. For every $\eta>0$ and $k\leq n-2$, there exists $C=C(n,\Lambda,\eta)$ such that for all $0<r<1$
\begin{gather}
 \Vol\qua{\T_r \ton{\cS^{k}_{\eta,r}(u)}\cap B_{1/2}(0)}\leq C r^{n-k-\eta}\, .
\end{gather}
\end{teo}

The proof uses a technique similar to the one introduced in \cite{ChNa1,ChNa2}. Instead of proving the statement for any $r>0$, we fix a convenient $0<\gamma<1$ (depending on $n$ and $\eta$) and restrict ourselves to the case $r=\gamma^j$ for any $j\in \N$. It is evident that the general statement follows. For the reader's convenience we restate the theorem under this convention.

\begin{teo}\label{th_main_proof}
Let $u$ be a harmonic function defined on $B_1(0)\subset \R^n$ with $\bar N^u(0,1)\leq \Lambda$. Then for every $j\in \N$, $\eta>0$ and $k\leq n-2$, there exists $0<\gamma(n,\eta)<1$ such that
\begin{align}
{\rm Vol}\ton{\T_{\gamma^j}\ton{\cS^k_{\eta,\gamma^j}(u)}\cap B_{1/2}(0)}\leq C(n,\Lambda,\eta) \ton{\gamma^j}^{n-k-\eta}\, .
\end{align}
\end{teo}

The scheme of the proof is the following: for some convenient $0<\gamma<1$ we prove that there exists a covering of $\cS^{k}_{\eta,\gamma^j}$ made of nonempty open sets in the collection $\{\cC^k_{\eta,\gamma^j}\}$. Each set $\cC^k_{\eta,\gamma^j}$ is the union of a controlled number of balls of radius $\gamma^j$. 

\begin{lemma}[Frequency Decomposition Lemma]\label{lemma_dec}
\index{frequency function!frequency decomposition}
There exist constants $c_0(n),c_1(n)>0$ and $D(n,\eta,\Lambda)>1$ such that for every $j\in \N$:
\begin{enumerate}
 \item $\cS^k_{\eta,\gamma^j}\cap B_{1/2}(0)$ is contained in the union of at most $j^D$ \textit{nonempty} open sets $C^k_{\eta,\gamma^j}$,
 \item each $C^k_{\eta,\gamma^j}$ is the union of at most $(c_1\gamma ^{-n})^D (c_0\gamma^{-k})^{j-D}$ balls of radius $\gamma^j$.
\end{enumerate}
\end{lemma}
Once this Lemma is proved, Theorem \ref{th_main_proof} easily follows.
\begin{proof}[Proof of Theorem \ref{th_main_proof}]
 Let $\gamma=c_0^{-2/\eta}<1$. Since we have a covering of $\cS^k_{\eta,\gamma^j}\cap B_{1/2}(0)$ by balls of radius $\gamma^j$, it is easy to get a covering of $\T_{\gamma^j}\ton{\cS^k_{\eta,\gamma^j}}\cap B_{1}(0)$, in fact it is sufficient to double the radius of the original balls. Now it is evident that
\begin{gather}
 \Vol\qua{\T_{\gamma^j}\ton{\cS^k_{\eta,\gamma^j}}\cap B_{1/2}(0)} \leq j^D \ton{(c_1\gamma^{-n})^D (c_0\gamma^{-k})^{j-D}} \omega_n 2^n \ton{\gamma^j}^n
\end{gather}
where $\omega_n$ is the volume of the $n$-dimensional unit ball. By plugging in the simple rough estimates
\begin{gather}
 j^D \leq c(n,\Lambda,\eta)\ton{\gamma^j}^{-\eta/2}\, ,\\
\notag (c_1\gamma^{-n})^D(c_0\gamma^{-k})^{-D}\leq c(n,\Lambda,\eta)\, ,
\end{gather}
and using the definition of $\gamma$, we obtain the desired result.
\end{proof}

\paragraph{Proof of the Frequency Decomposition Lemma}
Now we turn to the proof of the Frequency Decomposition Lemma. In order to do this, we define a new quantity which measures the non-symmetry of $u$ at a certain scale
\begin{definition}
 If $u$ is as in Theorem \ref{th_main_proof}, $x\in B_{1}(0)$ and $0<r<1$, define
\begin{gather}
 \cN(u,x,r) =\inf\{\alpha\geq 0 \ \ s.t. \ \ u \text{ is } (0,\alpha,r,x)\text{-symmetric}\}\, .
\end{gather}
\end{definition}
For fixed $\epsilon>0$, we divide the set $B_{1/2}(0)$ into two subsets according to the behaviour of the points with respect to their quantitative symmetry
\begin{align}
 H_{r,\epsilon}(u)=\{x\in B_{1/2}(0) \ s.t. \ \cN(u,x,r)\geq \epsilon\}\, , \\
\notag L_{r,\epsilon}(u)=\{x\in B_{1/2}(0) \ s.t. \ \cN(u,x,r)< \epsilon\}\, .
\end{align}
Next, to each point $x\in B_{1/2}(0)$ we associate a $j$-tuple $T^j(x)$ of numbers $\{0,1\}$ in such a way that the $i$-th entry of $T^j$ is $1$ if $x\in H_{\gamma^i, \epsilon}(u)$, and zero otherwise. Then, for each fixed $j$-tuple $\bar T^j$, we set
\begin{gather}
 E(\bar T^j) = \{x\in B_{1/2}(0) \ \ s.t. \ \ T^j(x)=\bar T^j\}\, .
\end{gather}
Also, we denote by $T^{ j-1}$, the $(j - 1)$-tuple obtained from $T^ j$ by dropping the last entry, and set $\abs{T^j}$ to be number of $1$ in the $j$-tuple $T^j$.

We will build the families $\{C^k_{\eta,\gamma^j}\}$ by induction on $j$ in the following way.
For $a=0$, $\{C^k_{\eta,\gamma^0}\}$ consists of the single ball $B_{1}(0)$.
\paragraph{Induction step}
For fixed $a\leq j$, consider all the $2^a$ $a$-tuples $\bar T^a$. Label the sets in the family $\{C^k_{\eta,\gamma^a}\}$ by all the possible $\bar T^a$. We will build $C^k_{\eta,\gamma^a}(\bar T^a)$ inductively as follows. For each ball $B_{\gamma^{a-1}}(y)$ in $\{C^k_{\eta,\gamma^{a-1}}(\bar T^{a-1})\}$ take a minimal covering of $B_{\gamma^{a-1}}(y)\cap \cS^{k}_{\eta,\gamma^j} \cap E(\bar T^a)$ by balls of radius $\gamma^a$ centered at points in $B_{\gamma^{a-1}}(x)\cap \cS^k_{\eta,\gamma^j}\cap E(\bar T^a)$. Note that it is possible that for some $a$-tuple $\bar T^a$, the set $E(\bar T^a)$ is empty, and in this case $\{C^k_{\eta,\gamma^{a}}(\bar T^{a})\}$ is the empty set.

Now we need to prove that the minimal covering satisfies points 1 and 2 in the Frequency Decomposition Lemma \ref{lemma_dec}.
\begin{remark}
\rm The value of $\epsilon>0$ will be chosen according to Lemma \ref{lemma_cov}. For the moment,
we take it to be an arbitrary fixed small quantity.
\end{remark}

\paragraph{Point 1 in Lemma}
As we will see below, we can use the monotonicity of $\bar N$ to prove that for every $\bar T^j$, $E(\bar T^j)$ is empty if $\abs{\bar T^j}\geq D$. Since for every $j$ there are at most $\binom j D\leq j^D$ choices of $j$-tuples with such a property, the first point will be proved.

\begin{lemma}\label{lemma_K}
 There exists  $D=D(\epsilon,\gamma,\Lambda,n)$ \footnote{in what  follows, we will fix $\epsilon$ as a function of $\eta,\Lambda,n$. Thus, $D$ will actually depend only on these three variables.} such that $E(\bar T^j)$ is empty if $\abs{\bar T^j}\geq D$. 
\end{lemma}

\begin{proof}
 Recall that $\bar N(x,r)$ is monotone  nondecreasing with respect to $r$, and, by Lemma \ref{lemma_strong2bar}, $\bar N(x,1/3)$ is bounded above by a function $C(n,\Lambda)$. For $s<r$,
we set
\begin{gather}
 \cW_{s,r}(x)=\bar N(x,r)-\bar N(x,s)\geq 0\, .
\end{gather}
If $(s_i,r_i)$ are \textit{disjoint} intervals with $\max\{r_i\}\leq 1/3$, then by monotonicity of $\bar N$
\begin{gather}\label{eq_sum}
 \sum_i \cW_{s_i,r_i}(x)\leq \bar N (x,1/3) -\bar N (x,0) \leq C(n,\Lambda) -1\, .
\end{gather}

Let $\bar i$ be such that $\gamma^{\bar i} \leq 1/3$, and consider intervals of the form $(\gamma^{i+1},  \gamma^{i})$ for $i=\bar i,\bar i+1,...\infty$. By Theorem \ref{th_N2hom} and Lemma \ref{lemma_strong2bar}, there exists a $0<\delta=\delta(\epsilon,\gamma,\Lambda,n)$ independent of $x$ such that
\begin{gather}
 \cW_{\gamma^{i+1},\gamma^{i}}(x)\leq \delta \implies u \text{ is }(0,\epsilon,\gamma^{i},x)\text{-symmetric}\, .
\end{gather}
in particular $x\in L_{ \gamma^{i},\epsilon}$, so that, if $i\leq j$, the $i$-th entry of $T^j$ is necessarily zero. By equation \eqref{eq_sum}, there can be only a finite number of $i$'s such that $\cW_{\gamma^{i+1}, \gamma^i}(x)>\delta$, and this number $D$ is bounded by
\begin{gather}\label{eq_estK}
 D\leq \frac{C(n,\Lambda)-1}{\delta(\epsilon,\gamma,\Lambda,n)} + \log_{\gamma}(1/3)\, .
\end{gather}
This completes the proof.
\end{proof}

\paragraph{Point 2 in Lemma}
The proof of the second point in Lemma \ref{lemma_dec} is mainly based on Corollary \ref{cor_symsum}. In particular, for fixed $k$ and $\eta$ in the definition of $\cS^k_{\eta,\gamma^j}$, choose $\epsilon$ in such a way that Corollary \ref{cor_symsum} can be applied with $r=\gamma^{-1}$ and $\tau = 7 ^{-1}$. Note that such an $\epsilon$ depends only on $n,\eta$ and $\Lambda$. Then we can restate point 2 in the Frequency Decomposition Lemma as follows.
\begin{lemma}\label{lemma_cov}
 Let $\bar T^j_a =0$. Then the set $A=\cS^{k}_{\eta,\gamma^j}\cap B_{\gamma^{a-1}}(x)\cap E(\bar T^j)$ can be covered by $c_0(n)\gamma^{-k}$ balls centered at $A$ of radius $\gamma^{a}$.
\end{lemma}
\begin{proof}
 First of all, note that since $\bar T^j_a =0$, all the points in $E(\bar T^j)$ are in $L_{\epsilon, \gamma^a}(u)$.

The set $A$ is contained in $B_{7^{-1}\gamma^a}(V^k)\cap B_{\gamma^{a-1}}(x)$ for some $k$-dimensional subspace $V^k$. Indeed, if there were a point $x\in A$, such that  $x\not\in B_{7^{-1} \gamma^a}(V^k)\cap B_{\gamma^{a-1}}(x)$, by Corollary \ref{cor_symsum} and Lemma \ref{lemma_strong2bar}, $u$ would be $(k+1,\eta,\gamma^{a-1},x)$-symmetric. This contradicts $x\in \cS^k_{\eta,\gamma^j}$. By standard geometry, $V^k \cap B_{\gamma^{a-1}}(x)$ can be covered by $c_0(n)\gamma^{-k}$ balls of radius $\frac 6 7 \gamma^a$, and by the triangle inequality it is evident that the same balls with radius $\gamma^a$ cover the whole set $A$.
\end{proof}

If instead $\bar T^j_a =1$, then it is easily seen that $A=\cS^{k}_{\eta,\gamma^j}\cap B_{a-1}(x)\cap E(\bar T^j)$ can be covered by $c_0(n)\gamma^{-n}$ balls of radius $\gamma^a$. Now a simple induction argument completes the proof.
\begin{lemma}
 Each (nonempty) $C^k_{\eta,\gamma^j}$ is the union of at most $(c_1\gamma ^{-n})^D (c_0\gamma^{-k})^{j-D}$ balls of radius $\gamma^j$.
\end{lemma}
\begin{proof}
 Fix a sequence $\bar T^j$ and consider the set $C^k_{\eta,\gamma^j}(\bar T^j)$. By Lemma \ref{lemma_K}, we can assume that $\abs {\bar T^j}\leq D$, otherwise $C^k_{\eta,\gamma^j}(\bar T^j)$ would be empty and there would be nothing to prove.

Note that at each step $a$, in order to get a (minimal) covering of $B_{\gamma^{a-1}}(x)\cap \cS^{k}_{\eta,\gamma^i}\cap E(\bar T^j) $ for $B_{\gamma^{a-1}}(x)\in C^k_{\eta,\gamma^{a-1}}(\bar T^j)$, we require at most $(c_0 \gamma^{-k})$ balls of radius $\gamma^{a}$ if $\bar T^j_a=0$, or $(c_0\gamma^{n})$ otherwise. Since the latter situation can occur at most $D$ times, the proof is complete.
\end{proof}

\subsection{Volume Estimates on the Critical Set}
Apart from the volume estimate on $\cS^k_{\eta,r}$, Theorem \ref{th_main} has a useful corollary for measuring the size of the critical set. Indeed, by Proposition \ref{prop_qcrit}, the critical set of $u$ is contained in $\cS^{n-2}_{\epsilon,r}$, thus we have proved Theorem \ref{t:crit_lip} for harmonic functions.
\begin{corollary}\label{cor_main}
 Let $u:B_1(0)\to \R$ be a harmonic function with $\bar N^u(0,1)\leq \Lambda$. Then, for every $\eta>0$, we can estimate
\begin{align}\label{eq_aaa}
{\rm Vol}(\T_r(\Cr(u))\cap B_{1/2}(0))\leq C(n,\Lambda,\eta)r^{2-\eta}\, .
\end{align}
\end{corollary}
\begin{proof}
By Proposition \ref{prop_qcrit}, for $\eta>0$ small enough we have the inclusion
\begin{gather}
 \T_{r/2}(\Cr(u)) \subset \cS^{n-2}_{\eta,r}\, .
\end{gather}
Using Theorem \ref{th_main}, we obtain the desired volume estimate for $\eta$ sufficiently small. However, since evidently
\begin{gather}
 \Vol(\T_r(\Cr(u))\cap B_{1/2}(0))\leq \Vol(B_{1/2}(0))\, ,
\end{gather}
it is easy to see that if \eqref{eq_aaa} holds for some $\eta$, then a similar estimate holds also for any $\eta'\geq \eta$.
\end{proof}

\begin{remark}
\rm As already mentioned in the introduction, this volume estimate on the critical set and its tubular neighborhoods immediately implies that $\dim_{Mink}(\Cr(u))\leq n-2$. This result is clearly optimal \footnote{consider for example the harmonic function $x_1^2-x_2^2$ defined in $\R^n$ with $n\geq 2$.}.
\end{remark}

\subsection{n-2 Hausdorff uniform bound for the critical set}

By combining the results of the previous sections with an $\epsilon$-regularity theorem from \cite{hanhardtlin}, in this subsection we prove an effective uniform bound on the $(n-2)$-dimensional
Hausdorff measure of $\Cr(u)$. The bound will not depend on $u$ itself, but only on the normalized frequency $\bar N^u (0,1)$. Specifically, the proof will be obtained by combining the $n-3+\eta$ Minkowski estimates available for $\cS^{n-3}_{\eta,r}$ with the following $\epsilon$-regularity lemma.
\index{epsilon reg@$\epsilon$-regularity theorem}
The lemma proves that if a harmonic function $u$ is sufficiently close to a homogeneous harmonic polynomial of only $2$ variables, then there is an effective upper bound on the $(n-2)$-dimensional Hausdorff measure of the whole critical set of $u$.

\begin{lemma}\label{lemma_n-2Ha}\cite[Lemma 3.2]{hanhardtlin}
Let $P$ be a homogeneous harmonic polynomial with exactly $n-2$ symmetries in $\R^n$. Then there exist positive constants $\epsilon$ and $\bar r$ depending on $P$, such that for any $u \in C^{2d^2}(B_1(0))$, if
\begin{gather}
 \norm{u-P}_{C^{2d^2}(B_1)}<\epsilon\, ,
\end{gather}
then for all $r\leq \bar r$
\begin{gather}
 H^{n-2}(\nabla u^{-1}(0)\cap B_r(0))\leq c(n)(d-1)^2 r^{n-2}\, .
\end{gather}
\end{lemma}

It is not difficult to realize that, if we assume $u$ harmonic in $B_1$ with $\bar N^u(0,1)\leq \Lambda$, then $\epsilon$ and $\bar r$ can be chosen to be independent of $P$, but dependent only on $\Lambda$. Indeed, up to rotations and rescaling, all polynomials with $n-2$ symmetries in $\R^n$ of degree $d$ look like $P(r,\theta,z)=r^d \cos(d\theta)$, where we used cylindrical coordinates on $\R^n$. Combining this with elliptic estimates yields the following corollary.
\begin{corollary}\label{cor_ereg}
Let $u:B_1\to \R$ be a harmonic function with $\bar N(0,1)\leq \Lambda$. Then there exist positive constants $\epsilon(\Lambda,n)$ and $\bar r(\Lambda,n)$ such that if there exists a normalized homogeneous harmonic polynomial $P$ with $n-2$ symmetries with
\begin{gather}
 \norm{T^u_{0,1}-P}_{L^2(\partial B_1)}<\epsilon\, \quad \text{and} \quad \, \fint_{\partial B_1(0)} P^2 = 1\, ,
\end{gather}
then for all $r\leq \bar r$
\begin{gather}
 H^{n-2}(\nabla u^{-1}(0)\cap B_r(0))\leq c(\Lambda,n)r^{n-2}\, .
\end{gather}
\end{corollary}

To prove the effective bound on the $(n-2)$-dimensional  Hausdorff measure, we combine the Minkowski estimates of Theorem \ref{th_main} with the above corollary. Using the quantitative stratification, we will use an inductive construction to split the critical set at different scales into a good part, the points where the function is close to an $(n-2)$-symmetric polynomial, and a bad part, whose tubular neighborhoods have effective bounds. Since we have estimates on the whole critical set in the good part, we do not have to worry any longer when we pass to a smaller scale. As for the bad part, by induction, we start  the process over and split it again into a good and a bad part. By summing the various contributions to the 
$(n-2)$-dimensional Hausdorff measure given by the good parts, we prove the following theorem \footnote{as mentioned in the introduction, this result has already been obtained in \cite{HLrank}. We provide a slightly different proof}.
\begin{teo}\label{th_n-2h}
 Let $u$ be a harmonic function in $B_1(0)$ with $\bar N(0,1)\leq \Lambda$. There exists a constant $C(\Lambda,n)$ such that
\begin{gather}
 H^{n-2}(\Cr(u)\cap B_{1/2}(0))\leq C(n,\Lambda)\, .
\end{gather}
\end{teo}
\begin{proof}
Note that by Lemma \ref{lemma_strong2bar}, for every $r\leq 1/3$ and $x\in B_{1/2}(0)$, the functions $T_{x,r}u$ have frequency uniformly bounded by $N^{T_{x,r}u}(0,1)\leq C(\Lambda,n)$. This will allow us to apply Corollary \ref{cor_ereg} to each $T_{x,r}u$ to obtain uniform constants $\epsilon(\Lambda,n)$ and $\bar r(\Lambda,n)$ such that the conclusion of the Corollary holds for all such $x$ and $r\leq \bar r$.

Now fix $\eta>0$ given by the minimum of $\eta(n,\Lambda)$ from Proposition \ref{prop_qcrit} and $\epsilon(n,\Lambda)$ from Corollary \ref{cor_ereg}. Let $0<\gamma\leq 1/3$ and define the following sets
\begin{gather}
 \Cr^{(0)}(u)=\Cr(u)\cap \ton{\cS^{n-2}_{\eta,1}\setminus \cS^{n-3}_{\eta,1}}\cap B_{1/2}(0)\, ,\\
 \Cr^{(j)}(u)=\Cr(u)\cap \ton{\cS^{n-2}_{\eta,\gamma^j}\setminus \cS^{n-3}_{\eta,\gamma^j}} \cap \cS^{n-3}_{\eta,\gamma^{j-1}}\cap B_{1/2}(0)\, .
\end{gather}
We split the critical set in the following parts
\begin{gather}
 \Cr(u)\cap B_{1/2}(0)=\bigcup_{j=0}^\infty \Cr^{(j)}(u) \bigcup \ton{\Cr(u) \bigcap_{j=1}^\infty\,  \cS^{n-3}_{\eta,\gamma^j}}\, .
\end{gather}
It is evident from Theorem \ref{th_main_proof} that
\begin{gather}
 H^{n-2}\ton{\Cr(u) \bigcap_{j=1}^\infty \cS^{n-3}_{\eta,\gamma^j}\cap B_{1/2}(0)}=0\, .
\end{gather}
As for the other set, we will prove by induction that
\begin{gather}
 H^{n-2} \ton{\bigcup_{j=0}^k \Cr^{(j)}(u)}\leq C(\Lambda,n,\eta)\sum_{j=0}^{k} \gamma^{(1-\eta)j}\, .
\end{gather}
Using Corollary \ref{cor_ereg} and a simple covering argument, it is easy to see that this statement is valid for $k=0$.

Choose a covering of the set $\Cr^{(k)}(u)$ by balls centered at $x_i \in \Cr^{(k)}(u)$ of radius $\gamma^k \bar r$, such that the same balls with half the radius are disjoint. Let $m(k)$ be the number of such balls. By the volume estimates in Theorem \ref{th_main}, we have
\begin{gather}
 m(k)\leq C(\eta,\Lambda,n)\gamma^{(3-\eta-n)k}\, .
\end{gather}
By construction of the set $\Cr^{(k)}(u)$, for each $x_i$ there exists a scale $s\in [\gamma^k,\gamma^{k-1}]$ such that for some normalized homogeneous harmonic polynomial of two variables $P$, we have
\begin{gather}
 \norm{T_{x_i,s}u-P}_{L^2(\partial B_1)}<\eta\, .
\end{gather}

Using Corollary \ref{cor_ereg} we can deduce that
\begin{gather}
 H^{n-2}\ton{\nabla u^{-1}(0)\cap B_{\gamma^k \bar r}(x_i)}\leq C(\Lambda,n)\gamma^{(n-2)k}\, ,
\end{gather}
and so
\begin{gather}
 H^{n-2}\ton{\Cr^{(k)}(u)}\leq C(\Lambda,n,\eta)\gamma^{(1-\eta)k}\, .
\end{gather}
Since $\eta<1$ and $0<\gamma<1/3$, the proof is complete.

\end{proof}

\section{Elliptic equations}\label{sec_ell}
With the necessary modifications, the results proved for harmonic functions are valid also for solutions to elliptic equations of the form \eqref{eq_Lu} with the set of assumptions \eqref{e:coefficient_estimates}. Indeed, a volume estimate of the form given in Theorem \ref{th_main} and Corollary \ref{cor_main} remains valid without any further regularity assumption on the coefficients $a^{ij}$ and $b^i$. While in order to get an effective bound on the $n-2$ Hausdorff measure of the critical set, we will assume some additional control on the higher order derivatives of the coefficients of the PDE. However, it is reasonable to conjecture that an $n-2$-Minkowski uniform bound can be proved assuming only the set of conditions \eqref{e:coefficient_estimates}. Note that the set of conditions \eqref{e:coefficient_estimates} is minimal if we want to have effective control on the critical set. Indeed, as noted in the introduction, in \cite{plis} there are counterexamples to the unique continuation principle for 
solutions of elliptic equations similar to \eqref{eq_Lu} where the coefficients $a^{ij}$ are H\"older continuous with any exponent strictly smaller than $1$. No reasonable estimates for $\Cr(u)$ are possible in such a situation.

As mentioned in the introduction, the basic ingredients and ideas needed to estimate the critical sets of solutions to elliptic equations are exactly the same as in the harmonic case, although there are some nontrivial technical issues to be addressed. For example, it is not completely straightforward to define the right frequency function for general elliptic equations, issue that is addressed in the next subsection.

\subsection{Generalized frequency function}\label{sec_freqell}
\index{frequency function!generalized frequency function}
In order to define and study a generalized frequency function for solutions to equation \eqref{eq_Lu}, we introduce a new metric related to the coefficients $a^{ij}$. For the sake of simplicity, we will occasionally use the terms and notations typical of Riemannian manifolds. For example we will denote by $B(g,x,r)$ the geodesic ball centered at $x$ with radius $r$ with respect to the metric $g$. 

It would seem natural to define a metric $g_{ij}=a_{ij}$ and exploit the weak version of the divergence theorem to estimate quantities similar to the ones introduced in definition \ref{deph_N}. However, for such a metric the geodesic polar coordinates at a point $x$ are well defined only in a small ball centered at $x$ whose radius is not easily bounded from below (it is related to the radius of injectivity of the metric under consideration). To avoid this problem, we define a similar but slightly different metric which has been introduced in \cite[eq. (2.6)]{toc}, and later used in \cite[Section 3.2]{hanlin}. In the latter paper paper, the authors use the new metric to define a frequency function which turns out to be almost monotone for elliptic equations in divergence form on $\R^n$ with $n\geq 3$, and only bounded at small enough scales for more general equations.

Using a slightly different definition, we will prove almost monotonicity at small scales for solutions of equation \eqref{eq_Lu} without any restriction on $n$ and valid also for equations not in divergence form.

First of all, we cite from \cite{toc} the definition and some properties of the new metric $g_{ij}$.

Following the standard convention, we denote by $a_{ij}$ the elements of the inverse matrix of $a^{ij}$, and by $a$ the determinant of $a_{ij}$. $g_{ij}$ denotes a metric on $B_1(0)\subset \R^n$ and $e_{ij}$ is the standard Euclidean metric.

Fix an origin $\bar x$, and define on the Euclidean ball $B_1(0)$
\begin{gather}
 r^2=r^2(\bar x,x)=a_{ij}(\bar x) (x-\bar x)^i (x-\bar x)^j\, ,
\end{gather}
where $x=x^ie_i$ is the usual decomposition in the canonical base of $\R^n$. Note that the level sets of $r$ are Euclidean ellipsoids centered at $\bar x$, and the assumptions on the coefficients $a_{ij}$ allow us to estimate
\begin{gather}
 \lambda^{-1} \abs {x-\bar x} ^2 \leq r^2(\bar x,x)\leq \lambda \abs {x-\bar x}^2\, .
\end{gather}
The following proposition is proved in \cite{toc}.
\begin{prop}\label{prop_gij}
 With the definitions above, set
\begin{gather}
 \eta(\bar x,x)={a^{kl}(x)\frac{\partial r(\bar x,x)}{\partial x^k}\frac{\partial r(\bar x,x)}{\partial x^l}}={a^{kl}(x)\frac{a_{ks}(\bar x)a_{lt}(\bar x)(x-\bar x)^s(x-\bar x)^t}{r^2}}\, ,\\
g_{ij}(\bar x,x)=\eta(\bar x,x)a_{ij}(x)\, .
\end{gather}
Then for each $\bar x \in B_{1}(0)$, the geodesic distance $d_{\bar x}(\bar x,x)$ in the metric $g_{ij}(\bar x,x)$ is equal to $r(\bar x,x)$. This implies that geodesic polar coordinates with respect to $\bar x$ are well-defined on the Euclidean ball
\begin{gather*}
 B_{\sqrt{\lambda}(1-\abs {\bar x})}(\bar x)=\cur{x\ s.t. \ \abs{x-\bar x}\leq \lambda^{-1/2}(1-\abs {\bar x})}\, .
\end{gather*}
Moreover in such coordinates the metric assumes the form
\begin{gather}\label{eq_bst}
 g_{ij}(\bar x, (r,\theta))=dr^2+ r^2 b_{st}(\bar x, (r,\theta))d\theta^s d\theta^t\, ,
\end{gather}
where $b_{st}(\bar x,r,\theta)$ can be extended to Lipschitz functions in $[0,\lambda^{-1/2}(1-\abs {\bar x})]\times \partial B_1$ with
\begin{gather}
 \abs{\frac{\partial b_{st}}{\partial r}}\leq C(\lambda)\, ,
\end{gather}
and $b_{st}(\bar x,0,\theta)$ is the standard Euclidean metric on $\partial B_1$.
\end{prop}
\begin{remark}\rm
 If $a^{ij}$ are Lipschitz, then so is also the metric $g_{ij}$. However, if the coefficients $a^{ij}$ have higher regularity, for example $C^1$ or $C^m$, it easily seen that $g_{ij}$ is of higher regularity away from the origin, but at the origin $g_{ij}$ in general is only Lipschitz.
\end{remark}

Before defining the generalized frequency formula, it is convenient to rewrite the differential equation \eqref{eq_Lu} in a Riemannian form with respect to the metric $g_{ij}$. The Riemannian (weak) Laplacian of the function $u$ is in coordinates
\begin{gather*}
 \Delta_g (u)=\dive{\nabla_g u}=\frac{1}{\sqrt g}\partial_i\ton{\sqrt g g^{ij}\partial_j u}\, .
\end{gather*}
Since $u$ solves equation \eqref{eq_Lu}, on the standard coordinates in $\R^n$, where by definition $g_{ij}=\eta a_{ij}$, we have
\begin{gather*}
 \Delta_g (u)=\frac{1}{\sqrt g}\partial_i\ton{\sqrt g g^{ij}\partial_j u}= \eta^{-1}\partial_i\ton{a^{ij} \partial_j u} + g^{ij}\partial_j u\ \partial_i\log\ton{g^{1/2} \eta^{-1}}=\\
 =g^{ij}\partial_j u\ \qua{-\eta^{-1} b_i +\partial_i\log\ton{g^{1/2} \eta^{-1}}}\, .
\end{gather*}
Define $B$ to be the vector field which, in the standard Euclidean coordinates, has components
\begin{gather}\label{eq_B}
 B_i = -\eta^{-1} b_i +\partial_i\log\ton{g^{1/2} \eta^{-1}}\, .
\end{gather}
It is evident that the Riemannian length of $B$ is bounded by
\begin{gather}
\abs B^2 _g= g^{ij}B_iB_j\leq C(\lambda,M,L)\, , 
\end{gather}
and $u$ satisfies, in the weak Riemannian sense
\begin{gather}
 \Delta_g u =\ps{B}{\nabla u}_g\, .
\end{gather}

Recall that $B(g,x,r)$ denotes the geodesic ball centered at $x$ with radius $r$ with respect to the metric $g$. Now we are ready to define the generalized frequency function for a (weak) solution $u$ to \eqref{eq_Lu}.
\begin{definition}\label{deph_LN}
Let $u$ be a (weak) solution to equation \eqref{eq_Lu} with conditions \eqref{e:coefficient_estimates}. For each $\bar x\in B_1(0)$ and $r\leq \lambda^{-1/2}(1-\abs {\bar x})$, define
\begin{gather}
\notag D(u,\bar x,g,r)=\int_{B(g(\bar x),\bar x,r)}\norm{\nabla u}_{g(\bar x)}^2dV_{g(\bar x)}=\int_{r(\bar x,x)\leq r} \eta^{-1}(\bar x,x) a^{ij}(x) \partial_i u \partial_j u\sqrt{\eta^n(\bar x,x)a(x)}dx\, ,\\
\notag I(u,\bar x,g,r)=\int_{B(g(\bar x),\bar x,r)}\norm{\nabla u}_{g(\bar x)}^2+ (u-u(\bar x))\Delta_{g(\bar x)} (u )dV_{g(\bar x)}=\\
\notag=\int_{r(\bar x,x)\leq r} \eta^{-1}(\bar x,x) a^{ij}(x) \ton{\partial_i u +\qua{u(x)-u(\bar x)}B_i}\partial_j u\sqrt{\eta^n(\bar x,x)a(x)}dx\, ,\\
\notag H(u,\bar x,g,r)=\int_{\partial B(g(\bar x),\bar x,r)} \qua{u-u(\bar x)}^2 dS_{g(\bar x)}=r^{n-1}\int_{\partial B_1} \qua{u(r,\theta)-u(\bar x)}^2\sqrt{b(\bar x,r,\theta)}d\theta\, ,\\
\bar N(u,\bar x,g,r)=\frac{rI(u,\bar x,g,r)}{H(u,\bar x,g,r)}\, .
\end{gather}
\end{definition}
By the unique continuation and maximum principles, if $u$ is nonconstant then $H(\bar x,r)>0$ for every $r>0$, and so $\bar N$ is well defined. Moreover, by elliptic regularity, $\bar N$ is a locally Lipschitz function.

Note that using the divergence theorem we have
\begin{gather}
 I(u,\bar x,g,r)=\int_{\partial B(g(\bar x),\bar x,r)} (u-u(\bar x)) u_n dS(g)\, ,
\end{gather}
where $u_n=\ps{\nabla u}{\hat n}_g$ is the normal (with respect to $g$) derivative on $\partial B_r^{g}$.

\paragraph{Estimates on $\bar N$} From now on, we will assume that $\bar x =0$ and $u(0)=0$ for the sake of simplicity.

First of all, note that a similar statement to Lemma \ref{lemma_Ninv} \footnote{without invariance with respect to the base point $\bar x$} holds also for this generalized frequency. Indeed we have the following lemma.
\begin{lemma}\label{lemma_Ninvell}
 Let $u$ be a nonconstant solution to \eqref{eq_Lu}, and consider the blow-up given in geodesic polar coordinates (with respect to $g_{ij}$, as defined in Proposition \ref{prop_gij}) by $(r(t),\theta)= (tr,\theta)$. If we define $w(r,\theta)=\alpha u(tr,\theta)+\beta$ and $g^t_{ij}(r,\theta)=g_{ij}(tr,\theta)$, we have
\begin{gather}
 \bar N(u,g,r)=\bar N(w,g^t,t^{-1}r)\, .
\end{gather}
\end{lemma}
By the properties of the metric $g_{ij}$, note that, on $B_1(0)$, as $t$ approaches zero $g_{ij}^t$ approaches in the Lipschitz sense standard Euclidean metric.

Mimicking the definition \ref{deph_Tharm}, we define the functions
\index{TUR@$T^u_{0,r}$}
\begin{gather}\label{eq_Tt}
 T^{u}_{0,t}(r,\theta)\equiv \frac{u(tr,\theta)}{\ton{\fint_{\partial B(g(0),0,t)} u(r,\theta)^2dS(g)}^{1/2} }\, , \quad \quad  T^{u}_{0,t}(0)=0\, .
\end{gather}
When there is no risk of confusion, we will omit the superscript $u$ and the subscript $0$. Note that elliptic regularity ensures that for all $t$, $T_t\in W^{2,2}(B_1(0))\cap C^{1,\alpha}(B_1(0))$ for all $\alpha<1$. Moreover, $T_t$ is normalized in the sense that
\begin{gather}
 \fint_{\partial B(g^t,0,1)} \abs{T_t}^2 dS(g^t)=1\, .
\end{gather}
Using a simple change of variables, it is easy to realize that $T_t$ satisfies (in the weak sense) the equation
\begin{gather}\label{eq_dt}
 \Delta_{g^t} T_t = t \ps{B\, }{\, \nabla u}_{g^t}\, ,
\end{gather}
where $B$ is defined by equation \eqref{eq_B}. The previous lemma has this immediate corollary.
\begin{corollary}\label{cor_Nt}
 Let $u$ be a nonconstant solution to \eqref{eq_Lu}, and $T_t$ as above. Then
\begin{gather}
 \bar N(u,g,r)=\bar N(T_t,g^t,t^{-1}r)\equiv \bar N_t(t^{-1}r)\, .
\end{gather}
\end{corollary}

An essential property of this frequency function is that the Poincaré inequality gives a positive lower bound for it. In particular we have the following proposition.
\begin{prop}\label{pa}
 Let $u$ be as above and $t< \lambda^{-1/2}$. Then there exists a constant $C(n,\lambda)$ such that
\begin{gather}
 \int_{\partial B(g^t,0,1)} \abs{T_t}^2 d\theta(g^t)\leq C \int_{B(g^t,0,1)} \norm{\nabla T_t}_{g^t}^2 dV(g^t)\, .
\end{gather}
\end{prop}
\begin{proof}
Since $T_t(0)=0$, we can apply the Poincaré inequality and get a $C(\lambda,M)$ such that
\begin{gather}
 \int_{B(g^t,0,1)} \abs{T_t}^2 dV(g^t)\leq C  \int_{B(g^t,0,1)} \norm{\nabla T_t}_{g^t}^2 dV(g^t)\, .
\end{gather}
 In the set $B_1(0)$ with geodesic polar coordinates relative to $g^t$, define the Lipschitz vector field given by $\vec v(r,\theta) = r \partial r$. Note that on $\partial B(g^t,0,1)$ this vector coincides with the normal vector to $B(g^t,0,1)$. Using the divergence theorem, we can estimate that
\begin{gather}
  \int_{\partial B(g^t,0,1)} \abs{T_t}^2 d\theta(g^t)=\int_{\partial B(g^t,0,1)} \abs{T_t}^2 \ps{\vec v}{\vec v}d\theta(g^t) =\int_{B(g^t,0,1)} \dive\ton{\abs{T_t}^2 \vec v } dV(g^t)\, .
\end{gather}
Note that, in geodesic polar coordinates, the weak divergence of $\vec v$ is
\begin{gather}
 \dive{\vec v}= \frac{1}{r^{n-1}\sqrt{b(tr,\theta)}}\frac{\partial}{\partial r} \ton{r^{n}\sqrt{b(tr,\theta)}}=n+\frac {tr} 2 \left.\frac{\partial \log(b)}{\partial r}\right\vert_{(tr,\theta)}\, .
\end{gather}
And so we have the estimate
\begin{gather}
\notag \int_{\partial B(g^t,0,1)} \abs{T_t}^2 d\theta(g^t)= \int_{B(g^t,0,1)} 2T_t \ps{\nabla T_t}{\vec v} dV(g^t)+\int_{B(g^t,0,1)} \abs{T_t}^2\dive{\vec v}dV(g^t)\leq\\
\notag\leq 2 \ton{\int_{B(g^t,0,1)} \abs{T_t}^2 dV(g^t)}^{1/2}\ton{\int_{B(g^t,0,1)} \norm{\nabla T_t}^2 dV(g^t)}^{1/2}    + C(n,\lambda)\int_{B(g^t,0,1)} \abs{T_t}^2 dV(g^t)\leq\\
\leq C \int_{B(g^t,0,1)} \norm{\nabla T_t}^2 dV(g^t)\, .
\end{gather}
\end{proof}
For $t$ small enough, we can also bound $D$ with $I$ and vice versa.
\begin{prop}\label{prop_ID}
 Let $u$ be as above and $t< \lambda^{-1/2}$. Then there exists a constant $C(n,\lambda)$ such that for all $t\leq \lambda^{-1/2}$ and $r\leq 1$
\begin{gather}
 I(T_t,g^t,r)\leq C(n,\lambda) \ D(T_t,g^t,r)\, .
\end{gather}
Moreover, there exist constants $r_0=r_0(n,\lambda)$ and $C(n,\lambda)$ such that for all $t\leq r_0$ and $r\leq 1$
\begin{gather}
 D(T_t,g^t,r)\leq C(n,\lambda) \ I(T_t,g^t,r)\, .
\end{gather}
\end{prop}
\begin{proof}
 By definition, we have
\begin{gather}
 I(T_t,g^t,r)=\int_{B(g^t,0,r)}\norm{\nabla T_t}_{g^t}^2+ T_t\Delta_{g^t} (T_t )dV_{g^t}= \int_{B(g^t,0,r)}\norm{\nabla T_t}_{g^t}^2+ tT_t\ps{B}{\nabla T_t}dV_{g^t}\, .
\end{gather}
Since we can easily estimate that
\begin{gather}
 \abs{\int_{B(g^t,0,r)} tT_t\ps{B}{\nabla T_t}dV_{g^t}}\leq t C(n,\lambda) \int_{B(g^t,0,r)}\norm{\nabla T_t}_{g^t}^2dV_{g^t}\, ,
\end{gather}
we have $\abs{I-D}\leq tC D$, and the claim follows.
\end{proof}
As a corollary to Propositions \ref{pa} and \ref{prop_ID} we get the following.
\begin{corollary}
 There exist constants $r_0=r_0(\lambda,n,M,L)>0$ and $c(n,\lambda)>0$ such that for all $t\leq r_0$
\begin{gather}
 \bar N(u,g,1)=\bar N(T_t,g^t,1)\geq c(n,\lambda)\, .
\end{gather}
\end{corollary}

\paragraph{Almost Monotonicity of $\bar N$}By an argument that is philosophically identical to the one for harmonic functions, although technically more complicated, we show that this modified frequency is ``almost'' monotone in the following sense.
\begin{theorem}\label{th_Nellmon}
 Let $u:B_1(0)\to \R$ be a nonconstant solution to equation \eqref{eq_Lu} with \eqref{e:coefficient_estimates} and let $x\in B_{1/2}(0)$. Then there exists a positive $r_0=r_0(\lambda)$ and a constant $C=C(n,\lambda)$ such that
\begin{gather}
 e^{C r} \bar N(r)\equiv  e^{C r} \bar N(u,x,g(x),r)
\end{gather}
is monotone nondecreasing on $(0,r_0)$.
\end{theorem}
\begin{proof}
For simplicity, we assume $x=0$ and $u(0)=0$. We will prove that, for $r\in (0, r_0)$
\begin{gather}\label{eq_dN}
 \frac{\bar N'(r)}{\bar N(r)}\geq -C(n,\lambda)\, .
\end{gather}
Define $T_{t} u=T_{0,t} u$ as in \eqref{eq_Tt}. Using lemma \ref{lemma_Ninvell}, the last statement is equivalent to
\begin{gather}
 \frac{\bar N'_t(1)}{\bar N_t(1)}\equiv \frac{\bar N'(T_tu,g^t,0,1)}{\bar N(T_tu,g^t,0,1)}\geq -C(n,\lambda)t\, .
\end{gather}

For the moment, fix $t$ and set $T=T_t u$. We begin by computing the derivative of $H$. Recalling that
\begin{gather*}
 H(r)=H(T,g^t,0,r)=r^{n-1}\int_{\partial B_1} T^2(r,\theta) \sqrt{b(tr,\theta)}d\theta\, ,
\end{gather*}
we have
\begin{gather*}
H'|_{r=1}=(n-1) H(1) +  2\int_{\partial B_1} T\ps{\nabla T\, }{\, \nabla r}\sqrt{b(t,\theta)}d\theta+\int_{\partial B_1}\ton{\frac t 2 \left.\frac{\partial \log(b)}{\partial r}}\right\vert_{(tr,\theta)} T^2(1,\theta) \sqrt{b(t,\theta)}d\theta\, .
\end{gather*}
By equation \eqref{eq_bst}, we obtain the estimate 
\begin{gather}\label{eq_N1}
 \abs{H'(1)- (n-1)  H(1) -  2\int_{\partial B(g^t,0,1)} T T_ndS(g^t)}\leq C(n,\lambda) t \ H(1)\, ,
\end{gather}
where $T_n=\ps{\nabla T}{\partial_r}$ is the normal derivative of $T$ on $\partial B(g^t,0,r)$. As for the derivative of $I$, we split it into two parts
\begin{equation}\label{eq_N2}
\begin{aligned}
 I'=\frac{d}{dr}I(T,g^t,r)&=\int_{\partial B(g^t,0,r)}\ton{\norm{\nabla T}_{g^t}^2+ T\Delta_{g^t} (T )}dS(g^t)=\\
&= \int_{\partial B(g^t,0,r)}\norm{\nabla T}_{g^t}^2dS(g^t)+\int_{\partial B(g^t,0,r)} T\Delta_{g^t} (T )dS(g^t)=\\
&\equiv I'_\alpha + I'_\beta\, .
\end{aligned}
\end{equation}
Using geodesic polar coordinates relative to $g^t$, set $\vec v=r \nabla r$. By the divergence theorem we get 
\begin{equation}
\begin{aligned}
I'_\alpha &= \frac 1 r \int_{\partial B(g^t,0,r)} \norm{\nabla T}_{g^t}^2\ps{\vec v}{r^{-1}\vec v}dS(g^t)
=\frac 1 r \int_{B(g^t,0,r)} \dive{\norm{\nabla T}_{g^t}^2\vec v }dV(g^t)=\\
&=\frac 1 r \int_{B(g^t,0,r)}\norm{\nabla T}_{g^t}^2 \dive{\vec v }dV(g^t)+
+ \frac 2 r \int_{B(g^t,0,r)} \nabla^i\nabla^j T\, \nabla_i T\,  \vec v _j\, dV(g^t)=\\
 &=\frac 1 r \int_{B(g^t,0,r)}\norm{\nabla T}_{g^t}^2 \dive{\vec v }dV(g^t)+
+\frac 2 r \int_{B(g^t,0,r)} \ps{\nabla \ps{\nabla T\,}{\, \vec v}\, }{\, \nabla T} dV(g^t)+\\
 &{}\qquad - \frac 2 r \int_{B(g^t,0,r)} \nabla^j T\nabla_i T \ton{\nabla^i\vec v }_j dV(g^t)=\\
&=\frac 1 r \int_{B(g^t,0,r)}\norm{\nabla T}_{g^t}^2 \dive{\vec v }dV(g^t) + 2 \int_{\partial B(g^t,0,r)} \ton{T_n}^2 dS(g^t) +\\
&{}\qquad-\frac 2 r\int_{B(g^t,0,r)} t\ps{\nabla T}{\vec v}\ps{B}{\nabla T} dV(g^t)-  \frac 2 r \int_{B(g^t,0,r)} \nabla^j T\nabla_i T \ton{\nabla^i\vec v }_j dV(g^t)\, .
\end{aligned}
\end{equation}
Using geodesic polar coordinates, it is easy to see that
\begin{gather}
 \left.\abs{\ton{\nabla^i \vec v}_j -\delta^i_j}\right\vert_{(r,\theta)}\leq r t C(\lambda)\, .
\end{gather}
Therefore,  we have the estimate
\begin{gather}
 \abs{I'_\alpha(1)- (n-2) D(1) - 2 \int_{\partial B(g^t,0,1)} \ton{T_n}^2 dS(g^t)}\leq t C(n,\lambda) D(1)\, .
\end{gather}
Using Proposition \ref{prop_ID} we conclude that for $t\leq r_0=r_0(\lambda)$,
\begin{gather}\label{eq_I}
 \abs{I'_\alpha(1)-(n-2) I(1)- 2 \int_{\partial B(g^t,0,1)}\ton{T_n}^2 dS(g^t)}\leq t C(n,\lambda) I(1)\, .
\end{gather}

To estimate $I'_\beta$, we use the divergence theorem to write
\begin{gather}
 I(r)=\int_{\partial B(g^t,0,r)} T T_n dS(g^t)\, .
\end{gather}
Note that for $tr\leq r_0$, $I(r)>0$.
 From Cauchy's inequality and Proposition \ref{prop_ID}, we get
\begin{gather}
 \notag  I^2(r)\leq H(r)\int_{\partial B(g^t,0,r)} T_n^2 dS(g^t)\leq \frac{rI(r)}{c(n,\lambda)}\int_{\partial B(g^t,0,r)} T_n^2 dS(g^t)\, ,\\
\!\!\!\!\!\!\!\!\!\!\!\!\!\!\!\!\!\!\!\!\!\!\!\!\!\!\!\!\!\!\!\!\!\!\!\!\!\!\!\!\!\!\!\!\!\!\!\!\!\!\!\!\!\!\!\!\!\!
 I(r)\leq \frac{r}{c(n,\lambda)}\int_{\partial B(g^t,0,r)} T_n^2 dS(g^t)\, ,
\end{gather}
and so, using equation \eqref{eq_I}, we get
\begin{gather}\label{eq_cau}
 \int_{\partial B(g^t,0,1)} \norm{\nabla T}^2_{g^t} dS(g^t)= I'_\alpha(1) \leq C(n,\lambda) \int_{\partial B(g^t,0,1)} T_n^2 dS(g^t)\, .
\end{gather}
 Following \cite[pag 56]{hanlin}, we divide the rest of the proof in two cases:
\paragraph{Case 1.}
\begin{gather}\label{eq_case1}
 \int_{\partial B(g^t,0,1)} T^2 dS(g^t) \ \int_{\partial B(g^t,0,1)} T_n^2 dS(g^t)\leq 2 \ton{\int_{\partial B(g^t,0,1)} TT_n dS(g^t)}^2= 2I^2(1)\, .
\end{gather}
In this case, using Cauchy's inequality and \eqref{eq_cau}, we have the estimate
\begin{gather}\label{eq_N3}
\abs{ I'_\beta(1)}= \abs{\int_{\partial B(g^t,0,1)} t T \ps{B}{\nabla T}dS(g^t)} \leq tC(n,\lambda) I(1)\, .
\end{gather}
So, from equations \eqref{eq_N1}, \eqref{eq_N2}, \eqref{eq_I} and \eqref{eq_N3}, we get for $t\leq r_0$,
\begin{gather}
\notag \frac{\bar N_t'(1)}{\bar N_t(1)}= 1 + \frac{I'(1)}{I(1)} - \frac{H'(1)}{H(1)}\geq \\
\geq \notag 0 + 2\ton{\frac{\int_{\partial B(g^t,0,1)}  T_n^2dS(g^t)}{\int_{\partial B(g^t,0,1)}  TT_ndS(g^t)}-\frac{\int_{\partial B(g^t,0,1)}  TT_ndS(g^t)}{\int_{\partial B(g^t,0,1)}  T^2dS(g^t)}} - tC(n,\lambda)\geq - tC(n,\lambda)\, ,
\end{gather}
where the last inequality comes from a simple application of Cauchy's inequality.
\paragraph{Case 2.}
To complete the proof, suppose
\begin{gather}\label{eq_case2}
 \int_{\partial B(g^t,0,1)} T^2 dS(g^t) \ \int_{\partial B(g^t,0,1)} T_n^2 dS(g^t)>2 \ton{\int_{\partial B(g^t,0,1)} TT_n dS(g^t)}^2= 2I^2(1)\, .
\end{gather}
Then we have the following estimate for estimate $I'_\beta$.
\begin{gather}
\abs{ I'_\beta(1)}= \abs{\int_{\partial B(g^t,0,1)} t T \ps{B}{\nabla T}dS(g^t)} \leq t \ton{\int_{\partial B(g^t,0,1)}  T^2dS(g^t)\int_{\partial B(g^t,0,1)} \norm{\nabla T}_{g^t}^2dS(g^t)}^{1/2}\leq \\
\notag \!\!\!\!\!\!\!\!\!\!\!\!\!\!\!\!\!\!\!\!\!\!\!\!\!\!\!\!\!\!\!\!\!\!\!\!\!\!\!\!\!\!\!\!\!\!\!\!\!\!\!\!\!\!\!\!
\!\!\!\!\!\!\!\!\!\!
\leq C(n,\lambda)t\ton{\int_{\partial B(g^t,0,1)}  T^2dS(g^t)\int_{\partial B(g^t,0,1)} T_n^2dS(g^t)}^{1/2}\, .
\end{gather}
Applying Young's inequality with the right constant and proposition \ref{prop_ID}, we obtain that for $t\leq r_0$,
\begin{gather}\label{eq_N4}
\abs{ I'_\beta(1)}\leq \int_{\partial B(g^t,0,1)} T_n^2dS(g^t) + C(n,\lambda)t^2 \int_{\partial B(g^t,0,1)}  T^2dS(g^t) \leq \int_{\partial B(g^t,0,1)} T_n^2dS(g^t) + C(n,\lambda)t^2 I(1)\, .
\end{gather}
Using equations \eqref{eq_N1}, \eqref{eq_N2}, \eqref{eq_I} and \eqref{eq_N4}, we get for $t\leq r_0$,
\begin{gather}
\frac{\bar N_t'(1)}{\bar N_t(1)}= 1 + \frac{I'(1)}{I(1)} - \frac{H'(1)}{H(1)}\geq \\
\geq \notag 0 + {\frac{\int_{\partial B(g^t,0,1)}  T_n^2dS(g^t)}{\int_{\partial B(g^t,0,1)}  TT_ndS(g^t)}-\frac{2\int_{\partial B(g^t,0,1)}  TT_ndS(g^t)}{\int_{\partial B(g^t,0,1)}  T^2dS(g^t)}} - tC(n,\lambda)\geq - tC(n,\lambda)\, ,
\end{gather}
where the last inequality follows directly from the assumption \eqref{eq_case2}.
\end{proof}

\begin{remark}\rm
 Even though we are not interested in doubling conditions in this work, it is worth mentioning that, as in the harmonic case, it is possible to use the monotonicity of the modified frequency function to prove doubling conditions on $H(\bar x,r)$, and the unique continuation principle as a corollary. Similar computations are carried out in \cite[Section 3.2]{hanlin}.
\end{remark}

It is evident that, with straightforward modifications, we can drop the assumption $u(0)=0$ and $\bar x=0$ in the previous theorem.

\paragraph{Uniform bounds on the frequency} If we want to adapt the proofs for harmonic functions to this more general case, a crucial property needed for the modified frequency function is a generalization of Lemma \ref{lemma_strong2}. Even though it might be possible to prove such a lemma using doubling conditions for $H(r)$ and generalized mean value theorems, here we use quite simple compactness arguments to prove our claims. Note that these arguments do not give a quantitative control on the constants $C$ and $r_0$, they only prove their existence, which is enough for the purposes of this work.

\begin{lemma}\label{lemma_prestrongell}
Let $u$ be a solution to \eqref{eq_Lu} with \eqref{e:coefficient_estimates} on $B_{{\lambda}^{-1/2}r_0}(0)$. There exists $r_0=r_0(n,\lambda,\Lambda)$ and $C=C(n,\lambda,\Lambda)$ such that if $\bar N(0,r)\leq \Lambda$, then for all $x\in B_{r/3}(0)$
\begin{gather}
 \bar N(x,r/2)\leq C\, .
\end{gather}
\end{lemma}
\begin{proof}
 Consider by contradiction a sequence of solutions $u_i$ to some operator $\L_i$ (satisfying \eqref{e:coefficient_estimates} with the same $\lambda$) such that $\bar N(u_i,0,i^{-1})\leq \Lambda$ but such that for some $x_i\in B_{i^{-1}/3}(0)$
 \begin{gather}
  \bar N(u_i,x_i,i^{-1}/2)\geq i\, .
 \end{gather}
Let $g^i=g^i(ir,\theta)$ be the metric associated to each operator $\L_i$. Since we are assuming \eqref{e:coefficient_estimates}, $g^i$ converges in the Lipschitz sense to the standard Euclidean metric on $B_1(0)$. Set for simplicity $T_i(r,\theta)= T^{u_i}_{0,i^{-1}} (r,\theta)$, where the latter is defined in equation \eqref{eq_Tt}.

The bound on the frequency $\bar N$ and Lemma \ref{prop_ID} imply that, for $i$ large enough
\begin{gather}
 \int_{B_1}\abs{\nabla T_i}^2 dV \leq \lambda^{-\frac{n-2}2} \int_{B_1(0)} \norm{\nabla T_i}^2_{g^i}dV(g^i)\leq C(n,\lambda)\bar N (0,i^{-1}) \leq C(n,\lambda)\Lambda\, .
\end{gather}
Since $T_i(0)=0$, $T_i$ have uniform bound in the $W^{1,2}(B_1(0))$ norm and, by elliptic estimates, also in the $C^{1,1/2}(B_{2/3})$ norm.

Consider a subsequence $T_i$ which converges in the weak $W^{1,2}$ sense to some $T$, and a subsequence of $i x_i$ converging to some $x\in \overline B_{1/3}$. $T$ is easily seen to be a nonconstant harmonic function, and, by the convergence properties of the sequence $T_i$, we also have
\begin{gather*}
 \lim_{i\to \infty} \bar N(T_i,0,g^i(0),1) = \bar N(T,0,e,1)\, ,\\
 \lim_{i\to \infty} \bar N(T_i,i x_i,g^i(i x_i),1/2) = \bar N(T,x,e,1/2)\, .
\end{gather*}
Recalling that $e$ is the standard Euclidean metric on $\R^n$, the contradiction is a consequence of Lemma \ref{lemma_strong2bar}.
\end{proof}
As for harmonic function, we will prove that $\bar N(0,1)\leq \Lambda$ gives a control on $\bar N(x,r)$ for all $x$ not close to the boundary of the ball of radius $1$. However, in the generic elliptic case $r$ cannot be arbitrary, but we have to assume $r\leq r_0$. Using a compactness argument similar to the one in the proof of Lemma \ref{lemma_strong2bar}, and the almost monotonicity of $\bar N$ we can prove the following
\begin{lemma}\label{lemma_Nec}
 Let $u$ be a nonconstant solution to \eqref{eq_Lu} in $B_1(0)$. Then there exist constants $r_1(n,\lambda,\Lambda)$ and $C(n,\lambda,\Lambda)$ such that, if
\begin{gather}
\frac{\int_{B_1} \abs{\nabla u}^2  dV}{\int_{\partial B_1} (u-u(0))^2 dS}\leq \Lambda\, ,
\end{gather}
then for all $x\in B_{1/2}(0)$ and $r\leq r_1$
\begin{gather}
 \bar N(u,x,r)\leq C(n,\lambda,\Lambda)\, .
\end{gather}
\end{lemma}

\subsection{Standard and Almost Symmetry}
\index{symmetry!standard symmetry}\index{symmetry!almost symmetry}
Similar properties to the one proved for harmonic function in Sections \ref{sec_homsym} and \ref{sec_almsym} are available also for solutions to more general elliptic equations, although it might be necessary to restrict ourselves to scales smaller than some $r_0(n,\lambda,\Lambda)$. With the necessary minor changes, the ideas of the proofs are completely analogous to the ones already set out in the harmonic case, so we will only sketch some of them in order to point out the technical issues involved.

For a solution $u$ of equation \eqref{eq_Lu}, the definition of $(\epsilon,r,k,x)$-symmetry is the same as the one given for harmonic functions in \ref{deph_hom2}, where $T^{u}_{x,r}:B_1(0)\to \R$ is given by \eqref{eq_Tt}.

\begin{teo}\label{th_N2homell}
Fix $\eta>0$, $0\leq \gamma<1$ and a solution $u$ to \eqref{eq_Lu} with \eqref{e:coefficient_estimates}. There exist $r_2=r_2(n,\lambda,\eta,\gamma,\Lambda)$ and $\epsilon=\epsilon(n,\lambda,\eta,\gamma,\Lambda)$ such that if
\begin{gather}
\frac{\int_{B_1} \abs{\nabla u}^2  dV}{\int_{\partial B_1} (u-u(0))^2 dS}\leq \Lambda
\end{gather}
and if for some $x\in B_{1/2}(0)$ and $r\leq r_2$, $\bar N(x,r)-\bar N(x,\gamma r)\leq \epsilon$, then $u$ is $(\epsilon,r,0,x)$-symmetric.
\end{teo}
\begin{proof}
 Using a contradiction argument similar to the one used for the proof of Lemma \ref{lemma_prestrongell}, we find a sequence of $T_i$ converging to some nonzero normalized harmonic $T$ with generalized frequency $\bar N(0,1)$ bounded by the $C(n,\lambda,\Lambda)$ given in Lemma \ref{lemma_Nec}. Note that the convergence is $W^{1,2}(B_1)$ weak and locally strong in $W^{2,2}\cap C^{1,1/2}$. Since the metric $g^i$ converge in the Lipschitz sense to $g^0=e$, the limit $T$ satisfies $\bar N^T(0,1)=\bar N^T(0,\gamma)\leq C$, and so it is a harmonic polynomial of degree $d\leq C(n,\lambda,\Lambda)$.
\end{proof}
In a completely similar way, we can also prove a generalization of Corollary \ref{cor_symsum}.
\begin{corollary}[Cone splitting theorem]\label{cor_symsumell}
\index{cone splitting theorem}
Fix $\eta,\tau,\chi>0$, $k\leq n-2$, and let $u$ be a solution $u$ to \eqref{eq_Lu} with conditions \eqref{e:coefficient_estimates}. There exists $\epsilon(n,\lambda,\eta,\tau,\chi,\Lambda)$ and $r_3=r_3(n,\lambda,\eta,\tau,\chi,\Lambda)$ such that if
\begin{gather}
 \frac{\int_{B_1} \abs{\nabla u}^2  dV}{\int_{\partial B_1} (u-u(0))^2 dS}\leq \Lambda
\end{gather}
and if for some $x\in B_{1/2}(0)$
\begin{enumerate}
 \item $u$ is $(\epsilon,\chi r_3,k,x)$-symmetric,
 \item for every affine subspace $V$ with $x\in V$ of dimension $\leq k$, there exists $y\in B_{\chi r_0}(x)\setminus T_{\tau}V$ such that $u$ is $(\epsilon,\chi r_3,0,y)$-symmetric;
\end{enumerate}
then $u$ is $(\eta,r_3,k+1,0)$-symmetric.
\end{corollary}

Since we are assuming uniform Lipschitz bounds on the leading coefficient $a_{ij}$ of our elliptic PDE, we have $C^{1,1/2}$ estimates and control on the solutions (see \cite{GT} for details). For this reason, it is not difficult to adapt the usual contradiction/compactness technique to prove the following generalization of Proposition \ref{prop_qcrit}.
\begin{prop}\label{prop_qcritell}
Let $u$ be a solution to \eqref{eq_Lu} with conditions \eqref{e:coefficient_estimates} such that
\begin{gather}
 \frac{\int_{B_1} \abs{\nabla u}^2  dV}{\int_{\partial B_1} (u-u(0))^2 dS}\leq \Lambda\, .
\end{gather}
 There exists $\epsilon=\epsilon(n,\lambda,\Lambda)>0$ and $r_0=r_0(n,\lambda,\Lambda)$ such that if $u$ is $(\epsilon,r_0,n-1,1,x)$-symmetric for some $x\in B_{1/2}(0)$, then $u$ does not have critical points in $B_{r_0/2}(x)$.
\end{prop}

\subsection{Quantitative stratification of the critical set}\label{sec_qsell}
Here we redefine the effective strata for a solution to \eqref{eq_Lu} in such a way that Proposition \ref{prop_ereg} still has an analogue. The definition is almost the same, the only difference is that if we want to apply the previous proposition we need to restrict our attention to scales smaller than $r_0$.
\begin{definition}\label{deph_effstraell}
 Let $u$ be a solution to \eqref{eq_Lu} with
\begin{gather}
 \frac{\int_{B_1} \abs{\nabla u}^2  dV}{\int_{\partial B_1} (u-u(0))^2 dS}\leq \Lambda\, ,
\end{gather}
and let $r_0=r_0(n,\lambda,\Lambda)$ be the scale obtained in Proposition \ref{prop_qcritell}. The $(k,\eta,r)$ effective singular stratum for $u$ is defined as
\index{skn@$\cS^k_{\eta,r}$}
\rm{\begin{gather}
 \cS^k_{\eta,r}(u)=\cur{x\in B_1(0) \ \text{ s.t. } \forall r\leq s \leq r_0, \ u \text{ is not } (\eta,s,k,x)-\text{symmetric }  }\, .
\end{gather}}
\end{definition}

The next corollary follows immediately from Proposition \ref{prop_qcritell}.
\begin{corollary}\label{cor_eregell}
Given a function $u$ as in Proposition \ref{prop_qcritell}, there exists $\eta=\eta(n,\lambda,\Lambda)>0$ such that, for $0<r\leq r_0$,
\begin{gather}
\Cr(u)\subset \T_{r/2}(\Cr(u))=\bigcup_{x\in \Cr(u)} B_{r/2}(x)\subset \cS^{n-2}_{\eta,r}(u)\, .
\end{gather}
\end{corollary}
\begin{proof}
 Let $x\not \in \cS^{n-2}_{\eta,r}(u)$. Then by definition $u$ is $(\epsilon,s,n-1,1,0)$-symmetric for some $r\leq s \leq r_0$. Proposition \ref{prop_qcritell} concludes the proof.
\end{proof}

\subsection{Volume Estimates for Elliptic Solutions}
Now we are ready to generalize Theorem \ref{th_main} to solutions to more general elliptic equations.
\begin{teo}\label{th_main_ell}
 Let $u$ be a ($W^{1,2}$ weak) solution to \eqref{eq_Lu} with \eqref{e:coefficient_estimates} in $B_1(0)\subset \R^n$ such that
\begin{gather}
 \frac{\int_{B_1} \abs{\nabla u}^2  dV}{\int_{\partial B_1} (u-u(0))^2 dS}\leq \Lambda\, .
\end{gather}
Then for every $\eta>0$ and $k\leq n-2$, there exists $C=C(n,\lambda,\Lambda,\eta)$ such that
\begin{gather}\label{eq_12chi}
 \Vol\qua{\T_r \ton{\cS^{k}_{\eta,r}(u)}\cap B_{1/2}(0)}\leq C r^{n-k-\eta}\, .
\end{gather}
\end{teo}

The proof of Theorem \ref{th_main_ell} is based on the Frequency Decomposition lemma,
\index{frequency function!frequency decomposition}
and is almost identical to the proof of Theorem \ref{th_main_proof}, so instead of rewriting all the details in the notationally more complicated setting of generic elliptic functions, we just point out how to adapt the proof for the harmonic case. As in that case, we prove the theorem only for $r=\gamma^j$ for some $0<\gamma<1$ and all $j\in \N$. Note that following these steps, we will prove Theorem \ref{th_main_ell} only for $r\leq r_0(n,\lambda,\Lambda)$. However, since $\Vol\qua{\T_r \ton{\cS^{k}_{\eta,r}(u)}\cap B_{1/2}(0)}\leq \Vol(B_{1/2}(0))$, it is evident how to extend the estimate for all positive $r$.

\begin{proof}[Proof of Theorem]
Fix any positive $\eta$ and let $\gamma=c_0^{-2/\eta}<1$, $\chi=\gamma^{-1}$ and choose any positive $\tau$ (say $\tau=7^{-1}$). Let $r_0$ be the minimum of $r_1$ given by Lemma \ref{lemma_Nec}, $r_2$ given by Corollary \ref{th_N2homell} and $r_3$ given by Corollary \ref{cor_symsumell}. Then, if $i$ is such that $\gamma^i\leq r_0$, the same proof as in the harmonic case applies also to this more general case with Lemma \ref{lemma_strong2} replaced by Lemma \ref{lemma_Nec}, Theorem \ref{th_N2hom} by \ref{th_N2homell} and Corollary \ref{cor_symsum} by \ref{cor_symsumell}. Note that there is only a finite number of $i'$s such that $\gamma^i> r_0$, and the number of such indexes is bounded by a uniform constant $D'=D'(n,\lambda,\eta,\Lambda)$. Finally, if we replace monotonicity of $\bar N$ with almost monotonicity of $\bar N$, it is straightforward to see that an estimate of the form given in equation \eqref{eq_estK} still holds. So all the ingredients of the proof, up to some easy, although 
perhaps annoying, details, are the same as in the harmonic case. 
\end{proof}

Just as in the harmonic case, the main application for this theorem concerns the critical set of the function $u$. In particular, it follows immediately from Corollary \ref{cor_eregell} that
\begin{corollary}\label{cor_main_ell}
 Let $u$ be a solution to \eqref{eq_Lu} with \eqref{e:coefficient_estimates} in $B_1(0)\subset \R^n$ such that
\begin{gather}
 \frac{\int_{B_1} \abs{\nabla u}^2  dV}{\int_{\partial B_1} (u-u(0))^2 dS}\leq \Lambda\, .
\end{gather}
Then, for every $\eta>0$, there exists $C=C(n,\lambda,\Lambda,\eta)$ such that
\begin{gather}
 \Vol\qua{\T_r \ton{\Cr(u)}\cap B_{1/2}(0)}\leq C r^{2-\eta}\, .
\end{gather}
\end{corollary}
Thus we have proved Theorem \ref{t:crit_lip}.

\subsection{n-2 Hausdorff estimates for elliptic solutions}
As for the Minkowski estimates, it is possible to generalize also the $n-2$ Hausdorff uniform estimates on the critical set to solutions to elliptic equations of the form \eqref{eq_Lu}. However to make such an extension we require some additional regularity on the coefficients $a^{ij}$ and $b^i$ of the PDE. If we were able to say that there exists a small enough scale $r_0$ and a $\delta$ small enough such that $\fint_{\partial B_1} \abs{T_{x,r}-P}^2 dS \leq \delta$ implies $\norm{T_{x,r}-P}_{C^{2d^2}(B_1)}<\epsilon$, then we would have an analogue of Corollary \ref{cor_ereg} for generic elliptic functions. 

There are two issues involved in the generalization of this corollary. First of all, if we want control over higher order derivatives of the solutions to \eqref{eq_Lu}, we need higher order control on the derivatives of the coefficients (see elliptic estimates in \cite{GT}). However, no matter what control we have on the coefficients $a_{ij}$, the metric $g_{ij}$ defined in Proposition \ref{prop_gij} is in general only Lipschitz at the origin. 

The first issue is solved by adding some hypothesis on the coefficients of the elliptic equation \footnote{although we feel that the set of assumptions \eqref{e:coefficient_estimates} should be sufficient to have uniform $n-2$ Hausdorff control on the critical set of solutions.}. As for the second issue, we need to study the critical set of $u$ in a better-behaved metric.

Recall from Proposition \ref{prop_gij} that the ball of radius $r$ in the metric $g_{ij}(\bar x)$ coincides with the ellipsoid $r(\bar x,x)<r$. We define a ``smooth'' replacement for $T^u_{\bar x,r}$ in the following way. 

Let $q^{ij}(\bar x)$ be the square root of the (positive) matrix $a^{ij}(\bar x)$, i.e., $\sum_k q^{ik}q^{kj}=a^{ij}$. Similarly, $q_{ij}$ is the inverse of $q^{ij}$, and so the square root of $a_{ij}$. For fixed $\bar x\in B_1(0)$, define the linear operator $Q_{\bar x}$ as
\begin{gather}\label{eq_Q}
 Q_{\bar x} (y) = \sum_j q_{ij}(x-\bar x)^i e_j\, .
\end{gather}
It is evident that, independently of $\bar x$, $Q$ is a bi-Lipschitz equivalence from $\R^n$ to itself with constant $\lambda^{1/2}$. Moreover, the push forward of $a_{ij}$ through $Q_{\bar x}$ has the property that $Q_{\bar x}^*(a_{ij})|_{0} = \delta_{ij}$ and so the ellipsoid $r(\bar x,r)<r$ is mapped onto the Euclidean sphere of radius $r$.

Define now the function $U_{\bar x,t}:B_1(0)\to \R$ by
\begin{gather}
 U_{\bar x,t}(y) = \frac{u(\bar x + t Q_{\bar x}^{-1}(y))}{\ton{\int_{\partial B_1} u^2(\bar x + t Q_{\bar x}^{-1}(y)) dS}^{1/2}}\, .
\end{gather}
Using a simple change of variables, it is easy to see that the function $U$ satisfies an elliptic PDE of the form
\begin{gather}
 \tilde \L(u)=\partial_i\ton{\tilde a^{ij} \partial _j U} + \tilde b^i \partial_i U =0\, ,
\end{gather}
with $\tilde a^{ij}(\bar x)=\delta^{ij}$. Moreover, as long as $t\leq 1$, if
\begin{gather}\label{eq_U}
 \sum_{ij}\norm{a^{ij}}_{C^{M}(B_1)} + \sum_i \norm{b^i}_{C^{M}(B_1)}\leq L^+\, ,
\end{gather}
then
\begin{gather}
 \sum_{ij}\norm{\tilde a^{ij}}_{C^{M}(B_1)} + \sum_i \norm{\tilde b^i}_{C^{M}(B_1)}\leq C(\lambda,L^+)\, .
\end{gather}
Note that, since $C(\lambda)^{-1}\delta_{ij}\leq g(\bar x,r)_{ij} \leq C(\lambda) \delta_{ij}$,
\begin{gather}
 \int_{B_1}\abs{\nabla U_{\bar x,r}}^2 dV \leq C(\lambda,n) \int_{B_1}\norm{\nabla T_{\bar x,r}}_{g(\bar x,r)}^2 dV(g(\bar x,r))\, .
\end{gather}
Note that if the coefficients $a_{ij}$ were constant in $B_1$, then $U_{\bar x,t}=T^u_{\bar x,t}$. Actually $U_{\bar x,r}$ is very similar in spirit to $T^u_{\bar x, r}$. Moreover, if $r\to 0$, then the metric $g(\bar x,r)$ becomes more and more similar to the constant metric $a_{ij}(\bar x)$ in the Lipschitz sense, and so $U_{\bar x,r}-T_{x,r}$ converges to $0$ in the $W^{1,2}(B_1)$ sense. However, there are some crucial differences between $T$ and $U$, for example, as remarked before, the function $T$ does not enjoy the property \eqref{eq_U}. Note also that it is almost immediate to see that
\begin{gather}\label{eq_n-2Ha}
 \Ha^{n-2}\ton{\abs{\nabla u}^{-1}(0)\cap B_r(0)} \leq C(\lambda)r^{n-2}\Ha^{n-2}\ton{\abs{\nabla U_{\bar x,r}}^{-1}(0)\cap B_1(0)}\, .
\end{gather}

With this technical definition, we are ready to prove the following.
\begin{lemma}\label{lemma_eregell}
Let $u:B_1\to \R$ be a solution to equation \eqref{eq_Lu} with
\begin{gather}
 \frac{\int_{B_1} \abs{\nabla u}^2 dV}{\int_{\partial B_1} (u-u(0))^2 dS} \leq \Lambda\, .
\end{gather}
Then there exists a positive integer $M=M(n,\lambda,\Lambda)$ such that if the coefficients of the equation satisfy
\begin{gather}
 \sum_{ij}\norm{a^{ij}}_{C^{M}(B_1)} + \sum_i \norm{b^i}_{C^{M}(B_1)}\leq L^+\, ,
\end{gather}
then there exist positive constants $C(n,\lambda,L^+,\Lambda)$, $\bar r (n,\lambda,L^+,\Lambda)$, $\eta(n,\lambda,L^+,\Lambda)$ and $\chi(n,\lambda,L^+,\Lambda)$ such that if, for some $r\leq \bar r$ and $x\in B_{1/2}(0)$, there exists a normalized homogeneous harmonic polynomial $P$ with $n-2$ symmetries with
\begin{gather}
 \norm{T_{x,r}-P}_{L^2(\partial B_1)}<\eta\, ,
\end{gather}
then for all $r\leq \bar r$
\begin{gather}
 \Ha^{n-2}(\nabla u^{-1}(0)\cap B_{\chi r}(x))\leq C(n,\lambda,L^+,\Lambda)\chi^{n-2}r^{n-2}= C(n,\lambda,L^+,\Lambda) r^{n-2}\, .
\end{gather}
\end{lemma}
\begin{proof}
By Lemma \ref{lemma_Nec}, we know that, if $\eta$ is sufficiently small, the degree $d$ of the polynomial $P$ is bounded by $d(n,\lambda,\Lambda)$. Choose $M\geq 2d(n,\lambda,\Lambda)^2 +1$, so that by elliptic regularity $U_{\bar x,r}\in C^{2d^2}(B_1)$. We first prove that there exists constants $\bar r,\eta$ such that, for $r\leq \bar r$ and $\bar x \in B_{1/2}(0)$
\begin{gather}\label{eq_fdsa}
 \norm{U_{\bar x, r}-P}_{C^{2d^2}(B_{1/2})}\leq \epsilon\, ,
\end{gather}
where $\epsilon$ is fixed by Lemma \ref{lemma_n-2Ha}. Note that, as in the harmonic case, $\epsilon$ can be chosen to depend only on the highest possible degree $d=d(n,\lambda,\Lambda)$ of the normalized homogeneous harmonic polynomial $P$. Suppose by contradiction that \eqref{eq_fdsa} is not true. Since
\begin{gather}
 \int_{B_1}\abs{\nabla U_{\bar x,r}}^2 dV \leq C(n,\lambda,\Lambda)\, ,
\end{gather}
we can extract a subsequence from $U_{\bar x_i,r_i}$ converging weakly to some harmonic $U$ in the $W^{1,2}$ sense. Since $T_{x_i,r_i}$ converge to $P$ and $r_i\to 0$, $U=P$. By elliptic estimates, the $C^{2d^2+1}(B_{1/2})$ norm of the sequence $U_{\bar x_i,r_i}$ is uniformly bounded by some $C(n,\lambda,L^+,\Lambda)$, and so $U_{\bar x_i,r_i}$ converges to $P$ also in the $C^{2d^2}(B_{1/2})$ sense, so we have our contradiction.

In order to complete the estimate, denote by $\chi=\chi(n,\lambda,\Lambda)$ the value of $\bar r(n,\lambda,\Lambda)$ coming from Lemma \ref{lemma_n-2Ha}. It is evident that
\begin{gather}
 \cur{\abs{\nabla u}^{-1}(0)\cap B_{\chi r }(\bar x)} =(\chi rQ)\cur{\abs{\nabla U_{\bar x,r}}^{-1}(0)\cap B_{\chi}(0)}\, ,
\end{gather}
where $Q$ is the affine operator defined in \eqref{eq_Q}.

Lemma \ref{lemma_n-2Ha} gives a bound on $ \Ha^{n-2}\ton{\cur{\abs{\nabla U_{\bar x,r}}^{-1}(0)\cap B_{\chi}(0)}}$, and using an estimate similar to \eqref{eq_n-2Ha}, the Lemma is proved.
\end{proof}

Now we are ready to prove the $n-2$ Hausdorff estimates \footnote{as mentioned in the introduction, this result has already been obtained in \cite{HLrank}. We provide a slightly different proof}.
\begin{teo}\label{th_ellHa}
Let $u:B_1\to \R$ be a solution to equation \eqref{eq_Lu} with
\begin{gather}
 \frac{\int_{B_1} \abs{\nabla u}^2 dV}{\int_{\partial B_1} (u-u(0))^2 dS} \leq \Lambda\, .
\end{gather}
Then there exists a positive integer $M=M(n,\lambda,\Lambda)$ such that if the coefficients of the equation satisfy
\begin{gather}
 \sum_{ij}\norm{a^{ij}}_{C^{M}(B_1)} + \sum_i \norm{b^i}_{C^{M}(B_1)}\leq L^+\, ,
\end{gather}
then there exist positive constants $C(n,\lambda,L^+,\Lambda)$ such that
\begin{gather}
 \Ha^{n-2}(\Cr(u)\cap B_{1/2}(0))\leq C(n,\lambda,L^+,\Lambda)\, .
\end{gather}
\end{teo}
\begin{proof}
Just as in the $n-2+\eta$ Minkowski estimates, the proof is almost the same as in the harmonic case, it is enough to replace Corollary \ref{cor_ereg} with Lemma \ref{lemma_eregell}.
\end{proof}

\section{Estimates on the Singular set}\label{sec_sing}
In this section, we briefly study elliptic equations of the form \eqref{eq_Lu2}, which we recall here:
\begin{gather}\label{eq_Luc}
 \L(u)=\partial_i \ton{a^{ij}\partial_j u} + b^i \partial_i u + c u =0\, .
\end{gather}
The zero order coefficient $c(x)$ makes (in general) constant functions not solutions to the equation, and for this reason we cannot say that, given a solution $u$ and a point $\bar x$, $u-u(\bar x)$ is still a solution to the same equation. This implies that all the proofs about almost monotonicity and doubling conditions are not valid in general any more.

However by inspecting the proofs given in section \ref{sec_ell}, one realizes that if we restrict our study to the level set $\{u^{-1}(0)\}$ all the properties still hold. In particular, in a similar way to definition \ref{deph_LN}, we can define the frequency function $N(x,r)$ as
\begin{gather}
N(u,\bar x,g,r)=\frac{r\int_{B(g(\bar x),\bar x,r)}\norm{\nabla u}_{g(\bar x)}^2+ u\Delta_{g(\bar x)} (u )dV_{g(\bar x)}}{\int_{\partial B(g(\bar x),\bar x,r)} u^2 dS_{g(\bar x)}}\, .
\end{gather}
This function turns out to be almost monotone as a function of $r$ on $(0,r_0(n,\lambda))$ if $u(\bar x)=0$. In case $u(\bar x)\neq 0$, it is still possible to prove that $N(u,\bar x,g,r_1)\leq C(n,\lambda)N(u,\bar x,g,r_2)$ for any $0<r_1\leq r_2\leq r_0$.

Once this is proved, it is not difficult to realize that a theorem similar to Theorem \ref{th_main_ell} holds for solutions to this kind of elliptic equation, although in this case the $n-2+\eta$ Minkowski estimate is available for the \textit{singular} set, not the \textit{critical} set.
\begin{teo}\label{th_si1}
  Let $u$ be a solution to \eqref{eq_Lu} in $B_1(0)\subset \R^n$ such that
\begin{gather}
 \frac{\int_{B_1} \abs{\nabla u}^2  dV}{\int_{\partial B_1} u^2 dS}\leq \Lambda\, .
\end{gather}
Then there exists $C=C(n,\lambda,\Lambda,\eta)$ such that
\begin{gather}
 \Vol\qua{\T_r \ton{\Cr(u)\cap u^{-1}(0)}\cap B_{1/2}(0)}\leq C r^{2-\eta}\, .
\end{gather}
\end{teo}

As for the uniform $n-2$ Hausdorff estimate, it is easy to see that Lemma \ref{lemma_eregell} is still valid for solutions to \eqref{eq_Lu2} if we assume also that $c\in C^M(B_1)$ with uniform bounds on the $C^M(B_1)$ norm. It is then evident how to generalize Theorem \ref{th_ellHa} for solutions of such equation.

As mentioned in the introduction, using a different proof Han, Hardt and Lin obtained the same result in \cite[Theorem 1.1]{hanhardtlin} \footnote{see also \cite[Theorem 7.2.1]{hanlin} or \cite[Theorem 7.21]{han} for the harmonic case}, which we cite here for completeness.
\begin{teo}\label{th_si2}
Suppose that $u$ is a $W^{1,2}(B_1)$ solution to equation \eqref{eq_Lu2} on $B_1\subset \R^n$ where:
\begin{enumerate}
 \item[a.] $a^{ij}$ are uniformly Lipschitz functions in $B_1$ with Lipschitz constant $\leq \lambda$
 \item[b.] the equation is strictly elliptic, i.e. there exists a $\lambda\geq1$ such that for all $x\in B_1$ and all $\xi\in \R^n$
\begin{gather}
(1+\lambda)^{-1} \abs \xi ^2\leq a^{ij}\xi_i \xi_j \leq (1+\lambda) \abs \xi ^2
\end{gather}
 \item[c.] $b_i, c\in L^{\infty}(B_1)$ with
\begin{gather}
 \sum_{ij} a^{ij}(x)^2 + \sum_i b^i(x)^2 + c(x)^2\leq \lambda\, .
\end{gather}
\end{enumerate}
Moreover, suppose that
\begin{gather}
 \frac{\int_{B_1} \abs{\nabla u}^ 2dV}{\int_{\partial B_1} u^2dS} \leq \Lambda\, .
\end{gather}
There exists a positive integer $M=M(n,\lambda,\Lambda)$ such that if we assume $a^{ij}, \ b^i, \ c\in C^{M}(B_1)$, then the $n-2$ Hausdorff measure of the singular set is bounded by
\begin{gather}
 \Ha^{n-2}\ton{\{x\in B_{1/2}(0) \ s.t. \ u(x)=0=\nabla u\}}\leq C\, ,
\end{gather}
where $C$ is a positive constant depending on $n,\lambda,\Lambda$ and the $C^M$ norms of $a^{ij}, \ b^i $ and $c$.
\end{teo}

\begin{remark}\rm
Note that, for solutions to equation \eqref{eq_Lu2}, it in not true in general that $\Cr(u)$ has Minkowski dimension bounded by $n-2$, indeed $\Cr(u)$ can even have nonempty interior. The following counterexample is given in \cite[Remark after Corollary 1.1]{hoste}.

Let $f$ be any smooth function in $B_1(0)$ and set $u=f^2+1$. Then $u$ solves the PDE
\begin{gather}
 \Delta u + c u =0\, ,
\end{gather}
where $c=-\frac{\Delta f^2}{f^2+1}$. It is evident that $f(x)=0 \ \Rightarrow\ \nabla u(x) =0$, and since for every closed $A\subset \R^n$ there exists a smooth function $f$ such that $A=f^{-1}(0)$, we have proved that the critical set of the solutions to \eqref{eq_Luc} can be rather wild, even if the coefficients of the equation are smooth.
\end{remark}

 \ifnum0\key=7
\chapter{Acknowledgments} \thispagestyle{myheadings} \markright{}
\else
\chapter*{Acknowledgments}
\fi

\ifnum0\key=7\ifnum\full=1
{\fontfamily{augie}\selectfont 
\fi\fi

\enlargethispage{30pt} First of all, I would like to thank my advisor, prof. Alberto Setti, for his insights, guidance and support, and also for bearing some of my occasional lack of diplomacy. He and my other colleagues \ifnum0\key=7 (actually friends) \fi have given me a lot, both in the field of mathematics and outside.
\ifnum0\key=7
Some, but not all, are Debora, Michele and Stefano.
\fi

In particular, I want to thank \ifnum0\key=7 Lucio (Dr. Luciano Mari) \else Dr. Luciano Mari \fi for the work we did together, especially because our cooperation was very effective. Prof. Aaron Naber has been of great help to me in these three years, and I would like to thank him for the time he spent with me and for the very interesting works we did together. As for Prof. 
\ifnum0\key=7
Jona
\else
Giona
\fi
Veronelli, I give him my thanks for the work we did together, but this is just a (strictly positive) $\epsilon$ of what I owe him, 
\ifnum0\key=7 since for example he's also a member of the 4th floor.

A special thanks to all the geeks who wrote/are writing the code for \LaTeX \ and Linux, especially because I have seen some thesis written some 20 years ago and... well thanks a lot! A special and strange thanks goes also to all the people that gave me the chance to study mathematics; my family of course, but, since we all are the dwarfs on the giant's shoulders, also all the people throughout the world that helped shape human society and maths. In particular, a special thanks to the Europeans; and also to the eurosceptics, for being so funnily anachronistic.

I would also like to mention my Windsor friends, and all the guys who played table football with me in the US, especially the ones stealing coffees (and time) around MIT...

And since I have already thanked them, a very special thanks goes, has gone and will go to the 4th floor... 
\fi

\ifnum\key=7\ifnum\full=1
}
\fi\fi

\vspace*{\fill}
\begingroup
And of course, I thank the Sorceress...\\
\textit{\ [...] Why had he ended up here, why had he chosen a trade that didn’t suit his character and his mental structure, the soldier’s trade, why had he betrayed mathematics… Oh how much he missed mathematics! How strongly he regretted having left it! It massages your meninges as a coach massages an athlete’s muscles, mathematics. It sprays them with pure thought, it purges them of the emotions that corrode intelligence, it transports them to greenhouses where the most astonishing flowers grow. The flowers of an abstraction composed of concreteness, a fantasy composed of reality... [...] No, it’s not true that mathematics is a rigid science, a severe doctrine. It’s a seductive, capricious art, a sorceress who can perform a thousand enchantments. A thousand wonders. It can make order out of disorder, it can give sense to senseless things, it can answer every question. It can even provide what you basically search for: the formula of Life. He had to return to it, he had to start all over again with the 
humility of a school-boy who during the summer has forgotten the Pythagorean tables [...]
%
%
%
%
%
%
Well, of course finding the formula of Life wouldn’t be so simple. To find a formula means solving a problem, to solve a problem we must enunciate it, to enunciate it we must start with a premise and … Oh! Why had he betrayed the sorceress? What had made him betray her? [...]
}
\begin{flushright}
 \textit{Oriana Fallaci: Inshallah\ifnum0\key=7 \cite{INSHALLAH}\fi. First Act, Chapter One}
\end{flushright}
\endgroup
\vspace*{\fill}

\clearpage
\ifnum0\key=7
\thispagestyle{empty}
\vspace*{\fill}
\begingroup

\textit{\large Original Italian version:}

\vspace{1cm}

\doublespace
\textit{\large [...] Ad esempio perché si trovasse qui, perché avesse scelto un mestiere che non si addiceva al suo carattere e alla sua struttura mentale 
cioè il mestiere di soldato, perché con quel mestiere avesse tradito la matematica. Quanto gli mancava la matematica, quanto la rimpiangeva! Massaggia 
le meningi come un allenatore massaggia i muscoli di un atleta, la matematica. Le irrora di pensiero puro, le lava dai sentimenti che corrompono 
l'intelligenza, le porta in serre dove crescono fiori stupendi. I fiori di un'astrazione composta di concretezza, di una fantasia composta di realtà [...] 
No, non è vero che è una scienza rigida, la matematica, una dottrina severa. \`E un'arte seducente, estrosa, una maga che può compiere mille incantesimi e 
mille prodigi. Può mettere ordine nel disordine, dare un senso alle cose prive di senso, rispondere ad ogni interrogativo. Può addirittuta fornire ciò 
che in sostanza cerchi: la formula della Vita. Doveva tornarci, ricominciare da capo con l'umiltà d'uno scolaro che nelle vacanze ha dimenticato la tavola 
pitagorica. [...]
Bè, naturalmente trovare la formula della Vita non sarebbe stato così semplice. Trovare una 
formula significa risolvere un problema, e per risolvere un problema bisogna enunciarlo, e per enunciarlo bisogna partire da un presupposto... Ah perché 
aveva tradito la maga? Che cosa lo aveva indotto a tradirla? [...]}

\vspace{1cm}
\begin{flushright}
 \textit{\large Oriana Fallaci: Insciallah \cite{INSCIALLAH}. Atto primo, Capitolo primo}
\end{flushright}

\endgroup
\vspace*{\fill}
\ifnum0\full=1
\enlargethispage{30pt}
\else
\enlargethispage{3.8cm}
\fi
\begin{flushright}
 \tiny{$\ {e^{-i\pi}+1=}$}\includegraphics[width=.015\textwidth]{EU}
\end{flushright}
\clearpage


\ifnum0\picture=1
\ifpdf

\thispagestyle{empty}
\vspace*{\fill}
\begin{center}
\includegraphics[width=0.9\textwidth]{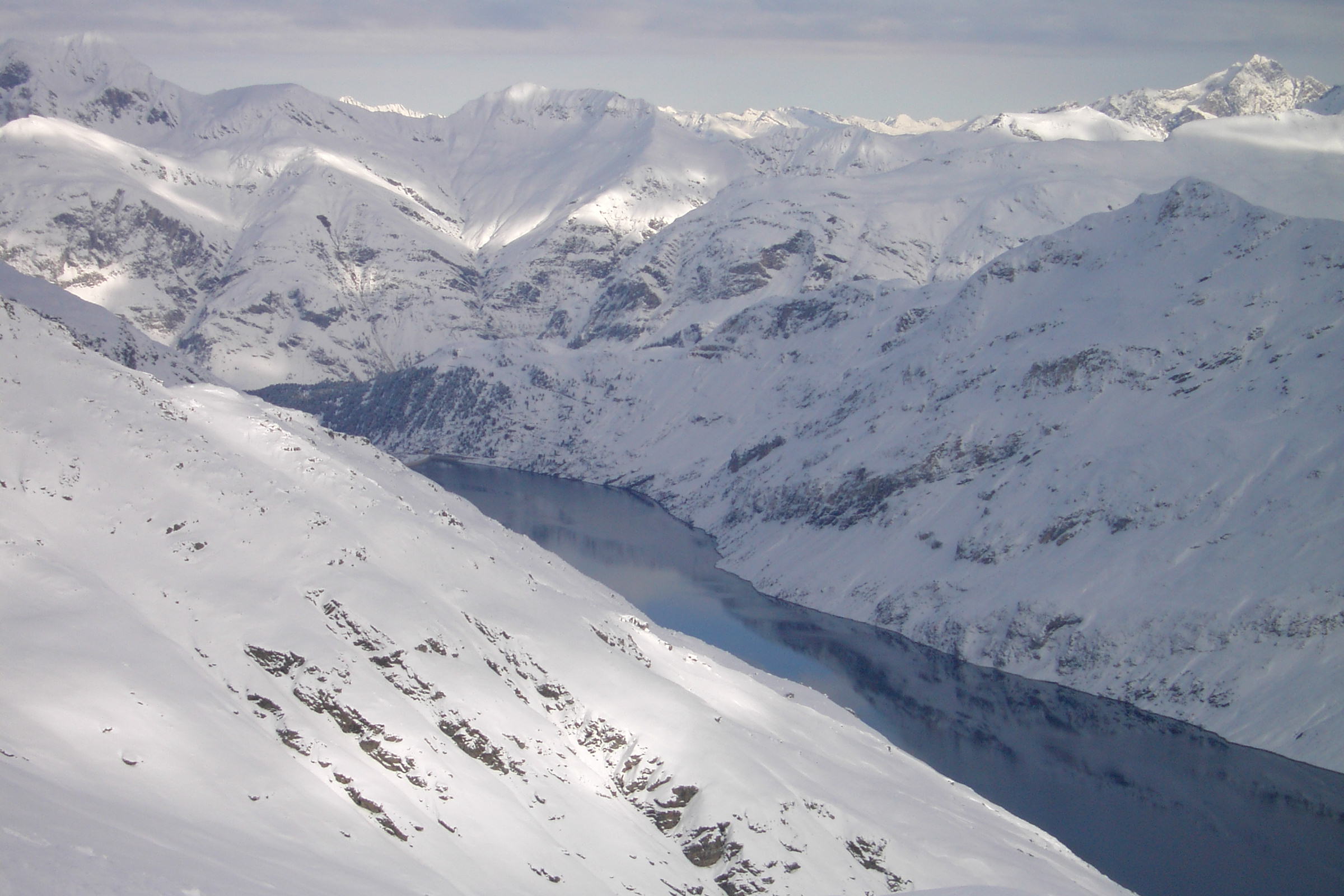}
\end{center}
\vspace*{\fill}

\else
\thispagestyle{empty}
\vspace*{\fill}
\begin{center}
\includegraphics[width=0.9\textwidth]{val_di_Lei}
\end{center}
\vspace*{\fill}
\fi 
\fi 

\else 
\fi 

\ifnum0\key=7
 \bibliographystyle{aomalpha}
  \ifnum\full=1 \bibliography{PhD_bib_short}
  \else \bibliography{PhD_bib_short}
  \fi
\else
 \bibliographystyle{alphanum}
 \bibliography{PhD_bib_short}
\fi

\ifnum0\key=7 \printindex \fi

\end{document}